\documentclass[11pt,letterpaper]{amsart}
 \usepackage{amsmath,amssymb, bbm}
 \usepackage[all,cmtip]{xy} 
  \usepackage{eucal, graphicx}
   \usepackage{caption}
   \usepackage{mathtools}
   \usepackage[new]{old-arrows}
   \usepackage{stmaryrd}
   \usepackage{pigpen}
\usepackage{mathrsfs}
\usepackage{relsize}
\usepackage{tikz,tikz-3dplot}
\usepackage{tikz-cd}
\usetikzlibrary{arrows,matrix,positioning, calc}
\usetikzlibrary{patterns,decorations.pathreplacing, decorations.markings}
\usetikzlibrary{shapes.misc}
\usepackage{pict2e,picture}
\usetikzlibrary{arrows}

 \usepackage[all]{xy}

\usepackage[plainpages=false,colorlinks,hyperindex,pdfpagemode=None,bookmarksopen,linkcolor=blue,citecolor=blue,urlcolor=blue]{hyperref}

\newtheorem{thm}{Theorem}[section]

\newtheorem{cor}[thm]{Corollary}
\newtheorem{lemma}[thm]{Lemma}

\newtheorem{prop}[thm]{Proposition}

\theoremstyle{definition}
\newtheorem{example}[thm]{Example}
\newtheorem{definition}[thm]{Definition}
\newtheorem{remark}[thm]{Remark}
\newtheorem{question}[thm]{Question}
\newtheorem{notations}[thm]{Notations}
\newtheorem{assumption}[thm]{Assumption}

\makeatletter
\def\iddots{\mathinner{\mkern1mu\raise\p@
\vbox{\kern7\p@\hbox{.}}\mkern2mu
\raise4\p@\hbox{.}\mkern2mu\raise7\p@\hbox{.}\mkern1mu}}
\makeatother

\newcommand{\cA}{\mathcal{A}}
\newcommand{\cB}{\mathcal{B}}
\newcommand{\cC}{\mathcal{C}}
\newcommand{\cD}{\mathcal{D}}

\newcommand{\cF}{\mathcal{F}}

\newcommand{\cH}{\mathcal{H}}
\newcommand{\cI}{\mathcal{I}}
\newcommand{\cJ}{\mathcal{J}}
\newcommand{\cK}{\mathcal{K}}
\newcommand{\cM}{\mathcal{M}}
\newcommand{\cN}{\mathcal{N}}
\newcommand{\cO}{\mathcal{O}}
\newcommand{\cP}{\mathcal{P}}
\newcommand{\cQ}{\mathcal{Q}}

\newcommand{\cS}{\mathcal{S}}
\newcommand{\cT}{\mathcal{T}}
\newcommand{\cU}{\mathcal{U}}
\newcommand{\cV}{\mathcal{V}}
\newcommand{\cZ}{\mathcal{Z}}

\newcommand{\OneCat}{1\text{-}\mathcal{C}\mathrm{at}}

\newcommand{\bA}{\mathbb{A}}
\newcommand{\bC}{\mathbb{C}}

\newcommand{\bP}{\mathbb{P}}
\newcommand{\bfP}{\mathbf{P}}
\newcommand{\bQ}{\mathbb{Q}}
\newcommand{\bR}{\mathbb{R}}

\newcommand{\bZ}{\mathbb{Z}}

\newcommand{\bD}{\mathbf{D}}

\newcommand{\ft}{\mathfrak{t}}
\newcommand{\fC}{\mathfrak{C}}
\newcommand{\cX}{\mathcal{X}}

\newcommand{\Mod}{\mathrm{Mod}}

\newcommand{\Hom}{\mathrm{Hom}}

\newcommand{\Res}{\mathrm{Res}}
\newcommand{\Ind}{\mathrm{Ind}}

\newcommand{\End}{\mathrm{End}}

\newcommand{\Spec}{\mathrm{Spec}}

\newcommand{\pr}{\mathrm{pr}}

\newcommand{\supp}{\mathrm{supp}}

\newcommand{\sfh}{\mathsf{h}}

\newcommand{\sfF}{\mathsf{F}}
\newcommand{\sfb}{\mathsf{b}}

\newcommand{\sing}{\mathrm{sing}}

\numberwithin{equation}{subsection}

\newcommand{\Loc}{\mathrm{Loc}}
\newcommand{\bk}{\mathbf{k}}

\newcommand{\Shv}{\mathrm{Shv}}

\newcommand{\Cone}{\mathrm{Cone}}
\newcommand{\proj}{\mathrm{proj}}

\newcommand{\reg}{\mathrm{reg}}

\newcommand{\frj}{\mathfrak{j}}

\newcommand{\bv}{\mathbf{v}}

\newcommand{\cY}{\mathcal{Y}}
\newcommand{\cW}{\mathcal{W}}

\newcommand{\std}{\mathrm{std}}
\newcommand{\cL}{\mathcal{L}}

\newcommand{\Ad}{\mathrm{Ad}}
\newcommand{\fg}{\mathfrak{g}}
\newcommand{\fb}{\mathfrak{b}}
\newcommand{\fn}{\mathfrak{n}}
\newcommand{\fU}{\mathfrak{U}}
\newcommand{\fc}{\mathfrak{c}}
\newcommand{\fp}{\mathfrak{p}}
\newcommand{\fl}{\mathfrak{l}}
\newcommand{\fsl}{\mathfrak{sl}}
\newcommand{\fz}{\mathfrak{z}}

\newcommand{\fM}{\mathfrak{M}}

\newcommand{\Core}{\mathrm{Core}}

\newcommand{\Coh}{\mathrm{Coh}}

\newcommand{\der}{\mathrm{der}}
\newcommand{\ad}{\mathrm{ad}}

\newcommand{\lng}{\langle}
\newcommand{\rng}{\rangle}

\newcommand{\fund}{\mathsf{fund}}

\newcommand{\sfLambda}{\mathsf{\Lambda}}

\newcommand{\fD}{\mathfrak{D}}

\newcommand{\norm}{\mathsf{N}}

\newcommand{\fq}{\mathsf{q}}
\newcommand{\cpt}{\mathsf{cpt}}

\newcommand{\fS}{\mathfrak{S}}
\newcommand{\fF}{\mathfrak{F}}

\newcommand{\fT}{\mathfrak{T}}

\newcommand{\fK}{\mathfrak{K}}

\newcommand{\fQ}{\mathfrak{Q}}

\newcommand{\Perf}{\mathrm{Perf}}

\newcommand{\sfq}{\mathsf{q}}
\newcommand{\scrZ}{\mathscr{Z}}

\newcommand{\dagg}{\dagger}
\newcommand{\ff}{\mathfrak{f}}

\newcommand{\po}{\ar@{}[dr]|{\text{\pigpenfont R}}}
\newcommand{\pb}{\ar@{}[dr]|{\text{\pigpenfont J}}}

\newcommand{\sfB}{\mathsf{B}}
\newcommand{\sfN}{\mathsf{N}}

\newcommand{\scrO}{\mathscr{O}}

\newcommand{\DMStk}{\mathrm{DMStk}}

\tikzset{cross/.style={cross out, draw=black, minimum size=2*(#1-\pgflinewidth), inner sep=0pt, outer sep=0pt},
cross/.default={2pt}}

\tikzset{
  saveuse path/.code 2 args={
    \pgfkeysalso{#1/.style={insert path={#2}}}
    \global\expandafter\let\csname pgfk@\pgfkeyscurrentpath/.@cmd\expandafter\endcsname
                                \csname pgfk@\pgfkeyscurrentpath/.@cmd\endcsname
    \pgfkeysalso{#1}},
  /pgf/math set seed/.code=\pgfmathsetseed{#1}}

\makeatletter
\newcommand{\pnrelbar}{
  \linethickness{\dimen2}
  \sbox\z@{$\m@th\prec$}
  \dimen@=1.1\ht\z@
  \begin{picture}(\dimen@,.4ex)
  \roundcap
  \put(0,.2ex){\line(1,0){\dimen@}}
  \put(\dimexpr 0.5\dimen@-.2ex\relax,0){\line(1,1){.4ex}}
  \end{picture}
}

\newcommand{\precneq}{\mathrel{\vcenter{\hbox{\text{\prec@neq}}}}}
\newcommand{\prec@neq}{
  \dimen2=\f@size\dimexpr.04pt\relax
  \oalign{
    \noalign{\kern\dimexpr.2ex-.5\dimen2\relax}
    $\m@th\prec$\cr
    \noalign{\kern-.5\dimen2}
    \hidewidth\pnrelbar\hidewidth\cr
  }
}
\makeatother

\begin{document}

\title[HMS for the universal centralizers]{Homological Mirror Symmetry for the universal centralizers }
\author{Xin Jin}
\address{Math Department, Boston College, Chestnut Hill, MA 02467.}
\email{xin.jin@bc.edu}

\maketitle

\begin{abstract}
We prove homological mirror symmetry for the universal centralizer $J_G$ (a.k.a the Toda space), associated to any complex reductive Lie group $G$. The A-side is a partially wrapped Fukaya
category on $J_G$, and the B-side is the category of coherent sheaves
on the categorical quotient of a dual maximal torus by the Weyl group
action (with some modification if the center of $G$ is not connected). 
\end{abstract}

\tableofcontents

\section{Introduction}\label{sec: introduction}

 \subsection{Background and main results}

For a (connected) complex reductive Lie group $G$, one can define a holomorphic symplectic variety $J_G$, called the \emph{universal centralizer} or the \emph{Toda space} (cf. \cite{Lusztig}\footnote{It was first introduced  in \cite{Lusztig} as $\cN_G$ (in the last paragraph) in the group-group setting, i.e. one considers centralizing pairs both coming from the group $G$ (or with one of the elements from a different group in the same isogeny class).}, \cite{Kostant2}, \cite{BFM}, \cite{Ginzburg}), which has the structure of a (holomorphic) complete integrable system over $\fc=\ft^*\sslash W$, where $\ft$ is a Cartan subalgebra of the Lie algebra $\fg$ of $G$, and $W$ is the Weyl group associated to the root system. Roughly speaking, one can build $J_G$ from an affine blowup of $T^*T$, where $T$ is a maximal torus, along the diagonal walls associated to the root data, and then take the orbit space of $W$.

There are many remarkable features of $J_G$, and here we list a couple of them. First, one has a canonical map 
\begin{align*}
\chi: J_G\rightarrow \fc=\ft^*\sslash W
\end{align*}
that exhibits $J_G$ as an abelian group scheme over $\fc$, and also a (holomorphic) complete integrable system. The fiber over any point in $\fc$, represented by a regular element $\xi$ in the Kostant slice $\cS\subset \fg^*$, is isomorphic to the centralizer of $\xi$ in $G$. In particular, the generic fiber is isomorphic to a maximal torus in $G$. Second, the ring of functions on $J_G$ (which defines $J_G$ as an affine variety) is identified with the $G^\vee(\cO)$-equivariant homology of the affine Grassmannian $Gr_{G^\vee}=G^\vee(\cK)/G^\vee(\cO)$ of the Langlands dual group $G^\vee$, where $\cK=\bC((z)), \cO=\bC[[z]]$. This is one of the main results in \cite{BFM} and it has led to interesting connections to various aspects of the geometric Langlands program. 

The integrable system structure on $J_G$ can be viewed as a non-abelian version of the familiar integrable system $T^*T\rightarrow\ft^*$, which is the most basic example of homological mirror symmetry (abbreviated as HMS below). Recall the HMS statement for $T^*T$ as the following. Let $T^\vee$ be the dual torus. Let $\cW(T^*T)$ denote for the partially wrapped Fukaya category of $T^*T$ (after taking twisted complexes), and let $\Coh(T^\vee)$ be the category of coherent sheaves on $T^\vee$. 
\begin{thm}[Well known]
There is an equivalence of categories
\begin{align*}
\cW(T^*T)\simeq \Coh(T^\vee).
\end{align*}
\end{thm}

We remark on the definition of $\cW(T^*T)$. Since $T$ is a non-compact manifold, one needs to specify the allowed wrapping Hamiltonians in the definition of the (partially) wrapped Fukaya category. Here we follow the recent work of \cite{GPS1}, \cite{GPS2} that gives a precise definition of (partially) wrapped Fukaya categories on Liouville sectors (also see \emph{loc. cit.} for previous work in this line). Roughly speaking, a Liouville sector is a class of Liouville manifolds $M$ with boundaries, that is in addition to the contact-type $\infty$-boundary $\partial^\infty M$ that a usual Liouville manifold has, it has a ``finite" non-contact-type boundary $\partial M$. The Lagrangian objects in the wrapped Fukaya category should have ends contained in $\partial^\infty M$. Any wrapping should take place on $\partial^\infty M$ as usual, but stops near $\partial M$ (the ``finite" boundary). In particular, for any non-compact manifold $X$, take a compactification $\overline{X}$ with smooth boundary (of codimension 1), then $T^*\overline{X}$ is a Liouville sector with finite boundary given by the union of cotangent fibers over $\partial X$. 

To simplify notations, we usually denote a Liouville sector by its interior, when the compactification has been introduced. So $\cW(T^*T)$ means the wrapped Fukaya category for the Liouville sector $T^*\overline{T}$, for a standard compactification of $T$, i.e. a maximal compact subtorus times a compact ball.

One of our results is that $J_G$ (together with a canonical Liouville 1-form) can be naturally partially compactified to be a Liouville sector, so that one has a well defined $\cW(J_G)$ as introduced above. 

\begin{prop}[cf. Proposition \ref{prop: partial compactify} and Remark \ref{remark: positive Ham} (iii)]
There are natural partial compactifications $\overline{J}_G$ of $J_G$ as Liouville sectors, all yielding canonically equivalent wrapped Fukaya categories. Moreover, there is an abundance of such compactifications making $\overline{J}_G$  a Weinstein sector. 
\end{prop}

The first main result of the paper is the following HMS statement for $J_G$, when $G$ is of adjoint type.

\begin{thm}[cf. Theorem \ref{thm: sec G adjoint}]\label{thm: J_G, adjoint}
For any complex semisimple Lie group $G$ of adjoint type (i.e. the center of $G$ is trivial), we have an equivalence of (pre-triangulated dg) categories
\begin{align}\label{eq: thm J_G, adjoint}
\cW(J_G) \simeq \Coh(T^\vee\sslash W). 
\end{align} 
\end{thm}

There is a more general statement for any complex reductive group $G$, but to state that we need to introduce some notations. For any such $G$, let $G^\der=[G,G]$ be the derived group of $G$, and let $\cZ(G^\der)^*$ be the Pontryagin dual of $\cZ(G^\der)$, the center of $G^\der$. 
Then $\cZ(G^\der)^*$ naturally acts on $(T/\cZ(G^\der))^\vee$ and on $(T/\cZ(G^\der))^\vee\sslash W$. 

If $G$ is semisimple, let $G^\vee_{sc}$ (resp. $G_{\ad}$) denote for the simply connected (resp. adjoint) form of $G^\vee$ (resp. $G$), i.e. the universal cover of $G^\vee$ (resp. $G/\cZ(G)$). Let $T^\vee_{sc}$ (resp. $T_{\ad}$) denote for a maximal torus of $G^\vee_{sc}$ (resp. $G_{ad}$). Then there is a canonical isomorphism $\cZ(G)^*\cong \pi_1(G^\vee)$ that acts naturally on $T^\vee_{sc}$ and on $T^\vee_{sc}\sslash W$. 

We have the following HMS result for a general reductive $G$. 

\begin{thm}[cf. Theorem \ref{thm: HMS for reductive}]\label{thm: J_G, general}
For any complex reductive group $G$, we have an equivalence of categories
\begin{align}\label{thm: eq mirror J_G}
\cW(J_G) \simeq \Coh((T/\cZ(G^\der))^\vee\sslash W)^{\cZ(G^\der)^*},
\end{align} 
where the category on the right-hand-side is the category of $\cZ(G^\der)^*$-equivariant coherent sheaves on $(T/\cZ(G^\der))^\vee\sslash W$. If $\cZ(G)$ is connected, then 
\begin{align*}
\cW(J_G) \simeq \Coh(T^\vee\sslash W). 
\end{align*}
\end{thm}

If $G$ is semisimple, then the theorem says 
\begin{align*}
\cW(J_G) \simeq \Coh(T_{sc}^\vee\sslash W)^{\pi_1(G^\vee)}. 
\end{align*}
In this case, the functor from the $A$-side $\cW(J_G)$ to the $B$-side $\Coh(T^\vee_{sc}\sslash W)^{\pi_1(G^\vee)}$ in (\ref{thm: eq mirror J_G}) on the object level can be described quite explicitly. The integrable system $J_G\rightarrow \fc$ has a collection of sections, called the Kostant sections, indexed by the center elements of $G$. These turn out to be a set of generators of the wrapped Fukaya category. On the other hand, the $\pi_1(G^\vee)$-equivariant coherent sheaves on $T^\vee_{sc}\sslash W$ (which can be identified with the affine space of dimension $n=\text{rank}(G)$) is generated by a collection of equivariant sheaves which come from putting different equivariant structures, indexed by $ \pi_1(G^\vee)^*$, on the structure sheaf $\cO_{T^\vee_{sc}\sslash W}$. The mirror functor matches these two collections of generators through the canonical isomorphism $\cZ(G)\cong \pi_1(G^\vee)^*$.

\subsection{Example of $G=SL_2(\bC)$ and idea of proof}

In this section, we illustrate some of the key geometric features of $J_G$ through the example of $G=SL_2(\bC)$, and we will give some sketch of the proof for Theorem \ref{thm: J_G, adjoint} in the adjoint type case. The general case Theorem \ref{thm: J_G, general} can be deduced from Theorem \ref{thm: J_G, adjoint} by the monadicity of a natural functor 
\begin{align*}
\cW(J_G)\rightarrow \cW(J_{G_\ad}\times T^*\big(\cZ(G)/\cZ(G^\der)\big))\simeq  \cW(J_{G_\ad})\otimes \cW(T^*\big(\cZ(G)/\cZ(G^\der)\big)), 
\end{align*}
where the latter equivalence is from the Kunneth formula in \cite{GPS2}, that is mirror to the pullback (i.e. forgetful) functor 
\begin{align*}
\Coh((T/\cZ(G^\der))^\vee\sslash W)^{\cZ(G^\der)^*}\longrightarrow \Coh((T/\cZ(G^\der))^\vee\sslash W).
\end{align*} 

\subsubsection{Example of $G=SL_2(\bC)$}

For $G=SL_2(\bC)$, the base of the integrable system $\fc=\ft^*\sslash W$ is identified with $\bA^1$, coming from taking the determinant of any traceless $2\times 2$-matrix. For any generic point $a\in \bA^1\backslash\{0\}$, we can represent it by the diagonal matrix $\mathrm{diag}[a, -a]$ (or any element in its conjugacy class), and the fiber over $a$ can be identified with its centralizer, the standard maximal torus $T$ (diagonal $2\times 2$-matrices with determinant 1). For the point $0\in \bA^1$, it should be represented by the (conjugacy class of) nilpotent matrix $\begin{bmatrix}0&0\\
1&0
\end{bmatrix}$, and the fiber over it can be identified with its centralizer in $G$, consisting of matrices of the form 
\begin{align*}
\begin{bmatrix}1&0\\
*&1
\end{bmatrix}, \begin{bmatrix}-1&0\\
*&-1
\end{bmatrix}, \text{ where }*\text{ can be any complex number}.
\end{align*}
In particular, the central fiber is a disjoint union of two affine lines. There is a canonical $\bC^\times$-action on $J_G$, whose flow lines are indicated in Figure \ref{figure: J_SL_2}. The corresponding $\bR_+$-action (after taking square root) is the flow of a Liouville vector field. 

\begin{figure}[htbp]
\centering
\includegraphics[width=3in]{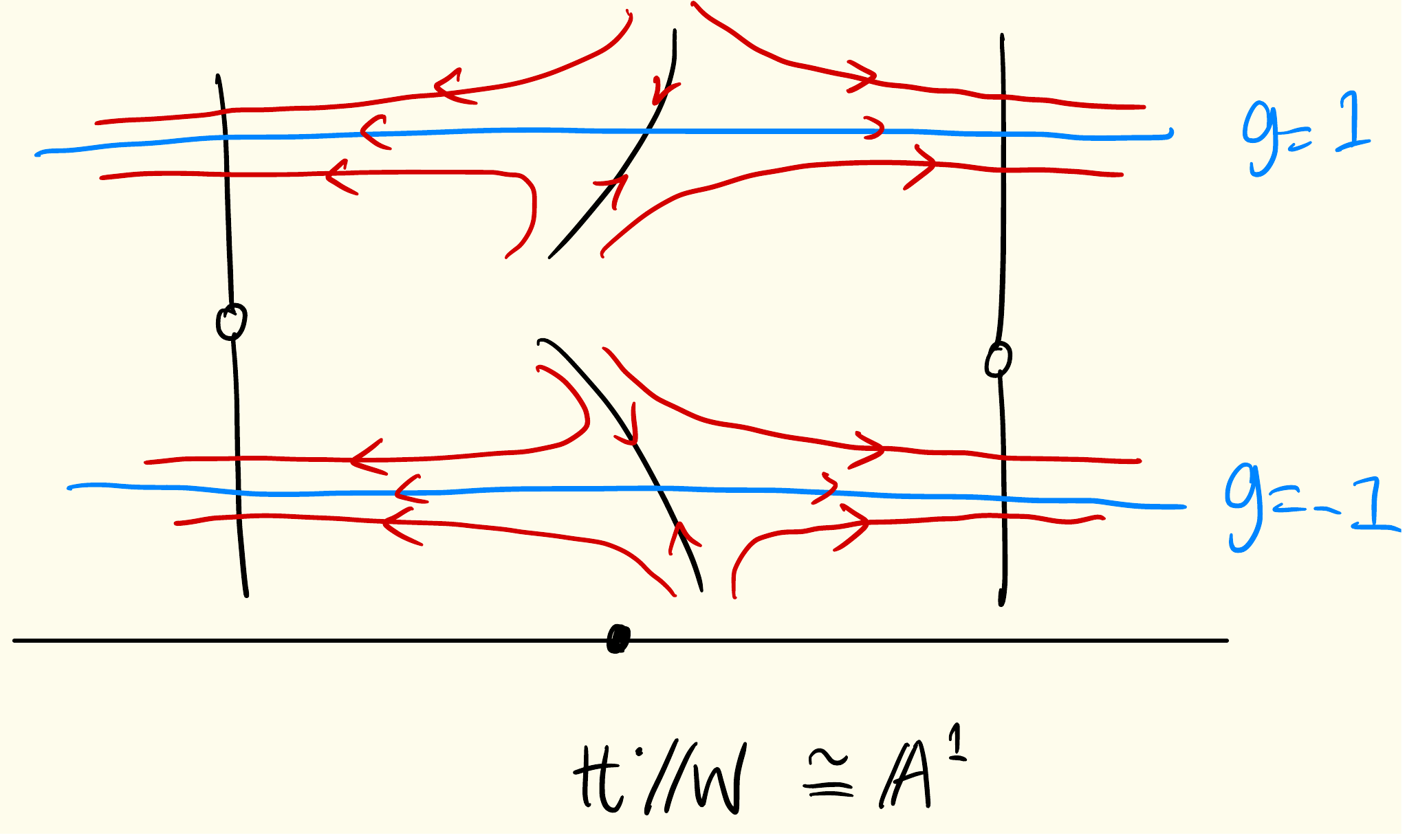} 
\caption{A picture of $J_{SL_2(\bC)}\rightarrow \fc\cong \bA^1$}\label{figure: J_SL_2}
\end{figure}

There are two horizontal sections of $\chi: J_{SL_2}\rightarrow \fc$, corresponding to the union of $g=\pm I$ in each fiber (recall each fiber is a centralizer and in particular a group). These are the Kostant sections. Away from the Kostant sections, there is an interesting symplectic identification 
\begin{align*}
J_{SL_2}- \{g=\pm I\}\cong T^*T,
\end{align*} 
which is \emph{not} obvious from the above picture (Figure \ref{figure: J_SL_2}). 
Using this, one can build $J_{SL_2}$ from a handle attachment by attaching two critical handles (a handle is called \emph{critical} if the core has the dimension of a Lagrangian), each has core a connected component of the central fiber, to $T^*T$\footnote{Here $T^*T$ is equipped with a different Liouville 1-form than the standard one. In particular, $J_{SL_2}$ as a Liouville sector is \emph{not} from attaching handles to the sector $T^*S^1\times T^*[0,1]$. In fact, the latter is replaced by $T^*S^1\times T^*(0,1]\cong T^*S^1\times \bC_{\Re z\leq 0}$.}. Then the Konstant sections become the ``linking discs" (i.e. normal slices to the cores). Furthermore, one can endow $J_{SL_2}$ with a Weinstein sector structure (in the sense of \cite{GPS1}), and obtains an arborealized Lagrangian skeleton in the sense of \cite{Nadler},  as follows (Figure \ref{figure: core}). Here we have two Lagrangian caps attached to a semi-infinite annulus $S^1\times [1, \infty)$ along two circles intersecting in an interesting way\footnote{Regarding microlocal sheaves on the Lagrangian skeleton, they should vanish near $S^1\times \{1\}$.}.  It is also easy to get the skeleton of $J_{PGL_2}$ by quotienting out the natural $\bZ/2\bZ$ symmetry in the picture, resulting in one Lagrangian cap attached to an annulus along an immersed circle wrapping around the puncture twice with one self-crossing. \\

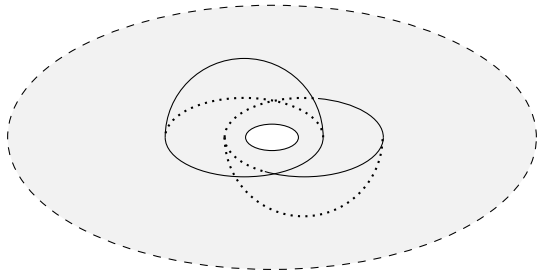
\begin{figure}[h]
\centering
 \begin{tikzpicture}
\path [draw=none,fill=gray, fill opacity = 0.1,even odd rule] (4.37,1.5) ellipse (10pt and 5pt) (4.37, 1.5) ellipse (100pt and 50pt);
   \draw[thick, dotted] (5.05,1.5)  arc[x radius = 10.5mm, y radius=5.25mm, start angle= 0, end angle= 180];
   \draw (5.05,1.5)  arc[x radius = 10.5mm, y radius=5.25mm, start angle= 360, end angle= 180];
   \draw (5.05,1.5)  arc[radius = 10.5mm, start angle= 0, end angle= 180];
    \draw[thick, dotted] (3.74,1.5)  arc[x radius = 10.5mm, y radius = 5.25mm, start angle= 180, end angle= 80];  
     \draw[thick, dotted] (3.74,1.5)  arc[x radius = 10.5mm, y radius = 5.25mm, start angle= 180, end angle= 240]; 
      \draw(5.85,1.5)  arc[x radius = 10.5mm, y radius = 5.25mm, start angle= 0, end angle= 80]; 
      \draw(5.85,1.5)  arc[x radius = 10.5mm, y radius = 5.25mm, start angle= 360, end angle= 240]; 
       \draw[thick, dotted] (3.74,1.5)  arc[radius = 10.5mm, start angle= 180, end angle= 360];  
      \draw (4.37,1.5) ellipse (10pt and 5pt);
      \draw[dashed]  (4.37, 1.5) ellipse (100pt and 50pt);
           \end{tikzpicture}
           \caption{Picture of an arborealized Lagrangian skeleton for $J_{SL_2(\bC)}$.}\label{figure: core}
\end{figure}

\subsubsection{Idea of proof of Theorem \ref{thm: J_G, adjoint}}\label{subsubsec: idea of proof}

First, for any semisimple Lie group $G$,  we prove that $J_G$ admits a Bruhat decomposition\footnote{During the preparation of the paper, the author learned that similar features have been observed in \cite{Teleman}.} indexed by subsets $S\subset \Pi$ of the set of simple roots $\Pi$ of $G$ (associated to a fixed principal $\fsl_2$-triple), based on an equivalent definition of $J_G$ as a Whittaker type Hamiltonian reduction. This roughly induces a Weinstein handle decomposition. For $G=SL_2(\bC)$, $\Pi$ has exactly one element, and we have $S=\Pi$ corresponding to the Kostant sections $\{g=\pm I\}$, and $S=\emptyset$ corresponding to the complement, which is isomorphic to $T^*\bC^\times$. For a general $G$, $S=\Pi$ always gives the Kostant section(s) and $S=\emptyset$ always gives $T^*T$ (but the Liouville form is somewhat different from the standard one). 

Second, we give natural partial compactifications of $J_G$ as Liouville sectors, among which there are (generalized) Weinstein sectors. With appropriate choices of such Weinstein sectors, we obtain a skeleton of $J_G$ as for the case $G=SL_2(\bC)$, with each Bruhat ``cell" contributing one component of the skeleton.  We further show that the cocores to some of the critical handles, which are the Kostant sections, generate the partially wrapped Fukaya category of $J_G$ (using general results from \cite{GPS1, GPS2, CDGG}).

Third, assuming $G$ is of adjoint type, the only Kostant section $\Sigma_I:=\{g=I\}$ generates $\cW(J_G)$. So to prove the HMS result (\ref{eq: thm J_G, adjoint}), we just need to compute $\End(\Sigma_I)$. The first step is to define appropriate wrapping Hamiltonians on $J_G$, so that $\End(\Sigma_I)$ matches with $\bC[T^\vee\sslash W]$ as vector spaces. The second step, which is the main step, is to use the funtoriality of inclusions of Weinstein sectors (plus other geometric information) to show the two rings are isomorphic. This step is somewhat indirect. The rough idea is that the Bruhat ``cell" corresponding to $S=\emptyset$, denoted by $\cB_{w_0}$, gives a sector inclusion $\cB_{w_0}^\dagg\cong T^*\overline{T}\hookrightarrow \overline{J}_G$ for a subsector $\cB_{w_0}^\dagg\subset \cB_{w_0}$ (see Subsection \ref{subsubsec: conic} for the precise formulation)\footnote{We remark that there is another adjoint pair for stop/handle removal, which is trivial because $\cW(\cB_{w_0})\simeq 0$.  
}, which induces an adjoint pair of functors between the ind-completion of $\cW(T^*T)$ and the ind-completion of $\cW(J_G)$ (cf. \cite{GPS1}; in the current case the adjoint pair is actually well defined between the wrapped Fukaya categories). For example, for the Lagrangian skeleton Figure \ref{figure: core}, the adjoint functors correspond to restriction and co-restriction between (wrapped) microlocal sheaves on the whole skeleton and local systems on the outer annulus which is disjoint from the attaching caps. Under mirror symmetry, this corresponds to the pushforward and pullback functors between $\Coh(T^\vee)$ and $\Coh(T^\vee\sslash W)$ along the projection $T^\vee\rightarrow T^\vee\sslash W$. Noting that the skyscraper sheaves on $T^\vee$ are mirror to conormal bundles $L_0$ of the maximal compact subtorus $T_{\cpt}\subset T$, equipped with a rank $1$ local system $\check{\rho}\in \Hom(\pi_1(T), \bC^\times)\cong T^\vee$, our approach is based on Floer calculations involving these conormal bundles and the Kostant section $\Sigma_I$. One of the key facts that we establish can be summarized as follows:

\begin{prop}[cf. Proposition \ref{prop: L_0, part 2} and \ref{prop: L_0, W} for the precise statement]
Under the natural functor $co\text{-}res: \cW(T^*T)\rightarrow \cW(J_G)$, the objects $(L_0, \check{\rho})$ are sent to ``skyscraper objects", i.e. their morphism spaces with $\Sigma_I$ are of rank $1$. Moreover, their images are $W$-invariant in the sense that $co\text{-}res(L_0, \check{\rho})\cong co\text{-}res(L_0, w(\check{\rho}))$ for all $w\in W$. 
\end{prop}

We also prove a non-exact version (though not logically needed for the proof of the main theorem) which is more intuitive from SYZ mirror symmetry perspective, and whose proof is relatively easier. For this, we consider generic \emph{shifted} conormal bundles of $T_{\cpt}$ and we work over the Novikov field $\sfLambda$. 

\begin{prop}[cf. Proposition \ref{prop: L_xi, S_e, part 1} for the precise statement]\label{prop: intro skyscraper}
Under the natural functor $\cW(T^*T;\sfLambda)\rightarrow \cW(J_G;\sfLambda)$, the (generic) shifted conormal bundles of $T_{\cpt}$ give ``skyscraper objects", i.e. their morphism spaces with $\Sigma_I$ are of rank 1. Moreover, their images are $W$-invariant under the natural $W$-action on $T^*T$. 
\end{prop}

We give a heuristic explanation why Proposition \ref{prop: intro skyscraper} holds. The integrable system $\chi: J_G\rightarrow \fc$ suggests that the ``skyscraper objects" in $\cW(J_G)$ are the fibers\footnote{We note that these fibers are not well defined objects in $\cW(J_G)$, because their boundaries are inside the ``finite" boundary of $J_G$.}, which follows from basic principles in SYZ mirror symmetry. The shifted conormal bundles of $T_{\cpt}$ can be thought as modeled on the generic torus fibers of $\chi$, with each $W$-orbit of shifted conormal bundles modeled on the same fiber. This reflects some intriguing geometric relations between a generic torus fiber of the integrable system and the base manifold $T$ in $\cB_{w_0}\cong T^*T$: while the generic shifted conormal bundles of $T_{\cpt}$ in a  $W$-orbit do \emph{not} talk to each other in $\cB_{w_0}$, they become ``close to'' Hamiltonian isotopic in $J_G$ and the bridge is given by the common torus fiber that they are modeled on (note that $W$ does \emph{not} act on $J_G$).

We make a couple of more remarks. First, there is a clear restriction and induction pattern among standard Levi subgroups (as in a related way expected in \cite{Teleman}) in terms of restriction and co-restriction functors between wrapped Fukaya categories for inclusions of the corresponding subsectors (and equivalently on microlocal sheaf categories). We use this in the proof of the main theorem and elaborate it more in Subsection \ref{subsec: induction}. Second, it is tempting to try to prove the HMS result by replacing $\cW(J_G)$ with $\mu\Shv^w(J_G)$, the wrapped microlocal sheaf category (cf. \cite{Nadler2, NaSh}) for the Lagrangian skeleton of $J_G$. However, due to the complicatedness of the singularities of the Lagrangian skeleton, the author does not know an effective way to directly compute the sheaf category in high dimensions. 

\subsection{Related works and future directions}

The main theorem (Theorem \ref{thm: J_G, adjoint}) can be viewed as an ``analytic" version of a theorem of Lonergan \cite{Lonergan} and Ginzburg \cite{Ginzburg} on the description of the category of bi-Whittaker $D$-modules (see \emph{loc. cit.} for the precise statement)
\begin{align}\label{eq: D-module}
D\text{-}mod(N\overset{\psi}{\backslash}G\overset{\psi}{/}N)\simeq \text{QCoh}(``\ft^*\sslash W_{\text{aff}}"),
\end{align}
where the generic Lie algebra character $\psi:\fn\rightarrow \bC$ of the maximal unipotent subgroup $N$ is the same as the $f$ in Subsection \ref{subsec: def of J_G} that realizes $J_G$ as a bi-Whittaker Hamiltonian reduction of $T^*G$, and $``\ft^*\sslash W_{\text{aff}}"$ is some coarse quotient $``(\ft^*/\Lambda)\sslash W"$ with $\Lambda$ the weight lattice of $T$ which is identified with the coweight lattice of $T^\vee$ (see also \cite{BZG}). Heuristically, if we replace the left-hand-side of (\ref{eq: D-module}) by the partially wrapped Fukaya category of $J_G$, and think of $``(\ft^*/\Lambda)\sslash W"$ analytically as $T^\vee\sslash W$ (and replace $\text{QCoh}$ by $\text{Coh}$), then this is exactly the equivalence of categories in the main theorem. However, there is no direct link between these two versions. 

As explained in \cite{BZG}, the result (\ref{eq: D-module}) is important for understanding module categories over the finite Hecke category $\widehat{\cH}_G$ of bimonodromic sheaves on $N\backslash G/N$, which is of particular interest in geometric representation theory. For example, in Betti Geometric Langlands program of Ben-Zvi and Nadler \cite{BZN2}, one studies sheaves with nilpotent singular support on the moduli of $G$-bundles on a curve $X$ with $N$-reductions on a finite set $S\subset X$. At each $s\in S$, there is an affine Hecke action and in particular an $\widehat{\cH}_G$-action. The $\widehat{\cH}_G$ module categories form the character field theory developed in \cite{BZN1,BZGN} that  assigns to a point a family of 3d topological field theories over $``\ft^*\sslash W_{\text{aff}}"$, thanks to the Ng$\hat{\text{o}}$-action of the bi-Whittaker category (cf. \cite{BZG}). In the Betti version, the natural action of $\Coh(T^\vee\sslash W)$ on the family of theories should correspond to the convolution action of $\cW(J_G)$. For example, using our theorem, the skyscraper sheaves on $T^\vee\sslash W$ in the B-model would give certain objects in the category of character sheaves (the assignment of the field theory to $S^1$) that act on it by convolution. The de Rham version of this has been studied in \cite{Chen}. We would like to investigate this aspect and its various applications in future work, e.g. along the line of the conjectural picture \cite[Remark 2.7]{BZG} and \cite{Teleman}.

As the symmetic monoidal structure on $\cW(J_G)$ (a consequence of the main theorem) plays an essential role in the above approach to categorical representation theory, we note that it is also expected to come naturally from the (abelian) group scheme structure on $J_G$ (cf. \cite{Pascaleff} for some developments in this direction). Roughly speaking, one can represent the functor for the monoidal structure $\cW(J_G)\otimes \cW(J_G)\rightarrow \cW(J_G)$ as a (smooth) Lagrangian correspondence $L_{\text{mon}}$ in $J_G^a\times J_G^a\times J_G$ (where the superscript $a$ means taking the opposite symplectic form). The main technical difficulty is caused by the ``finite" boundary of $J_G$. Namely, $L_{\text{mon}}$ will touch the ``finite" boundary of the product sector making it \emph{not} a well defined object in the wrapped Fukaya category. Alternatively, one can use microlocal sheaf theory on the Lagrangian skeleton, but we don't know how to realize this by a ``geometric" correspondence without appealing to the main theorem. 
We defer the study for a future work. Further desired results along this line would be to show that the restriction functors for sector inclusions are naturally symmetric monoidal, and there are natural compatibilities between compositions of restrictions as symmetric monoidal functors.  

Lastly, we would like to point out that the universal centralizers $J_G$ constitute an important class of the Coulomb branches mathematically defined in \cite{BFN}. It would be interesting to extend the present work to some other Coulomb branches whose HMS is currently unknown.

\subsection{Organization}
The organization of the paper goes as follows. In Section \ref{sec: Bruhat}, we review the definition(s) of $J_G$, and prove the Bruhat decomposition result. We give explicit descriptions of all the Bruhat ``cells" and some important symplectic subvarieties (associated to standard Levi subgroups) built from them. In Section \ref{sec: skeleton, sector}, we give the construction of partial compactifications of $J_G$ that are naturally Liouville sectors (with canonically equivalent wrapped Fukaya categories), and we present (easy) choices that make some of them Weinstein sectors. We describe the skeleton of a resulting Weinstein sector, and show that the Kostant sections generate $\cW(J_G)$. The discussions in Subsection \ref{subsubsec: H, sm} are quite technical. For this reason we would suggest the reader to skim through it and return to it later.
In Section \ref{sec: wrapping Ham}, we define certain positive linear Hamiltonians on $J_G$, so we have a convenient calculation of $\End(\Sigma_I)$ (and morphisms between different Kostant sections for a semisimple $G$), as a (graded) vector space. The upshot is that all intersection points are concentrated in degree 0, so $\End(\Sigma_I)$ is an ordinary algebra. In Section \ref{sec: HMS adjoint}, we first state the main theorem for $G$   of adjoint form and the key propositions that lead to its proof, then we develop some analysis in Subsection \ref{subsec: analysis cB_w0}-\ref{subsec: construct L_zeta} that are crucial for the proof of the key propositions. These subsections contain important geometric features of $J_G$, which in particular explain the intriguing picture behind Proposition \ref{prop: intro skyscraper}. We give the proof of the key propositions in Section \ref{subsec: proof prop}. Lastly, in Section \ref{sec: HMS reductive}, we prove the HMS result in the reductive case, and give a precise formulation of restriction and induction functors for sector inclusions associated with inclusions of Levi subgroups.

\subsection{Acknowledgement}
I would like to thank Harrison Chen, Sam Gunningham, Justin Hilburn, Oleg Lazarev, George Lusztig, David Nadler, John Pardon, Paul Seidel, Changjian Su, Dima Tamarkin and Zhiwei Yun for stimulating conversations at various stages of this project. I am grateful to David Nadler for valuable feedback on this work, and to Dima Tamarkin for help with proof of Lemma \ref{lemma: countable, 0}. I am also grateful to the anonymous referee for very helpful comments and suggestions. The author was partially supported by an NSF grant DMS-1854232.

\section{Definition(s) of $J_G$ and the Bruhat decomposition}\label{sec: Bruhat}

\subsection{Definition(s) of $J_G$ and a Lagrangian correspondence}\label{subsec: def of J_G}

In this subsection, we review some equivalent definitions of $J_G$ and a canonical Lagrangian correspondence, which will be used in later sections. The exposition is roughly following \cite[Section 2]{Ginzburg}, and we refer the reader to \emph{loc. cit.} for further details. We will focus on the semisimple case, since for a reductive group $G$, we have $J_G=J_{G^\der}\times_{\cZ(G^\der)}T^*\cZ(G)$.  

Let $G$ (resp. $\fg$) be any complex semisimple Lie group (resp. its Lie algebra). Let $\fg^{\reg}$ (resp. $\fg^{*,\reg}$) be the (Zariski open dense) subset of regular elements in $\fg$ (resp. $\fg^*$), i.e. the elements whose stabilizer with respect to the adjoint (resp. coadjoint) action by $G$ has dimension equal to $n:=\text{rank} G$ (which is the minimal possible dimension). To simplify notations, we often identify $\fg^*$ with $\fg$ using the Killing form unless otherwise specified, hence their regular elements. Let $\fc:=\fg\sslash G$ be the adjoint quotient of $\fg$. Fix any principal $\fsl_2$-triple $(e,f,h)$, and let $\cS:=f+\ker \ad_e\subset \fg^{\reg}$ be the Kostant slice. The Kostant slice gives a section of the adjoint quotient map $\fg\longrightarrow \fc$ (and its restriction to $\fg^{\reg}$), by a theorem of Kostant \cite{Kostant}. 

Let $T^{*,\reg}G\subset T^*G\cong G\times \fg$ (identified using left translations) be the regular part of the cotangent bundle of $G$, consisting of pairs $(g, \xi)\in G\times \fg^{\reg}$. Consider the locus in $T^{*,\reg}G$ 
defined by 
\begin{align}\label{eq: scrZ}
&\scrZ_G:=\{(g,\xi)\in T^{*,\reg}(G): \Ad_g\xi=\xi\},
\end{align}
which is acted by $G$ through the adjoint action on both factors. The obvious projection $\scrZ_G\longrightarrow \fg^{\reg}$ represents $\scrZ_G$ as a $G$-equivariant  abelian group scheme over $\fg^{\reg}$. The categorical quotient $\scrZ_G\sslash G$ can be identified with the affine variety
\begin{align}\label{eq: centralizer cS}
\{(g,\xi)\in G\times \cS: \Ad_g\xi=\xi\},
\end{align}
i.e. the centralizers of the elements in the Kostant slice $\cS$. 

\begin{definition}[First definition of $J_G$]
The \emph{universal centralizer} of $G$, denoted by $J_G$, is defined to be $\scrZ_G\sslash G$, which is isomorphic to (\ref{eq: centralizer cS}). 
\end{definition}

The virtue of this definition is that it explains the name ``universal centralizer", and it exhibits $J_G$ as an abelian group scheme over $\fc$:
\begin{align*}
\chi: J_G\longrightarrow \fc,
\end{align*}
which is actually a holomorphic integrable system. See Figure \ref{figure: J_SL_2} for the case when $G=SL_2(\bC)$.

Next, we give a second definition of $J_G$, which is given by a bi-Whittaker Hamiltonian reduction of $T^*G$. To define this, we fix a Borel subgroup $B\subset G$ and a maximal torus $T\subset B$, and let $N\subset B$ be the unipotent radical. Let $\fb, \ft, \fn$ be the respective Lie algebras.  Let $\Delta\subset \ft^*$ (resp. $\Delta^+$, $\Delta^-$) be the set of roots (resp. positive roots defined by $\fb$, negative roots). Let $\Pi$ be the set of simple roots in $\Delta^+$, and let $W$ be the Weyl group associated to the root system. 

Fix a \emph{regular} element $f\in \bigoplus\limits_{\alpha\in \Pi}\fg_{-\alpha}$, and an $\fsl_2$-triple $(e, f, \sfh_0:=h)$ as above. Note that $\sfh_0=\sum\limits_{\alpha\in \Delta^+}\alpha^\vee$, where $\alpha^\vee$ is the coroot corresponding to $\alpha$. Consider the $N\times N$-Hamiltonian action on $T^*G$, induced from the left and right $N$-action on $G$.  The moment map of the Hamiltonian action is given by 
\begin{align*}
\mu: T^*G&\longrightarrow \fn^*\oplus\fn^*\cong \fn^-\oplus \fn^-\\
(g,\xi)&\mapsto (\xi\text{ mod } \fb, \Ad_g\xi\text{ mod } \fb).
\end{align*}
Since $(f,f)\in \fn^-\oplus \fn^-$ is a regular character of $N\times N$, we have 
\begin{align*}
\mu^{-1}(f, f)=\{(g,\xi): \xi\in f+\fb, \Ad_g\xi\in f+\fb\}
\end{align*}
an $N\times N$-stable coisotropic subvariety in $T^*G$. The action turns out to be free (cf. \cite{Ginzburg} for more details), and we have an identification
\begin{align}\label{eq: Ham N reduction}
\mu^{-1}(f,f)/N\times N\cong \{(g,\xi)\in G\times \cS: \Ad_g\xi=\xi\},
\end{align}
which is exactly isomorphic to $J_G$. This uses the isomorphism
\begin{align*}
N\times\cS&\overset{\sim}{\longrightarrow} f+\fb\\
(u, \xi)&\mapsto \Ad_u\xi,
\end{align*}
which is an important feature of the Kostant slice that we will frequently use without referring to it explicit.

Hence we have a second definition/characterization of $J_G$ as follows.
\begin{definition}[Second definition of $J_G$]\label{def: second J_G}
The \emph{universal centralizer} $J_G$ is defined to be the Hamiltonian reduction (\ref{eq: Ham N reduction}), which is a smooth holomorphic symplectic variety. 
\end{definition}

We remark that there are several other equivalent definitions/characterizations of $J_G$, showing different features of it, as well as its prominent role in representation theory and mathematical physics. For example, it is calculated in \cite{BFM} that the ring of functions on $J_G$, as an affine variety, is isomorphic to the equivariant homology ring $H_{\bullet}^{G^\vee(\cO)}(Gr_{G^\vee})$ of the affine Grassmannian (with the convolution product structure). In particular, it belongs to the list of Coulomb branches defined in \cite{BFN}. On the other hand, $J_G$ is also identified with the moduli space of solutions of the Nahm equations, so it has a hyperKahler structure (cf. \cite{Bielawski}). Since we will not use these features, we will not provide any further details.

We now describe a canonical $\bC^\times$-action on $J_G$, which will define a Liouville vector field as follows. 
Let 
$\gamma: \bC^\times\rightarrow T$ denote the cocharacter corresponding to $\sfh_0$.
Then the canonical $\bC^\times$-action on $J_G$ is given by 
\begin{align}\label{eq: C star action}
&s\cdot (g, \xi)=(\Ad_{\gamma(s)}g, s^2\cdot\Ad_{\gamma(s)}\xi). 
\end{align}
Note that the $\bC^\times$-action scales the symplectic form $\omega=d(\lng\xi, g^{-1}dg\rng)$ by weight $2$, and it does not depend on the choice of representatives $(g,\xi)\in \mu^{-1}(f, f)$. 
Taking the square root of the restricted $\bR_+\subset \bC^\times$-action on $J_G$, we get a Liouville flow.  Let $Z$ denote for the corresponding Liouville vector field. Note that if $G$ is adjoint, then we can turn (\ref{eq: C star action}) into a weight $1$ action by using the cocharacter $\frac{1}{2}\sfh_0$ and changing the scaling $s^2$ on the second factor by $s$. Then the action gives the holomorphic Liouville flow on $J_G$.

Lastly, we recall the Lagrangian correspondence (cf. \cite[Section 2.3]{Ginzburg}, \cite{Teleman})
\begin{align}\label{eq: Lag corresp}
J_G\overset{\pi_{J_G}}{\longleftarrow} J_G\underset{\fc}{\times}\ft^*\overset{\pi_\chi}{\longrightarrow} T^*T,
\end{align}
in which the left map is the obvious projection, the middle term can be identified with
\begin{align}\label{eq: J_G, fc, ft}
J_G\underset{\fc}{\times}\ft^*&\cong\{(g,\xi, B_1)\in G\times \cS\times G/B: \Ad_{g}\xi=\xi, \xi\in \fb_1=\text{Lie} B_1, g\in B_1\}\\
\nonumber&\cong \{(g,\xi, B_1)\in \mathscr{Z}_G\times G/B: \Ad_{g}\xi=\xi, \xi\in \fb_1=\text{Lie} B_1, g\in B_1\}\sslash G
\end{align}
and the right map $\pi_\chi$ is given by 
\begin{align}\label{eq: pi_chi, B_1}
\pi_\chi: (g,\xi, B_1)\mapsto (g \text{ mod }[B_1, B_1], \xi\text{ mod }[\fb_1, \fb_1])\in T\times \ft^*. 
\end{align}
When we refer to this Lagrangian correspondence, we read the correspondence from left to right, i.e. we view $J_G\underset{\fc}{\times}\ft^*$ as a smooth Lagrangian submanifold in $J_G^a\times T^*T$, where $J_G^a$ is the same as $J_G$ but equipped with the opposite symplectic structure. We will refer to the opposite one that is read from right to left, as the \emph{opposite} correspondence.

We comment on some good and bad features of the correspondence (\ref{eq: Lag corresp}). Some useful features include: (1) the map $\pi_{\chi}$ is $W$-equivariant  with respect to the $W$-action on  $J_G\underset{\fc}{\times}\ft^*$ induced from the $W$-action on the $\ft^*$-factor and the natural $W$-action on $T^*T$; (2) the correspondence respects the canonical $\bC^\times$-action on $J_G$ and the square of the fiber dilating $\bC^\times$-action on $T^*T$; (3) it transforms the Kostant sections to cotangent fibers in $T^*T$; (4) it transforms a generic torus fiber of $\chi$ to $|W|$ copies of torus fibers (constant sections) in $T^*T$, inducing isomorphisms from the former to each component of the latter, and it respects the group scheme structure on $J_G$ and $T^*T$. 

An essential bad feature of the correspondence is that $\pi_\chi$ is neither proper nor open. For example, it transforms the central fiber $\chi^{-1}([0])$ to the discrete set $\cZ(G)\times \{0\}$ in $T^*T$, while the whole zero-section of $T^*T$, except for $\cZ(G)\times \{0\}$, is disjoint from the image of $\pi_\chi$. For this reason, it is hard to calculate the associated functors\footnote{Even the definition of the functors (as categorical bimodules) requires technical treatments, for the Lagrangian correspondence as a smooth Lagrangian submanifold in $J_G^a\times T^*T$ (and similarly for the inverse correspondence) will have ends intersect the ``finite" boundary of the product sector, so one needs to perturb the ends in a careful way.} between wrapped Fukaya categories by geometric compositions. However, we use the correspondence (not as a functor though) in our calculations of Floer cochains in Section \ref{sec: wrapping Ham} and \ref{subsec: proof 5.6, 5.7}.

\subsection{The Bruhat decomposition}\label{section: Bruhat}
Using the second definition of $J_G$ (Definition \ref{def: second J_G}) in Subsection \ref{subsec: def of J_G} and under the same setup, we will show a Bruhat decomposition for $J_G$.   The Bruhat decomposition is induced from the projection to the double coset $N\backslash G/N$
\begin{equation*}
J_G\rightarrow N\backslash G/N. 
\end{equation*}
For each element $w\in W$, we use $\cB_w$ to denote for the corresponding Bruhat ``cell"\footnote{Although we call $\cB_w$ a Bruhat cell, it does not mean that $\cB_w$ is contractible, and this is usually not the case (cf. Proposition \ref{prop: B_w0w}).} in $J_G$.

\begin{prop}\label{prop: B_w0w}
\begin{itemize}
\item[(a)]
For any semisimple Lie group $G$, the Bruhat decomposition of the group scheme $J_G$ is indexed by 
subsets $S$ of simple roots. The stratum indexed by $S$ is $\cB_{w_0w_S}$, 
where $w_0$ is the longest element in $W$ and $w_S$ is the longest element in the Weyl group of the standard parabolic subgroup $P_S$ determined by $S$. 
\item[(b)] Let $\cZ(L_S)$ be the center of the standard Levi factor $L_S$ of $P_S$, and let $L^{\der}_S=[L_S, L_S]$ be the derived group of $L_S$. Then 
\begin{align}\label{eq: prop B_w0w}
&\cB_{w_0w_S}
\cong T^*\cZ(L_S)\times (\fl^{\der}_S\sslash L_S^{\der})
\end{align}
and it is $\bC^\times$-invariant. 
\end{itemize}
\end{prop}
\begin{proof}
For any $w\in W$, let $\overline{w}$ be a representative of $w$ in the normalizer of $T$.  
For any $w_0w\in W$, the Bruhat cell $\cB_{w_0w}$ of $J_G$ consists of pairs $((\overline{w}_0)^{-1}\overline{w}h, f+t+\xi)$, $h\in T,\ t\in \ft,\ \xi\in\bigoplus\limits_{\alpha\in \Delta^+}\fg_\alpha$ (modulo the equivalences induced by the $N\times N$-action), such that 
\begin{align}\label{eq: Ad f Delta}
&\Ad_{(\overline{w}_0)^{-1}\overline{w}h} (f+t+\xi)\in f+\ft+\bigoplus\limits_{\alpha\in \Delta^+}\fg_\alpha.
\end{align}
Note that (\ref{eq: Ad f Delta}) implies that $w$ must send $-\Pi$ into $\Pi\cup \Delta^-$, equivalently, $w$ sends $\Pi$ into $(-\Pi)\cup \Delta^+$. Let $S=(-w(\Pi))\cap \Pi$ and let $\Gamma(S)$ be the set of positive roots that can be written as sums of elements in $S$. Let $\fp_S=\fb\oplus \sum\limits_{\alpha\in \Gamma(S)}\fg_{-\alpha}$ be the standard parabolic subalgebra determined by $S$, then $w=w_S$, the longest element in the Weyl group of the standard parabolic subalgebra $\fp_S$. 

Now fix a subset $S\subset \Pi$, and write 
\begin{align*}
&f=\sum\limits_{\alpha\in S}f_{\alpha}+\sum\limits_{\alpha\in \Pi\backslash S}f_{\alpha}\\
&\xi=\sum\limits_{\beta\in \Gamma(S)}\xi_\beta+\sum\limits_{\beta\in \Delta^+\backslash \Gamma(S)}\xi_\beta,
\end{align*}
then (\ref{eq: Ad f Delta}) is equivalent to the data of 
\begin{align}
\nonumber&t\in \ft,\ \Ad_{\overline{w}_Sh}(f+\xi)\in \Ad_{\overline{w}_0}f+\bigoplus\limits_{\alpha\in \Delta^-}\fg_\alpha\\
\Leftrightarrow&
\label{eq: condition w_0wh}\begin{cases}
&\Ad_{\overline{w}_Sh}\sum\limits_{\alpha\in S}f_{\alpha}=\Ad_{\overline{w}_0}\sum\limits_{\alpha\in -w_0(S)}f_{\alpha},\\
&\Ad_{\overline{w}_Sh}(\sum\limits_{\alpha\in \Delta^+\backslash \Gamma(S)}\xi_\alpha)=\Ad_{\overline{w}_0}(\sum\limits_{\alpha\in \Pi\backslash w_0(-S)}f_\alpha)
\end{cases}\\
\nonumber\Leftrightarrow
&\begin{cases}&h\in T \text{ satisfying } \Ad_{\overline{w}_Sh}\sum\limits_{\alpha\in S}f_{\alpha}=\Ad_{\overline{w}_0}\sum\limits_{\alpha\in -w_0(S)}f_{\alpha}\\
&\text{ which is a torsor over }\cZ(L_S),\\
& \xi\in\Ad_{(\overline{w}_Sh)^{-1}\overline{w}_0}(\sum\limits_{\alpha\in \Pi\backslash w_0(-S)}f_{\alpha})+\bigoplus\limits_{\alpha\in \Gamma(S)} \fg_\alpha\\
&t\in \ft
\end{cases}
\end{align}

Let $\phi_{S,h}=(\overline{w}_0)^{-1}\overline{w}_Sh$ and we identify the equivalence classes of solutions in $(\ref{eq: condition w_0wh})$ under the $N\times N$-action. We have $(\phi_{S,h}, f+t+\xi)$ identified with $(\phi_{S,h'}, f+t'+\xi')$ if and only if $h=h'$ and there exists $u\in N$ such that $\tilde{u}=\Ad_{\phi_{S,h}^{-1}}u^{-1}\in N$ and $f+t'+\xi'=\Ad_{\tilde{u}^{-1}}(f+t+\xi)$. 

Let $L_S^{\der}=[L_S, L_S]$ be the derived group of $L_S$.  For any $u=\exp(n)\in N$, $\Ad_{\phi_{S,h}^{-1}} u^{-1}=\exp(-\Ad_{\phi_{S,h}^{-1}} n)\in N$ if and only if $\Ad_{\phi_{S,h}^{-1}} n\in \fn$, and this happens if and only if $n\in \bigoplus\limits_{\alpha\in -w_0(\Gamma(S))}\fg_\alpha$ which is equivalent to $\tilde{u}=\Ad_{\phi_{S,h}^{-1}}u^{-1}\in N_{L^\der_S}$. Let $\fz_S$ be the subspace of $\ft$ defined by the equations $\alpha(\bullet)=0, \alpha\in S$, which is identified with the (dual of the) Lie algebra of $\cZ(L_S)$.  Since $\Ad_{\tilde{u}^{-1}}$ acts trivially on $\fz_S$, $\bigoplus\limits_{\alpha\in-(\Pi\backslash S)}\fg_\alpha$ and $\bigoplus\limits_{\alpha\in w_S^{-1}(\Pi\backslash S)}\fg_\alpha$,
 we have the following identification
\begin{align}\label{eq: proof B_w0w}
&\cB_{w_0w_S}\cong \cZ(L_S)\times \fz_S\times (\sum\limits_{\alpha\in S}f_{\alpha}+\fn_{\fl^{\der}_S}^\perp)/N_{L^{\der}_S}\\
\nonumber\cong&\cZ(L_S)\times \fz_S\times (\fl^{\der}_S\sslash L_S^{\der}),\\
\nonumber\cong&T^*\cZ(L_S)\times (\fl^{\der}_S\sslash L_S^{\der}). 
\end{align}
Note that the space of isomorphisms (\ref{eq: proof B_w0w}) is a torsor over $\cZ(L_S)$. The $\bC^\times$-invariance of $\cB_{w_0w_S}$ is obvious. 
\end{proof}

\begin{example}\label{example: B_w0}
If $S=\emptyset$, then $w_S=1$ and 
\begin{align*}
\cB_{w_0}\cong\{(\overline{w}_0^{-1}h, f+t+\Ad_{(\overline{w}_0^{-1}h)^{-1}}f): h\in T, t\in \ft\}\cong T^*T.
\end{align*}
\end{example}

\begin{remark}\label{remark: choices of w}

\begin{itemize}
\item[(a)] 
In the following, we will fix $\overline{w}_0$ and for each $S\subsetneq \Pi$, we will choose $\overline{w}_S\in N_{L_S^\der}(T\cap L_S^\der)$ (i.e. the normalizer of the maximal torus) satisfying
\begin{align}\label{eq: remark Ad, f}
f_{\alpha}=\Ad_{\overline{w}_S^{-1}\overline{w}_0}f_{w_0w_S(\alpha)},\ \forall\alpha\in S. 
\end{align}
Then for $S\subset S'$, we have 
\begin{align*}
&\Ad_{\overline{w}_{S'}^{-1}\overline{w}_{S}}f_{\alpha}=\Ad_{(\overline{w}_0^{-1}\overline{w}_{S'})^{-1}}(\Ad_{\overline{w}_0^{-1}\overline{w}_S}f_{\alpha})\\
=&\Ad_{(\overline{w}_0^{-1}\overline{w}_{S'})^{-1}}(f_{w_0w_{S}(\alpha)})=f_{w_{S'}w_{S}(\alpha)}, \forall\ \alpha\in S.
\end{align*}
Note that the last step uses $w_{S'}w_S(\alpha)\in S', \forall \alpha\in S$. 
Under such an assumption, the set of $h\in T$ in the second equivalent characterization in (\ref{eq: condition w_0wh}) is canonically identified with $\cZ(L_S)$. 

\item[(b)] Let $\ft_S$ denote for the Cartan subalgebra of $\fl_S^\der$. The condition of (\ref{eq: remark Ad, f}) gives an identification of the subrepresentation of $\Res_{L_{-w_0(S)}^\der}^G(V_{\lambda})$ generated by a highest weight vector $v_\lambda$, for any $\lambda\in X^*(T)^+$, with $V_{\pi_{\ft^*}^S(w_Sw_0(\lambda))}$ of $L_S^\der$, where $\pi_{\ft^*}^S: \ft^*\rightarrow \ft_S^*$ is the natural projection.

\end{itemize}
\end{remark}

For any $L_S$, we have $\cZ(L_S^{\der})$ acts on both $\cZ(L_S)$ and $J_{L_S^{\der}}$, and the twisted product $T^*\cZ(L_S)\underset{\cZ(L^\der_S)}\times J_{L^{\der}_S}$ is canonically a holomorphic symplectic variety. In the following, we use $N_S$ to denote for $N_{L_S^\der}$, and $f_S$ for $\sum\limits_{\alpha\in S}f_{\alpha}$. 

\begin{prop}\label{prop: U_S}
\begin{itemize}
\item[(a)] For any standard Levi $L_S$, we have 
\begin{align}\label{eq: prop fU_S splitting}
\fU_S=T^*\cZ(L_S)\underset{\cZ(L^\der_S)}\times J_{L^{\der}_S}
\end{align}
naturally embeds as an open (holomorphic) symplectic subvariety in $J_G$. 

\item[(b)] The Bruhat cell $\cB_{w_0w_S}$ is contained in $\fU_S$ as a coisotropic subvariety. More explicitly, using (\ref{eq: prop B_w0w}), we have 
\begin{align*}
\cB_{w_0w_S}\cong T^*\cZ(L_S)\underset{\cZ(L_S^\der)}{\times}\cB_{1, L_S^{\der}}\subset T^*\cZ(L_S)\underset{\cZ(L^\der_S)}\times J_{L^{\der}_S}.
\end{align*}
\end{itemize}
\end{prop}
\begin{proof}
We first prove (a). We continue to use the notations from the proof of Proposition \ref{prop: B_w0w}. 
We make the identification
\begin{align}\label{eq: J_L_S, f_S}
&J_{L_S^\der}\cong \{(g_S,\xi_S)|\xi_S, \Ad_{g_S}\xi_S\in f_S+\fn_{S}^\perp\subset\fl_S^\der\}/(N_S\times N_S)\\
\nonumber\cong&\mu_{N_S\times N_S}^{-1}(f_S,f_S)/(N_S\times N_S),
\end{align}
and let $\phi_{S}=(\overline{w}_0)^{-1}\overline{w}_S$ for the choices of $\overline{w}_0$ and $\overline{w}_S$ as in Remark \ref{remark: choices of w} (a). We have the following $N_S\times N_S$-equivariant embedding 
\begin{align}
\label{eq: prop U_S}&\iota_S: (\cZ(L_S)\times\fz_S)\underset{\cZ(L_S^\der)}{\times}\mu_{N_S\times N_S}^{-1}(f_S,f_S)\rightarrow G\times \fg\cong T^*G\\
\nonumber&(z, t; g_S, \xi_S)\mapsto (\phi_{S}zg_S, \xi_S+t+\Ad_{(\phi_{S}zg_S)^{-1}}(f-f_{-w_0(S)})+(f-f_S))\\
\nonumber& \ \ \ \ \ \ \ \ \ \ \ \ \ \ \ \ \ \ \ =:(\phi_{S}zg_S, \Xi_S),
\end{align}
whose image is in $\mu^{-1}_{N\times N}(f,f)$ and $N_S\times N_S$ acts on $G\times \fg$ through 
\begin{align*}
N_S\times N_S\overset{(\Ad_{\phi_{S}}, id)}\longhookrightarrow N\times N.
\end{align*} 
The validity of (\ref{eq: prop U_S}) follows from the simple fact that
\begin{align*}
\Ad_{(\phi_{S}zg_S)^{-1}}(f-f_{-w_0(S)})\in \bigoplus\limits_{\alpha\in \Delta^+\backslash \Gamma(S)}\fg_\alpha=\fn_{\fp_S}
\end{align*}
and
\begin{equation}
\Ad_{\phi_{S}zg_S}(f-f_S)\in  \bigoplus\limits_{\alpha\in \Delta^+\backslash(\Gamma(-w_0(S)))}\fg_\alpha.
\end{equation}
It is clear from (\ref{eq: condition w_0wh}) that the image of $\iota_S$ is independent of the choice of $\overline{w}_0, \overline{w}_S$.

Now we show that $\iota_S$ in (\ref{eq: prop U_S}) satisfies that $\iota_S^*\omega_{T^*G}=p_S^*\omega_{\fU_S},$ where 
\begin{align*}
p_S: (\cZ(L_S)\times\fz_S)\underset{\cZ(L_S^\der)}{\times}\mu_{N_S\times N_S}^{-1}(f_S,f_S)\rightarrow \fU_S,
\end{align*}
is the quotient map. 
Recall that $\omega_{T^*G}=-d(\langle \xi, g^{-1}dg\rangle)$, where $(g,\xi)\in G\times \fg$ and $g^{-1}dg$ denotes for the Maurer-Cartan form. In the following, let $\lambda_{T^*G}=-\langle \xi, g^{-1}dg\rangle$ and $\lambda_{\fU_S}=-(\langle t, z^{-1}dz\rangle+\langle \xi_S, g_S^{-1}dg_S \rangle)$ denote for the primitive of the symplectic forms on $T^*G$ and $\fU_S$ respectively.  We have 
\begin{align}\label{eq: i*lambda}
&-\iota_S^*\lambda_{T^*G}=\langle(\xi_S+t+\Ad_{(\phi_{S}zg_S)^{-1}}(f-f_{-w_0(S)})+(f-f_S)), (\phi_{S}zg_S)^{-1}d(\phi_{S}zg_S)\rangle\\
\nonumber=&\langle \xi_S+t+\Ad_{(\phi_{S}zg_S)^{-1}}(f-f_{-w_0(S)})+(f-f_S)), z^{-1}dz+g_S^{-1}dg_S\rangle\\
\nonumber=&\langle t, z^{-1}dz\rangle+\langle \xi_S, g_S^{-1}dg_S \rangle=-p_S^*\lambda_{\fU_S}.
\end{align}
Here the vanishing of $\langle \xi_S, z^{-1}dz\rangle$ and $\langle t, g_S^{-1}dg_S\rangle$ is clear, and the vanishing of $\langle \Ad_{(\phi_Szg_S)^{-1}}(f-f_{-w_0(S)})+(f-f_S), z^{-1}dz+g_S^{-1}dg_S\rangle$ comes from that 
$\Ad_{(\phi_{S}zg_S)^{-1}}(f-f_{-w_0(S)})+(f-f_S)\in \bigoplus\limits_{\alpha\in (\Delta^+\backslash \Gamma(S))\cup (-(\Pi\backslash S))}\fg_\alpha$.

Next, we show that $\iota_S$ induces a holomorphic symplectic embedding $\tilde{\iota}_S: \fU_S\hookrightarrow J_G$. By (\ref{eq: i*lambda}) and the fact that $\dim \fU_S=\dim J_G$,  
the image of $\iota_S$ is everywhere transverse to the $N\times N$-orbits in $\mu_{N\times N}^{-1}(f,f)$. So $\tilde{\iota}_S: J_{L_S^{\der}}\rightarrow J_G$ is a local holomorphic symplectic diffeomorphism. Now we observe that $\fU_S$ contains a Zariski open (dense) subset $T^*\cZ(L_S)\times_{\cZ(L_S^{\der})} \cB_{w_S, L_S^{\der}}$, where $\cB_{w_S, L_S^{\der}}$ is the open Bruhat cell in $J_{L_S^\der}$ and the restriction of $\tilde{\iota}_S$ to that open set is an embedding onto $\cB_{w_0}$. So we can conclude that $\tilde{\iota}_S$ is an embedding as well. 

Part (b) immediately follows, since $\cB_{w_0w_S}=\{g_S\in \cZ(L_S^\der)\}\subset \fU_S$. 
\end{proof}

\begin{prop}\label{prop: fU_S', sqcup}
For any $S\subset S'\subset \Pi$, we have a natural embedding $\tilde{\iota}_S^{S'}: \fU_S\hookrightarrow \fU_{S'}$. These form a compatible system of embeddings in the sense that for any $S\subset S'\subset S''$, we have $\tilde{\iota}_{S'}^{S''}\circ \tilde{\iota}_{S}^{S'}=\tilde{\iota}_{S}^{S''}$. Moreover,
\begin{equation}\label{eq: prop U_S decomp}
\fU_{S'}=\bigsqcup\limits_{S\subset S'} \cB_{w_0w_{S}},
\end{equation}
where $w_{S}$ as before is the longest element in the Weyl group of $L^{\der}_{S}$. 
\end{prop}
\begin{proof}
Since $L_S\subset L_{S'}$ and $\cZ(L_{S'})\subset \cZ(L_S)$, under the identification $L_{S'}\cong \cZ(L_{S'})\underset{\cZ(L_{S'}^{\der})}{\times} L_{S'}^\der$, we have $\cZ(L_S)=\cZ(L_{S'})\underset{Z(L_{S'}^{\der})}{\times}\cZ(L_S^{S'})$, where $L_{S}^{S'}=L_S\cap L_{S'}^{\der}$. This induces a splitting $\fz_S=\fz_{S'}\oplus \fz_{S}^{S'}$. Let $T_{S'}=T\cap L_{S'}^{\der}$ denote for the maximal torus in $L_{S'}^{\der}$, and choose representatives $\overline{w}_{S'}, \overline{w}_S\in N_{L_{S'}^\der}(T_{S'})$ as in Remark \ref{remark: choices of w}. Let $\phi^{S,S'}$ denote for $\overline{w}_{S'}^{-1}\overline{w}_S$, then we have $\Ad_{\phi^{S,S'}}f_{\alpha}=f_{w_{S'}w_S(\alpha)}$
 for all $\alpha\in S$.

Now we embed $\fU_S$ into $\fU_{S'}$ in a similar manner as of (\ref{eq: prop U_S}).
First, we have an $N_S\times N_S$-equivariant embedding
\begin{align}
\label{eq: prop U_S->U_S'}&\iota_{S}^{S'}: (\cZ(L_{S}^{S'})\times\fz_S^{S'})\underset{\cZ(L_S^\der)}{\times}\mu_{N_S\times N_S}^{-1}(f_S,f_S)\rightarrow\mu_{N_{S'}\times N_{S'}}^{-1}(f_{S'},f_{S'})\\
\nonumber&(z, t; g_S, \xi_S)\mapsto (\phi^{S,S'}zg_S, \xi_S+t+\Ad_{(\phi^{S,S'}zg_S)^{-1}}(f_{S'}-f_{-w_{S'}(S)})+(f_{S'}-f_S)).
\end{align}
By Proposition \ref{prop: U_S} (a), $\iota_S^{S'}$ descends to an embedding 
\begin{equation}\label{eq: iota, S, S'}
T^*\cZ(L_S^{S'})\underset{\cZ(L_S^{\der})}{\times} J_{L_S^{\der}}\hookrightarrow J_{L_{S'}^\der},
\end{equation}
which naturally extends to an embedding
\begin{equation}\label{eq: tilde i S, S'}
\tilde{\iota}_S^{S'}: \fU_S=T^*\cZ(L_S)\underset{\cZ(L_S^{\der})}{\times} J_{L_S^{\der}}\hookrightarrow T^*\cZ(L_{S'})\underset{\cZ(L_{S'}^{\der})}{\times} J_{L_{S'}^\der}=\fU_{S'}. 
\end{equation}
It is clear from (\ref{eq: prop U_S->U_S'}), that for any $S\subset S'\subset S''$, we have $\tilde{\iota}_{S'}^{S''}\circ \tilde{\iota}_{S}^{S'}=\tilde{\iota}_{S}^{S''}$, and the proposition follows. 
\end{proof}

\section{Weinstein Sector structures on $J_G$}\label{sec: skeleton, sector}
We will give natural partial compactifications of $J_G$ as Liouville sectors and present some of them with Weinstein sector structures. This allows us to define a partially wrapped Fukaya category $\cW(J_G)$ (independent of the choice of the partial compactifications) on it following \cite{GPS1}. 
We give the Lagrangian core and skeleton of $J_G$ (for a choice of partially compactified Weinstein sector), from which we can determine a set of generators of the partially wrapped Fukaya category. Throughout this section, we assume $G$ is semisimple. The reductive case follows easily from the semisimple case, which will be spelled out more explicitly in Section \ref{sec: HMS reductive}. 

\subsection{Some algebraic set-up}\label{subsec: algebraic setup}
Recall that the algebraic functions on $G/N$, denoted by $\mathbb{\bC}[G/N]$, as a $G$-representation has a decomposition into irreducibles using the right $T$-action
\begin{align}\label{eq: O, G/N}
\bC[G/N]\cong \bigoplus\limits_{\lambda\in X^*(T)^+} V_{-w_0(\lambda)},
\end{align}
where $X^*(T)^+$ is the set of dominant weights of $T$. Any highest weight vector in each $V_{-w_0(\lambda)}$ corresponds to a left $N$-invariant function. Let $G_{sc}$ be the simply connected form of $G$, and let $T_{sc}\subset G_{sc}$ be the maximal torus (from taking the inverse image of $T$). Then for each fundamental (dominant) weight $\lambda\in X^*(T_{sc})^+_{\fund}$, choose  
\begin{align*}
v_{\lambda}\in V_\lambda,\ v_{-w_0(\lambda)}\in V_{-w_0(\lambda)}\cong V_\lambda^*
\end{align*}
satisfying 
\begin{align*}
\lng\overline{w}_0^{-1}v_{\lambda}, v_{-w_0(\lambda)}\rng=1. 
\end{align*}
and let 
\begin{align}\label{eq: b_lambda}
b_\lambda(gN)=\lng gv_{\lambda}, v_{-w_0(\lambda)}\rng.
\end{align}
Then $b_\lambda$ is a highest weight vector in the factor $V_{-w_0(\lambda)}$ of (\ref{eq: O, G/N}). 

Since $b_{\lambda}(gzN)=\lambda(z)b_{\lambda}(gN)$ for any $z\in \cZ(G_{sc})$, the real function $|b_{\lambda}|$ descends to a left $N$-invariant function on $G/N$. In the following, unless otherwise specified, we will view $b_\lambda$ (resp. $|b_\lambda|$) as a function on $J_{G_{sc}}$ (resp. $J_G$) through the left $N$-equivariant map $\mu^{-1}(f, f)/N\rightarrow G_{sc}/N$ (resp. $\mu^{-1}(f, f)/N\rightarrow G/N$). Let $ac_{\gamma(s)}$ denote for the action of $s\in \bC^\times$ on $J_G$ defined in (\ref{eq: C star action}). It is easy to see that on $J_{G_{sc}}$, we have 
\begin{align}\label{eq: scale b_lambda}
ac_{\gamma(s)}^*b_\lambda=s^{(w_0(\lambda)-\lambda)(\sfh_0)}b_\lambda=s^{-2\lambda(\sfh_0)}b_\lambda.
\end{align}

 In the following lemma, we give a description of the canonical 
 $\bC^\times$-action on the factors in $T^*\cZ(L_S)$ and $J_{L_S^\der}$ under the symplectic embedding (\ref{eq: prop U_S}). For $S\subset \Pi$, let $\sfh_0=\sfh_{0,S}+\sfh'_{0,S^\perp}$ be the decomposition with respect to orthogonal decomposition $\ft\cong\ft_S\oplus \langle \alpha\in S\rangle^\perp$, where 
 \begin{align}\label{eq: h_0, S, S perp}
 \sfh_{0,S}=\sum\limits_{\delta\in \Gamma(S)}\delta^\vee,\ \sfh_{0, S^\perp}'=\sum\limits_{\beta\in \Delta^+\backslash\Gamma(S)}\beta^\vee
 \end{align}
 Let 
$\gamma_S: \bC^\times\rightarrow T_S=T_{L_S^\der}$ be the  map determined by the cocharacter $\sfh_{0,S}$. 

 \begin{lemma}\label{lemma: action factor}
The canonical $\bC^\times$-action on $J_G$ restricted to the open locus $T^*\cZ(L_S)\underset{\cZ(L_S^\der)}{\times}J_{L_S^\der}$ has its action on each factor as
\begin{align*}
&s\cdot (z, t)=(z\Ad_{\overline{w}_S^{-1}\overline{w}_0}\gamma(s)\gamma(s)^{-1}, s^2t), \ (z,t)\in T^*\cZ(L_S);\\
&s\cdot (g_S, \xi_S)=(\Ad_{\gamma_S(s)}g_S, s^2\Ad_{\gamma_S(s)}\xi_S),\ (g_S,\xi_S)\in J_{L_S^\der}, s\in \bC^\times.
\end{align*}
Here $\Ad_{\overline{w}_S^{-1}\overline{w}_0}\gamma(s)\gamma(s)^{-1}$ regarded as a one parameter subgroup in $\cZ(L_S)\subset T$ is given by the cocharacter
\begin{align}\label{eq: lemma: action factor }
w_S^{-1}w_0(\sfh_0)-\sfh_0=-2\sfh_{0, S^\perp}'.
\end{align}
\end{lemma}

\begin{proof}

\begin{align*}
s\cdot (z, t; g_S, \xi_S)\mapsto {\scriptstyle (\Ad_{\gamma(s)}(\phi_{S}zg_S), s^2\Ad_{\gamma(s)}(\xi_S+t+\Ad_{(\phi_{S}zg_S)^{-1}}(f-f_{-w_0(S)})+(f-f_S)))}
\end{align*}
\begin{align*}
&(\Ad_{\gamma(s)}(\phi_{S}zg_S), s^2\Ad_{\gamma(s)}(\xi_S+t+\Ad_{(\phi_{S}zg_S)^{-1}}(f-f_{-w_0(S)})+(f-f_S)))\\
=&(\phi_{S}z\Ad_{(\overline{w}_0^{-1}\overline{w}_S)^{-1}}\gamma(s)\gamma(s)^{-1}\Ad_{\gamma_S(s)}g_S, s^2\Ad_{\gamma_S(s)}(\xi_S+t))\\
&+s^2\Ad_{\Ad_{\gamma_S(s)}g_S^{-1}z^{-1}\gamma(s)\phi_{S}^{-1}}(f-f_{-w_0(S)})+(f-f_S)\\
=&(\phi_{S}z\Ad_{(\overline{w}_0^{-1}\overline{w}_S)^{-1}}\gamma(s)\gamma(s)^{-1}\Ad_{\gamma_S(s)}g_S, s^2\Ad_{\gamma_S(s)}(\xi_S+t))\\
&+\Ad_{\Ad_{\gamma_S(s)}g_S^{-1}z^{-1}\gamma(s)\Ad_{(\overline{w}_0^{-1}\overline{w}_S)^{-1}}\gamma(s)^{-1}\phi_{S}^{-1}}(f-f_{-w_0(S)})+(f-f_S)
\end{align*}

The cocharacter formula (\ref{eq: lemma: action factor }) is direct to check. 
\end{proof}

 The following lemma is easy to check. 
\begin{lemma}
For any $\lambda\in X^*(T_{sc})^+_\fund$, we have $|b_\lambda|\neq 0$ on $N\overline{w}TN$ if and if $ww_0\in W_\lambda=\{w\in W: w(\lambda)=w_0(\lambda)\}$. In particular, $|b_\lambda|\neq 0$ on the Bruhat cell $\cB_{w_0w_S}$ if and only if $w_S\in W_\lambda\Leftrightarrow \lambda\in \lng\alpha\in S\rng^\perp$. 
\end{lemma}

For any $\beta\in \Pi$, let $\beta^\vee$ be the the corresponding coroot, and let $\lambda_{\beta^\vee}$ (resp. $\lambda_\beta^\vee$) denote for the fundamental weight (resp. coweight) that is dual to $\beta^\vee$ (resp. $\beta$). 

In the following, for any $S$, let $\pi_{\ft^*}^S: \ft^*\rightarrow \ft_S^*$ be the natural projection map. 

\begin{lemma}\label{lemma: weights, b_lambda}
For any $G=G_{sc}$ and $S\subset \Pi$, under the embedding 
\begin{align*}
\iota_S: \fU_S=T^*\cZ(L_S)\underset{\cZ(L_S^\der)}{\times}J_{L_S^\der}\hookrightarrow J_G
\end{align*}
for a fixed choice of $\overline{w}_0, \overline{w}_S$ as in Remark \ref{remark: choices of w}, we have 
\begin{align}
\label{eq: lemma b,lambda,S}&\iota_S^*b_{\lambda_{\beta^\vee}}(g_S,\xi_S; z,t)=
\begin{cases}
&\lambda_{\beta^\vee}(z), \text{ if }\beta\not\in S,\\
&\lambda_{\beta^\vee}(z)b^S_{\lambda_{\beta^\vee}}(g_S),\text{ if }\beta\in S,
\end{cases}
\end{align}
where $b_{\lambda_{\beta^\vee}}^S\in \bC[L_S^\der]^{N_S\times N_S}$ corresponds to the fundamental weight\footnote{Note that in general, $b^S_{\lambda_{\beta^\vee}}(g_S)\neq b_{\lambda_{\beta^\vee}}(g_S)$.} $\pi_{\ft^*}^S(\lambda_{\beta^\vee})\in \ft_S^*$ that is dual to $\beta\in S$. 
\end{lemma}

\begin{proof}
Using the definition (\ref{eq: b_lambda}) of $b_\lambda$ and the formula (\ref{eq: prop U_S}), we have 
\begin{align*}
&\iota_S^*b_\lambda(z,t;g_S,\xi_S)=b_\lambda((\overline{w}_0)^{-1}\overline{w}_S zg_S)=\lambda(z)b_\lambda((\overline{w}_0)^{-1}\overline{w}_S g_S)\\
=&\lambda(z)(\overline{w}_0^{-1}\overline{w}_S g_S(v_\lambda), v_{-w_0(\lambda)})=\lambda(z)(g_S(v_\lambda), \overline{w}_S^{-1}\overline{w}_0v_{-w_0(\lambda)})\\
=&\lambda(z)b_{\pi_{\ft^*}^S(\lambda)}(g_S).
\end{align*}
The last line above uses
\begin{align*}
(\overline{w}_S^{-1}v_{\lambda}, \overline{w}_S^{-1}\overline{w}_0v_{-w_0(\lambda)})=1. 
\end{align*}

Since 
\begin{align*}
\pi_{\ft^*}^S(\lambda_{\beta^\vee})=\begin{cases}&0, \text{ if }\beta\not\in S\\
&\lambda_{\beta^\vee},\text{ if }\beta\in S
\end{cases},
\end{align*}
(\ref{eq: lemma b,lambda,S}) follows. 
\end{proof}

Recall that we use $\Gamma(S)$ to denote the set of positive roots that can be written as sums of elements in $S$, i.e. the set of positive roots of the standard Levi subalgebra generated by $S$. A direct corollary of Lemma \ref{lemma: weights, b_lambda} is the following.

\begin{cor}\label{cor: b_lambda, regular}
Assume $G=G_{sc}$. 
  For any $S\subsetneq \Pi$, the holomorphic map
\begin{align*}
b_{S^\perp}:=(\iota_S^*b_{\lambda_{\beta^\vee}})_{\beta\not\in S}: &\fU_S\longrightarrow (\bC^\times)^{\Pi\backslash S} 
\end{align*}
is submersive everywhere. 
\end{cor}

The following lemma is needed for proving Proposition \ref{prop: proper b map} below. 
Assume $G=G_{sc}$. Consider the holomorphic map 
\begin{align}\label{eq: pi_b, def}
\pi_b:=(b_{\lambda_{\beta^\vee}})_{\beta\in \Pi}: J_G\longrightarrow \bC^{\Pi}. 
\end{align}
Let 
\begin{align}\label{eq: pi_b, first, def}
\pi_{|b|}:=\sum\limits_{\beta\in \Pi}|b_{\lambda_{\beta^\vee}}|^{\frac{1}{\lambda_{\beta^\vee}(\sfh_0)}}: J_{G_\ad}\longrightarrow \bR_{\geq 0}. 
\end{align}
Note that the \emph{inverse} canonical $\bC^\times$-action scales each $b_{\lambda_{\beta^\vee}}$ with  weight $2\lambda_{\beta^\vee}(\sfh_0)>0$, making $\pi_{|b|}$ homogeneous of weight $2$.

\begin{lemma}\label{lemma: pi_|b|, epsilon}
\item[(i)] For any compact neighborhood of $\fK\subset \fc$ of $[0]$, there exists $\epsilon>0$ such that $\pi_{|b|}^{-1}([0,\epsilon])\cap \chi^{-1}(\fK)$ is compact. 

\item[(ii)] The restriction $\pi_b|_{\chi^{-1}([0])}: \chi^{-1}([0])\rightarrow \bC^{\Pi}$ is proper. 
\end{lemma}
\begin{proof}
(i) By the homogeneity of $\pi_{|b|}$ under the contracting $\bC^\times$-action, it suffices to show that there exists a compact neighborhood of $\fK\subset \fc$ of $[0]$ and $\epsilon>0$ such that $\pi_{|b|}^{-1}([0,\epsilon])\cap \chi^{-1}(\fK)$ is compact. 
Recall the log partial compactification for the adjoint group $\overline{J}^{\log}_{G_\ad}$ defined in \cite{Balibanu},  
\begin{align}
\label{eq: J, log, def}&\overline{J}^{\log}_{G_\ad}=\{(g^{-1}B,\xi\in \cS): \Ad_{g}\xi\in \bigoplus\limits_{\alpha\in \Pi}\fg_{-\alpha}\oplus \fb\}\subset G/B\times \cS,\\
\nonumber&G_{\ad}\times \cS\supset J_{G_\ad}\hookrightarrow \overline{J}^{\log}_{G_\ad}\overset{\overline{\chi}}{\longrightarrow}\cS\\
\nonumber&(g,\xi)\mapsto (g^{-1}B, \xi).
\end{align}
Here we need the presentation of $J_{G_\ad}$ in (\ref{eq: centralizer cS}) to make the embedding well defined. The \emph{inverse} canonical $\bC^\times$-action extends to $\overline{J}^{\log}_{G_\ad}$, given by 
\begin{align*}
s\cdot (g^{-1}B,\xi)=(s^{-\sfh_0}g^{-1}B, s^{-2}\Ad_{s^{-\sfh_0}}\xi).
\end{align*}
By Theorem 4.11 in \emph{loc. cit.}, the $\bC^\times$-fixed points of the \emph{contracting}  $\bC^\times$-flow are the $T$-fixed points of the Peterson variety identified with $\overline{\chi^{-1}([0])}\subset G/B$, and these are indexed by $\overline{w}^{-1}_S\overline{w}_0B, S\subset\Pi$\footnote{Here we use slightly different conventions from \cite{Balibanu} to be compatible with previous sections; the difference is essentially given by an additional factor of $\overline{w}_0$.}, and the dimension of the ascending manifold of $\overline{w}_S^{-1}\overline{w}_0B\in \overline{\chi^{-1}([0])}$ is $2|\Pi|-|S|=2n-|S|$. Note that the intersection of the ascending manifold with $J_{G_\ad}$ is exactly $\cB_{w_0w_S}$, which is an open dense part.

Suppose the contrary, there exists a sequence $(g_j,\xi_j)\in G_{\ad}\times \cS$ such that 
\begin{align*}
&\xi_j\rightarrow f,\ g_j^{-1}B\rightarrow g_\infty^{-1}B\in \overline{\chi^{-1}([0])}-\chi^{-1}([0])=\bigsqcup\limits_{S\subsetneq \Pi} \overline{w}_0A_S\overline{w}_S B,\\
&\pi_{|b|}(g_j)\rightarrow 0. 
\end{align*} 
where 
\begin{align}\label{eq: A_S}
A_S=N_{S}\cap C_G(\Ad_{\overline{w}_0^{-1}}f_{-w_0(S)})
\end{align}
 (cf. \cite[Proposition 6.3]{Balibanu1} and references cited therein for the Schubert decomposition of $\overline{\chi^{-1}([0])}$). As always, we fix a collection of representatives $\overline{w}_S, S\subset \Pi$ for $w_S\in W$. 

There exists a unique $S\subsetneq \Pi$ and a unique element $u_S\in A_S$, such that $g_{\infty}^{-1}B=\overline{w}_0u_S\overline{w}_S B$.
Since $\cB_{w_0}$ is open dense in $J_{G_\ad}$, after perturbing the sequence $(g_j, \xi_j)$ a little bit if necessary, we may assume that $(g_j,\xi_j)\in \cB_{w_0}$ for all $j$. 
Then we can write 
\begin{align}
\label{eq: g_j, xi_j, u_j}
(g_j,\xi_j)=&(\Ad_{u_j}(\widetilde{u}_j\overline{w}_0^{-1}z_j), \Ad_{u_j}(f+t_j+\Ad_{(\overline{w}_0^{-1}z_j)^{-1}}f)),\text{ for some }z_j\in T, t_j\in \ft, 
\end{align}
where $\widetilde{u}_j\in N$ is the unique element that makes the pair on the right-hand-side (without applying $\Ad_{u_j}$) a commuting pair, and $u_j\in N$ is the unique unipotent element whose adjoint action on the Lie algebra element in the presentation (\ref{eq: prop U_S}) is in the Kostant slice $\cS$. Then $\pi_{|b|}(g_j)\rightarrow 0$ is equivalent to $\lambda_{\beta^\vee}(z_j)\rightarrow 0$ for all $\beta\in \Pi$. 
We have the following implications
\begin{itemize}
\item[(i)]

\begin{align}
\nonumber&g_j^{-1}B=u_j\overline{w}_0B\rightarrow g_{\infty}^{-1}B=\overline{w}_0 u_S\overline{w}_SB,\\
\label{eq: y_jb_j, N-}\Rightarrow&N\ni u_j=\overline{w}_0u_S\overline{w}_S(y_j^-b_j)\overline{w}_0^{-1}\text{ for some }b_j\in B, N^-\ni y_j^-\rightarrow I\\
\nonumber\Rightarrow& u_S\overline{w}_S(y_j^-b_j)\in N^-\\
\nonumber\Rightarrow&y_j^-b_j\in \overline{w}_S^{-1}N_S\cdot N^- \cap N^-\cdot B \subset N^-\cdot h_1T_{S}N_{S}\\
\nonumber&\ \ \ \ \ \text{ for some fixed }h_1\in T\text{ such that }\overline{w}_S^{-1}\in h_1\cdot L_S^{\der} \\
\nonumber\Rightarrow& b_j\in h_1N_{S}T_S, N^-\ni y_j\rightarrow I
\end{align}
We will write $b_j=h_1n^{(j)}z^{(j)}$ with respect to the splitting above. \\
 
 \item[(ii)] Write $\xi_j=f+\eta_j$ for $\eta_j\in \fb$ (or more precisely in $\ker \ad_e$ which is not essential) with $|\eta_j|\rightarrow 0$. Recall $u_S\in A_S$ from (\ref{eq: A_S}). 
 \begin{align}
\nonumber\Ad_{\overline{w}_S^{-1}u_S^{-1}\overline{w}_0^{-1}}(f+\eta_j)=&{\scriptstyle \Ad_{\overline{w}_S^{-1}\overline{w}_0^{-1}}f_{-w_0(S)}+\Ad_{\overline{w}_S^{-1}u_S^{-1}\overline{w}_0^{-1}}(f-f_{-w_0(S)})+\Ad_{\overline{w}_S^{-1}u_S^{-1}\overline{w}_0^{-1}}\eta_j}\\
\label{eq: Ad, f, eta_j}=&\Ad_{h_2}f_{S}+(\text{a fixed term in }\fn_{\fp_S})+({\substack{\text{a term in }\fn_{\fp_S}^-\oplus \fl_S\\
 \text{ that is approaching to 0}}}),
 \end{align}
 where $h_2\in T$ is some fixed element.

 \begin{align}
 \nonumber&\Ad_{b_j\overline{w}_0^{-1}}(t_j+\Ad_{(\overline{w}_0z_j)^{-1}}f+f)\\
 \label{eq: Ad_b, 1}=&\Ad_{b_j\overline{w}_0^{-1}}(t_j)+\Ad_{h_1z^{(j)}}(\Ad_{\widetilde{z}_j^{-1}}(f-f_S))+\Ad_{b_j}(\Ad_{\widetilde{z}_j^{-1}}f_S)+\Ad_{b_j}(\Ad_{\overline{w}_0^{-1}}f)
 \end{align}
where 
\begin{align}
\nonumber&\widetilde{z}_j^{-1}=h_3\cdot w_0(z_j^{-1})=\overline{w}_0^{-1}z_j^{-1}\overline{w}_0^{-1}\text{ for some fixed }h_3\in T\\
\label{eq: beta tilde z_j}&\lambda_{\beta^\vee}(\widetilde{z}_j^{-1})=\lambda_{\beta^\vee}(h_3)\cdot \lambda_{\beta^\vee}(w_0(z_j^{-1}))\rightarrow 0, \ \beta\in \Pi.
\end{align} 
 \end{itemize}

Equation (\ref{eq: g_j, xi_j, u_j}), the relation (\ref{eq: y_jb_j, N-}) and $y_j^-\rightarrow I$ implies that the difference between (\ref{eq: Ad_b, 1}) and (\ref{eq: Ad, f, eta_j}) is approaching to $0$. In particular, with respect to the decomposition $\ft\oplus\fn^-\oplus \fn_{\fp_S}\oplus \fn_S$, we have 
\begin{align}
&w_0(t_j)+\proj_{\ft_S}\Ad_{b_j\widetilde{z}_j^{-1}}f_S\rightarrow 0\\
\label{eq: factor fn_-}&\Ad_{h_1z^{(j)}\widetilde{z}_j^{-1}}f\rightarrow \Ad_{h_2}f_S \\
\label{eq: factor fn_fp_S}&\Ad_{b_j\overline{w}_0^{-1}}(f-f_{-w_0(S)}))\rightarrow \Ad_{\overline{w}_S^{-1}u_S^{-1}\overline{w}_0^{-1}}(f-f_{-w_0(S)}),
\end{align}
where we omit the relation on the component $\fn_S$. (\ref{eq: factor fn_-}) implies that 
\begin{align*}
&\beta(z^{(j)}\widetilde{z}_j^{-1})\rightarrow\begin{cases}&\beta(h_1^{-1}h_2),\text{ if } \beta\in S\\
&\infty,\text{ if } \beta\in \Pi\backslash S
\end{cases}.
\end{align*}
However, since $z^{(j)}\in T_S$, for $\alpha\in \Pi\backslash S\neq \emptyset$, 
\begin{align*}
\lambda_{\alpha^\vee}(\widetilde{z}_j)=\lambda_{\alpha^\vee}(z^{(j)})\lambda_{\alpha^\vee}(z^{(j)}\widetilde{z}_j^{-1})^{-1}=\lambda_{\alpha^\vee}(z^{(j)}\widetilde{z}_j^{-1})^{-1}\rightarrow 0
\end{align*}
because $\lambda_{\alpha^\vee}$ as a \emph{nonnegative} linear combination of $\beta\in \Pi$ has a strictly positive component in $\alpha\in \Pi\backslash S$. This gives a contradiction to (\ref{eq: beta tilde z_j}), so the lemma is established.

(ii) follows from (i) since $\pi_{|b|}|_{\chi^{-1}([0])}$ is homogeneous with respect to the inverse $\bR_+$-action on the domain and the weight $2$ $\bR_+$-action on the codomain. 
\end{proof}

\begin{prop}\label{prop: proper b map}
For any compact region $\fK\subset \fc$, the restriction of $\pi_b$ from (\ref{eq: pi_b, def})
\begin{align*}
\pi_b|_{\chi^{-1}(\fK)}: \chi^{-1}(\fK)\longrightarrow \bC^{\Pi}
\end{align*}
is proper. 
\end{prop}
\begin{proof}

We prove by induction on two things.
(1) Suppose we have proved by induction on the rank of the group the proposition for all $J_{L_S^\der}$ with $S\subsetneq \Pi$. The base case $S=\emptyset$ is trivial. (2) Assume $S=\emptyset$ for the base case. For any compact $\fK'\subset \fc$ (here and after, always assuming containing a neighborhood of $[0]$) and any $\epsilon>0$, there exists a compact $\fK'_{\emptyset, \epsilon}\subset \ft$ such that for all $h\in T$ with $|b_{\lambda_{\beta^\vee}}(\overline{w}_0^{-1}h)|=|\lambda_{\beta^\vee}(h)|\geq \epsilon, \beta\in \Pi$, 
\begin{align*}
\chi_\fg(f+t+\Ad_{(\overline{w}_0^{-1}h)^{-1}}f)\in \fK'\Rightarrow t\in \fK'_{\emptyset, \epsilon}.
\end{align*}
The upshot is that $\Ad_{(\overline{w}_0^{-1}h)^{-1}}f$ is bounded under the assumption, so $\fK'_{\emptyset, \epsilon}$ does not depend on $h$. Note that for the same reason, the inverse implication is also true, i.e. for any compact $\fK_\emptyset\subset\ft$ and $h\in T$ as above, we have 
\begin{align*}
\chi(\iota_\emptyset (\{h: |\lambda_{\beta^\vee}(h)|\geq \epsilon\})\times \fK_\emptyset)\text{ is pre-compact}. 
\end{align*}
Now suppose we have proved for all $S'$ with $|S'|<k$ such that for any $\epsilon_1, \epsilon_2>0$ and compact $\fK'\subset \fc$ as above, there exists a compact $\fK_{S',\epsilon_1, \epsilon_2}'\subset \cS_{\fl_{S'}^\der}\times \fz_{S'}\cong \ft/W_{S'}$ such that for any $(g_{S'}, \xi_{S'}; z,t)\in \fU_{S'}$ satisfying 
\begin{align}\label{eq: region, b, S'}
&|b_{\lambda_{\beta^\vee}}(\overline{w}_0^{-1}\overline{w}_{S'}g_{S'}z)|\begin{cases}&\geq \epsilon_1, \beta\not\in S'\\
&< \epsilon_2 |\lambda_{\beta^\vee}(z)|, \beta\in S'
\end{cases} \\
\nonumber&\Leftrightarrow \begin{cases}&|\lambda_{\beta^\vee}(z)|\geq \epsilon_1, \beta\not\in S'\\
&|b^{S'}_{\lambda_{\beta^\vee}}(g_{S'})|< \epsilon_2, \beta\in S'
\end{cases}, 
\end{align}
we have 
\begin{align*}
\chi(\iota_{S'}(g_{S'}, \xi_{S'}; z, t))\in\fK'\Rightarrow \xi_{S'}+t\in \fK_{S',\epsilon_1, \epsilon_2}'. 
\end{align*}
Let $\fU_{S',<\epsilon_2}^{\geqslant\epsilon_1}$ be the region defined by (\ref{eq: region, b, S'}). 
We note that it is important that we assume $\xi_{S'}\in \cS_{\fl_{S'}^\der}$ in the presentation. 
On the other hand, the inverse implication also holds under the same assumption. Namely, if $\xi_{S'}+t\in\fK_{S'}$ for some compact $\fK_{S'}\subset \cS_{\fl_{S'}^\der}\times \fz_{S'}$ and $|b_{\lambda_{\beta^\vee}}^{S'}(g_S')|<\epsilon_2, \beta\in S'$, then by induction, $g_{S'}\in C_{L_{S'}^\der}(\xi_{S'})$ is uniformly bounded in $L_{S'}^\der$. This together with $|\lambda_{\beta^\vee}(z)|\geq \epsilon_2, \beta\not\in S'$ implies that $\Xi_{S'}$ from (\ref{eq: prop U_S}) has a uniformly bounded component in $\fn_{\fp_{S'}}$. Hence 
 \begin{align*}
 \chi(\iota_{S'}\{(g_{S'}, \xi_{S'}; z, t)\in \fU_{S',<\epsilon_2}^{\geqslant \epsilon_1}: \xi_{S'}+t\in \fK_{S'}\}) \text{ is pre-compact.}
 \end{align*}
The induction steps also include the above claim for all lower rank groups, in particular for $J_{L_S^\der}, S\subsetneq \Pi$, we have the claim holds for all $S'\subsetneq S$ and any $\epsilon_1, \epsilon_2>0$.

Now we look at any $S$ with $|S|=k<n=|\Pi|$. For any $\epsilon_1,\epsilon_2$, let $\cZ(L_S)_{\geq \epsilon_1}:=\{|\lambda_{\beta^\vee}(z)|\geq \epsilon_1, \beta\not\in S\}$. 
Then there exists $\epsilon'_1, \epsilon_2'>0$ such that 
\begin{align}\label{eq: fU_S: geq epsilon}
\fU_{S, \geqslant \epsilon_2}^{\geqslant \epsilon_1}:=(J_{L_S^\der}\underset{\cZ(L_S^\der)}{\times}T^*\cZ(L_S)_{\geq \epsilon_1})-\fU_{S,<\epsilon_2}^{\geqslant \epsilon_1}\subset \bigcup\limits_{S'\subsetneq S}\tilde{\iota}_{S'}^S(\fU^{\geqslant \epsilon'_1}_{S', <\epsilon'_2}).
\end{align}
We have the fibration $\pi_{S,<\epsilon_2}: \fU_{S,<\epsilon_2}^{\geqslant \epsilon_1}\rightarrow T^*\cZ(L_S)_{\geq \epsilon_1}/\cZ(L_S^\der)$ with fiber at any point $(\dot{z}, t)\in T^*\cZ(L_S)_{\geq \epsilon_1}/\cZ(L_S^\der)$ the open subset 
\begin{align}\label{eq: F_dot z, t}
F_{(\dot z, t)}\cong \{(g_S, \xi_S)\in (L_S^\der\times \cS_{\fl_S^\der})\cap \mathscr{Z}_{L_S^\der}: |b_{\lambda_{\beta^\vee}}(g_S)|< \epsilon_2, \beta\in S\}\subset J_{L_S^\der}.
\end{align}
By induction, for any given compact $\fK_S\subset \fc_S$ and $(\dot{z},t)$ as above, $\chi_S^{-1}(\fK_S)\cap F_{(\dot{z},t)}$ is compact.  Let 
\begin{align}\label{eq: fU_S, geq epsilon}
\fU_{S,\fK_S, \epsilon_2}^{\geqslant \epsilon_1}:=\fU_{S,\geqslant \epsilon_2}^{\geqslant\epsilon_1}\cup \bigcup\limits_{\dot{z}\in \cZ(L_S)_{\geq \epsilon_1}/\cZ(L_S^\der)}(F_{(\dot{z},t)}\cap \chi_S^{-1}(\fK_S)). 
\end{align}
Now the fiber of $\pi_{S, \fK_S}: \fU_{S,\fK_S,\epsilon_2}^{\geqslant \epsilon_1}\rightarrow T^*(\cZ(L_S)_{\geq \epsilon_1}/\cZ(L_S^\der))$ at $(\dot{z}, t)$ has two parts of (finite) boundaries: (1) $\chi_S^{-1}(\partial\fK_S)\cap \overline{F}_{(\dot{z}, t)}$ and (2) $\partial F_{(\dot{z}, t)}\cap \chi_{S}^{-1}(\fc_S-\fK_S)$, whose union over all $T^*(\cZ(L_S)_{\geq \epsilon_1}/\cZ(L_S^\der))$ gives the ``horizontal" boundary of $\fU_{S,\fK_S,\epsilon_2}^{\geqslant \epsilon_1}$. We denote these two parts of boundaries by $\sfB_{\partial \fK_S, \epsilon_1,\epsilon_2}$ and $\sfB^{\not\fK_S}_{\partial F, \epsilon_1,\epsilon_2}$, respectively.

We show that
\begin{align}\label{eq: claim fK_S}
&\text{for sufficiently large }\fK'_S,\ \chi(\sfB_{\partial \fK'_S,\epsilon_1,\epsilon_2})\text{ and }\chi(\sfB^{\not\fK'_S}_{\partial F,\epsilon_1,\epsilon_2})\text{ are outside}\\
\nonumber& \text{any given compact }\fK'\subset \fc,
\end{align}
which is exactly the induction step for $S$ in the second part. 
We set up some notations. For any interval $J\subset \bR_{>0}$, we set $\cZ(L_S)_{J}=\{z\in \cZ(L_S): |\lambda_{\beta^\vee}(z)|\in J, \beta\not\in S\}$. 
Denote for the preimage of $T^*(\cZ(L_S)_{J}/\cZ(L_S^\der))$ through $\pi_{S,<\epsilon_2}$ (resp. $\pi_{S, \fK_S}$) in $\fU_{S,<\epsilon_2}^{\geqslant \epsilon_1}$ (resp. $\fU_{S, \fK_S,\epsilon_2}^{\geqslant \epsilon_1}$) as $\fU_{S, <\epsilon_2}^{J}$ (resp. $\fU_{S, \fK_S,\epsilon_2}^{J}$). 

First, choose any $R_1>2\epsilon_1$, and consider $\fU_{S, <\epsilon_2}^{[\epsilon_1, R_1]}\subset \fU_{S,<\epsilon_2}^{\geqslant\epsilon_1}$.  Since $\pi_{|b|}$ is bounded on this region, by Lemma \ref{lemma: pi_|b|, epsilon} (ii), $\chi^{-1}([0])$ intersects $\overline{\fU}_{S, <\epsilon_2}^{[\epsilon_1, R_1]}$ in a compact region, so there exists a compact $\fK^{(1)}_S\subset \fc_S$ such that 
\begin{align}\label{eq: chi, [0], fU, R_1}
\chi^{-1}([0])\cap \fU_{S,<2\epsilon_2}^{[\epsilon_1, R_1]}\subset (\chi_S^{-1}(\fK^{(1)}_S)\underset{\cZ(L_S^\der)}{\times} T^*\cZ(L_S))\cap \fU_{S,<2\epsilon_2}^{[\epsilon_1, R_1]}.
\end{align}

Choose $\epsilon_1''>0$ such that $\fU_{S',<\epsilon_2'}^{\geqslant \epsilon_1'}\subset (\fU_{S'}^S)_{<\epsilon_2'}^{\geqslant\epsilon_1''}\underset{\cZ(L_S^\der)}{\times} T^*\cZ(L_S)_{\geq \epsilon_1'}$ (cf. (\ref{eq: fU_S: geq epsilon}) for $\epsilon_j'$), for all $S'\subsetneq S$, where $\fU_{S'}^S$ denotes for the left-hand-side of (\ref{eq: iota, S, S'}) with the containment relation between $S'$ and $S$ swapped, and $(\fU_{S'}^S)_{<\epsilon_2'}^{\geqslant \epsilon_1''}$ is defined similarly using (\ref{eq: region, b, S'}) for the group $L_S^\der$. 
By induction, for any compact $\fK'\subset \fc$ containing a neighborhood of $[0]$ and any $S'\subsetneq S$, there exists $\fK'_{S', \epsilon_1', \epsilon_2'}\subset \cS_{\fl_{S'}^\der}\times\fz_{S'}$ such that for any $(g_{S'},\xi_{S'}; z, t)\in \fU_{S',<\epsilon_2'}^{\geqslant \epsilon_1'}$ with $\xi_{S'}+t\not\in \fK'_{S', \epsilon_1',\epsilon_2'}$, we have $\chi(\iota_{S'}(g_{S'}, \xi_{S'}; z, t))\not\in \fK'$. 
Also by induction, there exists $\fK^{(2)}_S\subset \fc_S$ such that for any $S'\subsetneq S$, 
\begin{align*}
&\chi_S(\iota_{S'}^S(g_{S'}, \xi_{S'}; z^{(S)}, t^{(S)}))\not\in \fK^{(2)}_S, (g_{S'}, \xi_{S'}; z^{(S)}, t^{(S)})\in (\fU^S_{S'})_{<\epsilon_2'}^{\geqslant \epsilon_1''}\\
&\Rightarrow \xi_{S'}+t^{(S)}\not\in \proj_{\cS_{\fl_{S'}^\der}\times \fz_{S'}^S}\fK'_{S', \epsilon_1', \epsilon_2'},
\end{align*}
with respect to the splitting $\fz_S=\fz_{S'}\oplus \fz_{S'}^S$ as in the proof of Proposition \ref{prop: fU_S', sqcup}. 
Fix any $\fK_S$ containing an open neighborhood of $\fK_S^{(1)}\cup \fK_S^{(2)}$. By induction, for any $(g_S, \xi_S;z, t)\in \sfB_{\partial \fK_S, \epsilon_1,\epsilon_2}$, we have $g_S\in C_{L_S^\der}(\xi_S)$ and $\xi_S$ are bounded and $|\beta(z^{-1})|, \beta\not\in S$ are bounded from above,
so by (\ref{eq: prop U_S}) 
\begin{align}\label{eq: sfB, |t|}
 \chi(\sfB_{\partial \fK_S,\epsilon_1,\epsilon_2}\cap \{|t|\gg 1\})\cap \fK'=\emptyset.
 \end{align}

Combining the above observations (and the compatibility of the open embeddings in Proposition \ref{prop: fU_S', sqcup}) and using the relation (\ref{eq: fU_S: geq epsilon}), we have 
\begin{itemize}
\item[(a)] There exists a compact neighborhood $\fK_1$ of $[0]$ in $\fc$ such that 
\begin{align*}
\chi(\sfB_{\partial \fK_S, \epsilon_1,\epsilon_2}\cap \overline{\fU}_{S,<\epsilon_2}^{[\epsilon_1,R_1]})\cap \fK_1=\emptyset
\end{align*}

\item[(b)] 
$\chi(\sfB^{\not\fK_S}_{\partial F,\epsilon_1,\epsilon_2})\cap \fK'=\emptyset.$ 
\end{itemize}

Second, we claim that there exists $R_2\gg R_1$ and a compact neighborhood $\fK_2$ of $[0]$ in $\fc$ such that 
\begin{align*}
\chi(\sfB_{\partial \fK_S,\epsilon_1,\epsilon_2}\cap \overline{\fU}_{S, <\epsilon_2}^{\geqslant R_2})\cap \fK_2=\emptyset.
\end{align*}
Indeed, recall $\Xi_S$ is from (\ref{eq: prop U_S}), by the same consideration as above from induction, 
\begin{align*}
&(g_S, \xi_S;z, t)\in \sfB_{\partial \fK_S,\epsilon_1,\epsilon_2}\cap \overline{\fU}_{S, <\epsilon_2}^{\geqslant R_2}\Rightarrow \Xi_S\overset{\substack{\text{uniformly}\\ \text{close to}}}{\sim} f+\xi_S+t, \chi_{\fl_S^\der}(\xi_S)\in \partial \fK_S\\
\Rightarrow& \chi_\fg(\Xi_S)\overset{\substack{\text{uniformly}\\ \text{close to}}}{\sim}\chi_\fg(f+\xi_S+t)=\chi_\fg(\xi_S+t). 
\end{align*}
By assumption on $\xi_S$, it is clear that $\chi_\fg(\xi_S+t)$ is outside a fixed compact neighborhood of $[0]\in \fc$. 

Third, for $\sfB_{\partial \fK_S,\epsilon_1,\epsilon_2}\cap \overline{\fU}_{S, <\epsilon_2}^{[R_1, R_2]}$, by (\ref{eq: chi, [0], fU, R_1}) and the invariance of $\chi^{-1}([0])$ under the inverse $\bC^\times$-action, we have 
\begin{align}\label{eq: chi, [0], fU, R_1, R_2}
\chi^{-1}([0])\cap \fU_{S,<2\epsilon_2}^{[R_1, R_2]}\subset (\chi_S^{-1}(\fK_S)\underset{\cZ(L_S^\der)}{\times}T^*\cZ(L_S))\cap \fU_{S,<2\epsilon_2}^{[R_1, R_2]}.
\end{align}
Combining with (\ref{eq: sfB, |t|}), we see that there exists a compact neighborhood $\fK_3\subset \fc$ of $[0]$ such that 
\begin{align*}
\chi(\sfB_{\partial \fK_S, \epsilon_1,\epsilon_2}\cap \overline{\fU}_{S, <\epsilon_2}^{[R_1, R_2]})\cap \fK_3=\emptyset. 
\end{align*}

In summary, we have found a $\fK_S$ so that the following hold:
\begin{itemize}
\item[(a')] There exists a compact neighborhood $\widetilde{\fK}\subset \fc$ of $[0]$ such that\\
 $\chi(\sfB_{\partial \fK_S,\epsilon_1,\epsilon_2})\cap \widetilde{\fK}=\emptyset.$

\item[(b)] $\chi(\sfB_{\partial F,\epsilon_1,\epsilon_2}^{\not\fK_S})\cap \fK'=\emptyset$. Note that if we enlarge $\fK_S$ to be sufficiently large, then the corresponding $\chi(\sfB^{\not\fK_S}_{\partial F,\epsilon_1,\epsilon_2})$ is disjoint from any given compact $\fK''\subset \fc$. 
\end{itemize}

Now we use the (inverse, i.e. contracting) $\bR_{\geq 1}$-action (as a multiplicative monoid) to find a $\fK'_S$ so that claim (\ref{eq: claim fK_S}) holds. Without loss of generality, we may assume that $\fc-\fK'$ is invariant under the $\bR_{\leq 1}$-action. 
Let $\tau\gg 1$ such that $\tau\cdot \fK'\subset\widetilde{\fK}^\circ$. Choose $\fK_{S}'\supset \fK_S$ such that $\tau_1\cdot \sfB_{\partial \fK'_S, \epsilon_1,\epsilon_2}\cap \sfB_{\partial \fK_S, \epsilon_1,\epsilon_2}=\emptyset$ for all $1\leq \tau_1\leq \tau$. This is achievable because the $\bR_{\geq 1}$-action on $\fU_{S}^{\geqslant \epsilon_1}:=\fU_{S,<\epsilon_2}^{\geqslant \epsilon_1}\cup \fU_{S, \geqslant \epsilon_2}^{\geqslant \epsilon_1}$ is the ``product" $\bR_{\geq 1}$-action on the fiber (canonically identified with  $J_{L_S^\der}$ up to $\cZ(L_S^\der)$) and on the base $T^*(\cZ(L_S)_{\geq \epsilon}/\cZ(L_S^\der))$. So the condition on $\fK_S'$ can be checked for the open subset in (\ref{eq: F_dot z, t}) quotient out by $\cZ(L_S^\der)$. It is not hard to see that $\fK'_S$ makes the claim (\ref{eq: claim fK_S}) valid. Indeed, for any $(g_S, \xi_S;z, t)\in \fU_{S}^{\geqslant\epsilon_1}-\fU_{S, \fK'_S}^{\geqslant\epsilon_1}$, we look at the flow $\tau_1\cdot (g_S, \xi_S;z, t), \tau_1\in \bR_{\geq 1}$, which will intersect $\sfB_{\partial \fK_S,\epsilon_1,\epsilon_2}\cup \sfB_{\partial F, \epsilon_1,\epsilon_2}^{\not\fK_S}$ at a finite time. There are two cases
\begin{itemize}
\item[Case 1.] the flow line first intersects $\sfB_{\partial F, \epsilon_1,\epsilon_2}^{\not\fK_S}$, then by (b) above and that $\fc-\fK'$ is invariant under the $\bR_{\leq 1}$-action, $\chi(\iota_S(g_S, \xi_S;z, t))\cap \fK'=\emptyset$.  

\item[Case 2.] the flow line first intersects $\sfB_{\partial \fK_S, \epsilon_1,\epsilon_2}$ at $\widetilde{\tau}_1\cdot (g_S, \xi_S;z, t)$ for some $\widetilde{\tau}_1\in \bR_{\geq 1}$. By assumption and (a') above, $\widetilde{\tau}_1>\tau$, therefore, 
\begin{align*}
\chi(\iota_S(g_S, \xi_S;z, t))=\widetilde{\tau}_1^{-1}\cdot \chi(\iota_S(\widetilde{\tau}_1\cdot (g_S, \xi_S;z, t)))\subset \widetilde{\tau}_1^{-1}(\fc-\widetilde{\fK}).
\end{align*}
Since $\widetilde{\tau}_1^{-1}(\fc-\widetilde{\fK})\cap \fK'=\emptyset$, the claim follows in this case. 
\end{itemize} 
Thus, we have proved claim (\ref{eq: claim fK_S}). 

Lastly, we finish the proof of the proposition. Using Lemma \ref{lemma: pi_|b|, epsilon} (i), we fix an $\epsilon>0$ and a compact neighborhood $\fK\subset \fc$ of $[0]$, so that $\pi_{|b|}^{-1}([0,\epsilon])\cap \chi^{-1}(\fK)$ is compact. We have $\pi_{|b|}^{-1}[\epsilon, \infty)\subset \bigcup\limits_{S\subsetneq \Pi} \fU_{S, <\epsilon_2'}^{\geqslant \epsilon_1'}$ for some $\epsilon_1', \epsilon_2'>0$. Fix any finite interval $[0, K]\subset\bR_{\geq 0}$. By the induction steps above, for any $S\subsetneq\Pi$, $\fU_{S, <\epsilon_2'}^{\geqslant \epsilon_1'}\cap \chi^{-1}(\fK)\cap \pi_{|b|}^{-1}([0,K])$ is pre-compact in $\fU_{S}$. Therefore $\chi^{-1}(\fK)\cap \pi_{|b|}^{-1}([0,K])$ is a finite union of compact subsets, so it is compact. The proof is complete. 
\end{proof}

\begin{remark}\label{remark: handle}
Implicit in the proof above is an inductive process of handle attachments to get $J_G$. 
Namely, the step of getting from (\ref{eq: fU_S: geq epsilon}) to (\ref{eq: fU_S, geq epsilon}), for a fixed $\fK_S$ (assuming it is a closed ball in $\fc_S$ containing $[0]$ in the interior) and sufficiently small $\epsilon_2>0$,  should be viewed as joining a (Morse-Bott) index $(n+|S|)$-handle. 
\end{remark}

We fix some standard (local) coordinates for the open cell $\cB_{w_0}\cong T^*T$. First, the functions $b_{\lambda_{\beta^\vee}},\beta\in \Pi$ give local coordinates on $w_0T\subset G$ (if $G=G_{sc}$ these are also global coordinates). Let $\widetilde{p}_{\beta^\vee}\in \ft, \beta\in \Pi$ be the dual coordinate on $\ft^*$, which are the same as pairing with the simple coroots $\beta^\vee$. Let
\begin{align}\label{eq: q, p}
&q_{\lambda_{\beta^\vee}}=\log |b_{\lambda_{\beta^\vee}}|^{1/\lambda_{\beta^\vee}(\sfh_0)}, \ \theta_{\lambda_{\beta^\vee}}=\Im \log b_{\lambda_{\beta^\vee}} (\text{this is multivalued})\\
\nonumber&p_{\beta^\vee}=\lambda_{\beta^\vee}(\sfh_0)\Re \widetilde{p}_{\beta^\vee}-i\Im \widetilde{p}_{\beta^\vee}. 
\end{align}
The symplectic form on $\cB_{w_0}\cong T^*T$ in such coordinates is given by 
\begin{align*}
\omega&=-\Re(d\sum\limits_{\beta\in \Pi}\widetilde{p}_{\beta^\vee} b_{\lambda_{\beta^\vee}}^{-1}db_{\lambda_{\beta^\vee}})=-\sum\limits_{\beta\in \Pi}d \Re p_{\beta^\vee}\wedge dq_{\lambda_{\beta^\vee}}+d\Im p_{\beta^\vee}\wedge d\theta_{\lambda_{\beta^\vee}}. 
\end{align*}

Similarly, for any $S\subset \Pi$, we can define (local) symplectic dual coordinates 
\begin{align}\label{eq: dual coordinate S}
(q_{\lambda_{\beta^\vee}}, \theta_{\lambda_{\beta^\vee}}; \Re p_{\beta^\vee_{S^\perp}}, \Im p_{\beta^\vee_{S^\perp}}), \beta\not\in S
\end{align}
for the factor $T^*\cZ(L_S)$ in $\fU_S$, where $\beta^\vee_{S^\perp}=\pi_{\fz_S}(\beta^\vee)$ denote for the orthogonal projection of $\beta^\vee$ onto $\fz_S$ with respect to the Killing form. This amounts to replacing $\widetilde{p}_{\beta^\vee}$ in  (\ref{eq: q, p}) with the linear function on $\fz_S^*$ given by $\beta^\vee_{S^\perp}$, for $\beta\not\in S$. Then
\begin{align*}
\omega_{\fU_S}=-\sum\limits_{\beta\not\in S}(d \Re p_{\beta_{S^\perp}^\vee}\wedge dq_{\lambda_{\beta^\vee}}+d\Im p_{\beta_{S^\perp}^\vee}\wedge d\theta_{\lambda_{\beta^\vee}})+\omega_{J_{L_S^\der}}. 
\end{align*}
Note that for $S_1\subsetneq S_2$, the function $\Re p_{\beta^\vee_{S_1^\perp}}$ and $\Re p_{\beta^\vee_{S_2^\perp}}, \beta\not\in S_2$,  are usually different: 
\begin{align}\label{eq: relation Rp_S1, S2}
\Re p_{\beta^\vee_{S_2^\perp}}=\Re p_{\beta^\vee_{S_1^\perp}}+\sum\limits_{\gamma\in S_2\backslash S_1}a_\gamma\Re p_{\gamma^\vee_{S_1^\perp}}
\end{align}
for some constants $a_\gamma$, on $\fU_{S_1}$.

\subsection{Partial compactifications of $J_G$ as Liouville/Weinstein sectors}\label{subsec: partial compact}
In this section, we introduce partial compactifications of $J_G$ as  Liouville/Weinstein sectors, depending certain choices of data. Recall that we have assumed that $G$ is semisimple. 
The key idea is to first partially compactify $J_G-\cB_1$ as a Liouville sector of the form $\fF\times \bC_{\Re z\leq 0}$ where $\fF$ is a Liouville manifold. Then $\overline{J}_G$ is obtained from attaching $|\cZ(G)|$ many critical handles (corresponding to the connected components of $\chi^{-1}([0])$) to $\fF\times \bC_{\Re z\leq 0}$. The main results are Proposition \ref{prop: hypersurface F}, \ref{prop: partial compactify} and \ref{prop: split generation}.

\subsubsection{A smooth hypersurface $H^{sm}$ in $\bR^{\Pi}_{\geq 0}$}\label{subsubsec: H, sm}

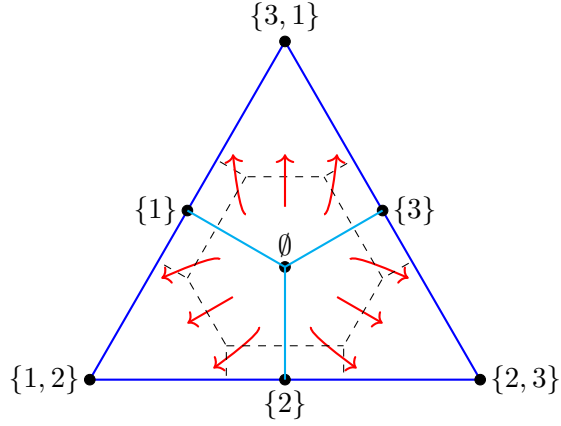
\begin{figure}[h]
\begin{tikzpicture}
\draw[thick, blue] (1.73*1.5,-1.5) to (0,3);
\draw[thick, blue] (-1.73*1.5, -1.5) to (0,3);
\draw[thick, blue] (-1.73*1.5,-1.5) to  (1.73*1.5, -1.5);
\draw[dashed] (-1.73*0.5, 1.4)--(-1.73*0.3, 1.2)-- (1.73*0.3, 1.2)--(1.73*0.5, 1.4);
\draw[red, thick, ->] (1.73*0.3, 0.7) parabola (1.73*0.4, 1.5); 
\draw[red, thick, ->] (-1.73*0.3, 0.7) parabola (-1.73*0.4, 1.5); 
\draw[red, thick, ->,rotate=120] (1.73*0.3, 0.7) parabola (1.73*0.4, 1.5); 
\draw[red, thick, ->, rotate=120] (-1.73*0.3, 0.7) parabola (-1.73*0.4, 1.5); 
\draw[red, thick, ->,rotate=240] (1.73*0.3, 0.7) parabola (1.73*0.4, 1.5); 
\draw[red, thick, ->, rotate=240] (-1.73*0.3, 0.7) parabola (-1.73*0.4, 1.5); 
\draw[red, thick, ->] (0, 0.8)--(0, 1.5);
\draw[red, thick, ->, rotate=120] (0, 0.8)--(0, 1.5);
\draw[red, thick, ->, rotate=240] (0, 0.8)--(0, 1.5);
\draw[dashed, rotate=120] (-1.73*0.5, 1.4)--(-1.73*0.3, 1.2)-- (1.73*0.3, 1.2)--(1.73*0.5, 1.4);
\draw[dashed, rotate=240] (-1.73*0.5, 1.4)--(-1.73*0.3, 1.2)-- (1.73*0.3, 1.2)--(1.73*0.5, 1.4);
\draw[dashed] (-1.73*0.3, 1.2)--(-1.73*0.75, 1.2-1.35);
\draw[dashed,  rotate=120] (-1.73*0.3, 1.2)--(-1.73*0.75, 1.2-1.35);
\draw[dashed, rotate=240] (-1.73*0.3, 1.2)--(-1.73*0.75, 1.2-1.35);
\filldraw (1.73*1.5,-1.5) circle (2pt) node[right]{$\{2,3\}$};
\filldraw (0,3) circle (2pt) node[above]{$\{3,1\}$};
\filldraw (-1.73*1.5,-1.5) circle (2pt) node[left]{$\{1,2\}$};
\filldraw (1.73*0.75, 0.75) circle (2pt) node[right]{$\{3\}$};
\filldraw (-1.73*0.75, 0.75) circle (2pt) node[left]{$\{1\}$};
\filldraw (0, -1.5) circle (2pt) node[below]{$\{2\}$};
\filldraw (0,0) circle (2pt) node[above]{$\emptyset$};
\draw[thick, cyan] (0,0)--(-1.73*0.75, 0.75);
\draw[thick, cyan] (0,0)--(1.73*0.75, 0.75);
\draw[thick, cyan] (0,0)--(0, -1.5);
\end{tikzpicture}
\caption{A picture for $\fC^2$ with $\Pi=\{1,2,3\}$: the barycenters are indexed by $S\subsetneq \Pi$;  the center of $\fC^2$, the cyan segments, and the three open region in the complement give the radial projection of the stratification $\{\fS_S^\circ\}_{S\subsetneq \Pi}$ in $\mathring{\fC}^{2}$; some enlargements of the open regions enclosed by the dashed lines give the collection of $U_S$ (\ref{eq: U_S, D_S}); the red flow lines indicate the flow of $Z_{H^{sm}}$ for model (A) (under the radial projection) with prescribed features.}\label{figure: Morse}
\end{figure}

Fixing a labeling of the set of simple roots $\Pi$, we identify $\bR_{\geq 0}^{\Pi}$ with $\bR_{\geq 0}^n$, and we denote the coordinates by $(r_\beta)_{\beta\in \Pi}$. 
Define the following functions on $J_G$: 
\begin{equation}\label{eq: pi_b, main}
\begin{tikzcd}[column sep=4.6em]
&\pi_{r}: J_G\ar[r,"|b_{\lambda_{\beta^\vee}}|^{1/\lambda_{\beta^\vee}(\sfh_0)}"]&(\bR_{\geq 0})^{\Pi}
\end{tikzcd}
\end{equation}  
Note that by (\ref{eq: scale b_lambda}), the map $\pi_r$ is equivariant with respect to the canonical $\bR_+$-action by restriction from the canonical $\bC^\times$-action (resp. Liouville flow) and the weight $-2$ (resp. $-1$) radial $\bR^+$-action on $(\bR_{\geq 0})^{\Pi}$.

Let $\fC^{n-1}$ be the $(n-1)$-simplex 
\begin{align*}
\{\sum_{\beta\in \Pi} r_\beta=1\}\subset \bR^n_{\geq 0},
\end{align*}
depicted in Figure \ref{figure: Morse}.  
The cells in $\fC^{n-1}$, indexed by $S\subsetneq \Pi$, are given by 
\begin{align}\label{eq: C_S}
C_S=\{r_\beta=0\Leftrightarrow\beta\in S\}. 
\end{align}
We mark the barycenter of $C_S$ by $S$ (cf. Figure \ref{figure: Morse}). For each $\alpha\in \Pi$, let $\Pi_{\alpha}=\Pi\backslash\{\alpha\}$.  Let $\widehat{C}_S$ be the coordinate plane in $\bR^n_{\geq 0}$ defined by $r_\beta=0, \beta\in S$. 

We are going to ``bend" $\fC^{n-1}$ inside $\bR^n_{\geq 0}$ in the following steps. The point is that the function $(\sum_\beta r_\beta)\circ \pi_r$ is \emph{not} smooth on $J_G-\cB_1$, and our goal is to define a smooth hypersurface $H^{sm}\subset \bR^n_{\geq 0}$ which will be the level hypersurface of a homogeneous function replacing $\Sigma_{\beta\in \Pi} r_\beta$ that will become smooth on $J_G-\cB_1$ under the composition with $\pi_r$ (\ref{eq: pi_b, main}). Moreover, we will have a slightly detailed construction of $H^{sm}$ so that we obtain an explicit description of $\pi_{r}^{-1}(H^{sm})$ as the product of a Liouville hypersurface $\fF$ and $\bR$. By appropriate choices of $H^{sm}$, we obtain (generalized) Weinstein structures on $\fF$ (cf. Proposition \ref{prop: hypersurface F}). Due to the technical nature of the discussions, we suggest the reader to skim through the rest of this subsection and return to it later. 

First, for every $\Pi_{\alpha}$, viewed as a vertex in $\fC^{n-1}$ (up to scaling by $1/2$), take the hyperplane 
\begin{align}\label{eq: H_alpha}
H_\alpha=\{r_\alpha=1/2\}. 
\end{align}
The hyperplanes $H_\alpha, \alpha\in\Pi$, together cut out the cubic region $Q^n=\{r_\alpha\in [0,1/2]\}$. Let $\fC^{n-1}_Q=\overline{\partial Q^n\cap \bR_{>0}^n}$. Then $\fC^{n-1}_Q$ is naturally (minimally) stratified, and the collection of strata whose closure does not contain the origin projects to a stratification on $\fC^{n-1}$ along the radial rays, depicted in Figure \ref{figure: Morse}. The strata in the interior of $\fC^{n-1}_Q$ are indexed by $S\subsetneq \Pi$, given by $(\bigcap\limits_{\alpha\not\in S}H_\alpha\cap Q^n)^\circ\subset \partial Q^n$. We denote the strata in $\fC_Q^{n-1}\cap \bR^n_{>0}$  by $\fS^\circ_S, S\subsetneq \Pi$. Also introduce 
\begin{align}\label{eq: fS_S, def}
\fS_S=\overline{\fS^\circ_S}\backslash\bigcup\limits_{S'\subsetneq S}\overline{\fS^\circ_{S'}}\subset \fC_Q^{n-1}. 
\end{align}
The collection $\{\fS_S: S\subsetneq \Pi\}$ gives a stratification of $\fC_Q^{n-1}$ as a singular  space \emph{with boundary}, where we do \emph{not} separately stratify the boundary.

Second, we perform a smoothing of $\fC^{n-1}_Q$ using induction on the dimension of strata $\dim \fS_S=|S|$. For $|S|=n-1$ (which is the base case), we delete a tubular neighborhood of the lower dimensional strata. Suppose we have defined the smoothing of $\fC^{n-1}_Q$ away from a tubular neighborhood of the union of strata of dimension $\leq \ell$, such that along each stratum $\fS_{S'}$ with $|S'|>\ell$, the smoothing is locally defined by an equation of the form 
\begin{align}
\label{eq: f_S'}&f_{S'}(r_\beta;\beta\not\in S')=0,\\
\label{eq: f_S', partial}&\sum_{\beta\not\in S}\frac{\partial f_{S'}}{\partial r_\beta}r_\beta<0,
\end{align}
i.e. it is a product of a smooth star-shaped (meaning the radial vector field is transverse to it everywhere) hypersurface in the coordinate plane $\widehat{C}_{S'}$ with an open neighborhood of $0$ in $\bR^{S'}_{\geq 0}$ (again as a manifold with boundary). 
Here we take 
\begin{align*}
f_{\Pi_\alpha}(r_\alpha)=-r_\alpha+\frac{1}{2}, \forall \alpha\in \Pi. 
\end{align*}
For nice geometric properties, we can assume that all functions belong to a fixed analytic geometric category. 
For any $S$ with $|S|=\ell$,
we look at the intersection of $\widehat{C}_S$ with the existing partial smoothing, which can be extended to a smoothing of $\widehat{C}_S\cap\fC_Q^{n-1}$ satisfying (\ref{eq: f_S', partial}) with $S'$ replaced by $S$. Take the product of the smoothing with an open neighborhood $D_{S}$ of $0$ in $\bR^{S}_{\geq 0}$. 
Note that by Lemma \ref{lemma: weights, b_lambda}, for a fixed point $(r_\beta)_{\beta\not\in S}$, $D_S$ is parametrizing the values of the functions $(|b^S_{\lambda_{\beta^\vee}}|^{1/\lambda_{\beta^\vee}(\sfh_{0})})_{\beta\in S}$ on the factor $J_{L_S^\der}$ (and also of  $(|b^S_{\lambda_{\beta^\vee}}|^{1/\lambda_{\beta^\vee}(\sfh_{0;S})})_{\beta\in S}$), near the cone point $0\in D_S$. 
Now the smoothing is obtained in the complement of a tubular neighborhood of the strata of dimension $<\ell$, and (\ref{eq: f_S'}) and (\ref{eq: f_S', partial}) are satisfied for all $|S|\geq \ell$. Repeat the step until no stratum is left. 

Take a collection of functions $\{(f_S(r_\beta;\beta\not\in S): S\subsetneq \Pi\}$ as above, which defines a global smoothing of $\fC_Q^{n-1}$, denoted by $H^{sm}$. We can define a stratification $\fS$ on $H^{sm}$ with similar features as on $\fC_Q^{n-1}$ (since from now on we will forget $\fC_Q^{n-1}$ and only work with $H^{sm}$, this abuse of notation should not cause any confusion): in a neighborhood of $H^{sm}\cap \widehat{C}_S$, $\fS_S$ is the product of a point in $H^{sm}\cap \widehat{C}_S$ and $D_S$ (with respect to the product structure described above). 
Now for each $S\subsetneq \Pi$, let $\sfN_S\subset \widehat{C}_S$ be an open neighborhood of $\fS_S\cap \widehat{C}_S$, and let 
\begin{align}\label{eq: U_S, D_S}
U_S=(\sfN_S\cap H^{sm})\times D_S.
\end{align}
The collection $\{U_S: S\subsetneq \Pi\}$ (for appropriate choices of $\sfN_S$) defines an open cover of $H^{sm}$, depicted as the domains enclosed by the dashed lines in Figure \ref{figure: Morse} (after some enlargement for each of them, and under the radial projection to $\fC^{n-1}$). Without loss of generality, we may assume that $U_S$ is contained in the $(n-|S|)$-th step of smoothing of $\fC^{n-1}_Q$ above.

Let $Z_{r}$ denote the standard negative radial vector field on $\bR^n_{\geq 0}$, which is the same as the pushforward of the Liouville vector field $Z$ along the projection $\pi_r$. 

For each $S$, let $Z=Z_{S^\perp}+Z_{S}$ be the splitting of $Z$ on $\fU_S$ as in Lemma \ref{lemma: action factor}. The projection of $Z_S$ along $\pi_r$ gives a well defined vector field on $U_S$ (\ref{eq: U_S, D_S}), denoted by $Z_{r;S}$, which is the direct sum of a vector field on $D_S$ and the zero vector field on $\sfN_S\cap H^{sm}$. The flow of $Z_{S}$ scales each $|b^S_{\lambda_{\gamma^\vee}}(g_S)|^{1/\lambda_{\gamma^\vee}(\sfh_{0;S})}, \gamma\in S$, by weight $-1$, and consequently $Z_{r; S}$ scales each $r_\gamma, \gamma\in S$, by weight $-\frac{\lambda_{\gamma^\vee}(\sfh_{0;S})}{\lambda_{\gamma^\vee}(\sfh_0)}$.

Consider the following function on an open neighborhood of $\pi_r^{-1}(\overline{U}_S)$ in $\fU_S$: 
\begin{align}
F_S=\sum\limits_{\gamma\in S}|b^S_{\lambda_{\gamma^\vee}}(g_S)|^{2}=\sum\limits_{\gamma\in S}\frac{(|b_{\lambda_{\gamma^\vee}}|^{1/\lambda_{\gamma^\vee}(\sfh_0)})^{2\lambda_{\gamma^\vee}(\sfh_{0})}}{|\lambda_{\gamma^\vee}(z)|^2}. 
\end{align} 
(cf. Lemma \ref{lemma: weights, b_lambda}). Here $F_\emptyset=0$. 
If we write $\gamma^\vee=\sum_{\beta\not\in S} m_{\gamma}^{\beta}\cdot \beta^\vee$ with $m_{\gamma}^{\beta}\in \bZ$, then $\lambda_{\gamma^\vee}(z)=\prod_{\beta\not\in S}\lambda_{\beta^\vee}(z)^{m_\gamma^\beta}$. Let 
\begin{align*}
&F_{r; S, \gamma}=\frac{r_\gamma}{(\prod_{\beta\not\in S}r_\beta(z)^{m_\gamma^\beta})^{\frac{1}{\lambda_{\gamma^\vee(\sfh_0)}}}},\\
&F_{r;S}=\sum_{\gamma\in S}F^{2\lambda_{\gamma^\vee}(\sfh_0)}_{r; S, \gamma}.
\end{align*}
Then $F_{r;S,\gamma}, \gamma\in S$,  are  well defined on $\sfN_S\times \bR_{\geq 0}^{S}$, and $F_S=\pi_r^*F_{r;S}$. 
Now  $(r_\beta, \beta\not\in S; F_{r;S,\gamma}, \gamma\in S)$ naturally extends to be a coordinate system on $\sfN_S\times \bR^{S}\subset \bR^n$. In particular, the restriction of the function $(F_{r; S,\gamma})_{\gamma\in S}$ on $(\sfN_S\cap H^{sm})\times \bR^S$ with values in $\bR^S$ is everywhere regular. 
Fix $\delta_S>0$ such that $F_{r;S}^{-1}[0,2\delta_S]\subset U_S$.  
For every fixed value of $(r_\beta, \beta\not\in S)$ in $\sfN_S\cap H^{sm}$, any level hypersurface $\{F_{r;S}=\eta\}, 0<\eta<2\delta_S$, cuts out a contractible portion of a sphere in $D_S$. 
 The vector field $Z_{r;S}$ is transverse to all level hypersurfaces and points from higher levels to lower ones.  

The projection of $Z_{S^\perp}$ on the coordinate plane $\widehat{C}_S=\{r_\gamma=0, \gamma\in S\}$ is just the negative standard radial vector field. Let $Z'_{H^{sm};S^\perp}$ be the orthogonal projection of the negative standard radial vector field to $\sfN_S\cap H^{sm}$. The vector field uniquely lifts to a smooth vector field on $U_S$, denoted by $Z_{H^{sm};S^\perp}$, through the projection $U_S\rightarrow \sfN_S\cap H^{sm}$ satisfying the condition that $Z_{H^{sm};S^\perp}(F_{r;S, \gamma})=0$ for all $\gamma\in S$. 

Let 
\begin{align}\label{eq: def U_S'}
U_S'=U_S\cap F_{r;S}^{-1}[0,\delta_S).
 \end{align}
Let $\partial(U_S')_v$ be the vertical boundary of $U_S'$ given by $(\partial (\sfN_S\cap H^{sm})\times D_S)\cap \overline{U'_S}$, and let $\partial(U_S')_h$ be the horizontal boundary of $U_S'$ given by $F_{r;S}^{-1}(\delta_S)\cap \overline{U'_S}$. For any $P\subsetneq S$, let $U'_{S;P}$ be the portion of the boundary of $U_S'$ given by $\overline{U_S'}\cap \widehat{C}_P$. We similarly define $\partial (U'_{S;P})_v$ (resp. $\partial (U'_{S;P})_h$) by the intersection of $U'_{S;P}$ with the hypersurfaces $(\partial (\sfN_S\cap H^{sm})\times D_S)$ (resp. $F_{r;S}^{-1}(\delta_S)$). Note that $\partial (U'_{S;\emptyset})_v=\partial(U_S')_v$ and $\partial (U'_{S;\emptyset})_h=\partial(U_S')_h$.

In order to define a Liouville/Weinstein hypersurface in $\pi_r^{-1}(H^{sm})$, we make the following assumptions on the choice of 
\begin{align}\label{eq: N_S, f_S, choice}
(\sfN_S, f_S(r_\beta;\beta\not\in S)), 
\end{align}
in the previous induction steps:
\begin{itemize}

\item the restriction of the distance squared function $r_{S^\perp}^2:=\sum_{\beta\not\in S}r_\beta^2$ 
on $(\widehat{C}_S\cap H^{sm})^\circ$ (which is obtained in the $(n-|S|)$-th step of smoothing of $\fC_Q^{n-1}$) is Morse.
\end{itemize}
 If $S=\emptyset$, then we also denote  $r_{\emptyset^\perp}^2$ by $r^2$.

We will focus on the following two types of models:
\begin{itemize}
\item[(A)] Let $c_S\in \sfN_S\cap H^{sm}$ be the point that radially projects to the barycenter of the cell $C_S$ in $\fC^{n-1}$. We assume that $r_{S^\perp}^2$ on $(\widehat{C}_S\cap H^{sm})^\circ$ has a unique critical point at $c_S$ which is a local maximum.

\item[(B)] Consider the following stratification of $\fC^{n-1}$ as in Figure \ref{figure: Morse, old}, whose codimension $k$ strata are indexed by strictly increasing $(k+1)$-chains $(S_j)_j=S_0\subsetneq S_1\subsetneq\cdots\subsetneq S_k\subsetneq \Pi$ (as before we don't stratify the boundary separately; equivalently, the strata are in one-to-one correspondence with the strata in $\mathring{\fC}^{n-1}$). We assume the top dimensional stratum for each $S$ has closure contained in $U_S'$. For each stratum indexed by a $(k+1)$-chain $(S_j)_j$, let $c'_{(S_j)_j}$ be the barycenter of the intersection of the stratum with the cell $C_{S_0}$. Let $\widetilde{\fS}_{(S_j)_j}$ (resp. $c_{(S_j)_j}$) denote the corresponding stratification (resp. point of $c'_{(S_j)_j}$) on $H^{sm}$ under the radial projection. We assume that for any $S\subsetneq \Pi$, the restriction of $r_{S^\perp}^2$ on $(\widehat{C}_S\cap H^{sm})^\circ$ has a critical point at $c_{(S_j)_j}$, for any chain $(S_j)_j$ starting with $S_0=S$, whose unstable manifold is a small perturbation of $\widetilde{\fS}_{(S_j)_j}$, and the function has no other critical point. In particular, if $(S_j)_j=(S_0)$ is just a chain of length one, then the distance squared function $r^2$ on $H^{sm}$ has a nondegenerate local minimum at $c_{(S_0)}$.

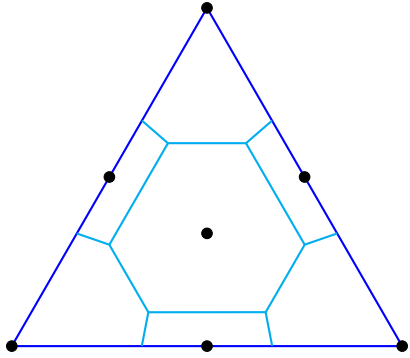
\begin{figure}[h]
\begin{tikzpicture}
\draw[thick, blue] (1.73*1.5,-1.5) to (0,3);
\draw[thick, blue] (-1.73*1.5, -1.5) to (0,3);
\draw[thick, blue] (-1.73*1.5,-1.5) to  (1.73*1.5, -1.5);
\draw[thick, cyan] (-1.73*0.5, 1.5)--(-1.73*0.3, 1.2)-- (1.73*0.3, 1.2)--(1.73*0.5, 1.5);
\draw[thick, cyan, rotate=120] (-1.73*0.5, 1.5)--(-1.73*0.3, 1.2)-- (1.73*0.3, 1.2)--(1.73*0.5, 1.5);
\draw[thick, cyan, rotate=240] (-1.73*0.5, 1.5)--(-1.73*0.3, 1.2)-- (1.73*0.3, 1.2)--(1.73*0.5, 1.5);
\draw[thick, cyan] (-1.73*0.3, 1.2)--(-1.73*0.75, 1.2-1.35);
\draw[thick, cyan, rotate=120] (-1.73*0.3, 1.2)--(-1.73*0.75, 1.2-1.35);
\draw[thick, cyan, rotate=240] (-1.73*0.3, 1.2)--(-1.73*0.75, 1.2-1.35);
\filldraw (1.73*1.5,-1.5) circle (2pt);
\filldraw (0,3) circle (2pt);
\filldraw (-1.73*1.5,-1.5) circle (2pt);
\filldraw (1.73*0.75, 0.75) circle (2pt);
\filldraw (-1.73*0.75, 0.75) circle (2pt);
\filldraw (0, -1.5) circle (2pt);
\filldraw (0,0) circle (2pt);
\end{tikzpicture}
\caption{The stratification of $\fC^{n-1}$ for model (B).}\label{figure: Morse, old}
\end{figure}

 \end{itemize}

The above two models can be easily achieved along the inductive steps of defining $H^{sm}$ as follows. Suppose we have defined $U_{>\ell}:=\bigcup_{|S'|>\ell}U_{S'}$, then for any $S$ with $|S|=\ell$, the gradient vector field of $r^2_{S^\perp}|_{U_{S'}\cap \widehat{C}_S}$ is pointing outward everywhere along $(\partial U_{S';S})_h$ if $S'\supset S$, and so we can pick a level hypersurface $H_{>\ell, S^\perp}$ with some value $a$ of $r^2_{S^\perp}|_{U_{>\ell}\cap \widehat{C}_S}$ near its boundary inside $\widehat{C}^\circ$ (this can be achieved since we can make the smoothing $C^0$-close to a round sphere, so that the deviation of the values of $r^2_{S^\perp}|_{U_{>\ell}\cap \widehat{C}_S}$ is sufficiently small). Then to complete $U_{>\ell}\cap \widehat{C}_S$ to a star-shaped hypersurface in $ \widehat{C}_S$ with the prescribed feature of the restriction of $r^2_{S^\perp}$ on it, it amounts to viewing $U_{>\ell}\cap \widehat{C}_S$ as a  graph of a smooth function from an open subset of $S^{n-1}\cap \widehat{C}_S$ to $\bR_{>0}$, and then extending it to be a complete graph across a neighborhood of the level hypersurface $H_{>\ell, S^\perp}$ with prescribed Morse singularities at the points in $S^{n-1}\cap \widehat{C}_S$
corresponding to the barycenters ($c_S$ in (A) and $c_{(S_j)_j}, S_0=S$ in (B), respectively) under the radial projection. 

Moreover, define a vector field $Z_{H^{sm}}$ on $H^{sm}$ as follows (cf. Figure \ref{figure: Morse} for model (A)). Choose a partition of unity $\{\varphi_{U_S'}\}_{S\subsetneq \Pi}$ for the open covering $\{U_S'\}_{S\subsetneq \Pi}$ of $H^{sm}$. Choose a smooth function $\epsilon_{S}\geq 0$ on each $U_S'$, which is the pullback of a function $\overline{\epsilon}_{S}$ defined on $U_S'\cap \widehat{C}_S$. Then let 
\begin{align}\label{eq: Z_Hsm, inductive}
Z_{H^{sm}}=\sum_{S\subsetneq \Pi}\varphi_{U_{S}'}\cdot (\epsilon_{S}\cdot Z_{H^{sm}, S^\perp}+Z_{r;S}). 
\end{align}
It is clear from construction that if we choose $\epsilon_S>0$ for all $S$, then $Z_{H^{sm}}\cdot \mathrm{grad}(-r^2)>0$ everywhere on $H^{sm}$ except in a sufficiently small neighborhood of the critical points of $r^2$ (i.e. those barycenters) if we choose $\epsilon_S$ sufficiently small. In that neighborhood, the critical points are the only zeros of $Z_{H^{sm}}$.

We have the following lemma. 

\begin{lemma}\label{lemma: Z_H, sm}
\item[(a1)] At any point in $U_{S}^{\flat}:=U_{S}'\backslash \bigcup_{S^\dagg\subsetneq S}\supp(\varphi_{U_{S^\dagg}'})$, the difference $Z_r-Z_{H^{sm}}$ satisfies 
\begin{align}\label{eq: Z_r-Z_Hsm}
(Z_r-Z_{H^{sm}})(F_{r;S, \gamma})=0, \gamma\in S, 
\end{align}
In particular, with respect to the splitting of $\fU_S$ (\ref{eq: prop fU_S splitting}) and the Darboux coordinates on the factor $T^*\cZ(L_S)$ in (\ref{eq: dual coordinate S}), there is a unique lifting of $Z_r-Z_{H^{sm}}$ to $TJ_G|_{\pi_r^{-1}(U_S^\flat)}$ of the form
\begin{align}\label{eq: lemma Z_b-Z_Hsm}
Z-\widetilde{Z}_{H^{sm}}=\sum\limits_{\beta\not\in S}\pi_r^*(a_{S;\beta})\partial_{q_{\lambda_{\beta^\vee}}}=\sum\limits_{\beta\not\in S}\pi_r^*(a_{S;\beta}) \cdot |b_{\lambda_{\beta^\vee}}|^{1/\lambda_{\beta^\vee}(\sfh_0)}\partial_{|b_{\lambda_{\beta^\vee}}|^{1/\lambda_{\beta^\vee}(\sfh_0)}}. 
\end{align}
where $a_{S;\beta}$ is a real function on $U_S^\flat$ satisfying 
\begin{align}\label{eq: partial f, a_S,beta}
\sum_{\beta\not\in S}\frac{\partial f_S}{\partial r_\beta}a_{S;\beta}r_\beta>0. 
\end{align}

\item[(a2)] For any $S_2\subset S_1$, the above lifting of $Z_{r}-Z_{H^{sm}}$ on $\pi_{r}^{-1}(U_{S_1}^\flat)$ and $\pi_{r}^{-1}(U_{S_2}^\flat)$ coincide on their intersection. Hence there is a canonical lifting $Z-\widetilde{Z}_{H^{sm}}$ of $Z_r-Z_{H^{sm}}$ in $TJ_G|_{\pi_{r}^{-1}(H^{sm})}$. 

\item[(b)]Assume we are in the case of model (A) and $\epsilon_{S}>0$ for $|\Pi\backslash S|\geq 2$. Then for every $S\subsetneq \Pi$, the vector field $Z_{H^{sm}}$ has exactly one zero on $U_S'$ at  $c_S$. Moreover, $Z_{H^{sm}}$ is pointing inward to $U_S'$ along $\partial(U_{S;P}')_h$ and pointing outward along $\partial(U_{S;P}')_v$, for all $P\subsetneq S$. 
\end{lemma}

\begin{proof}
(a1) If $S_1\supsetneq S_2$, then on $U_{S_1}'\cap U_{S_2}'$, any vector field $X$ satisfying $X(F_{r;S_1,\gamma})=0, \forall \gamma\in S_1$, implies that $X(F_{r;S_2,\gamma})=0, \forall \gamma\in S_2$, since $F_{r;S_2,\gamma}, \gamma\in S_2$ are entirely determined by $F_{r;S_1,\gamma'}, \gamma'\in S_1$. 
Hence by the definition of $Z_{H^{sm}}$, we have (\ref{eq: Z_r-Z_Hsm}) hold on $U_S^\flat$ for all $S\subsetneq \Pi$. It then follows that $Z_r-Z_{H^{sm}}$ can be lifted in the form of (\ref{eq: lemma Z_b-Z_Hsm}). The uniqueness of the lifting is clear. The functions $a_{S;\beta}, \beta\not\in S$, satisfy (\ref{eq: partial f, a_S,beta}) because $Z_r-Z_{H^{sm}}|_{U_S^\flat}$ is pointing inward along $H^{sm}$ (towards the origin) and is of the form 
\begin{align*}
\sum_{\beta\not\in S}a_{S;\beta}r_\beta\partial_{r_\beta}+\sum_{\alpha\in S}c_{S;\alpha}r_\alpha\partial_{r_\alpha},
\end{align*}
so its pairing with the inward conormal vector $\sum_{\beta\not\in S}\frac{\partial f_S}{\partial r_\beta}d r_\beta$ is strictly positive.

(a2) By the relation 
\begin{align*}
&T^*\cZ(L_{S_2})\underset{\cZ(L_{S_2}^\der)}{\times}J_{L_{S_2}^\der}\cong (T^*\cZ(L_{S_1})\underset{\cZ(L_{S_1}^\der)}{\times}T^*(\cZ(L_{S_2})\cap L_{S_1}^\der))\underset{\cZ(L_{S_2}^\der)}{\times}J_{L_{S_2}^\der}\\
&\hookrightarrow T^*\cZ(L_{S_1})\underset{\cZ(L_{S_1}^\der)}{\times}J_{L_{S_1}^\der}
\end{align*}
and the definition of the coordinates in (\ref{eq: q, p}), it is clear that the unique lifting of $Z_{r}-Z_{H^{sm}}$ in the chart $\pi_{r}^{-1}(U_{S_1}^\flat)$ satisfies that its restriction to $\pi_{r}^{-1}(U_{S_1}^\flat\cap U_{S_2}^\flat)$ is of the form (\ref{eq: lemma Z_b-Z_Hsm}) with respect to the chart $\pi_{r}^{-1}(U_{S_2}^\flat)$. The claim then follows from the uniqueness property. 

(b) is straightforward. 
 \end{proof}

There is an analogous statement of Lemma \ref{lemma: Z_H, sm} (b) for model (B), where one chooses a covering of $H^{sm}$ by open neighborhoods of the critical points of $(-r^2)|_{H^{sm}}$ (indexed by chains $(S_j)_j$), on each of which one sees $Z_{H^{sm}}$ behaving similarly as the gradient vector field near the Morse critical point. We omit the details.

 \subsubsection{Some structural results on $J_G- \cB_1$}\label{subsec: tilde N, I}

Let $\norm(r_\beta;\beta\in \Pi)$ be the function on $\bR^n_{\geq 0}$, homogeneous with respect to the Liouville flow with \emph{weight $-\frac{1}{2}$},  whose value on $H^{sm}$ (defined in Subsection \ref{subsubsec: H, sm}) is constantly $1$. We use $\widetilde{\norm}$ to denote for its pullback to $J_G-\cB_1$ along the projection
\begin{align*}
\pi_{r}|_{J_G-\cB_1}: J_G- \cB_1\longrightarrow \bR^n_{\geq 0}. 
\end{align*}
The upshot is that $\widetilde{\norm}$ is everywhere differentiable and regular (i.e. submersive), which follows from the fact that $|b_{\lambda_{\beta^\vee}}|^{1/\lambda_{\beta^\vee}(\sfh_0)}$ is bounded below by a positive number on $U'_S$ for any $\beta\not\in S$ and from Corollary \ref{cor: b_lambda, regular}. The Hamiltonian vector field $X_{\widetilde{\norm}}$ on $\pi_{r}^{-1}(H^{sm})=\{\widetilde{\norm}=1\}\subset J_G$ generates the characteristic foliation on the hypersurface. 

\begin{prop}\label{prop: hypersurface F}
\item[(a)]
For both model (A) and (B), there exists a Liouville hypersurface $\fF$ in $\widetilde{\norm}^{-1}(1)$ and a diffeomorphism 
\begin{align}\label{eq: lemma, hypersurface F}
\widetilde{\norm}^{-1}(1)\cong \bR\times \fF
\end{align}
such that each leaf of the characteristic foliation on the left-hand-side is sent to $\bR\times \{y\}$ for some $y\in \fF$. 

\item[(b)] For both model (A)  and (B), with appropriate choices of the functions $\{\epsilon_S\}_{S\subsetneq\Pi}$, the Liouville structure on $\fF$ admits a presentation as a (generalized) Weinstein handle decomposition\footnote{By a \emph{generalized} Weinstein structure, we mean the function $\phi$ in the Weinstein manifold structure in \cite[Section 11.4, Definition 11.10]{CE} is Morse-Bott (rather than Morse).}, and hence up to  deformations\footnote{Here a deformation of a Liouville structure means adding the Liouville 1-form by $df$ of a compactly supported function $f$.}, they are Weinstein manifolds. 
The (generalized) critical Weinstein handles of model (B) (resp. model (A)) are indexed by  $(\sigma, S)$ with $S\subsetneq \Pi$ and $\sigma\in \pi_0(\cZ(L_S))$ (resp. $(\sigma, S)$ with $|S|=n-1$). 
\end{prop}

\begin{proof}
(a) In the proof, we assume that we are in model (A). All the discussions below can be carried out without much change to model (B)---the only difference is that model (B) has more critical points which results in a slightly longer description. 

First, by the construction of $H^{sm}$, on $\pi_{r}^{-1}(U'_S)$ we have 
\begin{align*}
X_{\widetilde{\norm}}=\sum\limits_{\beta\not\in S}\pi_r^*(\frac{\partial f_S}{\partial r_\beta})|b_{\lambda_{\beta^\vee}}|^{1/\lambda_{\beta^\vee}(\sfh_0)}\partial_{\Re p_{\beta_{S^\perp}^\vee}}
\end{align*}
with respect to the splitting $T^*\cZ(L_{S})\underset{\cZ(L_{S}^\der)}{\times} J_{L_{S}^\der}$ over $U_S'$ (cf. (\ref{eq: q, p}) for the notations on dual symplectic coordinates). 

Second, recall the vector field $Z_{H^{sm}}$ constructed above, which depends on a choice of $U_S', S\subsetneq \Pi$ with a partition of unity, and the functions $\epsilon_S\geq 0$ on $U_S'$, for $|\Pi\backslash S|\geq 2$.  Let $a_{S;\beta}, \beta\not\in S$ be the function on $U_S^\flat$ as in (\ref{eq: lemma Z_b-Z_Hsm}).  Let $\fF_S\subset \pi_{r}^{-1}(U_S^\flat)$ be the symplectic hypersurface cut out by  the equation
\begin{align}\label{eq: F_(S_i)_i}
F_{U^\flat_S}(|b_{\lambda_{\beta^\vee}}|^{-1/\lambda_{\beta^\vee}(\sfh_0)}, \Re p_{\beta^\vee_{S^\perp}}; \beta\not\in S):=\sum\limits_{\beta\not\in S}\pi_r^*(a_{S;\beta})\cdot \Re p_{\beta_{S^\perp}^\vee}=0. 
\end{align}
Since by construction $X_{\widetilde{\norm}}(F_{U^\flat_S})>0$ everywhere on $\pi_{r}^{-1}(U^\flat_S)$ (cf. Lemma \ref{lemma: Z_H, sm} (a1)), 
$\fF_{S}$ gives a section of the principal $\bR$-bundle $\pi_{r}^{-1}(U^\flat_S)\rightarrow \pi_{r}^{-1}(U^\flat_S)/\bR$, generated by the Hamiltonian flow of $\widetilde{\norm}$. 

Third, we claim that over any intersection $U^\flat_{S_1}\cap U^\flat_{S_2}$ with $S_2\subset S_1$, $\fF_{S_1}$ and $\fF_{S_2}$ coincide, so $\{\fF_S: S\subsetneq \Pi\}$ glue to be a global symplectic hypersurface. Let $\cZ(L_S)_0\subset \cZ(L_S)$ be the identity component, $\cZ(L_S^\der)_0:=\cZ(L_S^\der)\cap \cZ(L_S)_0$, and let $\pi_r^S: J_{L_S^\der}/\cZ(L_S^\der)_0\rightarrow \bR^S_{\geq 0}$ denote the projection (\ref{eq: pi_b, main}) for the group $L_S^\der/\cZ(L_S^\der)_0$. Since both $\fF_{S_1}$ and $\fF_{S_2}$ are cut out by the \emph{linear equation in each cotangent fiber} of the factor $T^*\cZ(L_{S_2})_0$  given by $\omega(Z-\widetilde{Z}_{H^{sm}},-)=0$, where $Z-\widetilde{Z}_{H^{sm}}$ (cf. Lemma \ref{lemma: Z_H, sm} (a2)) is a vector field tangent to the base with respect to the natural splitting $T^*\cZ(L_{S_2})_0\cong \cZ(L_{S_2})_0\times \fz_{\fl_{S_2}}^*$, we are done.

 Let $H_{S^\perp}\subset \cZ(L_S)_{0,\bR}$ be the (non-closed) smooth hypersurface that projects to $\sfN_S\cap H^{sm}$ isomorphically, under taking the components in $\pi_r$ for $r_\beta,\beta\not\in S$. Then $\fF_{S}$ (up to shrinking a little bit) can be canonically symplectically identified with
 \begin{align}\label{eq: fF_S, splits}
 \cU^\der_S\times_{\cZ(L_S^\der)_0} T^*\cZ(L_S)_{0,\cpt}\times T^*H_{S^\perp}
 \end{align}
 for some open subset $ \cU^\der_S\subset J_{L_S^\der}$. 
The Liouville vector field $Z_{\fF_S}$ also splits as 
\begin{align*}
Z_{\fF_S}=Z_{\cU_S^\der}+Z_{T^*\cZ(L_S)_{0,\cpt}}+Z'_{T^*H_{S^\perp}}, 
\end{align*}
where $Z_{\cU_S^\der}$ and $Z_{T^*\cZ(L_S)_{0,\cpt}}$ are the standard Liouville vector fields. It is clear that the projection of $Z_{\fF}$ to $H^{sm}$ is exactly $Z_{H^{sm}}$. Note the (original) Liouville 1-form on $\fU_S$ in the factor $T^*\cZ(L_S)_{0,\bR}$ is of the form
\begin{align}\label{eq: theta, S, bR}
\vartheta_{S^\perp, \bR}=&-\lng t_\bR, \rho^{-1}d\rho\rng-d\lng t_\bR, \sfh'_{0,S^\perp}\rng,\ (\rho, t_\bR)\in \cZ(L_S)_{0,\bR}\times \fz^*_\bR\\
\nonumber=&-\sum_{\beta\not\in S} \Re p_{\beta_{S^\perp}^\vee} dq_{\lambda_{\beta^\vee}}-d\sum_{\beta\not\in S}\Re p_{\beta_{S^\perp}^\vee},
\end{align}
where the second line uses the coordinates from (\ref{eq: dual coordinate S}). 
Let $U'_{>\ell}=\bigcup_{|S|>\ell}U'_S$. To figure out the Liouville vector field $Z'_{T^*H_S}$, we first note that if $\epsilon_{S}=0$, then in $\pi_r^{-1}(U^\flat_S\backslash \overline{U'_{>|S|}})$, the vector field $Z-\widetilde{Z}_{H^{sm}}$, which is tangent to the factor $T^*\cZ(L_S)_{0,\bR}$, gives exactly the difference between the Liouville vector field for $\vartheta_{S^\perp,\bR}$ and the Euler vector field. Equivalently, we have 
\begin{align*}
\omega_{T^*\cZ(L_S)_{0,\bR}}(Z-\widetilde{Z}_{H^{sm}},-)=-d\sum_{\beta\not\in S}\Re p_{\beta_{S^\perp}^\vee}.
\end{align*}
It follows that the part
\begin{align*}
\fF'_S:=\fF_S\cap \pi_r^{-1}(U^\flat_S\backslash \overline{U'_{>|S|}})
\end{align*}
is cut out by $\lng t_\bR, \sfh_{0, S^\perp}'\rng=\sum_{\beta\not\in S}\Re p_{\beta_{S^\perp}^\vee}=0$ inside $\pi_r^{-1}(U^\flat_S\backslash \overline{U'_{>|S|}})$, so $Z'_{H^{sm}}$ is equal to the Euler vector field on $T^*H_{S^\perp}$, assuming $\epsilon_S=0$. 
For the general case, choose a coordinate system $(x_1,\cdots, x_{n-|S|})$ on $\cZ(L_S)_{0,\bR}$ in a neighborhood of $H_{S^\perp}$, so that $H_{S^\perp}$ is defined by $x_{n-|S|}=0$ and $c_S$ has coordinate $0$, and let 
\begin{align*}
(x_1,\cdots, x_{n-|S|}; y_1,\cdots, y_{n-|S|})
 \end{align*}
 be the Darboux coordinates on the cotangent bundle of that neighborhood. The vector field $Z_{S^\perp,H^{sm}}$ lifts to a (unique) vector field on $H_{S^\perp}$, which we write as 
 \begin{align*}
 \sum_{j=1}^{n-|S|-1}v_j(x_1,\cdots, x_{n-|S|-1})\partial_{x_j}. 
 \end{align*}
Similarly $\epsilon_S\cdot Z_{S^\perp,H^{sm}}$ lifts to $\overline{\epsilon}_S(x_1,\cdots, x_{n-|S|-1})\cdot \sum_{j=1}^{n-|S|-1}v_j\partial_{x_j}$. Then the symplectic hypersurface $T^*H_{S^\perp}$ inside $T^*\cZ(L_S)_{0,\bR}|_{H_{S^\perp}}$ is cut out by the equation
\begin{align*}
-\sum_{\beta\not\in S}\Re p_{\beta_{S^\perp}^\vee}-\overline{\epsilon}_S(x_1,\cdots, x_{n-|S|-1})\cdot \sum_{j=1}^{n-|S|-1}v_j(x_1,\cdots, x_{n-|S|-1})y_j=0. 
\end{align*}
Now we have 
\begin{align*}
\vartheta_{S^\perp, \bR}|_{T^*H_{S^\perp}}=-\sum_{j=1}^{n-|S|-1}y_jdx_j+d(\overline{\epsilon}_S\cdot \sum_{j=1}^{n-|S|-1}v_jy_j).
\end{align*}
One can directly calculate the Liouville vector field to be
\begin{align}\label{eq: Z'HSperp}
Z'_{T^*H_{S^\perp}}=\sum_{j=1}^{n-|S|-1}y_j\partial_{y_j}+ \overline{\epsilon}_S\cdot \sum_{j=1}^{n-|S|-1}v_j\partial_{x_j}-\sum_{1\leq i,j\leq n-|S|-1} y_j\frac{\partial{(\overline{\epsilon}_Sv_j)}}{\partial x_i}\partial_{y_i}
\end{align}

Now assume $\epsilon_S>0$ everywhere with $\max_{U_S'}(|D\epsilon_S|+\epsilon_S)$ sufficiently small. Then 
\begin{itemize}
\item we can make the last component in (\ref{eq: Z'HSperp}) have norm squared
\begin{align*}
\sum_{i,j}|y_j|^2|\frac{\partial{(\overline{\epsilon}_Sv_j)}}{\partial x_i}|^2\ll \sum_j|y_j|^2=|\sum_{j=1}^{n-|S|-1}y_j\partial_{y_j}|^2, 
\end{align*}
so $Z'_{T^*H_{S^\perp}}$ is a small perturbation of the sum of the Euler vector field and the lifting of $\overline{\epsilon}_S\cdot Z_{S^\perp, H^{sm}}$ on $H_{S^\perp}$;

\item for any convex open pre-compact region $B_S\subset \fl_S^{\der}$ and $\{|t|< R\}\subset \fz_S^*$, we can choose such $\epsilon_S$, so that for the pre-compact region (due to Proposition \ref{prop: proper b map}):
\begin{align*}
\fD'_{B_S, R}:=\fF_S'\cap (\chi_S^{-1}(B_S)\times_{\cZ(L_S^\der)_0}\big((\cZ(L_S)_0\times \{|t|< R\})\big)),
\end{align*}
the Liouville vector field is pointing \emph{outward} along the boundaries $\{|t|=R\}$, (the preimage of) $\chi_S^{-1}(\partial B_S)$ and (the preimage of) $\partial_v(U_S^\flat\backslash U'_{>|S|})$ (defined similarly as for $\partial_v(U_S')$), while it is pointing \emph{inward} along the boundary (the preimage of) $\partial_h(U_S^\flat\backslash U'_{>|S|})$;

\item The only zero locus of the Liouville vector field is along $\{(g_S, \xi_S;z,t): g_S=I, \xi_S=0, t=0, \pi_r(g_S, \xi_S;z,t)=c_S\}$, which is $|\pi_0(\cZ(L_S))|$-many orbit(s) of $Z(L_S)_{0, \cpt}$ over $c_S\in H^{sm}$. 

\end{itemize}

In the following, choose $0<\widetilde{\delta}_S<\delta_S$, and let $\widetilde{U}_S'=U_S\cap F_{r;S}^{-1}([0, \widetilde{\delta}_S))$, defined similarly as $U_S'$ in (\ref{eq: def U_S'}), be a smaller open subset satisfying $\widetilde{U}_S'\cap \big(\bigcup_{S^\dagg\subsetneq S} \supp(\varphi_{U'_{S^\dagg}})\big)=\emptyset$.  By easily achieved appropriate choices, we may assume that $\{\widetilde{U}'_S\}_{S\subsetneq\Pi}$ forms an open covering of $H^{sm}$ and  Lemma \ref{lemma: Z_H, sm} (b) applies to $\widetilde{U}_S'$.
Now we can define a Liouville domain $\fD\subset \fF$, fibered over $H^{sm}$, whose completion is $\fF$, inductively on $|S|$ starting with $|S|=n-1$. First, define  
\begin{align*}
\widetilde{\fD}'_{B_{S}, R_S}:=\fD'_{B_S, R_S}\cap \pi_r^{-1}(\widetilde{U}_S'),
\end{align*}
for sufficiently large $B_{S}$ and $R_S$. 
For $S=\Pi_\alpha$, let 
\begin{align*}
\pi_{r, >n-2}:\fD_{>n-2}:=\bigcup_{|S|=n-1}\widetilde{D}'_{B_S, R_S}\longrightarrow \widetilde{U}_{>n-2}:=\bigcup_{|S|=n-1}\widetilde{U}_S'
\end{align*}
be the projection from restriction of $\pi_r$. Note that the projection extends by taking the closure of both the source and target, and denote the resulting projection by $\overline{\pi}_{r, >n-2}$. 

Suppose we have defined $\fD_{>\ell}$, which is a manifold with corners so that $\partial \fD_{>\ell}$ is stratified with respect to the corner structure, equipped with the projection $\overline{\pi}_{r, >\ell}: \overline{\fD}_{>\ell}\rightarrow \overline{\widetilde{U}}_{>\ell}\subset \bigcup_{|S'|>\ell} U_{S'}^\flat$ restricted from $\pi_r$, for some open $\widetilde{U}_{>\ell}\supset \bigcup_{|S'|>\ell} \widetilde{U}_{S'}'$, so that $\overline{\pi}_{r,>\ell}^{-1}(\partial \widetilde{U}_{>\ell})$ is a union of strata on which the Liouville vector field is pointing inward, and on its complement (just the open strata) the Liouville vector field is pointing outward. We denote the former by $(\partial \fD_{>\ell})_{in}$ and the latter by $(\partial \fD_{>\ell})_{out}$. Now take any $S$ with $|S|=\ell$, and choose $B_S$ and $R_S$ sufficiently large, so that $\fD_{B_S, R_S}'$ satisfies that 
\begin{itemize}
\item $\partial \fD_{B_S, R_S}'\cap \pi_r^{-1}\big(\partial_v(U_S^\flat\backslash U_{>|S|}')\big)$ contains all the points in $\fF\cap \pi_r^{-1}\big(\partial_v(U_S^\flat\backslash U_{>|S|}')\big)$  that can be flowed into $\overline{\fD}_{>\ell}$ under the (positive) Liouville flow in $\fF$. This is possible because by construction (cf. Lemma \ref{lemma: Z_H, sm} (b)), the  flow of the projection of the Liouville vector field, which is $Z_{H^{sm}}$, will take all points in $\partial_v(U_S^\flat\backslash U_{>|S|}')$ into $\widetilde{U}_{>\ell}$ in a uniformly bounded time. 
\end{itemize}

Let $B_S'\supsetneq \overline{B_S}, R_S'>R_S$, and let $\fT_{B_S', R_S'}'$ be the union of flow lines of the Liouville flow starting from $\partial \fD_{B_S', R_S'}'\cap \pi_r^{-1}\big(\partial_v(U_S^\flat\backslash U_{>\ell}')\big)$ and ending at $\fF\cap \pi_r^{-1}(\partial\widetilde{U}_{>\ell})$, i.e. an embedded suspension of $\partial \fD_{B_S', R_S'}'\cap \pi_r^{-1}\big(\partial_v(U_S^\flat\backslash U_{>\ell}')\big)$ under the flow.  After some (obvious) renormalization of the Liouville vector field,  we have the flow time from any initial point on the former to the corresponding ending point on the latter is always $1$. This defines a smooth function $\tau: \fT_{B_S', R_S'}'\rightarrow [0,1]$ that records the flow time. 
Choose any smooth function $\psi_S: B_S'\times \{|t|<R_S'\}\longrightarrow [0,1]$, such that $\psi_S|_{B_S\times \{|t|<R_S\}}=1$, and $\psi_S=0$ on a neighborhood of $\partial (B_S'\times \{|t|<R_S'\})$. With some easily achieved appropriate choices, we can make sure that 
$\pi_r(\fT_{B_S', R_S'}')$ contains the region bounded by $\{F_{r;S}=\delta_S'\}$, $\partial_v(U_S^\flat\backslash U_{>\ell}')$ and $\partial\widetilde{U}_{>\ell}$, for some $\delta_S'\in (\widetilde{\delta}_S,\delta_S)$. Then  
\begin{align*}
\fT_{B_S', R_S'}:=\fT_{B_S', R_S'}'\cap \{\tau< \psi_S(\chi_S(g_S, \xi_S), t)\}\cap\pi_r^{-1}(\{F_{r;S}<\delta_S'\})
\end{align*}
has $Z_{\fF}$, the Liouville field on $\fF$, pointing inward along the portion of the boundary cut out by $\pi_r^{-1}(\{F_{r;S}=\delta_S'\})$ and $\pi_r^{-1}\big(\partial_v(U_S^\flat\backslash U_{>\ell}')\big)$, denoted by $(\partial \fT_{B_S', R_S'})_{in}^{h}$ and $(\partial \fT_{B_S', R_S'})_{in}^{v}$ respectively, and pointing outward along that cut out by $\{\tau= \psi_S(\chi_S(g_S, \xi_S), t)\}$, denoted by $(\partial \fT_{B_S', R_S'})_{out}$. 

Lastly, let 
\begin{align*}
\fD_{>\ell-1}=\big(\overline{\fD}_{>\ell}\cup \bigcup_{|S|=\ell}(\widetilde{\fD}'_{B'_{S}, R'_S}\cup \fT_{B_S', R_S'})\big)^\circ.
\end{align*}
Let $\widetilde{U}_{>\ell-1}=\pi_r(\fD_{>\ell-1})$, which is open in $H^{sm}$. Then 
\begin{itemize}
\item $\widetilde{U}_{>\ell-1}\supset  \bigcup_{|S'|>\ell-1} \widetilde{U}_{S'}'$ and $\overline{\widetilde{U}}_{>\ell-1}\subset \bigcup_{|S'|>\ell-1} U_{S'}^\flat$\\

\item Since each $\fT_{B_S', R_S'}, |S|=\ell$, is attached to $\widetilde{\fD}'_{B'_{S}, R'_S}$ and $\fD_{>\ell}$, with its boundary portion $(\partial \fT_{B_S', R_S'})_{out}$ completely covering $(\partial \fD_{>\ell})_{in}\cap \pi_r^{-1}(\{(r_\beta)_{\beta\in\Pi}\in U_S': F_{r;S}(r_\beta)<\delta_S'\})$, and its boundary portion $(\partial \fT_{B_S', R_S'})_{in}^{v}\cap \pi_r^{-1}(\partial(\widetilde{U}_S'\backslash \overline{U'_{>\ell}}))$ completely covered by $\partial(\widetilde{\fD}_{B'_S, R'_S})_{out}$, we see that $\overline{\pi}_{r,>\ell-1}^{-1}(\partial \widetilde{U}_{>\ell-1})$ is a union of strata on which the Liouville vector field is pointing inward, and on its complement the Liouville vector field is pointing outward. 
\end{itemize}
Therefore, $\pi_{r, >\ell-1}: \fD_{>\ell-1}\rightarrow \widetilde{U}_{>\ell-1}$ completes the inductive step. See Figure \ref{figure: fD_ell-1} below for an illustration of the constructions. 
Let $\fD=\fD_{>-1}$, then $\widetilde{U}_{>-1}=H^{sm}$, and $\pi_r|_{\fD}: \fD\rightarrow H^{sm}$ gives the desired Liouville domain (since $Z_{\fF}$ is pointing outward along its boundary) fibered over $H^{sm}$, whose completion is $\fF$. This is because the Liouville flow on $\fF$ is certainly complete, and there is no zeros of $Z_\fF$ outside $\fD$. 

\begin{figure}[htbp]
\centering
\includegraphics[width=5in]{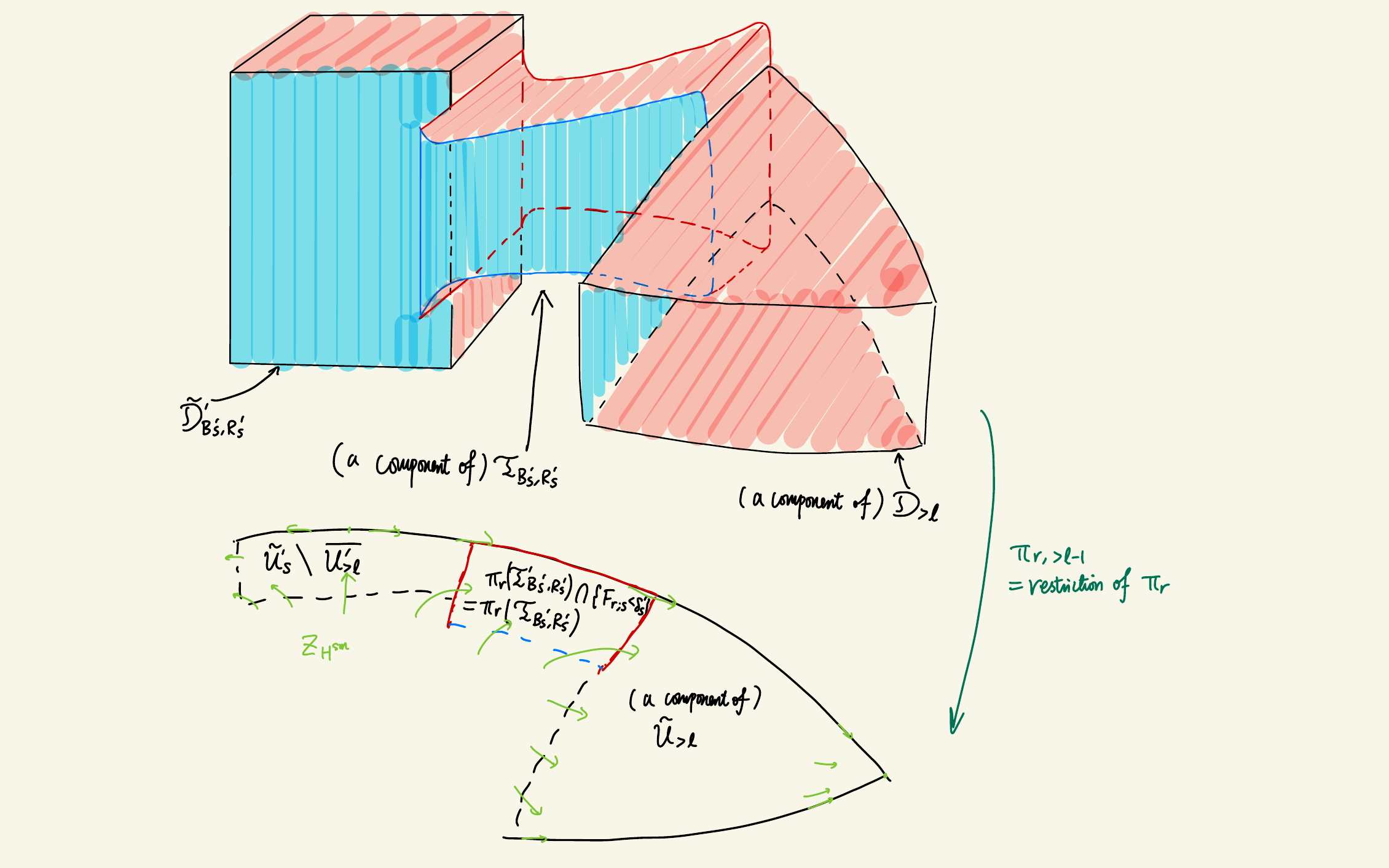} 
\caption{An illustration of the construction of $\pi_{r,>\ell-1}:\fD_{>\ell-1}\rightarrow \widetilde{U}_{>\ell-1}$. The picture shows how (a component of) $\fT_{B_S', R_S'}$ is attached to $\widetilde{\fD}'_{B_S',R_S'}$ and $\fD_{>\ell}$ so that the resulting (portion of) $\fD_{>\ell-1}$ has $(\partial \fD_{>\ell-1})_{in}$ (colored by blue) exactly over $\partial \widetilde{U}_{>\ell-1}$, and $(\partial \fD_{>\ell-1})_{out}$ (colored by red) is its complement in $\partial\fD_{>\ell-1}$. Note that $\pi_r^{-1}(\partial H^{sm})$ is \emph{not} part of the boundary of $\fF$ or $\fD_{>\ell-1}$ (although in the picture it seems so). That is why it is not included in $\partial\widetilde{U}_{>\ell-1}$, and its preimage under $\pi_{r,>\ell-1}$ is not colored. }\label{figure: fD_ell-1}
\end{figure}

(b) It follows directly from part (a) that there is a Morse-Bott type handle decomposition of the Liouville hypersurface $\fF$ for model (A). Namely, over each $U_S^\flat$, there are $|\pi_0(\cZ(L_S))|$ many handles of index $n-|S|+2|S|=n+|S|$ associated with the zero loci of the Liouville vector field, whose core (i.e. ascending manifold of $Z_{\fF}$) is isomorphic to $(\cZ(L_S)_{0,\cpt})\times D^{2|S|}$. Now we show that with appropriate choices of the functions $\epsilon_S$, the handles are (generalized) Weinstein handles. We will examine each of the factors in (\ref{eq: fF_S, splits}). 

Recall the formula of $Z'_{T^*H_{S^\perp}}$ from (\ref{eq: Z'HSperp}). We will choose $\epsilon_S$ so that it is gradient-like for a Morse function in a neighborhood of the fixed point in $T^*H_{S^\perp}$ as follows. By the assumption that $-r^2_{S^\perp}$ has a Morse singularity at $c_S$ which is a local minimum, we can choose the local coordinates $(x_1,\cdots, x_{n-|S|})$ near $c_S$, so that the (lifting of the) vector field $Z_{S^\perp, H^{sm}}$ is of the form $\sum_{j=1}^{n-|S|-1}x_j\partial_{x_j}+\widetilde{v}$, where $|\widetilde{v}|\leq O(|x|^2)$. Let $\epsilon_S=\epsilon$, for a sufficiently small constant $\epsilon>0$, in a small neighborhood of $c_S$. Then 
\begin{align*}
Z'_{T^*H_{S^\perp}}=(1-\epsilon)\sum_{j=1}^{n-|S|-1}y_j\partial_{y_j}+\epsilon\cdot \sum_{j=1}^{n-|S|-1}x_j\partial_{x_j}+Z'', \text{ with }|Z''|\leq O(|x|^2+|y|^2).
\end{align*}
Then $Z'_{T^*H_{S^\perp}}$ is clearly gradient-like for the Morse function $\phi(x,y)=|x|^2+|y|^2$, and so we get a Weinstein $0$-handle. 

For the factor $T^*\cZ(L_S)_{0,\cpt}$, it is a standard generalized Weinstein handle. For the factor $\cU_{S}^\der$, the zeros of $Z_{\cU_{S}^\der}$ are the centers of the Kostant sections of $J_{L_S^\der}$. In a small neighborhood of each of these zeros, we see the attachment of a Weinstein $2|S|$-handle. 
Hence on the product (after the finite quotient by $\cZ(L_S^\der)_0$) we see the attached handles are (generalized) Weinstein handles.  
Therefore, up to deformations, $\fD$ is a Weinstein domain and $\fF$ as its completion is a Weinstein manifold. 

In model (B), for each stratum $\widetilde{\fS}_{(S_j)_j}$ of codimension $k$, we can associate $|\pi_0(\cZ(L_{S_0}))|$ many index $(2n-1-k)$ (generalized) Weinstein handle(s), whose core is given by $\cZ(L_{S_0})_{0,\cpt}\times D^{n-1-k-|S_0|}\times D^{2|S_0|}$. The construction is completely similar to model (A), and we leave the details to the interested reader. 
\end{proof}

\begin{remark}\label{rmk: isotopy fF}
\begin{itemize}
\item[(i)] It is not hard to see that the (generalized) Weinstein structure on $\fF$ from model (A) is (generalized) Weinstein isotopic to that from model (B). The isotopy is essentially governed by a 1-parameter family of $H_t^{sm}$, with $H_0^{sm}$ (resp. $H_1^{sm}$) from model (A) (resp. model (B)), where the function $r^2$ on $H_t^{sm}$ is Morse for all $t$ except for some finite collection $t_1,\cdots, t_k\in (0,1)$, where $r^2$ has birth-death type critical points (cf. \cite[Section 9.1]{CE}). Since in the definition of a Weinstein domain $(X, \phi)$ from \emph{loc. cit.}, one allows $\phi: X\rightarrow\bR$ to have birth-death type of critical points, we get the desired isotopy, which becomes a Weinstein isotopy after a generic deformation.  

\item[(ii)] Fixing $H^{sm}$, for different choices of $(\epsilon_S)$, the Liouville hypersurfaces $\fF$ associated with them are canonically symplectomorphic, hence their Liouville structures are all isotopic (cf. \cite[Section 1]{El}). Also an isotopy of the eligible $H^{sm}$ and $(\epsilon_S)$, gives an isotopy of Liouville hypersurfaces $\fF$. 
\end{itemize}
\end{remark}

\subsubsection{Partial compatifications of $J_G$ and its Lagrangian skeleton}\label{subsec: compactify, J_G}

Fix a Liouville hypersurface $\fF\subset \widetilde{\sfN}^{-1}(1)$ as in Proposition \ref{prop: hypersurface F} above, which is Weinstein up to deformations using appropriate choices in the construction. Let $\tilde{I}: \widetilde{\sfN}^{-1}(1)\rightarrow \bR$ be a function such that 
\begin{align*}
&\tilde{I}|_{\fF}=0,\ X_{\frac{1}{\widetilde{\norm}^2}}(\tilde{I})=1.
\end{align*}
Recall that $\widetilde{\sfN}$ is homogeneous of weight $-\frac{1}{2}$ with respect to the Liouville flow. Then the flow of $X_{\frac{1}{\widetilde{\norm}^2}}$ gives an identification $\widetilde{\sfN}^{-1}(1)\cong \fF\times \bR$ from which we see that $\widetilde{\sfN}^{-1}(1)$ is a contact manifold with contact form $d\tilde{I}+\vartheta_\fF$. 
Furthermore, we have an isomorphism of exact symplectic manifolds (which in particular identifies the respective Liouville flows)
\begin{align}
&(J_G-\cB_1, \vartheta|_{J_G-\cB_1}=-\frac{1}{\widetilde{\sfN}^2}d\tilde{I}+\vartheta_{\fF})\overset{\sim}\longrightarrow (\fF\times T^{*,>0}\bR, -\tau dt+\vartheta_{\fF}),
\end{align}
 where $T^{*, >0}\bR=\{(t, \tau\cdot dt): t\in \bR, \tau>0\}$. Using the exact embedding
 \begin{align*}
 T^{*,>0}\bR&\longhookrightarrow \bC_{\Re z\leq 0}\\
 (t, \tau)&\mapsto (-\tau^{\frac{1}{2}}, -2\tau^{\frac{1}{2}}t), 
 \end{align*}
we can embed $J_G-\cB_1$ into $\fF\times \bC_{\Re z\leq 0}$, which gives
the partial compactification of $J_G$
\begin{align*}
\overline{J}_G:=J_G\underset{J_G-\cB_1}{\coprod}(\fF\times \bC_{\Re z\leq 0}). 
\end{align*}
Moreover, We define the completion of $J_G$ as 
\begin{align*}
\widehat{J}_G:=J_G\underset{J_G-\cB_1}{\coprod}(\fF\times \bC_{z}).
\end{align*}

For later reference, we define the function 
\begin{align*}
&I: J_G-\cB_1\longrightarrow \bR
\end{align*}
determined by the properties that $I|_{\widetilde{\sfN}^{-1}(1)}=-2\tilde{I}$, and it is homogeneous with \emph{weight $\frac{1}{2}$} with respect to the Liouville flow. Then under the embedding of $J_G-\cB_1$ into $\fF\times \bC_{\Re z\leq 0}$, the $\bC_{\Re z\leq 0}$ has coordinate $q=\Re z=I$ and $p=\Im z=-\frac{1}{\widetilde{\sfN}}$.

For example, on a conic (with respect to the Liouville flow) open subset in $\cB_{w_0}\cong T^*T$ whose projection to $\bR^n_{\geq 0}$ is disjoint from an open neighborhood of the codimension $<n$ and $\geq 1$ faces, we can take
\begin{align}
\label{eq: tildeN, B_w0}&\widetilde{\norm}=\pi_r^*(\sum_{\beta\in \Pi}r_\beta)^{\frac{1}{2}}\\
\label{eq: I, B_w0}&I=2\pi_r^*(\sum_{\beta\in \Pi}r_\beta)^{\frac{1}{2}}\sum\limits_{\beta\in \Pi}\Re p_{\beta^\vee}.
\end{align}

Similarly, for any $\emptyset\neq S\subsetneq \Pi$, on a conic open subset in $\fU_S$ whose projection to $\bR^n_{\geq 0}$ is disjoint from an open neighborhood of the faces that have vanishing $r_\beta$ for some $\beta\not\in S$, we can set 
\begin{align}
\label{eq: tildeN, B_w0wS}&\widetilde{\norm}_S=\pi_r^*(\sum_{\beta\not\in S}r_\beta)^{\frac{1}{2}}\\
\label{eq: I, B_w0wS}&I_S=2\pi_r^*(\sum_{\beta\not\in S}r_\beta)^{\frac{1}{2}}\sum\limits_{\beta\not\in S}\Re p_{\beta^\vee_{S^\perp}}.
\end{align}
using (\ref{eq: dual coordinate S}).

Note that for the above choice of $(\widetilde{\norm}, I)$ (resp. $(\widetilde{\norm}_S, I_S)$), the projection of $Z_{\fF}$ to $\fC^{n-1}$ is completely zero in a neighborhood of the center (resp. $Z_{r;S}$  in a neighborhood of the barycenter $c_S$). These come from choosing $\epsilon_S=0$ on some open $\Omega_S$ with $\overline{\Omega}_S\subsetneq U_S'$, for any $S\subsetneq \Pi$.

\begin{prop}\label{prop: partial compactify}
The partial compactification  $\overline{J}_G$ from the given choice of $\fF$ and $\widetilde{\sfN}$ is a Liouville sector, and is a (generalized) Weinstein sector that is obtained from attaching  $|\cZ(G)|$ many critical handles to $\fF\times \bC_{\Re z\leq 0}$ if $\fF$ is a (generalized) Weinstein hypersurface. The completion $\widehat{J}_G$ is a Liouville completion of $\overline{J}_G$, and it is a (generalized) Weinstein manifold that is obtained from attaching  $|\cZ(G)|$ many critical handles to $\fF\times \bC_{z}$ if $\fF$ is (generalized) Weinstein. 
\end{prop}
\begin{proof}
This is obvious from the fact that $\overline{J}_G$ (resp. $\widehat{J}_G$) is obtained from attaching $|\cZ(G)|$-many critical Weinstein handles to $\fF\times \bC_{\Re z\leq 0}$ (resp. $\fF\times \bC_z$). 
\end{proof}

\begin{prop}\label{prop: split generation}
\item[(i)] Using the Weinstein structure on $\fF$ from Proposition \ref{prop: hypersurface F} (b) model (B), 
the Lagrangian skeleton of $\overline{J}_G$ inside $\widehat{J}_G$ has $|\pi_0(\cZ(L_S))|$ many  Lagrangian component(s) for each $S\subset \Pi$. 

\item[(ii)]
The Kostant sections $\{g=z\}, z\in \cZ(G)$ generate the partially wrapped Fukaya category of the Weinstein sector $\overline{J}_G$. 
\end{prop}
\begin{proof}
(i) The Lagrangian component(s) for each $S\subset \Pi$ is given by: 
\begin{itemize}
\item If $S=\Pi$, then $\chi^{-1}([0])$ gives $|\cZ(G)|$ many Lagrangian components in the skeleton, and the Kostant sections give their cocores;\\
\item If $S\subsetneq \Pi$, then it gives $|\pi_0(\cZ(L_S))|$ many Lagrangian components in $\Core(\fF)$, the product of which with $\bR_{\geq 0}$ gives the same amounts of Lagrangian components in the skeleton of $\overline{J}_G$. Let $H_{S^\perp}\subset \cZ(L_S)_{0,\bR}$ be the (non-closed) smooth hypersurface that projects to $\sfN_S\cap H^{sm}$ isomorphically, under taking the components in $\pi_r$ for $r_\beta,\beta\not\in S$ (defined in the same way as for model (A)), and let $c'_S$ be the unique fixed point of the Liouville flow. Then 
the Lagrangian
\begin{align*}
T^* H_{S^\perp}|_{\{c'_{S}\}}\times T^*\cZ(L_S)_{\cpt}|_{\{z^c_S\}}\times (\{g=I\}\subset J_{L_S^\der})\times \{p=1\}\subset\fF_S\times \bC_z,
\end{align*}
for any fixed $z^c_S$ in each component of $\cZ(L_S)_{\cpt}$ gives a cocore\footnote{To be more precise, we need to replace $\{p=1\}$ by a cylindrical Lagrangian in $\bC_z$, which is similar to that on the right-hand-side of  Figure \ref{figure: proj L_xi, R}.} of the corresponding Lagrangian component. 
\end{itemize}

(ii) By \cite{GPS2}, the above Lagrangian cocores genearate $\cW(\overline{J}_G)$. 
On the other hand, for any Lagrangian cocore corresponding to $S\subsetneq \Pi$, we can perform a Hamiltonian isotopy on the $\bC_z$ factor which pushes $\{p=1\}$ away from $\bR_{\geq 0}$, and so moves the cocore away from $\Core(\fF)\times \bR_{\geq 0}$ and makes it intersect $\chi^{-1}([0])$ only.  Therefore, such cocores are generated by the Kostant sections, using wrapping exact triangles in the stabilization $\overline{J}_G\times T^*[0,1]$ from \cite{GPS2} (see also the proof of Proposition \ref{prop: A_G commutative} (i) in Subsection \ref{sec: proof_prop5.2.3} for more details in a similar situation). 
\end{proof}

\begin{remark}
\item[(i)] Following Remark \ref{rmk: isotopy fF}, any Liouville/Weinstein isotopy of $\fF$ will induce in an obvious way Liouville/Weinstein isotopy of $\overline{J}_G$ as sectors. 

\item[(ii)]We remark on some obvious relations between $\overline{J}_G$ and the log-compactification $\overline{J}^{\log}_G$ (\ref{eq: J, log, def}) for $G$ of adjoint type. The smooth function $1/\widetilde{\sfN}$ extends to $\overline{J}^{\log}_G$ by $0$, and defines a decreasing sequence of tubular neighborhoods $\{1/\widetilde{\sfN}<\frac{1}{j}\}, j>1$ (with smooth boundary) of the log-boundary divisor $\partial\overline{J}^{\log}_G$, which are related by the contracting $\bR_+$-flow and whose intersection is equal to the boundary divisor. An alternative way to see the normal crossing divisor $\partial\overline{J}^{\log}_G$ is as follows. Let $\cC_S$ be the ascending manifold of $\overline{w}_0\overline{w}_SB\in \overline{\chi^{-1}([0])}$ in $\overline{J}^{\log}_G$ with respect to the contracting $\bC^\times$-flow, whose union over all $S\subset \Pi$ gives the Bialynicki-Birula decomposition of $\overline{J}^{\log}_G$ (cf. \cite{Balibanu}). On $\cB_{w_0w_S}\cong \Sigma_{I;S}\times T^*\cZ(L_S)$, where $\Sigma_{I;S}\subset J_{L_S^\der}$ is the Kostant section associated to $g_S=I$, define 
\begin{align*}
\sfF_{S,\beta}(g_S=I, \xi_S;z, t):=\beta(z^{-1}),\beta\in \Pi\backslash S,
\end{align*}
which extend to be affine coordinates (completed by a choice of affine coordinates on $\cS_{\fl_S^\der}$ and $\fz_S^*$) on the affine space $\cC_S$. 
The zero locus of $\prod\limits_{\beta\in \Pi\backslash S}\sfF_{S,\beta}$ gives the normal crossing divisor inside $\cC_S$. 
\end{remark}

\section{The Wrapping Hamiltonians and one calculation of wrapped Floer cochains}\label{sec: wrapping Ham}

In this section, we calculate the wrapped Floer complexes for the Kostant sections. Throughout this section, we assume $G$ is semisimple. The main results are 
Proposition \ref{prop: A_G, vector} and \ref{prop: A_G, center}, which show that the Floer complexes are all concentrated in degree zero, and the generators are indexed by the \emph{dominant} coweights of $T$ for $G$ of adjoint form. 

We mention a few basic set-ups for the wrapped Fukaya category of $J_G$, and give some references on the foundations of Fukaya categories instead of going into any detail of them. To set up gradings for Lagrangians in $\cW(J_G)$, we need to choose a compatible almost complex structure $\cJ$ and trivialize the square of the canonical bundle $\kappa^{\otimes 2}$. For this, we use that $J_G$ is hyperKahler and let $\cJ$ be the complex structure that is compatible with the real part of the present holomorphic symplectic form on it. Since $(J_G, \cJ)$ is again holomorphic symplectic using the hyperKahler rotated holomorphic symplectic form, $c_1(TJ_G)=0$ and we can trivialize $\kappa$ (hence $\kappa^{\otimes 2}$) by the top exterior power of this holomorphic symplectic form. Using this, holomorphic Lagrangians all have constant integer gradings (cf. \cite[Proposition 5.1]{Jin1}). We remark that since the choice of a grading for a (smooth) Lagrangian is completely topological, we usually don't stick to a single $\cJ$ or trivialization of $\kappa^{\otimes 2}$. 

For a friendly introduction of Fukaya categories, we refer the reader to \cite{Auroux}. For the foundations of Fukaya categories, we refer the reader to \cite{Seidel1}. For the more recent development of partially wrapped Fukaya categories on Liouville/Weinstein sectors, we refer the reader to \cite{GPS1, GPS2, Sylvan}.

\subsection{Choices of wrapping Hamiltonians}\label{subsec: Hamiltonians}

The Killing form on $\fg$ induces a $W$-invariant Hermitian inner product on $\ft^*\cong \ft$, namely 
\begin{align*}
\lng \xi, \eta\rng_{\text{Herm}}:=\lng\xi,\overline{\eta}\rng, 
\end{align*}
and let $\|\xi\|^2$ (or $|\xi|^2$) be $\lng \xi,\overline{\xi}\rng_{\text{Herm}}$. 
For any $R>0$, let $y_R: [-1,\infty)\rightarrow \bR$ be any smooth function such that 
\begin{align}\label{eq: condition on y}
y_R(x)=\begin{cases}&\frac{1}{2}x^2,\ x\leq R, \\
&\frac{1}{2}R x, x>2R. 
\end{cases}
\end{align} 
Let $\pi_\ft: \ft\cong \ft^*\longrightarrow \fc$ denote the quotient map.

Let $(\sigma_1,\cdots, \sigma_n)$ be a set of \emph{homogeneous} complex affine coordinates on $\fc$ with respect to the induced $\bC^\times$-action from the weight $1$ dilating action on $\ft$.  Let $u_1, \cdots, u_n$ be the respective weights of the affine coordinates, which are all positive integers. Let $\tilde{u}:=\max\{u_1,\cdots, u_n\}+1$.

Assume $f(\xi)$ is any $W$-invariant homogeneous smooth (real-valued) function with weight $2$ on $\ft$ such that $f|_{\ft-\{0\}}>0$ and $f(\xi)$ descends to a $C^2$-function on $\fc-\{[0]\}$. 
For any $\delta>0$ small, let 
\begin{align}\label{eq: p_tilde u, delta}
p_{\tilde{u},\delta}: \bR_{\geq 0}\rightarrow \bR_{\geq 0}
\end{align}
be a smooth function such that (1) $0<p'_{\tilde{u},\delta}(s)\leq 2$, for $s>0$; (2) $p_{\tilde{u},\delta}(s)=s, s\in [3\delta,\infty)$, and $p_{\tilde{u},\delta}(s)=s^{\tilde{u}}, s\in [0, \delta)$. Then $p_{\tilde{u},\delta}\circ f$ is a $W$-invariant $C^2$-function on $\ft$ that descends to a $C^2$-function\footnote{If $f$ descends to a $C^k$-function on $\fc-\{0\}$, then by sufficiently increasing $\tilde{u}$, we can make $\tilde{f}_{\tilde{u},\delta}$ a $C^k$-function as well.} on $\fc$, denoted by $\tilde{f}_{\tilde{u}, \delta}: \fc\rightarrow \bR_{\geq 0}$. Note that $[0]\in \fc$ is the only critical point (which is a global minimum) of $\tilde{f}_{\tilde{u}, \delta}$. We can always perturb $\tilde{f}_{\tilde{u}, \delta}$ a little bit near $[0]$ so that $[0]$ is a non-degenerate global minimum, without introducing new critical points. 
Moreover, we have 
\begin{align}
\|D(p_{\tilde{u},\delta}\circ f)(\xi)\|\leq 2\|D f(\xi)\|, \text{ on }\{\xi\in \ft: f(\xi)\leq 3\delta\}. 
\end{align}

Now we describe the induction steps to define a smooth $W$-invariant function $\widetilde{F}$ that descends to a $C^2$-function $F$ on $\fc$, and which will serve (after some modifications) as a collection of  desired positive wrapping Hamiltonian functions on $J_G$. Let $\fS_{\ft}$ be the standard stratification on $\ft$ indexed by $S\subset \Pi$, with each stratum $\fz_S^\circ$ consisting of points whose stabilizer under the $W$-action is equal to $W_S=N_{L_S}(T)/T$. For any $S\subset \Pi$, let 
\begin{align}\label{eq: U_S, epsilon}
U_{S,\epsilon}:=\{\xi\in \ft: \|\xi-\proj_{\fz_S}\xi\|<\epsilon\cdot \|\proj_{\fz_S}\xi\|, \proj_{\fz_S}\xi\in \fz_S^\circ\}
\end{align}
be a $\bC^\times$-invariant tubular neighborhood of $\fz_S^\circ$. In each of the following steps, we will choose some $\epsilon_j>0, j=1,\cdots, n$, sufficiently small, such that 
\begin{align}\label{eq: mathring U_S, epsilon}
\mathring{U}_{S, \epsilon_{|S|}}:=U_{S, \epsilon_{|S|}}-\bigcup\limits_{S'\supsetneq S}U_{S',\frac{1}{2}\epsilon_{|S'|}}
\end{align}
 are all disjoint for any pair of $S$ without any containment relation. 

\emph{Step 1. } The base case $F_{\leq 0}$ on $\ft^{\reg}$.

We start with the function $F_{\leq 0}(\xi):=\|\xi\|^2$ on $\ft^\reg$. It is clear that the function $F_{\leq 0}$ descends to a smooth function on $\fc^{\reg}=\ft^{\reg}/W$. 

\emph{Step 2.} Assumptions on the $j$-th step function $F_{\leq j}$. 

Suppose we have defined $F_{\leq j}$ on $\ft_{\leq j}:=\ft-\bigcup\limits_{|S|> j}U_{S, \frac{1}{2}\epsilon^{(j)}_{|S|}}$, for some choice of $(\epsilon^{(j)}_k)_{k=1,\cdots, n}$ as above, such that $F_{\leq j}$ is a $W$-invariant homogeneous $C^2$-function with weight $2$ and the following hold: 
\begin{itemize}
\item[(i)] For any $S$ with $|S|\leq j$,  on $\mathring{U}_{S,\epsilon^{(j)}_{|S|}}$ we have 
\begin{align}\label{eq: F_leq j, sum}
F_{\leq j}(\xi)=\|\proj_{\fz_S}\xi\|^2(1+f_S(\frac{\xi-\proj_{\fz_S}\xi}{\|\proj_{\fz_S}\xi\|})),
\end{align}
for some smooth function\footnote{Here we only need $f_S$ in the $\epsilon^{(j)}_{|S|}$-neighborhood of $0\in \ft_S$.} $f_S: \ft_S\rightarrow \bR_{\geq 0}$ that descends to a smooth function on $\fc_{S}:=\ft_S\sslash W_S$. In particular, $F_{\leq j}$ descends to a $C^2$-function on $\fc_{\leq j}:=\ft_{\leq j}/W$ (the image of $\ft_{\leq j}$ under $\pi_\ft: \ft\rightarrow \fc$). 

\item[(ii)] The function $f_S: \ft_S\rightarrow \bR_{\geq 0}$ satisfies $f_S(0)=0$ and $f_S>0$ on $\ft_S-\{0\}$. Let $Z_S$ be the radial vector field on $\ft_S$, i.e. the vector field generating the weight $1$ $\bR_+$-action. Then $\iota_{Z_S}df_S>0$ on $\ft_S-\{0\}$. 
In particular, this implies that the origin is the only critical point (global minimum) of $f_S$. For the induced function $\tilde{f}_S: \fc_S\rightarrow \bR_{\geq 0}$, we require that $[0]$ is a non-degenerate critical point. 

\item[(iii)]  
\begin{align}\label{eq: norm Df_S}
\|Df_S\|\leq 4^j\|D(\|\xi_S\|^2)\|.
\end{align} 
\end{itemize}

\emph{Step 3. }Modifying and extending $F_{\leq j}$ to $F_{\leq j+1}$.

For any $S$ with $|S|=j+1<n$, 
consider the following intersection:
\begin{align*}
\cT_{S, \epsilon_{j+1}^{(j)}}:=U_{S,\epsilon^{(j)}_{j+1}}\cap \{\xi\in \ft:\|\proj_{\fz_S}\xi\|=1\}\cap  \bigcup\limits_{S_1\subsetneq S}\mathring{U}_{S_1, \epsilon^{(j)}_{|S_1|}}.
\end{align*}
By the requirement on $F_{\leq j}$ (\ref{eq: F_leq j, sum}), we have for any $S_1\subsetneq S$
\begin{align}
\nonumber&F_{\leq j}(\xi)|_{\cT_{S, \epsilon^{(j)}_{j+1}}\cap \mathring{U}_{S_1,\epsilon^{(j)}_{|S_1|}}}\\
\label{eq: F_leq j, cT}&=(1+\|\proj_{\fz_{S_1}}\xi-\proj_{\fz_{S}}\xi\|^2) (1+f_{S_1}(\frac{\xi-\proj_{\fz_{S_1}}\xi}{\sqrt{1+\|\proj_{\fz_{S_1}}\xi-\proj_{\fz_{S}}\xi\|^2}})). 
\end{align}
In particular, $F_{S}:=F_{\leq j}|_{\cT_{S, \epsilon^{(j)}_{j+1}}}-1$, which only depends on $\xi-\proj_{\fz_S}\xi$ but not on $\proj_{\fz_S}(\xi)$ (here we use that $\xi-\proj_{\fz_{S}}\xi$ is equal to $\xi-\proj_{\fz_{S_1}}\xi$ plus $\proj_{\fz_{S_1}}(\xi-\proj_{\fz_{S}}\xi)=\proj_{\fz_{S_1}}\xi-\proj_{\fz_{S}}\xi$), descends to a smooth positive function defined on an open ``annulus" $\{\epsilon^{(j)}_{j+1}/2<\|\xi_S\|<\epsilon^{(j)}_{j+1}\}$ around the origin in $\ft_{S}$, satisfying $\iota_{Z_S}dF_S>0$. 

Now modify $F_{S}$ inside $\{\epsilon^{(j)}_{j+1}/2<\|\xi_S\|<(2/3)\epsilon^{(j)}_{j+1}\}$, extend it to be homogeneous with weight $2$ (or better modify its induced function on a portion of $\fc_S$) on a small neighborhood of the origin in $\ft_S$, then compose it with $p_{\tilde{u}, \delta}$ (\ref{eq: p_tilde u, delta}) for appropriate $\tilde{u}$ and $\delta>0$.  The resulting function is denoted by $f_S$, and it is clear that, with some careful choices,  $f_S$ satisfies all the conditions in \emph{Step 2}. Note that we can always perturb $\tilde{f}_S: \fc_S\rightarrow \bR_{\geq 0}$ a little bit near $[0]$ so that it becomes $C^2$-smooth and $[0]$ is a non-degenerate minimum, without creating new critical points. 

Lastly, define $F_{\leq {j+1}}(\xi)$ on $\mathring{U}_{S, \epsilon_{|S|}^{(j)}}, |S|=j+1$ by the formula in (\ref{eq: F_leq j, sum}). Since it matches with $F_{\leq j}$ near the boundary of $\cT_{S,\epsilon_{j+1}^{(j)}}$, it extends $F_{\leq j}$ (restricted to a smaller domain) to a desired function on $\ft_{\leq j+1}$, for some new choices of $(\epsilon^{(j+1)}_{k})_{k=1,\cdots, n}$.

For the case when $j+1=n$, the modification is simpler: we directly modify $F_{\leq n-1}$ inside $\{\epsilon^{(n-1)}_{n}/2<\|\xi\|<(2/3)\epsilon^{(n-1)}_{n}\}$ to get $F_{\leq n}$ as above. 
In the end, we set $\widetilde{F}:=F_{\leq n}$ on $\ft$, and this finishes the induction step. Let $F$ be the induced function on $\fc$.

Define 
\begin{align}
\label{eq: tilde H_R def}&\widetilde{H}_R:=y_R\circ \sqrt{\widetilde{F}}: \ft\longrightarrow \bR_{\geq 0}\\
\label{eq: H_R def}&H_R:=y_R\circ \sqrt{F}: \fc\longrightarrow \bR_{\geq 0}
\end{align}
It is clear that both $\widetilde{H}_R$ is smooth and $H_R$ is $C^2$-smooth on their respective defining domains. By some abuse of notations, we will denote their respective pullback functions on $T^*T$ and $J_G$ by the same notations. Since $J_G\rightarrow\fc$ is a complete integrable system, the Hamiltonian flows of $H_R$ on $J_G$ are complete. 

\begin{definition}\label{def: positive Ham}
Assume a Liouville sector $\overline{X}$ has an increasing sequence of Liouville subsectors $\overline{X}_k\subset X, k\geq 1$ such that $\bigcup\limits_{k}\overline{X}_k=X$ (the interior of $\overline{X}$). We say a Hamiltonian function $H:X\rightarrow\bR$, whose Hamiltonian flows are complete, is (\emph{nonnegative}/\emph{positive}) \emph{linear} if each $H|_{\overline{X}_k}, k\geq 1$ is (nonnegative/strictly positive) homogeneous of weight $1$ with respect to the Liouville flow outside a compact region in $\overline{X}_k$. 
\end{definition}

\begin{remark}\label{remark: positive Ham}
\item[(i)] Strictly speaking, by the definition of a linear Hamiltonian on a Liouville sector $\overline{X}$ in \cite{GPS1}, one needs the Hamiltonian and its differential to vanish along $\partial\overline{X}$. 
In the setting of Definition \ref{def: positive Ham}, 
we can extend $H|_{\overline{X}_k}$ to be $H_k: \overline{X}\rightarrow \bR$ which vanishes in a neighborhood of $\partial\overline{X}$. Given any cylindrical $L\subset X$, for any $t\in \bR$, define $\varphi^t_{X_H}(L):=\varphi^t_{X_{H_k}}(L)$ for $k\gg 1$, which is well defined and obviously stabilizes by the completeness of the Hamiltonian flows of $H$. In particular, the argument in \cite[Lemma 3.28]{GPS1} still works with $\text{Ham}(\overline{X})$ replaced by $\text{Ham}(X)$ consisting of linear Hamiltonian functions with complete Hamiltonian flows in the sense of the above definition. \\

\item[(ii)] The Liouville sectors $\overline{J}_G$ and $T^*\overline{M}$ for a smooth compact manifold $\overline{M}$ with boundary both satisfy the conditions in Definition \ref{def: positive Ham}. The latter is easy to see. For $\overline{J}_G$, this follows from Proposition \ref{prop: proper b map} and the handle attachment description in Proposition \ref{prop: partial compactify}. By the notations from Subsection \ref{subsubsec: subsector} below, we can form $(\overline{J}_G)_k=J_G-\fF\overset{\triangle}{\times}\mathring{\cP}_k$, for a decreasing sequence $\cP_k$ such that $\bigcap\limits_k\cP_k=\emptyset$. Then it is clear that $H_R$ is a positive linear Hamiltonian on $J_G$. \\

\item[(iii)]  From (i) and (ii), we see that $\cW(\overline{J}_G)$ is independent of the choice of compactifications presented in Section \ref{subsec: partial compact}, i.e. for any two compactifications using different choices of $(\widetilde{\sfN}, I)$ from Subsection \ref{subsec: tilde N, I}, the wrapped Fukaya categories are canonically equivalent, since all the calculations of the Fukaya categories are performed within $J_G$ in the exactly same way.

\end{remark}

\subsection{One calculation of wrapped Floer cochains}\label{subsec: Floer cochains}

Let $G$ be an adjoint group. Let $\Sigma_I$ denote for the (only) Kostant section. In this subsection, we calculate $\Hom(\Sigma_I, \Sigma_I)$ using the Hamiltonians defined in (\ref{eq: H_R def}). The idea is to use the Lagrangian correspondence (\ref{eq: Lag corresp}) to transform the wrapping process in $J_G$ to a wrapping process in $T^*T$, with the latter easier to understand. Indeed, since the Hamiltonian function $-\proj_1^*H_R+\proj_2^*(\widetilde{H}_{R})$ on $J_G^a\times T^*T$ vanishes on the Lagrangian subvariety $J_G\underset{\fc}{\times}\ft$, we have the Lagrangian correspondence equivariant with respect to the Hamiltonian flow $\varphi_{H_R}^s$ on $J_G$ and $\varphi_{\widetilde{H}_R}^s$ on $T^*T$.  For any Lagrangian $L\subset J_G$, let $\widehat{L}$ be the transformation under  (\ref{eq: Lag corresp}), e.g. $\widehat{\Sigma}_I=T_{\{I\}}^*T$. 
Then we have 
\begin{align}\label{eq: equivariant flow H_R}
\widehat{\varphi_{H_R}^1(L)}=\varphi_{\widetilde{H}_{R}}^1(\widehat{L}),
\end{align}
for any $L\subset J_G$. 

Implicitly in the definitions (\ref{eq: tilde H_R def}), (\ref{eq: H_R def}) are the choices of $(\epsilon_k^{j})_{k=1,\cdots, n}$. In the following, we assume $\widetilde{H}_R, R\geq 0$ (resp. $H_R$) as $R$ increases satisfies that the choices of $(\epsilon_k^{j})_{k=1,\cdots, n}$ depending on $R$ have limit values $0$.

\begin{lemma}
For any cylindrical Lagrangian $L$, the Lagrangians $\{\varphi_{H_{R}}^1(L)\}_{R\geq 0}$ is cofinal in the wrapping category $(L\rightarrow -)^+$ (in the sense of \cite[Section 3.4]{GPS1}).
\end{lemma}
\begin{proof}
Note that $\varphi_{H_R}^1$ on $\partial^\infty J_G$ is the same as the time $R$ map of the positive contact flow induced by the linear Hamiltonian $\frac{1}{2}\sqrt{F}$ on its symplectization. So the lemma follows from the argument in \cite[Lemma 3.28]{GPS1}.  
\end{proof}

\begin{prop}\label{prop: A_G, vector}
Assume $G$ is of adjoint type. For a sequence of $R_n\rightarrow \infty$, the intersections of $\varphi_{H_{R_n}}^1(\Sigma_I)$ and $\Sigma_I$ are all transverse and are in degree $0$. Morover, as $R\rightarrow \infty$, the intersection points are naturally indexed by the dominant coweights of $T$. 
\end{prop}
\begin{proof}
Using (\ref{eq: equivariant flow H_R}), we just need to examine the intersection points $\varphi_{\widetilde{H}_{R}}^1(\widehat{\Sigma}_I)\cap \widehat{\Sigma}_I$ and understand their corresponding intersection points in $J_G$. 

By construction, given any $(\epsilon_j)_{j=1,\cdots, n}$ and $M\gg 1$, for any $S\subset \Pi$,  over $\mathring{U}_{S,\epsilon_{|S|}}\cap \{\|\xi\|\leq M\}\subset \ft$ (\ref{eq: mathring U_S, epsilon}), we have 
the intersections $\varphi_{\widetilde{H}_R}^1(\widehat{\Sigma}_I)\cap \widehat{\Sigma}_I$ stabilize for $R\rightarrow\infty$. Using the form of $\widetilde{H}_R$ in (\ref{eq: norm Df_S}) and the assumptions on $f_S$, we can conclude that the intersection points there are naturally indexed by $\mathring{U}_{S,\epsilon_{|S|}}\cap \{\|\xi\|\leq M\}\cap X_*(T)$. Now transforming these intersection points to $J_G$ using the opposite Lagrangian correspondence  (\ref{eq: Lag corresp}), and using the non-degeneracy of the minimum of $\widetilde{f}_S: \fc_S\rightarrow\bR_{\geq 0}$, we can conclude that all intersections are transverse and have degree $0$, and they are naturally indexed by the dominant coweights $X_*(T)^+$. The proposition thus follows. 
\end{proof}

We can do a similar calculation for any semisimple $G$ with center $\cZ(G)$. For any $z\in \cZ(G)$, let $\mu^\vee(z)$ be any coweight representative of $z$ under the canonical isomorphism $X_*(T_\ad)/X_*(T)\cong \cZ(G)$. 
\begin{prop}\label{prop: A_G, center}
Assume $G$ is semisimple. Let $z_1, z_2\in \cZ(G)$. For a sequence of $R_n\rightarrow \infty$, the intersections of $\varphi_{H_{R_n}}^1(\Sigma_{z_1})$ and $\Sigma_{z_2}$ are all transverse and are in degree $0$. Morover, as $R\rightarrow \infty$, the intersection points are naturally indexed by $(\mu^\vee(z_1)-\mu^\vee(z_2)+X_*(T))\cap X_*^+(T_\ad)$. 

\end{prop}

\begin{proof}
The proof is very similar to that of Proposition \ref{prop: A_G, vector}. Here we first look at $\varphi_{\widetilde{H}_{R}^1(\widehat{\Sigma}_{z_1})}\cap \widehat{\Sigma}_{z_2}$, and then transform back to $\varphi_{H_{R_n}}^1(\Sigma_{z_1})\cap \Sigma_{z_2}$. The intersection points  $\varphi_{\widetilde{H}_{R}^1(\widehat{\Sigma}_{z_1})}\cap \widehat{\Sigma}_{z_2}$ as $R\rightarrow \infty$ are exactly indexed by $\mu^\vee(z_1)-\mu^\vee(z_2)+X_*(T)$. Transforming the intersection points to $J_G$ gives $(\mu^\vee(z_1)-\mu^\vee(z_2)+X_*(T))\cap X_*^+(T_\ad)$.  
\end{proof}

\section{Homological mirror symmetry for adjoint type $G$}\label{sec: HMS adjoint}

For $z\in \cZ(G)$, let $\Sigma_z$ denote for the Kostant section $\{g=z\}$ (which makes sense for all reductive $G$). In particular, $\Sigma_I$ is the Kostant section $\{g=I\}$.  For $G$ of adjoint type, let
\begin{align*}
\cA_G:=End(\Sigma_I)^{op}. 
\end{align*}
From now on, we will work with ground field $\bC$. The calculation in Proposition \ref{prop: A_G, vector} says that $\cA_G$ is isomorphic to $\bC[T^\vee\sslash W]$ as a \emph{vector space}. 
In this section, we prove the main theorem for $G$ of adjoint type:
\begin{thm}\label{thm: sec G adjoint}
Assume $G$ is of adjoint type. There is an algebra isomorphism $\cA_G\cong \bC[T^\vee\sslash W]$ yielding the HMS result:
\begin{align*}
\cW(J_G)\simeq \Coh(T^\vee\sslash W). 
\end{align*} 
\end{thm}
Recall that $\cW(J_G)$ is generated by $\Sigma_I$ (cf. Proposition \ref{prop: split generation}), so the only remaining nontrivial part of Theorem \ref{thm: sec G adjoint} is the isomorphism $\cA_G\cong \bC[T^\vee\sslash W]$. The proof of this isomorphism occupies the last two sections. It uses the functorialities of wrapped Fukaya categories under inclusions of Liouville sectors, developed in \cite{GPS1, GPS2}.

\subsection{Statement of main propositions}\label{subsec: key propositions}

From the Weinstein handle attachment description of $J_G$ in Section \ref{sec: skeleton, sector}, we see that the inclusion $\cB_{w_0}\cong T^*T\hookrightarrow J_G$ 
restricted to a Liouville subsector $\cB_{w_0}^\dagg\simeq T^*\overline{T}$ (with isotopic sector structures), 
gives an inclusion of Liouville sectors (see Subsection \ref{subsubsec: conic} for the precise formulation). Thus we have the restriction (right adjoint) and co-restriction (left adjoint) functors as adjoint pairs on the (large) dg-categories 
\begin{equation}\label{eq: res, co-res, large}
\begin{tikzcd}[arrow style=tikz,>=stealth,row sep=4em]
\cA_G-\Mod \arrow[rr, shift left=.4ex, "res"]
  &&\bC[T^\vee]-\Mod\ar[ll, shift left=.4ex, "co\text{-}res"],
\end{tikzcd}
\end{equation}
where $co\text{-}res$ preserves compact objects (i.e. perfect modules).

\begin{prop}\label{prop: A_G commutative}
For $G$ of adjoint type,  we have the following. 
\begin{itemize}
\item[(i)] The co-restriction functor is given by an $\cA_G-\bC[T^\vee]$-bimodule $\cM$ that is isomorphic to $\cA_G^{\oplus |W|}$ (resp. $\bC[T^\vee]$) as a left $\cA_G$-module (resp. right $\bC[T^\vee]$-module).

\item[(ii)]  The restriction functor sends $\cA_G$ to $\bC[T^\vee]$. In particular, we have  
\begin{equation}\label{eq: res, co-res}
\begin{tikzcd}[arrow style=tikz,>=stealth,row sep=4em]
\cA_G-\Perf \arrow[rr, shift left=.4ex, "res"]
  &&\bC[T^\vee]-\Perf\simeq \cW(\cB^\dagg_{w_0}).\ar[ll, shift left=.4ex, "co\text{-}res"]
\end{tikzcd}
\end{equation}

\item[(iii)] The algebra $\cA_G$ is embedded as a subalgebra of $\bC[T^\vee]$, hence commutative. 

\item[(iv)] The (commutative) algebra $\cA_G$ is finitely generated. 
\end{itemize}
\end{prop}

\begin{prop}\label{prop: A_G, W-inv}
\item[(i)] The restriction and co-restriction functors in (\ref{eq: res, co-res}) can be identified as the $!$-pullback and pushfoward functors respectively on the (bounded) dg-category of coherent sheaves for a map of affine varieties
\begin{align}\label{eq: f, schemes}
\mathsf{f}: T^\vee\longrightarrow \Spec \cA_G.
\end{align}

\item[(ii)] The map $\mathsf{f}$ (\ref{eq: f, schemes}) is $W$-invariant.
\end{prop}

Assuming Proposition \ref{prop: A_G commutative} and Proposition \ref{prop: A_G, W-inv}, we can give a direct proof of Theorem \ref{thm: sec G adjoint}. 

\begin{proof}[Proof of Theorem \ref{thm: sec G adjoint}]
Since $\mathsf{f}$ from (\ref{eq: f, schemes}) is $W$-invariant, it factors as
\begin{align*}
\mathsf{f}:T^\vee\longrightarrow T^\vee\sslash W\overset{\hat{\mathsf{f}}}{\longrightarrow} \Spec \cA_G.
\end{align*}
By Proposition \ref{prop: A_G commutative} and the Pittie--Steinberg Theorem (cf. \cite{Steinberg}, \cite [Theorem 6.1.2]{CG}), we have isomorphisms
\begin{align*}
\mathsf{f}_*\cO_{T^\vee}\cong (\hat{\mathsf{f}}_*\cO_{T^\vee\sslash W})^{\oplus |W|}\cong \cO_{\Spec\cA_G}^{\oplus |W|}.
\end{align*}
So $\hat{\mathsf{f}}_*\cO_{T^\vee\sslash W}$ is a line bundle on $\Spec \cA_G$, which on the other hand must be trivial, i.e. 
$\hat{\mathsf{f}}_*\cO_{T^\vee\sslash W}\cong \cO_{\Spec\cA_G}$. Hence, $\hat{\mathsf{f}}$ is an isomorphism, and the theorem follows. 

\end{proof}

We will give the proof of Proposition \ref{prop: A_G commutative} and \ref{prop: A_G, W-inv} in Section \ref{subsec: proof prop}. The key technical results for the proof are Proposition \ref{prop: L_0, part 2} and   \ref{prop: L_0, W} below, whose proof will be provided in the same section. Since the motivation for the latter results comes from a relatively easier calculation for certain non-exact Lagrangians, with coefficients in the Novikov field, we will first state the non-exact version in Proposition \ref{prop: L_xi, S_e, part 1}. Although it is not logically necessary for the proof of the main theorem, it gives the geometric intuition, and the techniques in its proof in Subsection \ref{subsec: proof non-exact} will be used for the proof of the exact version.

Let $\sfLambda=\{\sum_{j=0}^\infty a_j\sfq^{\gamma_j}: a_j\in \bC, \gamma_j\in \bR, \gamma_j\rightarrow\infty\}$ be the Novikov field over $\bC$. Let $\cW(J_G;\sfLambda)$ be the wrapped Fukaya category linear over $\sfLambda$ consisting of tautologically unobstructed, tame and asymptotically cylindrical Lagrangian branes (equipped with local systems\footnote{In general, one allows $\sfLambda$-local systems with unitary monodromy. Here we restrict to a simpler situation.} induced from finite rank local systems over $\bC$). 
When writing the morphism space between two Lagrangian objects, if a Lagrangian (brane) does not come with a local system, we mean the underlying local system is the trivial rank 1 local system.  In the following, we fix the grading on $\Sigma_I$ to be the constant $n=\dim_\bC T$ (cf. \cite{Jin1} for the constant property of gradings on a holomorphic Lagrangian). Since $\Sigma_I$ is contractible, the Pin structure is uniquely assigned.

\begin{prop}\label{prop: L_xi, S_e, part 1}
Assume $G$ is of adjoint type. For any $\zeta\in \ft^\reg_c\cong i\ft_\bR^\reg$, there exists a non-exact Lagrangian brane $\cL_{\zeta}\in \cW(\cB_{w_0}^\dagg;\sfLambda)$, with the projection $\pi_\zeta: \cL_\zeta\rightarrow T$ a homotopy equivalence and $(\pi_\zeta^*)^{-1}[\alpha_{J_G}|_{\cL_\zeta}]=\zeta\in H^1(T, \bC)\cong \ft^*$, 
such that 
\begin{itemize}
\item[(i)] The object $(\cL_\zeta, \check{\rho})\in \Perf_\Lambda(\cW(\cB_{w_0}^\dagg;\sfLambda))\simeq \Perf_\sfLambda(\bC[T^\vee]\underset{\bC}{\otimes}\sfLambda)$ corresponds to the simple module $\bC[T^\vee]\underset{\bC}{\otimes}\sfLambda/(x^{\lambda_\alpha^\vee}-\lambda_\alpha^\vee(\check{\rho})\cdot \sfq^{i\lambda_\alpha^\vee(\zeta)}: \alpha\in \Pi)$, up to some renormalization $\sfq\mapsto \sfq^c$, for some fixed constant $c\in \bR^\times$. 

\item[(ii)] Viewing $(\cL_{\zeta},\check{\rho})$ as an object in $\cW(J_G;\sfLambda)$, we have 
\begin{align}
\label{eq: prop skyscraper 1}&\Hom_{\cW(J_G; \sfLambda)}((\cL_{\zeta},\check{\rho}), \Sigma_I)\cong \sfLambda [-n]\\
\label{eq: prop skyscraper 2}&\Hom_{\cW(J_G;\sfLambda)}(\Sigma_I, (\cL_{\zeta},\check{\rho}))\cong \sfLambda.
\end{align} 

\item[(iii)] For any two objects $(\cL_{\zeta},\check{\rho}_1)$ and $(\cL_{w(\zeta)},w(\check{\rho}_2))$ in $\cW(J_G;\sfLambda)$, we have 
\begin{align*}
\Hom_{\cW(J_G;\sfLambda)}((\cL_{\zeta},\check{\rho}_1), (\cL_{w(\zeta)},w(\check{\rho}_2)))\cong\begin{cases}&H^*(T, \sfLambda), \text{ if }\check{\rho}_1=\check{\rho}_2,\\
&0, \text{ otherwise}.
\end{cases}
\end{align*}
In particular, the objects $(\cL_\zeta, \check{\rho})$ and $(\cL_{w(\zeta)}, w(\check{\rho}))$ in $\cW(J_G;\sfLambda)$ are isomorphic, for all $\zeta\in \ft_c^\reg$ and $w\in W$.  

\end{itemize}
\end{prop}

\begin{remark}
In Proposition \ref{prop: L_xi, S_e, part 1}, the objects $(\cL_\zeta, \check{\rho})$ and $(\cL_{w(\zeta)},w(\check{\rho}))$ are geometrically modeled on the complex torus fiber $\chi^{-1}([\zeta])$ (which is not a well defined object in $\cW(J_G;\sfLambda)$). 
More explicitly, it will follow from the construction in Subsection \ref{subsec: construct L_zeta} that $\cL_\zeta\cap \chi^{-1}([\zeta])$ is a compact torus homotopy equivalent to $\chi^{-1}([\zeta])$ (more precisely a $(\chi^{-1}([\zeta]))_{\cpt}$-orbit), and $\cL_{w(\zeta)}\cap \chi^{-1}([\zeta])$ can be thought as (though not identical to) $w(\cL_\zeta\cap \chi^{-1}([\zeta]))$. Then $w(\check{\rho})$ on ${w(\cL_\zeta\cap \chi^{-1}([\zeta]))}$ is the pullback local system of $\check{\rho}$ on $\cL_\zeta\cap \chi^{-1}([\zeta])$ under $w^{-1}$. In particular, they define the same local system on $\chi^{-1}([\zeta])$. This morally explains why they are isomorphic in $\cW(J_G;\sfLambda)$. 
 \end{remark}

Now we state the key propositions in the exact setting. 
Let $L_0\subset \cB_{w_0}\cong T^*T$ be a ``cylindricalization" of the conormal bundle of an orbit of the maximal compact subtorus in $T$ (cf. Subsection \ref{subsubsec: conic} for an explicit construction).

\begin{prop}\label{prop: L_0, part 2}
We have in $\cW(J_G)$, 
\begin{align}
\label{eq: skyscraper 1 no q}&\Hom_{\cW(J_G)}((L_0,\check{\rho}), \Sigma_I)\cong \bC[-n]\\
\label{eq: skyscraper 2 no q}&\Hom_{\cW(J_G)}(\Sigma_I, (L_0,\check{\rho}))\cong \bC.
\end{align} 
\end{prop}

\begin{prop}\label{prop: L_0, W}
For all regular $\check{\rho}\in \Hom(\pi_1(T), \bC^\times)\cong T^\vee$, i.e. $\check{\rho}\in (T^\vee)^\reg$, we have 
\begin{align}\label{eq: prop L_0, W}
\Hom_{\cW(J_G)}((L_0, \check{\rho}), (L_0, w_1(\check{\rho})))\cong H^*(T,\bC), w_1\in W. 
\end{align}
In particular, in such cases, the objects $(L_0, \check{\rho})$ and $(L_0, w_1(\check{\rho}))$ viewed as objects in $\cW(J_G)$ are isomorphic. 
\end{prop}

In the remaining parts of this section, we develop some analysis in Subsection \ref{subsec: analysis cB_w0} and \ref{subsec: walls} that are crucial for the proof of the key propositions. Strictly speaking, the analysis in Subsection \ref{subsec: analysis fU_S} about $\fU_S, \emptyset\neq S\subsetneq \Pi$ is not logically needed for the proofs, but it is a natural generalization of the analysis done in Subsection \ref{subsubsec: B_w_0} about $\cB_{w_0}$. We include this for the sake of completeness and for recording some interesting geometric aspects about $\fU_S$ that may be of independent interest (see Question \ref{question: mu_D, K, rho} for the main points addressed). In Subsection \ref{subsec: construct L_zeta}, we give the explicit construction of $L_0$ and $\cL_{\zeta}$ that appeared in the above key propositions.

\subsection{Some analysis inside $\cB_{w_0}$ and $\fU_S, S\subsetneq \Pi$}\label{subsec: analysis cB_w0}

This subsection is motivated by the following simple observation, and it is crucial for the proof of the main theorem in Section \ref{sec: HMS adjoint}. Recall the identification $\cB_{w_0}\cong T^*T$ in Example \ref{example: B_w0}. We observe that for a fixed $t\in \ft$ and $h\in T$, as we multiply $h$ by $\epsilon^{-\sfh_0}$ for $|\epsilon|\rightarrow 0$, the characteristic map  
\begin{align*}
\chi|_{\cB_{w_0}}: &\cB_{w_0}\longrightarrow \fc\\
&(\overline{w}_0^{-1}h, f+t+\Ad_{(\overline{w}_0^{-1}h)^{-1}}f)\mapsto \chi(f+t+\Ad_{(\overline{w}_0^{-1}h)^{-1}}f)
\end{align*}  
is getting closer and closer to $\chi(f+t)$, which is the same as the composition of projecting to $t\in \ft$ and the quotient map $\ft\rightarrow\fc$. Geometrically, this suggests that for any $[\xi]\in \fc^{\reg}$, $\chi^{-1}([\xi])\cap \cB_{w_0}$ will split into $|W|$ many disjoint sections over a region in $T$ of the form $\bigcup\limits_{|\epsilon|<\eta_0}\epsilon^{-\sfh_0}\cdot \cV$, for any pre-compact domain $\cV\subset T$ and for sufficiently small $\eta_0>0$. In the following, we make these into rigorous statements. In particular, we establish a link between the standard integrable system structure $T^*T\rightarrow \ft$ and that inherited from the embedding into $\chi: J_G\rightarrow \fc$ (the latter is certainly incomplete, i.e. having incomplete torus orbits) through an interpolating family of ``integrable systems" on certain pre-compact regions in $T^*T$. We also have the general discussions for $\fU_S$  (\ref{eq: prop fU_S splitting}) where the torus with Hamiltonian action(s) is replaced by $\cZ(L_S)$.

For any $S\subsetneq \Pi$, it would be more convenient to use the identity component of $\cZ(L_S)$, denoted by $\cZ(L_S)_0$, instead of $\cZ(L_S)$ for discussions of Hamiltonian actions. We state the following lemma about the relation between $\cZ(L_S^\der)$ and $\pi_0(\cZ(L_S))$ for concreteness. 
\begin{lemma}\label{lemma: pi_0, L_S}
For any semisimple Lie group $G$, we have canonical identifications
\begin{align}\label{eq: pi_0, L_S}
\pi_0(\cZ(L_S))\cong X_*(T_{S,\ad})/\pi_{\ft_S}(X_*(T))
\end{align}
\begin{align}\label{eq: Z, L_S, der}
\cZ(L_S^\der)\cong X_*(T_{S,\ad})/(X_*(T)\cap \ft_S),
\end{align}
where $T_{S,\ad}$ is a maximal torus of $L_{S,\ad}$. 
In particular, we have a short exact sequence 
\begin{align*}
1\rightarrow \pi_{\ft_S}(X_*(T))/(X_*(T)\cap \ft_S)\rightarrow \cZ(L_S^\der)\rightarrow \pi_0(\cZ(L_S))\rightarrow 1,
\end{align*}
which gives an identification 
\begin{align}\label{eq: cZ_S, der0}
\cZ(L_S^\der)_0:=\cZ(L_S^\der)\cap \cZ(L_S)_0\cong\pi_{\ft_S}(X_*(T))/(X_*(T)\cap \ft_S).
\end{align} 
\end{lemma}
\begin{proof}
First, we have the preimage of $\cZ(L_S)$ in the universal cover $\ft$ of $T$ given by $\{t\in \ft_S: (\alpha,t)\in i\bZ, \forall\alpha\in S\}+\fz_S$. So 
\begin{align*}
\pi_0(\cZ(L_S))&\cong (\{t\in \ft_S: (\alpha,t)\in  i\bZ, \alpha\in S\}+\fz_S)/(iX_*(T)+\fz_S)\\
&\cong X_*(T_{S,\ad})/\pi_{\ft_S}(X_*(T)). 
\end{align*}
Similarly, we have the preimage of $\cZ(L_S^\der)$ in the universal cover $\ft$ given by $iX_*(T_{S,\ad})\subset \ft_S$ modulo $iX_*(T)$, and so (\ref{eq: Z, L_S, der}) follows. 
\end{proof}

It follows from Lemma \ref{lemma: pi_0, L_S} that for $G$ of adjoint type, $\cZ(L_S)_0=\cZ(L_S)$ and $\cZ(L^\der_S)_0=\cZ(L^\der_S)$. Although we assume $G$ of adjoint type for the rest of this section, we use $\cZ(L_S)_0$ and $\cZ(L^\der_S)_0$ in the following, since most of the results work directly for a semisimple $G$.

Let $D_S\subset \ft_S$ be any $W_S$-invariant pre-compact open neighborhood of $0\in \ft_S$. Let $\cK_{S^\perp}\subset \fz_S^\circ$ be any connected pre-compact open region such that 
\begin{align}\label{eq: cQ_D,cK}
\cQ_{D, \cK}:=D_S+\cK_{S^\perp}\subset \ft
\end{align} 
(cf. Figure \ref{figure: cK}) satisfies 
\begin{align}\label{eq: condition K_S, perp}
w(\overline{\cQ_{D, \cK}})\cap \overline{\cQ_{D, \cK}}=\emptyset, \forall w\not\in W_S. 
\end{align}
Let $\widetilde{\pr}_{\cK'_{S^\perp}}: \cQ_{D,\cK}/W_S\rightarrow \cK_{S^\perp}$ be the natural (analytic) projection. 
Let 
\begin{align}\label{eq: frX_0, D, K}
\fU_{S, D, \cK}:=\chi_S^{-1}(D_S/W_S)\underset{\cZ(L_S^\der)_0}{\times} (\cZ(L_S)_0\times \cK_{S^\perp})=\chi_S^{-1}(D_S/W_S)\underset{\cZ(L_S^\der)}{\times} (\cZ(L_S)\times \cK_{S^\perp}),
\end{align}
where $\chi_S: J_{L_S^\der}\rightarrow \fc_S$ is the characteristic map. For any $\cZ(L_S^\der)_0$-invariant pre-compact open region 
\begin{align}\label{eq: cY_S, cV}
\cY_S\subset \chi_S^{-1}(D_S/W_S)\text{ and }\cV_{S^\perp}\subset \cZ(L_S)_0, 
\end{align}
let 
\begin{align}\label{eq: cW_cY_S}
\cW_{\cY_S, \cV, \cK}:=\cY_S\underset{\cZ(L_S^\der)_0}{\times} (\cV_{S^\perp}\times \cK_{S^\perp})
\end{align}

Define for any $\rho\in \cZ(L_S)_0$
\begin{align}\label{eq: frj_S;rho}
\frj_{S;\rho}: \fU_S&\longrightarrow \fU_S\\
\nonumber(g_S,\xi_S;z, t)&\mapsto (g_S,\xi_S;z\rho, t),
\end{align}
which preserves the canonical holomorphic symplectic and Liouville 1-form on $\fU_S$ given explicitly by 
\begin{align*}
\omega|_{\fU_S}=-(d\lng \xi_S, g_S^{-1}dg_S\rng+d\lng t, z^{-1}dz\rng)
\end{align*} 
\begin{align}\label{eq: hat vartheta_S}
\vartheta|_{\fU_S}=-(\lng \xi_S, g_S^{-1}dg_S\rng+\lng t, z^{-1}dz\rng -\frac{1}{2}d\lng \xi_S, \Ad_{g_S^{-1}}\sfh_{0,S}-\sfh_{0,S}\rng+d\lng t, \sfh_{0, S^\perp}'\rng)
\end{align}

Let 
\begin{align}
\gamma_{-\Pi\backslash S}:=(-\beta\in-\Pi\backslash S): \cZ(L_S)_0\longrightarrow (\bC^\times)^{\Pi\backslash S}\hookrightarrow \bC^{\Pi\backslash S}. 
\end{align}
For $\rho\in \cZ(L_S)_0$ satisfying $|\gamma_{-\Pi\backslash S}(\rho)|\ll 1$, and some slightly larger open neighborhood $\cK'_{S^\perp}$ of $\overline{\cK}_{S^\perp}$, the map 
\begin{align}\label{eq: mu_D,cK,rho}
\mu_{D,\cK', \rho}:\pr_{\cK'_{S^\perp}}\circ \chi\circ \frj_{S;\rho}: \cW_{\cY_S, \cV, \cK}\longrightarrow \cK'_{S^\perp}
\end{align}
is well defined, and it fits into an $(n-|S|)$-dimensional family of deformations of $\pr_{\cK_{S^\perp}}$ through $\rho\mapsto (c_\beta)_\beta=\gamma_{-\Pi\backslash S}(\rho)$ (after inserting $\Ad_{\rho}$ between $\chi$ and $\frj_{S;\rho}$ in (\ref{eq: mu_D,cK,rho}) which has \emph{no} effect on (\ref{eq: mu_D,cK,rho}))\footnote{Here to simplify notations, we have suppressed the dependence of $\mu_{D,\cK, \rho}$ on the domain $\cW_{\cY_S, \cV, \cK}$. Note also that the family of maps $\mu_{D,\cK,\rho}$ does not necessarily embed into the universal family $\widetilde{\mu}_{D,\cK,(c_\beta)_{\beta\in \Pi\backslash S}}$, because $\gamma_{-\Pi\backslash S}$ is not always injective.}, given by 
\begin{align}\label{eq: mu_D,K, canonical}
&\widetilde{\mu}_{D,\cK',(c_\beta)_{\beta\in \Pi\backslash S}}:=\widetilde{\pr}_{\cK'_{S^\perp}}\circ \chi(\sum\limits_{\beta\in \Pi\backslash S}c_\beta\cdot f_\beta+\xi_S+t+\Ad_{g_S^{-1}}\psi), \\
\nonumber&\text{ for }(c_\beta)\in \bC^{\Pi\backslash S}, |(c_\beta)|\ll 1, \text{ and }\psi=\Ad_{z^{-1}\overline{w}_S^{-1}\overline{w}_0}(f-f_{-w_0(S)}),
\end{align}
from the same domain. 
Note that since $G$ has trivial center, there is a one-to-one correspondence between $\psi$ and $z\in \cZ(L_S)_0$. 
We will refer to (\ref{eq: mu_D,K, canonical}) as the \emph{universal $(\Pi\backslash S)$-deformations} of $\widetilde{\mu}_{D,\cK, 0}:=\pr_{\cK_{S^\perp}}$. One can view $\widetilde{\mu}_{D,\cK, 0}$ (originally defined on $\fU_{S,D,\cK}$) as the moment map for the obvious Hamiltonian $\cZ(L_S)_0$-action on the right-hand-side of (\ref{eq: frX_0, D, K}), and the Hamiltonian reduction is isomorphic to $\chi_S^{-1}(D_S/W_S)/\cZ(L_S^\der)_0\subset J_{L_{S}^\der/\cZ(L_S^\der)_0}$.  
Proposition \ref{eq: tilde mu, embed} below shows that for every element in the family, functions on $\cK_{S^\perp}'$ induce Poisson commuting Hamiltonian functions on $\cW_{\cY_S, \cV, \cK}$ through pullback, and it is part of an integrable system with complete $\cZ(L_S)_0$-orbits.

\begin{prop}\label{eq: tilde mu, embed}
For any $(c_\beta)_\beta$ with $|(c_\beta)_\beta|\ll 1$,  the image of 
\begin{align*}
\widetilde{\mu}_{D, \cK', (c_\beta)_\beta}^*: C^\infty(\cK'_{S^\perp};\bR)\longrightarrow C^\infty(\cW_{\cY_S, \cV, \cK};\bR)
\end{align*}
are Poisson commuting Hamiltonian functions on $\cW_{\cY_S, \cV, \cK}$, with respect the real symplectic structure. The same holds for pullback of holomorphic functions with respect the holomorphic symplectic structure. In fact, letting $S'=S\cup\{\beta\in \Pi\backslash S: c_\beta\neq 0\}$, we have a natural commutative diagram 
\begin{align}\label{lemma: integrable S}
\xymatrixcolsep{5pc}\xymatrix{\cW_{\cY_S, \cV, \cK}\ar@/_3pc/[dd]_{\widetilde{\mu}_{D,\cK',(c_\beta)_{\beta}}}\ar[d]\ar@{^{(}->}[r]^{\tilde{\iota}_S^{S'}\circ \frj_{S, \rho_0}}&J_{L_{S'}}\ar[d]^{\widetilde{\chi}_{S'}}\\
\cQ_{D,\cK'}/W_{S}\ar[d]\ar@{^{(}->}[r]^{\jmath}&\ft\sslash W_{S'}\cong\fc_{S'}\times \fz_{S'}\ar[d]\\
\cK'_{S^\perp}\ar[r]&\fz_{S'}
},
\end{align}
for some $\rho_0\in \cZ(L_S)_0$, that embeds $\widetilde{\mu}_{D, \cK', (c_\beta)_\beta}$ holomorphically symplectically into the integrable system 
\begin{align}\label{eq: lemma tilde chi_S'}
\widetilde{\chi}_{S', \cK'}: \widetilde{\chi}_{S'}^{-1}(\jmath(\cQ_{D,\cK'}/W_{S}))\overset{\widetilde{\chi}_{S'}}{\longrightarrow} \jmath(\cQ_{D,\cK'}/W_{S})\overset{\pr_{\cK_{S^\perp}'}\circ \jmath^{-1}}{\longrightarrow}  \cK'_{S^\perp}
\end{align}
with complete $\cZ(L_S)_0$-orbits. 

\end{prop}

\begin{proof}
Fix any $(c_\beta)_\beta$ and let $S'$ be as above. Choose $\rho_0\in \cZ(L_S)_0$ such that $\Ad_{\rho_0^{-1}}(\sum\limits_{\beta\in \Pi\backslash S}c_\beta\cdot f_\beta)=f_{S'\backslash S}$. We do the following embedding using $\tilde{\iota}_S^{S'}$ from (\ref{eq: tilde i S, S'})
\begin{align*}
\tilde{\iota}_S^{S'}\circ \frj_{S, \rho_0}: &\cW_{\cY_S, \cV, \cK}\longhookrightarrow J_{L_{S'}}=J_{L_{S'}^\der}\underset{\cZ(L_{S'}^\der)}{\times}T^*\cZ(L_{S'})
\end{align*}
Then comparing $\Ad_{\rho_0^{-1}}(\sum\limits_{\beta\in S'\backslash S}c_\beta\cdot f_\beta+\xi_S+t+\Ad_{g_S^{-1}}\psi)$ with the second component of  $\tilde{\iota}_S^{S'}\circ \frj_{S, \rho_0}(g_S,\xi_S;z, t)$, we see that their difference 
is contained in $\fn_{\fp_{S'}}$ (the nilpotent radical of the standard parabolic subalgebra for $S'$). This can be directly seen from the equality $\tilde{\iota}^\Pi_{S'}\circ\tilde{\iota}_S^{S'}=\tilde{\iota}_S^\Pi$ established in Proposition \ref{prop: fU_S', sqcup}. Hence, we have the commutative diagram (\ref{lemma: integrable S}), and the embedding of $\widetilde{\mu}_{D, \cK', (c_\beta)_\beta}$ into the integrable system with complete $\cZ(L_S)_0$-orbits.  
\end{proof}

\begin{remark}\label{remark: cK in fz_S}
We remark that it is important to view (i.e. fix an embedding of) $\cK'_{S^\perp}$ inside $\fz_S^\circ$ to specify a Hamiltonian $\cZ(L_S)_0$-action on\\
 $\widetilde{\chi}_{S'}^{-1}(\jmath(\cQ_{D,\cK'}/W_{S}))$ in Proposition \ref{eq: tilde mu, embed}. In particular, in the following whenever we are talking about integrable systems over $\cK'_{S^\perp}$ with $\cZ(L_S)_0$-actions, it only makes sense after fixing such an embedding. Changing $\cK_{S^\perp}$ by $w\in N_{W_{S'}}(W_S)$ induces the following commutative diagram, where the left $\cZ(L_S)_0$ and right $\cZ(L_S)_0$ actions on $\chi_{\fl_{S'}}(\cK'_{S^\perp})$ at the top are respectively induced from identifying $\chi_{\fl_{S'}}(\cK'_{S^\perp})$ with $\cK'_{S^\perp}$ and $w(\cK'_{S^\perp})$. They are related by the automorphism $w$ on $\cZ(L_S)_0$. 
\begin{figure}[htbp]
\begin{tikzpicture}
\node (A0) at  (-4,0) {$\cK'_{S^\perp}$};
\node (B0) at (0,0)  {$\chi_{\fl_{S'}}(\cK'_{S^\perp})$};
\node (C0) at  (4,0) {$w(\cK'_{S^\perp})$};
\draw[->] (A0) -- (B0) node[pos=0.5, above]{$\sim$};
\draw[->] (C0) -- (B0) node[pos=0.5, above]{$\sim$};
\draw[->] (A0) to[bend right] (C0) node[pos=0.5, yshift=-0.4in]{$w$};
\node (A1) at  (-4,1.5) {$\cW_{\cY_S, \cV, \cK}$};
\node (B1) at (0,1.5)  {$\widetilde{\chi}_{S'}^{-1}(\jmath(\cQ_{D,\cK'}/W_{S}))$};
\node (C1) at  (4,1.5) {$\cW_{\cY_S, \cV, w(\cK)}$};
\draw[right hook->] (A1) -- (B1);
\draw[left hook->] (C1)-- (B1);
\draw[->] (A1) -- (A0) node[pos=0.5, left]{$\widetilde{\mu}_{D,\cK',(c_\beta)_{\beta}}$};
\draw[->] (B1) -- (B0) node[pos=0.5, left]{$\widetilde{\chi}_{S'}$};
\draw[->] (C1) -- (C0) node[pos=0.5, right]{$\widetilde{\mu}_{D,w(\cK'),(c_\beta)_{\beta}}$};
\node (A2) at  (-1.5,3) {$\cZ(L_S)_0$};
\node (B2) at (1.5,3)  {$\cZ(L_S)_0$};
\draw[->] (A2) -- (B2) node[pos=0.5, above]{$w$};
\draw [->] (-1.5, 1.9) arc(260:-80:0.4);
\draw [->] (1.5, 1.9) arc(260:-80:0.4);
\end{tikzpicture}
\end{figure}
\end{remark}

Let 
\begin{align}\label{eq: chi_fg}
\chi_\fg: \fg\longrightarrow \fc\ (\text{resp. }\chi_\ft: \ft\longrightarrow \fc)
\end{align}
be the adjoint quotient map, and let 
\begin{align}
\chi_{\cK_{S^\perp}}:=\pr_{\cK_{S^\perp}}\circ \chi: \chi^{-1}(\chi_\fg(\cQ_{D,\cK}))\longrightarrow \cQ_{D,\cK}/W_S\longrightarrow \cK_{S^\perp}
\end{align} 
be the moment map for the Hamiltonian $\cZ(L_S)_0$-action on $ \chi^{-1}(\chi_\fg(\cQ_{D,\cK}))$. 
For some slight enlargement $D'_S\supset \overline{D}_S$ contained in $\ft_S$, we have the commutative diagram 
\begin{align}\label{diagram: mu, S, rho}
\xymatrixcolsep{5pc}\xymatrix{\cW_{\cY_S, \cV, \cK}\ar[d]_{\widetilde{\mu}_{D,\cK', (\gamma_{-\Pi\backslash S}(\rho))}=\mu_{D,\cK', \rho}}\ar@{^{(}->}[r]^{\frj_{S,\rho}\ \ \ \ \ }&\chi^{-1}(\chi_\fg(\cQ_{D',\cK'}))\ar[dl]^{\chi_{\cK'_{S^\perp}}}\\
\cK'_{S^\perp}&
}
\end{align}
 for $\rho\in \cZ(L_S)_0$ satisfying $|\gamma_{-\Pi\backslash S}(\rho)|\ll 1$. 
 By Lemma \ref{lemma: xi, t, regular} below, there is an isomorphism 
\begin{align}\label{eq: chi_S, mathscr Z}
\chi_S^{-1}(D'_S/W_S)\underset{\cZ(L_S^\der)_0}{\times}(\cZ(L_S)_0\times \cK'_{S^\perp})&\longrightarrow \chi^{-1}(\chi_\fg(\cQ_{D',\cK'}))\\
\nonumber(((g_S, \xi_S)\in \mathscr{Z}_{L_S^\der}\sslash L_S^\der), (z, t))&\mapsto ((g_Sz, \xi_S+t)\in \mathscr{Z}_{G}\sslash G),
\end{align}
where the second line of the presentation (with the elements understood from the respective sublocus) emphasizes that the elements $(g_S, \xi_S)$ are from the centralizer presentation of $J_{L_S^\der}$ (\ref{eq: scrZ}), rather than the Whittaker Hamiltonian reduction perspective (in particular, $\xi_S+t$ is \emph{not} in $f+\fb$ unless $S=\Pi$) Then the Hamiltonian reduction of $\chi_{\cK'_{S^\perp}}$ at any point in $\cK'_{S^\perp}$ is then canonically isomorphic to $\chi_S^{-1}(D'_S/W_S)/\cZ(L_S^\der)_0$. 

\begin{lemma}\label{lemma: xi, t, regular}
Let $D_S, \cK_{S^\perp}$ satisfy the condition (\ref{eq: condition K_S, perp}). Then for any $\xi_S^\natural\in (f_S+\fb_S)\cap \chi_{\fl_S^\der}^{-1}(D_S/W_S)$ and $t^\natural\in \cK_{S^\perp}$, we have $\xi_S^\natural+t^\natural$ is regular in $\fg$. 
\end{lemma}
\begin{proof}
Up to adjoint action by $N_S$, we may assume that $\xi_S^\natural=f_S+\tau\in \fb_S^-$ for some $\tau\in D_S\subset \ft_S$. We claim that for any $\eta=\sum\limits_{\alpha\in \Delta^+\backslash \Gamma(S)}c_\alpha e_\alpha\in \fn_{\fp_S}$, $[\xi_S^\natural+t^\natural, \eta]=0\Rightarrow \eta=0$. Suppose $\eta\neq 0$, let $\alpha_0$ be a maximal root (under the standard partial order) such that $c_{\alpha}\neq 0$. Then the root component of $[\xi_S^\natural+t^\natural, \eta]$ in $\fg_{\alpha_0}$ is equal to $[\tau+t^\natural, c_{\alpha_0}e_{\alpha_0}]=c_{\alpha_0}\alpha_0(\tau+t^\natural)e_{\alpha_0}$. By assumption on $D_S+\cK_{S^\perp}$, $\alpha_0(\tau+t^\natural)\neq 0$, for the annihilators in $\Delta^+$ of any element in $D_S+\cK_{S^\perp}$ is contained in $\Gamma(S)$. So the claim follows. Similarly, we have 
for any $\eta=\sum\limits_{\alpha\in \Delta^+\backslash \Gamma(S)}c_\alpha f_\alpha\in \fn^-_{\fp_S}$, $[\xi_S^\natural+t^\natural, \eta]=0\Rightarrow \eta=0$. Thus the Lie algebra centralizer of $\xi_S^\natural+t^\natural$ is contained in $\fl_S$. Since $\xi_S^\natural+t^\natural$ is regular in $\fl_S$, the lemma follows. 
\end{proof}

In the following, fix any  $D_S^\dagger, \cK_{S^\perp}^\dagger$ with the same property as $D_S, \cK_{S^\perp}$, respectively, satisfying
\begin{align}\label{eq: D, K, dagger}
\overline{D_S^\dagger}\subset D_S,\ \overline{\cK^\dagger}_{S^\perp}\subset \cK_{S^\perp}. 
\end{align}
and we consider 
\begin{align}\label{eq: cY_S^dagg}
\cY^\dagger_S\subset \chi_S^{-1}(D_S^\dagg/W_S), \overline{\cY_S^\dagg}\subset \cY_S
\end{align}
satisfying the similar property as for $\cY_S$ (\ref{eq: cY_S, cV}). 

Here is the main question that we will answer in this section. 
\begin{question}\label{question: mu_D, K, rho}
Since $\mu_{D,\cK',\rho}$ fits into the universal $(\Pi\backslash S)$-deformation of $\widetilde{\mu}_{D,\cK', (c_\beta)_{\beta\in \Pi\backslash S}}$, in particular for $|\gamma_{-\Pi\backslash S}(\rho)|\rightarrow 0$, it converges to $\widetilde{\mu}_{D,\cK,0}=\pr_{\cK_{S^\perp}}$, it can be viewed as an interpolating family of incomplete Hamiltonian $\cZ(L_S)_0$-systems between the complete systems $\chi_{\cK'_{S^\perp}}$ and $\widetilde{\mu}_{D,\cK,0}$ (the latter viewed on $\fU_{S, D,\cK}$). Can we understand the relations between these two complete Hamiltonian $\cZ(L_S)_0$-systems through the interpolating family? More concretely, we want to investigate the following two aspects of their relations:
\begin{itemize}
\item[(i)] The relation(s) between their $\cZ(L_S)_0$-orbits: for this (and (ii) below) we take $\cV_{S^\perp}\subset \cZ(L_S)_0$ to be $\cZ(L_S)_{0,\cpt}\times \exp(B_R(0))$ for some standard ball $B_R(0)$ centered at $0$ inside $\fz_{S, \bR}$, and we will relate $\frj_{S,\rho}(\{(g_S,\xi_S)\}\times \cV_{S^\perp}\times \{\kappa\})$ with a $\cZ(L_S)_0$-orbit inside $\chi_{\cK'_{S^\perp}}^{-1}(\kappa)$, for any $\kappa\in \cK_{S^\perp}^\dagger$.

\item[(ii)] The relation(s) between the Hamiltonian reductions through $\frj_{S,\rho}$: for the universal $(\Pi\backslash S)$-deformations $\widetilde{\mu}_{D,\cK',(c_\beta)_{\beta\in \Pi\backslash S}}$ (\ref{eq: mu_D,K, canonical}) with $|(c_\beta)_\beta|$ sufficiently small, we have the characteristic foliations in $\widetilde{\mu}_{D,\cK',(c_\beta)_\beta}^{-1}(\kappa)$ arbitrarily close to the ``standard" foliations 
\begin{align*}
\{\{(g_S,\xi_S)\}\times \cV_{S^\perp}\times \{\kappa\}: (g_S,\xi_S)\in \cY^\dagg_S, \kappa\in \cK_{S^\perp}\}.
\end{align*}
In particular, fixing the $|\cZ(L_S^\der)_0|$-to-$1$ multi-section of the standard foliation given by $\cY^\dagg_S\times \{u_0\}\times \cK_{S^\perp}$ for some $u_0\in \cV_{S^\perp}$, it is transverse to the characteristic foliations in $\widetilde{\mu}_{D,\cK',(c_\beta)_\beta}^{-1}(\kappa)$ for all $|(c_\beta)_\beta|$ small. After a modification of 
\begin{align*}
(\cY^\dagg_S\times \{u_0\}\times\cK_{S^\perp})\cap \widetilde{\mu}_{D,\cK',(c_\beta)_\beta}^{-1}(\kappa) 
\end{align*}
to be a $\cZ(L_S^\der)_0$-equivariant multi-section of the symplectic quotient; see the definition of $\cY^\dagg_{S, \kappa, (c_\beta)_\beta}$ in Remark \ref{remark: modify multi-section}. 
We get an embedding 
\begin{align*}
\cY^\dagg_{S, \kappa, (c_\beta)_\beta}/\cZ(L_S^\der)_0\hookrightarrow\widetilde{\mu}_{D,\cK',(c_\beta)_\beta}^{-1}(\kappa)/(\text{characteristic leaves}) 
\end{align*}
where the latter has the quotient symplectic structure\footnote{The latter quotient space might not have a good structure near $\partial \cY_S\underset{\cZ(L_S^\der)_0}{\times} (\cV_{S^\perp}\times \cK_{S^\perp})$. The embedding from $\cY^\dagg_{S, \kappa, (c_\beta)_\beta}/\cZ(L_S^\der)_0$ does not touch such ``bad" places.}, for all $(c_\beta)_\beta$ near $0$ and $\kappa\in \cK^\dagg_{S^\perp}$. Now for $\rho\in \cZ(L_S)_0$ with $|\gamma_{-\Pi\backslash S}(\rho)|$ sufficiently small, $\frj_{S,\rho}$ induces a map (which is a \emph{local} symplectic isomorphism) on the ``Hamiltonian reductions", 
\begin{align}\label{eq: j_S,rho, reduction}
\overline{\frj}_{S,\rho;\kappa}:&\cY^\dagg_{S,\kappa, (c_\beta)_\beta}/\cZ(L_S^\der)_0\hookrightarrow \mu_{D,\cK',\rho}^{-1}(\kappa)/(\text{characteristic leaves})\\
\nonumber&\longrightarrow \chi_S^{-1}(D'_S/W_S)/\cZ(L_S^\der)_0.
\end{align}

We would like to understand this map. More specifically, we will show that as we enlarge $\cY_S$ (then so is $\cY^\dagg_{S, \kappa, (c_\beta)_\beta}$) and letting $|\gamma_{-\Pi\backslash S}(\rho)|\rightarrow 0$, the map (\ref{eq: j_S,rho, reduction}) covers any fixed compact region inside $\chi_S^{-1}(D^\dagg_S/W_S)/\cZ(L_S^\der)_0$ in the codomain and it is one-to-one (see Proposition \ref{prop: Ham reduction emb} below). 
\end{itemize}
\end{question}

We remark that Question \ref{question: mu_D, K, rho} (i), (ii) are nontrivial and are quite useful for understanding the geometry of $J_G$. The reason is that it is a highly nonlinear question to deduce explicit formulas for the torus fibers $\chi^{-1}([\xi])\cong C_G(\xi)$ for general $\xi\in \cS$ and similarly $\cZ(L_S)_0$-orbits in $\chi_{\cK'_{S^\perp}}^{-1}(\kappa)$, especially (the portion) inside $\cB_{w_0}$ or $\fU_S$. The following two subsections analyze the asymptotic behaviors in certain directions, i.e. $|\gamma_{-\Pi\backslash S}(\rho)|\ll 1$, making the questions approachable, while giving geometric information about the orbits that is sufficient for many purposes.

\subsubsection{Analysis inside $\cB_{w_0}$}\label{subsubsec: B_w_0}
For $S=\emptyset$, many of the discussions as well as notations can be simplified. We will omit the null inputs $D_{\emptyset}, \cY_{\emptyset}, \cZ(L_{\emptyset}^\der)_0$ in all notations, and we will denote $\cK_{\emptyset^\perp}$ (resp. $\cV_{\emptyset^\perp}$) simply by $\cK$ (resp. $\cV$), for which we make the analytic identification $\chi_\ft|_{\cK}: \cK\cong \chi_\ft(\cK)\subset \fc^{\reg}$ (cf. Figure \ref{figure: cK}). Note that $\chi_{\cK}=\chi$, and diagram (\ref{diagram: mu, S, rho}) specializes to be the commutative diagram 
\begin{align}\label{diagram: mu, rho}
\xymatrix{\cV\times \ft\ar[d]_{\mu_{\rho}}&\ar@{_{(}->}[l]\cV\times \cK= \cW_{\cV, \cK}\ar[d]_{\widetilde{\mu}_{\cK', \gamma_{-\Pi}(\rho)}=\mu_{\cK', \rho}}\ar@{^{(}->}[r]^{\ \ \ \ \frj_{\rho} }&\chi^{-1}(\chi_\ft(\cK'))\ar[dl]^{\chi}\\
\fc&\ar@{_{(}->}[l]\cK'\overset{\chi_\ft}{\cong} \chi_\ft(\cK')&
},
\end{align}
where we add a left column on the deformed $\mu_\rho$ well defined on $\cV\times\ft$. We emphasize again that if we want to talk about $T$-action on $\chi^{-1}(\chi_\ft(\cK'))$, we need to specify an embedding of $\cK'$ into $\ft$. This is by default through the definition of $\cK'$ as a subset of $\ft$. 

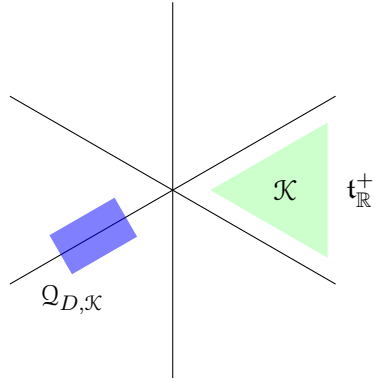
\begin{figure}
\begin{tikzpicture}
\draw (0,2.5)--(0,-2.5);
\draw ({2.5*cos(30)},{2.5*sin(30)})--({-2.5*cos(30)},{-2.5*sin(30)});
\draw ({2.5*cos(150)},{2.5*sin(150)})--({-2.5*cos(150)},{-2.5*sin(150)});
\fill[fill=green!20] ({1.8*cos(30)+0.5},{1.8*sin(30)})--(0.5,0)--({-1.8*cos(150)+0.5},{-1.8*sin(150)})--({1.8*cos(30)+0.5},{1.8*sin(30)});
\draw (1.5, 0) node {$\cK$};
\draw (2.5, 0) node {$\ft_\bR^+$};
\fill[fill=blue, opacity=0.5] ({0.6*cos(240)-0.2*cos(30)},{0.6*sin(240)-0.2*sin(30)})--({-0.6-0.2*cos(30)},{-0.2*sin(30)})--({-0.6-1.2*cos(30)}, {-1.2*sin(30)})--({0.6*cos(240)-1.2*cos(30)},{0.6*sin(240)-1.2*sin(30)})--({0.6*cos(240)-0.2*cos(30)},{0.6*sin(240)-0.2*sin(30)});
\draw ({0.6*cos(240)-1.2*cos(30)},{0.6*sin(240)-1.2*sin(30)}) node[below] {$\cQ_{D,\cK}$};
\end{tikzpicture}
\caption{A real picture of $\cK$ (green triangular region; here we make it inside the closed cone $\ft_\bR^+$) and $\cQ_{D,\cK}=D_S+\cK_{S^\perp}$ (blue rectangular region) inside $\ft$. }\label{figure: cK}
\end{figure}

For $S=\emptyset$, part (ii) of Question \ref{question: mu_D, K, rho} is trivial. For Question \ref{question: mu_D, K, rho} (i), our main result not only gives relations between the individual (incomplete) $T$-orbits, but also establishes an ``equivalence" between the integrable systems, restricted to certain pre-compact regions.

Since any $\kappa \in \cK$ is a regular value of $\chi$ and $\widetilde{\mu}_{\cK', 0}^{-1}(\kappa)=T\times \{\kappa\}$, for any pre-compact open region $\cV\subset T$ as described in Question \ref{question: mu_D, K, rho} (i) (the inclusion is in particular a homotopy equivalence), there exist $r_{\cV}>0$ such that for any $(c_\beta)_{\beta\in \Pi}$ satisfying $|(c_\beta)_\beta|<r_{\cV}$, we have $\kappa$ a regular value of $\widetilde{\mu}_{\cK', (c_\beta)}$ and $\widetilde{\mu}_{\cK', (c_\beta)}^{-1}(\kappa)\cap \cW_{\cV, \cK'}$ is a smooth Lagrangian section over $\cV$. 

For a general $[\xi]\in \fc$, we have the following: 

\begin{lemma}\label{lemma: hat chi_epsilon, gamma}
\item[(i)] For any pre-compact open $\cV\subset T$ as above, any compact region $\fK\subset \fc$ and $\delta>0$, there exists $r_{\cV,\fK, \delta}>0$ such that for all $[\xi]\in \fK$ and $\rho\in T$ satisfying $|\gamma_{-\Pi}(\rho)|\leq r_{\cV,\fK, \delta}$, we have
\begin{align}\label{eq: chi, xi, general}
\chi^{-1}([\xi])\cap (\frj_\rho(\cV)\times \ft)\subset \frj_\rho(\cV)\times \bigcup\limits_{\xi'\in \ft: \chi_\ft(\xi')=[\xi]}\{|t-\xi'|<\delta\}. 
\end{align}

\item[(ii)]
Let $\fK'\subset \fc^{\reg}$ be a compact subset. Then for any small $\delta>0$, there exists $r>0$ such that 
\begin{align}\label{lemma eq: chi_epsilon in delta_U}
\chi^{-1}([\xi])\cap (\bigcup\limits_{|\gamma_{-\Pi}(\rho)|\leq r}\frj_\rho(\cV)\times \ft)\subset (\bigcup\limits_{|\gamma_{-\Pi}(\rho)|\leq r} \frj_\rho (\cV))\times (\bigcup\limits_{\xi'\in \ft: \chi_\ft(\xi')=[\xi]}\{|t-\xi'|<\delta\}),
\end{align}
for all $[\xi]\in \fK'$, and the intersection has $|W|$ many connected components with each projecting to $\bigcup\limits_{|\gamma_{-\Pi}(\rho)|\leq r} \frj_\rho (\cV)$ isomorphically. 

\end{lemma}
\begin{proof}
First, by applying $\frj_\rho^{-1}$ on both sides, (\ref{eq: chi, xi, general}) is equivalent to 
\begin{align}\label{eq: proof chi, xi, general}
\mu_{\rho}^{-1}([\xi])\cap (\cV\times \ft)\subset  \cV\times \bigcup\limits_{\xi'\in \ft: \chi(\xi')=[\xi]}\{|t-\xi'|<\delta\}. 
\end{align}
Second, we have the homogeneity relation for $(h,t)\in T^*T\cong \cB_{w_0}$
\begin{align*}
&\epsilon^2\cdot \mu_{\rho}(h, t)=\mu_\rho(h\cdot \epsilon^{-\sfh_0}, \epsilon^2\cdot t)=\mu_{\rho\cdot \epsilon^{-\sfh_0}}(h, \epsilon^2\cdot t)\\
\Leftrightarrow &\mu_{\rho}(h, t)=\frac{1}{\epsilon^2} \mu_{\rho\cdot \epsilon^{-\sfh_0}}(h, \epsilon^2\cdot t)
\end{align*}
This implies that $\proj_\ft(\mu_{\rho}^{-1}([\xi])\cap (\cV\times \ft))$ is contained in a compact region  for all $\rho$ with $|\gamma_{-\Pi}(\rho)|$ small. Then (\ref{eq: proof chi, xi, general}) follows from that $\mu_\rho$ is a small deformation of $\chi_\ft\circ \proj_\ft$. 

Assuming $[\xi]\in\fc^{\reg}$, there exists $\cK\subset \ft^{\reg}$ as above, such that $\chi_{\ft}^{-1}([\xi])=\{w(\xi'): w\in W\}\subset \coprod\limits_{w\in W}w(\cK)$ for some $\xi'\in \cK$. Then for $|\gamma_{-\Pi}(\rho)|$ sufficiently small, $\mu_{w(\cK'), \rho}^{-1}(w(\xi'))$ is a Lagrangian section in $\cW_{\cV, w(\cK')}$ over $\cV$. This directly implies the second part of the lemma. 
\end{proof}

In particular, Lemma \ref{lemma: hat chi_epsilon, gamma} implies that for any $[\xi]\in \fc^{\reg}$, inside the preimage of $\pi_{r}$ over  
\begin{align}\label{eq: cone cap bigger than R}
\Cone(\cV_{\log})\cap \{\sum_{\beta\in \Pi}r_\beta>R\}\subset \bR^\Pi_{\geq 0}, 
\end{align}
for a pre-compact open subset $\cV_{\log}$ in the interior of $\fC^{n-1}$ and $R\gg 1$, $\chi^{-1}([\xi])$ splits into $|W|$ disjoint sections over that region. Moreover, the sections are getting closer and closer to the constant sections indexed by $\{\xi'\in \ft: \chi_\ft(\xi')=[\xi]\}$ as $R\rightarrow\infty$. 
Here we are using the notations in Subsection \ref{subsubsec: H, sm}. Near the end of this subsection, we will give a more precise description of these $|W|$ disjoint sections inside $\chi^{-1}([\xi])$.

Now we work specifically with the setting written before Lemma \ref{lemma: hat chi_epsilon, gamma}. 
\begin{lemma}\label{lemma: L_t_gamma}
Under the above settings, for each $\kappa\in \cK^\dagg$, the Lagrangian 
\begin{align}\label{eq: lemma L_t, epsilon}
\cS_{\kappa, \rho}:=\mu_{\cK', \rho}^{-1}(\kappa)\subset \cW_{\cV, \cK},\text{ where } |\gamma_{-\Pi}(\rho)|<r_\cV
\end{align}
(resp. 
\begin{align*}
\cS_{\kappa, (c_\beta)}:=\widetilde{\mu}_{\cK', (c_\beta)_\beta}^{-1}(\kappa)\subset \cW_{\cV, \cK},\text{ where } |(c_\beta)_\beta|<r_\cV)
\end{align*}
satisfies
\begin{itemize}

\item[(i)] $\cS_{\kappa, \rho}$ (resp. $\cS_{\kappa, (c_\beta)}$) is a smooth Lagrangian section over $\cV$ that is Hamiltonian isotopic to $\cV\times \{\kappa\}$ inside $\cW_{\cV, \cK'}$. The same holds for $\cS_{\kappa, (c_\beta)}$.

\item[(ii)] The natural inclusion $\cS_{\kappa, \rho}\overset{\frj_\rho}\longhookrightarrow \chi^{-1}([\kappa])$ is a homotopy equivalence. Moreover, if we use the canonical identification with respect to $\xi=\kappa$ and $B_1=B$ in (\ref{eq: J_G, fc, ft}), (\ref{eq: pi_chi, B_1}), $\nu_\kappa: \chi^{-1}([\kappa])\overset{\sim}{\rightarrow} C_G(\kappa)\cong T$, then the sequence of maps 
\begin{align}\label{eq: lemma T, j, nu}
T\overset{h.e.}{\hookleftarrow}\cV\cong \cS_{\kappa, \rho}\overset{\frj_\rho}\longhookrightarrow \chi^{-1}([\kappa])\overset{\sim}{\rightarrow} C_G(\kappa)\cong T
\end{align}
induces a homotopy equivalence from $T$ (identified with $B/[B,B]$) to itself that is isotopic to the identity map. 
\end{itemize}
\end{lemma}
\begin{proof}

(i) Fix a basis for $H_1(\cV,\bZ)\cong H_1(\cS_{\kappa,\rho},\bZ)$ (the isomorphism is the canonical one induced from the projection $\cS_{\kappa,\rho}\overset{\sim}{\rightarrow} \cV$) and denote each 1-cycle by $\Gamma_i$. First, the family of embeddings 
\begin{align}\label{eq: cV, cS, j_rho}
\cV\cong \cS_{\kappa, \rho}\overset{\frj_\rho}{\longhookrightarrow} \chi^{-1}([\kappa])
\end{align}
induces the same map on homology $\widetilde{\frj}: H_1(\cV,\bZ)\longrightarrow H_1( \chi^{-1}([\kappa]);\bZ)$. Since $\frj_\rho$ preserves holomorphic Liouville forms (\ref{eq: hat vartheta_S}) in the case when $S=\emptyset$, we have for any $\Gamma_i$, 
\begin{align}\label{eq: lemma L_t_gamma}
\int_{\Gamma_i}\vartheta|_{\cS_{\kappa, \rho}}=\int_{\widetilde{\frj}(\Gamma_i)}\vartheta|_{\chi^{-1}([\kappa])},
\end{align}
where the right-hand-side does \emph{not} depend on $\rho$. 
On the other hand, we have  
\begin{align}\label{eq: limit Gamma_i}
\lim\limits_{|\gamma_{-\Pi}(\rho)|\rightarrow 0}\int_{\Gamma_i}\vartheta|_{\cS_{\kappa, \rho}}=\int_{\Gamma_i}\vartheta|_{\cV\times\{\kappa\}}=\lng \kappa, \Gamma_i\rng.
\end{align}
So we have 
\begin{align*}
\int_{\Gamma_i}\vartheta|_{\cS_{\kappa, \rho}}=\lng \kappa, \Gamma_i\rng, \forall i.
\end{align*}
The same holds for $\cS_{\kappa, (c_\beta)}$ because every $(c_\beta)_\beta$ is in the closure of $\gamma_{-\Pi}(T)$.  These imply (i).

(ii) 
Since (\ref{eq: cV, cS, j_rho}) gives an isotopy class of embeddings over $\kappa\in \cK$, it suffices to prove (ii) for generic $\kappa\in \cK^\dagg$. For generic choices of $\kappa$, we may assume that $\lng\kappa, -\rng$ on an integral basis of $H_1(\cV;\bZ)$ is a set of linearly independent complex numbers over $\bQ$, equivalently the map $\lng \kappa,-\rng: H_1(\cV,\bQ)\rightarrow \bC$ is an embedding of vector spaces over $\bQ$. Note that the right-hand-side of (\ref{eq: lemma L_t_gamma}) is equal to $\lng \kappa, \nu_\kappa(\widetilde{\frj}(\Gamma_i)\rng$ with respect to the canonical identification $\nu_\kappa: \chi^{-1}([\kappa])\overset{\sim}{\longrightarrow} C_G(\kappa)\cong T$. This can be directly seen from the Lagrangian correspondence (\ref{eq: Lag corresp}) that induces an exact symplectomorphism $\chi^{-1}(\chi_\ft(\cK))\cong T\times \cK$. Now from the equality between (\ref{eq: lemma L_t_gamma}) and (\ref{eq: limit Gamma_i}), we see that $\Gamma_i, i=1,\cdots, n$, contained in $\cS_{\kappa, \rho}$, gives a basis of $H_1(\chi^{-1}([\kappa]);\bZ)$, and this shows that $\cS_{\kappa, \rho}\hookrightarrow \chi^{-1}([\kappa])$ is a homotopy equivalence. Moreover, by the same consideration, the sequence of maps (\ref{eq: lemma T, j, nu}) induces the identity map on $H_1(T;\bZ)$, hence it induces a homotopy equivalence that is isotopic to the identity. 
\end{proof}

\begin{prop}\label{lemma: empty, Ham isotopy}
Under the same setting as for Lemma \ref{lemma: L_t_gamma}, for any pre-compact open $\cV^\dagg\subsetneq \cV$ (defined similarly as for $\cV$) and any smooth curve $(c_\beta(s))_\beta, s\in (-\epsilon', \epsilon')$ with $(c_\beta(0))_\beta=0$ in $\bC^{\Pi}$, there exists a compactly supported Hamiltonian isotopy $\varphi_s, 0\leq s\leq \epsilon$ (with $\varphi_0=id$) on $\cW_{\cV, \cK}=\cV\times \cK$, for some sufficiently small $\epsilon>0$, such that 
\begin{align*}
\varphi_s(\cV^\dagg\times \{\xi\})\subset \widetilde{\mu}_{\cK', (c_\beta(s))_\beta}^{-1}([\xi]), 
\end{align*}
for every $\xi\in \cK^\dagg\subset \cK$ and $s\in [0,\epsilon]$. 
\end{prop}
\begin{proof}
Fix a reference point $u_0\in \cV$. The Lagrangian section $\{u_0\}\times \cK$ of the Lagrangian fibration $\cV\times \cK\rightarrow \cK$ gives a Lagrangian section for 
\begin{align}\label{eq: mu, cK', c_beta, s}
\widetilde{\mu}_{\cK', (c_\beta(s))_\beta}: \widetilde{\mu}_{\cK', (c_\beta(s))_\beta}^{-1}(\cK^\dagg)\longrightarrow \cK^\dagg, |s|\text{ sufficiently small}. 
\end{align} 
Without loss of generality, we may assume that $\epsilon'$ is sufficiently small, so that the above holds for all $s\in (-\epsilon', \epsilon')$. 
By Proposition \ref{eq: tilde mu, embed}, for every $s$, (\ref{eq: mu, cK', c_beta, s}) is part of a complete integrable system with each fiber a complete $T$-orbit. 
So with respect to some fixed real linear coordinates $(p_{c}^j; p_\bR^j), 1\leq j\leq n$ on $\ft^*\cong \ft\cong \ft_c\oplus \ft_\bR$ (e.g. those introduced in (\ref{eq: q, p})), there are canonical (locally defined) affine coordinates on the fibers $(q_{c,s}^j; q_{\bR,s}^j)$, with base points defined by the Lagrangian section $\{u_0\}\times \cK$, such that the real symplectic form $\omega$ is of the form $-\sum\limits_j dp_{c,s}^j\wedge dq_{c,s}^j+dp_{\bR,s}^j\wedge dq_{\bR,s}^j$. 
Here 
\begin{align*}
p_{c,s}^j(u_0, \xi)=p_c^j(\widetilde{\mu}_{\cK', (c_\beta(s))_\beta}(u_0, \xi)),\  p_{\bR,s}^j(u_0, \xi)=p_\bR^j(\widetilde{\mu}_{\cK', (c_\beta(s))_\beta}(u_0, \xi))
\end{align*} 
on $\{u_0\}\times \cK$. 
We will use $(q_c^j, q_{\bR}^j; p_c^j, p_\bR^j)$ to denote for $(q_{c, 0}^j, q_{\bR,0}^j; p_{c,0}^j, p_{\bR,0}^j)$. 

For any $0<s\leq \epsilon'$, using Proposition \ref{eq: tilde mu, embed}, we can define a $T$-equivariant symplectomorphism
\begin{align}\label{eq: tilde varphi_s}
\widetilde{\varphi}_{s,\rho_0}: T\times \cK^\dagg\longrightarrow \widetilde{\chi}_{S'}^{-1}(\jmath(\cK^\dagg))
\end{align}
over $\cK^\dagg$, that respects the restriction of the chosen Lagrangian sections $\{u_0\}\times \cK$ and $\widetilde{\iota}_{\emptyset}^{S'}\circ \frj_{\rho_0}(\{u_0\}\times \cK)$, where $S'$ and $\rho_0$ depend on $(c_{\beta}(s))_\beta$. 
Now the restriction of $\widetilde{\varphi}_{s,\rho_0}$ gives  
\begin{align*}
\varphi_s: &\cV^\dagg\times \cK^\dagg\longrightarrow  \widetilde{\mu}_{\cK', (c_\beta(s))_\beta}^{-1}(\cK^\dagg)\hookrightarrow \cW_{\cV,\cK}\\
&(q_c^j, q_\bR^j; p_c^j, p_\bR^j)\mapsto (q_{c,s}^j=q_c^j, q_{\bR,s}^j=q_c^j; p_{c,s}^j=p_c^j, p_{\bR,s}^j=p_{\bR}^j),
\end{align*}
with respect to the coordinates defined above, which is \emph{independent} of the choice of $\rho_0$. Since the coordinates $(q_{c,s}^j, q_{\bR,s}^j; p_{c,s}^j, p_{\bR,s}^j)$ change smoothly with respect to $s$, for sufficiently small $s>0$, $\varphi_s$ is well defined and smoothly depending on $s$.  Note that this actually gives an alternative proof of Lemma \ref{lemma: L_t_gamma} (ii).

Lastly, by Lemma \ref{lemma: L_t_gamma} (i), we see that $\varphi_s^*\vartheta_{\std}-\vartheta_{\std}$ is an exact 1-form (which is bounded because all the constructions can be extended to the larger neighborhood $\cV\times \cK$), using the restriction of the standard real Liouville form $\vartheta_{\std}$ on $T^*T$. Equivalently, one can use $\vartheta|_{\cB_{w_0}}$ instead of $\vartheta_{\std}$. Hence $\varphi_s$ can be extended to be a compactly supported Hamiltonian isotopy on $\cV\times \cK$. This completes the proof of the proposition. 
\end{proof}

\begin{notations}\label{notations}
For inclusion of open cones $C_{\lhd}\subset C_{\lhd}'\subset \ft_{\bR}^+- \{0\}$ (recall $\ft_\bR^+$ is closed), we use the notation $C_{\lhd}\dot{\subset}C_{\lhd}'\dot{\subset} \ft_{\bR}^+- \{0\}$ to indicate the condition that $\overline{C}_{\lhd}-\{0\}\subset C_{\lhd}'$ and $\overline{C'}_{\lhd}-\{0\}\subset\mathring\ft_{\bR}^+$. 
\end{notations}

\begin{lemma}\label{lemma: proj X_eta, c}
Assume the same setting as for Lemma \ref{lemma: L_t_gamma}. Fix any open cones $C_{\lhd}\dot{\subset} C_{\lhd}'\dot{\subset} \ft_{\bR}^+- \{0\}$.  
Then there exists $\epsilon_{C_\lhd}>0$ and $M>0$, such that for all $|(c_\beta)_{\beta\in \Pi}|<\epsilon_{C_\lhd}$ and all $\eta\in C_{\lhd}\subset C^\infty(\cK;\bR)$ (or equivalently viewed as a holomorphic function in the holomorphic setting), the Hamiltonian vector field $X_{\eta;(c_\beta)}$ of the pullback function $\widetilde{\mu}_{\cK', (c_\beta)_\beta}^*(\eta)$ satisfies the following:

for any $(u,\xi)\in \widetilde{\mu}_{\cK', (c_\beta)_\beta}^{-1}(\cK^\dagg)\subset \cV\times\cK$, the projection of $X_{\eta;(c_\beta)}(u,\xi)$ in $T_u\cV\cong \ft$ is contained in $C_{\lhd}'+\ft_c$ and 
\begin{align*}
|X_{\eta;(c_\beta)}(u,\xi)-(\frj_u)_*\eta(u,\xi)|\leq M\cdot |(c_\beta)_\beta|\cdot |\eta|.
\end{align*} 
\end{lemma}

\begin{proof}
It is clear from the definition (\ref{eq: mu_D,K, canonical})
\begin{align*}
\widetilde{\mu}_{\cK', (c_\beta)_\beta}(u,\xi)=\xi+\sum\limits_{\beta\in \Pi} c_\beta P_\beta(u,\xi)+\cdots
\end{align*}
has a convergent analytic expansion in $c_\beta$ with coefficients in analytic $\ft^*$-valued functions of $(u,\xi)$. Thus the holomorphic Hamiltonian vector field $X^{hol}_{\eta;(c_\beta)}$ has an analytic expansion 
\begin{align*}
X^{hol}_{\eta, (c_\beta)}(u,\xi)=(\frj_u)_*\eta+\sum\limits_{\beta\in \Pi} c_\beta X^{hol}_{\eta;\beta}(u,\xi)+\cdots
\end{align*}
where $\eta\in \ft$ is the invariant vector field on each fiber $\cV\times \{\kappa\}\subset T\times\{\kappa\}$ and $X^{hol}_{\eta;\beta}$ is the holomorphic Hamiltonian vector field of $\lng \eta, P_{\beta}(u,\xi)\rng$. Note that the corresponding real Hamiltonian vector field is $X_{\eta;(c_\beta)}=2\Re X^{hol}_{\eta;(c_\beta)}$. Since $\cW_{\cV, \cK}$ is pre-compact, the lemma follows.  
\end{proof}

Similarly as for $\widetilde{\varphi}_{s, \rho_0}$ (\ref{eq: tilde varphi_s}) in the proof of Proposition  \ref{lemma: empty, Ham isotopy}, we define for $|(c_\beta)_\beta|\ll 1$
\begin{align}\label{eq: tilde varphi_c}
\widetilde{\varphi}_{(c_\beta), \rho_0}: T\times \cK^\dagg\longrightarrow \widetilde{\chi}_{S'}^{-1}(\jmath(\cK^\dagg)) 
\end{align} 
to be the $T$-equivariant symplectomorphism over $\cK^\dagg$ that sends the Lagrangian section $\{u_0\}\times \cK^\dagg$ to the restriction of $\widetilde{\iota}_\emptyset^{S'}\circ \frj_{\rho_0}(\{u_0\}\times \cK)$, where $S'$ and $\rho_0$ depending on $(c_\beta)_\beta$ are as in Proposition \ref{eq: tilde mu, embed}. In particular, $\gamma_{-S'}(\rho_0)=(c_\beta)_{\beta\in S'}$.  

For any subset $C\subset \ft_\bR^+-\{0\}$, let $T_{C}$ denote for the preimage of $C$ through the real logarithmic map $\log_{\bR}: T\rightarrow \ft_\bR$. 

\begin{prop}\label{prop: C_lhd, j}
Fix any open cone $C_\lhd\dot{\subset} \ft_\bR^+-\{0\}$. 
Under the same setting as for Lemma \ref{lemma: L_t_gamma}, there exists $\epsilon_{C_\lhd}>0$ such that for all $|(c_\beta)_\beta|<\epsilon_{C_\lhd}$ and $\rho'\in T_{C_\lhd}$, 
\begin{align}\label{eq: cor rho}
\rho'\star (\widetilde{\iota}_\emptyset^{S'}\circ \frj_{\rho_0}(\overline{\widetilde{\mu}_{\cK', (c_\beta)_\beta}^{-1}(\cK^\dagg)}))\subset \widetilde{\iota}_\emptyset^{S'}\circ \frj_{\rho_0\rho'}(\cW_{\cV', \cK'}),
\end{align}
where $\rho_0$ is associated with $(c_\beta)_\beta$ as above, and the action on the left-hand-side is from the $T$-action on the right-hand-side of (\ref{eq: tilde varphi_c}) with respect to $\chi_\ft(\cK')\cong\cK'$ (cf. Remark \ref{remark: cK in fz_S}) . 
Moreover, for any chosen $\delta>0$, we can choose $\epsilon_{C_\lhd}>0$ so that there is a uniform bound 
\begin{align}\label{eq: cor dist}
dist ((\widetilde{\iota}_\emptyset^{S'}\circ \frj_{\rho_0\rho'})^{-1}(\rho'\star (\widetilde{\iota}_\emptyset^{S'}\circ \frj_{\rho_0}(u,\xi))), (u,\xi))<\delta 
\end{align}
for all $(u, \xi)\in \widetilde{\mu}_{\cK', (c_\beta)_\beta}^{-1}(\cK^\dagg)\subset \cV'\times \cK'$ and $\rho'\in T_{C_\lhd}$. Here the distance is taken with respect to the standard $T$-invariant metric on $\cB_{w_0}\cong T^*T$. 
\end{prop}
\begin{proof}
First, choose $C_{\lhd}\dot{\subset} C_{\lhd}'\dot{\subset} \ft_{\bR}^+- \{0\}$,  $\epsilon_{C_\lhd}>0$ and $M>0$ satisfying the assumption and conclusion in Lemma \ref{lemma: proj X_eta, c}.  
By fixing the embedding $\cB_{w_0}$ into $J_{L_{S'}}$ through $\widetilde{\iota}_{\emptyset}^{S'}$, we can view everything inside $J_{L_{S'}}$, so we will omit $\widetilde{\iota}_{\emptyset}^{S'}$ in the proof. 
Since the embedding $\widetilde{\iota}_{\emptyset}^{S'}$ is $\cZ(L_{S'})$-equivariant for the obvious $\cZ(L_{S'})$-action on the source and target, the proposition can be reduced to the case when $S'=\Pi$ and $(c_\beta)_\beta=\gamma_{-\Pi}(\rho_0)$. 
It suffices to prove (\ref{eq: cor dist}) for the chosen Lagrangian section $\{u_0\}\times \cK$, and it is equivalent to saying
\begin{align}\label{eq: proof dist estimate}
\rho'\star \frj_{\rho_0}(u_0, \widetilde{\xi})\in \cB_{w_0}, \text{ and } dist(\rho'\star \frj_{\rho_0}(u_0, \widetilde{\xi}), \frj_{\rho_0\rho'}(u_0,\xi))<\delta, 
\end{align}
where $(u_0, \widetilde{\xi})=(\{u_0\}\times \cK)\cap \widetilde{\mu}_{\cK', (c_\beta)_\beta}^{-1}(\xi)$. 

For any $\eta\in C_{\lhd}$ and $\rho'_c\in T_c$, let $\Upsilon_{\eta}(s)=\rho'_c\cdot \exp(s\cdot \eta), s\geq 0$. With given $(c_\beta)$, we have 
\begin{align}\label{eq: Upsilon_eta}
\frac{d}{ds}\Upsilon_{\eta}(s)\star\frj_{\rho_0}(u_0,\widetilde{\xi})= X_{\eta, (c_\beta)}(\Upsilon_{\eta}(s)\star \frj_{\rho_0}(u_0,\widetilde{\xi})).
\end{align}
We claim that 
\begin{align}\label{eq: gamma_eta, s, cB}
\Upsilon_{\eta}(s)\star \frj_{\rho_0}(u_0,\widetilde{\xi})\subset (T_{C_\lhd'}\cdot \frj_{\rho_0}(\cV))\times \cK
\end{align}
for all $s\geq 0$. Suppose the contrary, there exists $r>0$ such that (\ref{eq: gamma_eta, s, cB}) holds for $s\in [0, r)$ but $\Upsilon_{\eta}(r)\star \frj_{\rho_0}(u_0,\widetilde{\xi})$ is outside $(T_{C_\lhd'}\cdot\frj_{\rho_0}(\cV))\times \cK$. Pick $r_1<r$ that is very close to $r$, and let $\rho_1=u_0^{-1}\cdot \proj_T (\Upsilon_{\eta}(r_1)\star \frj_{\rho_0}(u_0,\widetilde{\xi}))\in u_0^{-1}\cdot T_{C_\lhd'}\cdot \frj_{\rho_0}(\cV)$, then 
\begin{align*}
\frj_{\rho_1}^{-1}(\Upsilon_{\eta}(r_1)\star \frj_{\rho_0}(u_0,\widetilde{\xi}))\in \{u_0\}\times \cK. 
\end{align*}
Since 
\begin{align*}
\frj_{\rho_1}^{-1}(\Upsilon_{\eta}(r_1)\star \frj_{\rho_0}(u_0,\widetilde{\xi}))\in \widetilde{\mu}_{\cK', (\gamma_{-\Pi}(\rho_1))}^{-1}(\xi),
\end{align*}
and $|\gamma_{-\Pi}(\rho_1)|<\epsilon_{C_\lhd}$, we have $\frj_{\rho_1}^{-1}(\Upsilon_{\eta}(r_1)\star \frj_{\rho_0}(u_0,\widetilde{\xi}))$ very close to $(u_0, \xi)$. Hence by a similar argument as in Proposition \ref{lemma: empty, Ham isotopy}, there exists a fixed interval $[0, \nu], \nu>0$, depending only on $\eta$, such that for any $\epsilon\in [0,\nu]$, 
\begin{align*}
&\frj_{\rho_1}^{-1}(\Upsilon_{\eta}(r_1+\epsilon)\star \frj_{\rho_0}(u_0,\widetilde{\xi}))\subset T_{C'_{\lhd}}\cdot \frj_{\rho_1}^{-1}(\proj_T(\Upsilon_{\eta}(r_1)\star \frj_{\rho_0}(u_0,\widetilde{\xi})))\times \cK\\
\Rightarrow& \Upsilon_{\eta}(r_1+\epsilon)\star \frj_{\rho_0}(u_0,\widetilde{\xi})\subset T_{C_{\lhd}'}\cdot (\proj_T(\Upsilon_{\eta}(r_1)\star \frj_{\rho_0}(u_0,\widetilde{\xi})))\times \cK\subset (T_{C_\lhd'}\cdot \frj_{\rho_0}(\cV))\times \cK.
\end{align*}
Choosing $r_1>r-\nu$ gives a contradiction to the assumption that (\ref{eq: gamma_eta, s, cB}) does not hold at $r$. 

We show the estimate on distance in (\ref{eq: proof dist estimate}). Let $\rho_\eta(s)=\proj_T (\Upsilon_\eta(s)\star\frj_{\rho_0}(u_0, \widetilde{\xi}))$, then we have 
\begin{align}\label{eq: DE, rho_eta}
(\frj_{u_0\rho_\eta(s)^{-1}})_{*}\frac{d}{ds}\rho_\eta(s)=\proj_T X_{\eta, \gamma_{-\Pi}(u_0^{-1}\rho_\eta(s))}(\frj_{u_0^{-1}\rho_\eta(s)}^{-1}(\Upsilon_\eta(s)\star\frj_{\rho_0}(u_0, \widetilde{\xi})))
\end{align}
where both sides are contained in $T_{u_0}\cV$. Using the estimate from Lemma \ref{lemma: proj X_eta, c}, we get 
\begin{align*}
|u_0\rho_\eta(s)^{-1}\frac{d}{ds}\rho_\eta(s)-(\frj_{u_0})_*\eta|&\leq M\cdot |\gamma_{-\Pi}(u_0^{-1}\rho_\eta(s))|\cdot |\eta|.
\end{align*}
By the assumption on $C_\lhd'$, there exists $\varepsilon>0$ such that for all $\beta_i\in \Pi$,  
\begin{align*}
&\varepsilon\leq \frac{\beta_j}{\sum\limits_{i=1}^n\beta_i}\leq 1-\varepsilon \text{ on }C_\lhd'\\
\Rightarrow& |\gamma_{-\Pi}(u_0^{-1}\rho_\eta(s))|\leq n|\beta_j(u_0^{-1}\rho_\eta(s))|^{-\frac{1}{K}}
\end{align*}
for every $j$ and a uniform constant $K>0$ only depending on $\varepsilon$. Therefore, looking at each component $\beta(\rho_\eta(s))\in \bC^\times$ for (\ref{eq: DE, rho_eta}), we get 
\begin{align}\label{eq: beta, rho_eta, d}
&\beta(\rho_\eta(s))^{-1}\frac{d}{ds}\beta(\rho_\eta(s))=\beta(\eta)+O(|\beta(\rho_\eta(s))|^{-\frac{1}{K}}\cdot |\eta|). 
\end{align}
Let $F_\beta(s)=\log |\beta(\rho_\eta(s))e^{-\beta(\eta)s}|$, then the above on the real parts implies
\begin{align*}
&|\frac{d}{ds} F_\beta(s)|\leq \widetilde{M}\cdot e^{-\frac{F_\beta(s)}{K}-\frac{\beta(\eta)s}{K}}|\eta|,\ \beta\in \Pi\\
\Rightarrow & |\frac{d}{ds}e^{\frac{F_\beta(s)}{K}}|\leq \frac{\widetilde{M}}{K}e^{-\frac{\beta(\eta)s}{K}}|\eta|\\
\Rightarrow &|e^{\frac{F_\beta(s)}{K}}-|\beta(\rho_\eta(0))|^{1/K}|\leq \widetilde{M}' (1-e^{-\frac{\beta(\eta)s}{K}})\leq \widetilde{M'}\\
\Rightarrow&K\log( |\beta(\rho_\eta(0))|^{1/K}-\widetilde{M}')\leq F_\beta(s)\leq K\log(|\beta(\rho_\eta(0))|^{1/K}+ \widetilde{M}')\\
\Rightarrow & K\log(1- \frac{\widetilde{M}'}{|\beta(\rho_\eta(0))|^{1/K}})\leq F_\beta(s)-\log|\beta(\rho_\eta(0))|\leq K\log(1+ \frac{\widetilde{M}'}{|\beta(\rho_\eta(0))|^{1/K}}).
\end{align*}
Here $K, \widetilde{M}'$ only depend on $u_0, C_\lhd, C'_{\lhd}$. Assume that we have chosen $|\beta(\rho_\eta (0))|, \beta\in \Pi$, sufficiently large, equivalently $|(c_\beta)_\beta|$ sufficiently small, then 
\begin{align}\label{}
|\log |\beta(\rho_\eta(s)\rho_\eta(0)^{-1})e^{-\beta(\eta) s}||=|F_\beta(s)-\log|\beta(\rho_\eta(0))||<\delta', \forall \beta\in \Pi, s\geq 0
\end{align}
for arbitrarily small $\delta'>0$.

Lastly, taking the imaginary part of (\ref{eq: beta, rho_eta, d}) and using the above, we get 
\begin{align*}
&\frac{d}{ds}\arg \beta(\rho_\eta(s))=O(|\beta(\rho_\eta(0))|^{-\frac{1}{K}}e^{-\frac{\beta(\eta)s}{K}}\cdot |\eta|)\\
\Rightarrow& |\arg \beta(\rho_\eta(s))-\arg \beta(\rho_\eta(0))|\leq \widetilde{M}'|\beta(\rho_\eta(0))|^{-\frac{1}{K}}.
\end{align*}
By choosing $|(c_\beta)_\beta|$ sufficiently small, we can make the right-hand-side arbitrarily small, and also make $\rho_\eta(0)$ very close to $\rho_c'\rho_0u_0$. Thus we have proved the distance estimate in (\ref{eq: proof dist estimate}).  
\end{proof}

Now we are ready to give a refinement of Lemma \ref{lemma: hat chi_epsilon, gamma}. 

\begin{cor}\label{cor: dist star}
Under the same setting as in Proposition \ref{prop: C_lhd, j}, for any $\delta>0$, there exists $\epsilon_{C_\lhd}>0$ such that for any $(u,\xi)\in \cV\times \cK^\dagg$, $\rho_1\in T$ satisfying $|\gamma_{-\Pi}(\rho_1)|<\epsilon_{C_\lhd}$ and $\rho'\in T_{C_\lhd}$, we have 
\begin{align}
\label{eq: dist, rho'rho_1}&dist(\rho'\star \frj_{\rho_1}(u,\xi), \frj_{\rho'\rho_1}(u,\xi))<\delta
\end{align}
Moreover, 
\begin{align}
\label{eq: dist, W}&dist(w^{-1}(\rho')\star \frj_{\rho_1}(u,w(\xi))), \frj_{\rho'\rho_1}(u,w(\xi)))<\delta, \forall w\in W,
\end{align}
where both the $T$-action denoted by $\star$ are taken with respect to $\chi_\ft(\cK')\cong\cK'\subset \ft^\reg$ (cf. Remark \ref{remark: cK in fz_S}). The distance is taken with respect to the standard $T$-invariant metric on $T^*T$. 
\end{cor}

\begin{proof}
First, (\ref{eq: dist, rho'rho_1}) is the special case of Proposition \ref{prop: C_lhd, j} (\ref{eq: cor dist}) for $S'=\Pi$. Although in the proposition, it is stated for $(u, \xi)\in \mu_{\cK', \rho_1}^{-1}(\cK^\dagg)\subset \cV'\times \cK'$, it also holds for $\cV\times \cK^\dagg$ by enlarging the former $\cK^\dagg$ slightly.  

Second, for any $w\in W$, using $(u,w(\xi))\in \cV\times w(\cK^\dagg)$, we have 
\begin{align*}
dist(\rho'\star \frj_{\rho_1}(u,w(\xi))), \frj_{\rho'\rho_1}(u,w(\xi)))<\delta,
\end{align*}
where the $T$-action denoted by $\star$ here is \emph{with respect to $\chi_\ft(\cK')\cong w(\cK')\subset \ft^\reg$}. By Remark \ref{remark: cK in fz_S}, this $T$-action differs from the $T$-action in (\ref{eq: dist, W}) by $w$, hence (\ref{eq: dist, W}) follows.  
\end{proof}

\subsubsection{Analysis inside $\fU_S, \emptyset\neq S\subsetneq \Pi$}\label{subsec: analysis fU_S}

In this section, we generalize several results from Subsection \ref{subsubsec: B_w_0} to $\emptyset\neq S\subsetneq \Pi$. We also give an answer to Question \ref{eq: D, K, dagger} (ii), which was trivial for $S=\emptyset$. Recall the notations from Question \ref{eq: D, K, dagger}. In particular, we are under the settings depicted in Figure \ref{figure: mu_D, cK', c}. 

\begin{figure}[htbp]
\begin{tikzpicture}
\draw (-2.5,-2)--(2.5,-2) node[pos=0.5, below] {$\ \ \ \ \ \kappa\in \cK'_{S^\perp}$};
\draw (-2.5,-1.6)--(-2.5,1.2) node[pos=0.5, left] {$\cV_{S^\perp}$};
\draw (-1, -1.6)--(1,-0.6) to [bend right](1, 1.4)--(-1, 0.4) to [bend left](-1, -1.6);
\draw[dashed] (0, -1.1)--(0, -2);
\draw[blue, thick] (-0.8, -1.2)--(1.2, -0.2); 
\draw[blue, thick] (-0.7, -0.4)--(1.3, 0.6); 
\draw[green] (-0.6, -1.4) to [bend right] (-0.6, 0.6);
\draw[green] (-0.2, -1.2) to [bend right] (-0.1, 0.8);
\draw[green] (0.2, -1) to [bend right] (0.5, 1);
\draw (2.6, 0) node {$\widetilde{\mu}_{D, \cK', (c_\beta)_\beta}^{-1}(\kappa)$};
\draw[blue] (-0.8, -1.2) node[left] {$\cY_S$};
\end{tikzpicture}
\caption{A picture of the fiber $\widetilde{\mu}_{D, \cK', (c_\beta)_{\beta\in \Pi\backslash S}}^{-1}(\kappa)$, where the blue multi-section (it is connected although we draw it disconnected in this low dimensional picture) indicates the intersection of $\cY_S\times \{u_0\}\times \cK'_{S^\perp}$ with the fiber, and the green curves represent the characteristic foliations. }\label{figure: mu_D, cK', c}
\end{figure}
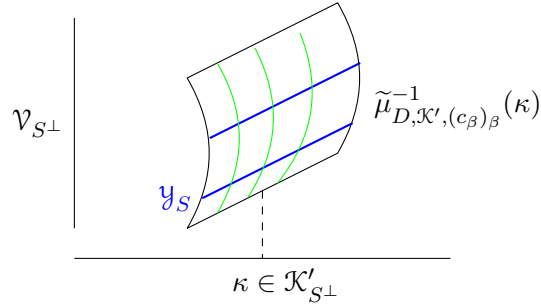

First, we state some direct generalizations of results from Subsection \ref{subsubsec: B_w_0}. For any $\kappa\in \cK^\dagg_{S^\perp}$, let $\cS_{\kappa, (g_S, \xi_S), (c_\beta)}$ denote for the characteristic leaf in $\widetilde{\mu}_{D, \cK', (c_\beta)_\beta}^{-1}(\kappa)\subset \cW_{\cY_S, \cV, \cK'}$ that passes through the point $(g_S, \xi_S; u_0, \widetilde{\kappa})$. Note that $\widetilde{\kappa}$ is uniquely determined for the restriction of $\widetilde{\mu}_{D, \cK', (c_\beta)}$ from $\{(g_S,\xi_S)\}\times\{u_0\}\times \cK_{S^\perp}$ to $\cK'_{S^\perp}$ is an open embedding. Let $\cD_{(g_S, \xi_S)}$ be a contractible neighborhood of $(g_S, \xi_S)$ in $\cY_S$ that is contained in a fundamental domain of the $\cZ(L_S^\der)_0$-action. 
Then Lemma \ref{lemma: L_t_gamma} immediately generalizes to the following form.

\begin{lemma}\label{lemma: cS_kappa, S}
Under the above settings, there exists $r_{\cV}>0$ such that for all $|(c_\beta)_{\beta\in \Pi\backslash S}|<r_{\cV}$, we have for each $\kappa\in \cK_{S^\perp}^\dagg$, the characteristic leaf $\cS_{\kappa, (g_S, \xi_S), (c_\beta)}$
satisfies
\begin{itemize}

\item[(i)] $\cS_{\kappa, (g_S, \xi_S), (c_\beta)}$ is a smooth section over $\cV_{S^\perp}$ that is Hamiltonian isotopic to $\{(g_S,\xi_S)\}\times \cV_{S^\perp}\times \{\kappa\}$ inside 
\begin{align}\label{eq: lemma cS_kappa, S}
\cD_{(g_S, \xi_S)}\times (\cV_{S^\perp}\times \cK')\subset \cW_{\cY_S, \cV, \cK'}.
\end{align}

\item[(ii)] The natural inclusion $\cS_{\kappa, (g_S, \xi_S), (c_\beta)}\overset{\widetilde{\iota}_S^{S'}\circ\frj_{S, \rho_0}}\longhookrightarrow \widetilde{\chi}_{S',\cK'}^{-1}(\kappa)$ from (\ref{eq: lemma tilde chi_S'}) induces a homotopy equivalence from the former to the $\cZ(L_S)_0$-orbit that contains it. Moreover, by reverting the first homotopy equivalence, the sequence
\begin{align*}
\cZ(L_S)_0\longhookleftarrow \cV_{S^\perp}\cong \cS_{\kappa, (g_S, \xi_S), (c_\beta)} \longhookrightarrow  \cZ(L_S)_0\star \widetilde{\iota}_S^{S'}(g_S, \xi_S; u_0\rho_0, \widetilde{\kappa})\cong \cZ(L_S)_0
\end{align*}
induces a homotopy equivalence from $\cZ(L_S)_0$ to itself that is isotopic to the identity. 

\end{itemize}
\end{lemma}
\begin{proof}
It follows from the same proof for Lemma \ref{lemma: L_t_gamma}. 
\end{proof}

\begin{remark}\label{remark: modify multi-section}
For $(\kappa, (c_\beta)_\beta)$ ranging in the space 
\begin{align}\label{eq: space kappa, c_beta}
\cK_{S^\perp}\times \{(c_\beta)_\beta\in \bC^{\Pi\backslash S}:|(c_\beta)_\beta|<r_\cV\}, 
\end{align}
the intersection $(\cY_S\times \{u_0\}\times \cK'_{S^\perp})\cap \widetilde{\mu}_{D^\dagg, \cK', (c_\beta(s))_\beta}^{-1}(\kappa)$ gives a $|\cZ(L_S^\der)_0|$-multi-section over its image in the reduced space, i.e. the quotient of $\widetilde{\mu}_{D^\dagg, \cK', (c_\beta)_\beta}^{-1}(\kappa)$ by the characteristic leaves. In the following, we modify these multi-sections to be $\cZ(L_S^\der)_0$-equivariant with respect to the ``moment map" $\widetilde{\mu}_{D^\dagg, \cK', (c_\beta)_\beta}$. 
For $(c_\beta)_\beta=0$, the multi-section is $\cZ(L_S^\der)_0$-invariant with respect to $\widetilde{\mu}_{D^\dagg, \cK',0}$. For close by $(c_\beta)_\beta$, we can do an averaging process, to make the multi-section $\cZ(L_S^\der)_0$-invariant with respect to $\widetilde{\mu}_{D^\dagg, \cK', (c_\beta(s))_\beta}$ after applying Proposition \ref{eq: tilde mu, embed}. 
More precisely, since the multi-section is very close to be $\cZ(L_S^\der)_0$-invariant, for any characteristic leaf, we can use the respective $\cZ(L_S^\der)_0$-action to move the points in the original multi-section to a small neighborhood of any chosen one of the points (the result will be independent of the chosen point), then we do an average in that small neighborhood (using the $\cZ(L_S)_0$-action from group elements near the identity) which is well defined, and we turn its $\cZ(L_S^\der)_0$-orbit to be the new multi-section restricted to that leaf. This gives the modification, and we denote the resulting multi-section for $(c_\beta)_\beta$ as $\cY^\dagg_{S, \kappa, (c_\beta)}$. If $(c_\beta)_\beta=\gamma_{-\Pi\backslash S}(\rho)$ for some $\rho\in \cZ(L_S)_0$, we also denote $\cY^\dagg_{S, \kappa, (c_\beta)}$ by $\cY^\dagg_{S, \kappa, \rho}$. 
\end{remark}

\begin{lemma}\label{lemma: cY_S, R}
Fix $\kappa\in \cK_{S^\perp}^\dagg$. Assume that $\cY_{S,R}$ is defined by 
\begin{align}\label{eq: def cY_S,R}
\cY_{S,R}:=\{\sum\limits_{\beta\in S}|b^S_{\lambda_{\beta^\vee}}|^{\frac{1}{\lambda_{\beta^\vee}(\sfh_{0;S})}}<R\}
\end{align}
inside $\chi_S^{-1}(D_S/W_S)\subset J_{L_{S}^\der}$. Then
\begin{itemize}
\item[(i)] Fix any compact region in the Hamiltonian reduction of $\chi_{\cK'_{S^\perp}}$ at $\kappa$, which is also canonically identified with $\chi_S^{-1}(D'_S/W_S)/\cZ(L_S^\der)_0$. For any $R>0$ sufficiently large, there exists $\epsilon_R>0$ such that for all $\rho\in \cZ(L_S)_0$ satisfying $|\gamma_{-\Pi\backslash S}(\rho)|<\epsilon_R$, the image of some fixed tubular neighborhood of the ``horizontal boundary" of $\cY^\dagg_{S,R; \kappa, \rho}$ (i.e. the intersection of $\cY^\dagg_{S,R; \kappa, \rho}$ with a tubular neighborhood of 
\begin{align*}
\{\sum\limits_{\beta\in S}|b_{\lambda_{\beta^\vee}}^S(g_S)|^{\frac{1}{\lambda_{\beta^\vee}(\sfh_{0;S})}}=R\}, 
\end{align*}
in $\widetilde{\mu}_{D^\dagg, \cK', (c_\beta(s))_\beta}^{-1}(\kappa)$) under $\overline{\frj}_{S,\rho;\kappa}$ (\ref{eq: j_S,rho, reduction}), is outside the fixed compact region.  

\item[(ii)] Fixing $R>0$, the image of $\cY^\dagg_{S,R;\kappa, \rho}/\cZ(L_S^\der)_0$ under $\overline{\frj}_{S,\rho;\kappa}$ is contained in some fixed compact region, for all $\rho\in \cZ(L_S)_0$ with sufficiently small $|\gamma_{-\Pi\backslash S}(\rho)|$.   
\end{itemize}
The same claims hold with $\mu_{D', \cK', \rho}$ replaced by $\widetilde{\mu}_{D',\cK', (c_\beta)_\beta}$ for $|(c_\beta)_{\beta\in \Pi\backslash S}|$ sufficiently small.

\end{lemma}

\begin{proof}
First, the statements about $\widetilde{\mu}_{D',\cK', (c_\beta)_\beta}$ can be deduced from those about $\mu_{D', \cK', \rho}$ by replacing the group $G$ with $L_{S'}$ and using Proposition \ref{eq: tilde mu, embed}. So it suffices to prove the statements for $\mu_{D', \cK', \rho}$.

For any $\cY_{S,R}$, we choose $\rho$ with $|\gamma_{-\Pi\backslash S}(\rho)|$ sufficiently small so that $\mu_{D', \cK', \rho}: \cW_{\cY_{S,R}, \cV, \cK}\rightarrow \cK'_{S^\perp}$ is well defined. Fix any point  $(g_S, \xi_S;z, t)$ in  $\mu_{D', \cK', \rho}^{-1}(\kappa)$. 
Without loss of generality, we may assume $\xi_S$ is from the Kostant slice $\cS_{\fl^\der_S}$ for the semisimple Lie algebra $\fl^\der_S$, and $g_S$ be the respective centralizing element. Recall the notation from (\ref{eq: prop U_S}) $(\phi_Sg_Sz, \Xi_S(g_S, \xi_S;z, t))$. For $\frj_{S,\rho}(g_S, \xi_S;z, t)$, there exists a (unique) $u_\rho\in N$ such that 
\begin{align}\label{eq: Xi, g_S, xi_S, z, t}
(u_\rho\phi_Sg_Sz\rho, \Xi_\rho:= \Xi_S(g_S, \xi_S; z\rho,t))
\end{align}
is a centralizing pair. As $|\gamma_{-\Pi\backslash S}(\rho)|\rightarrow 0$, $\Xi_\rho$ is approaching $\Xi_0:=(f-f_S)+\xi_S+t$. 

On the other hand, let $(g_{S,\rho}^\natural, \xi_{S,\rho}^\natural; z_\rho^\natural, t_\rho^\natural)$ be a representative of $\frj_{S,\rho}(g_S, \xi_S;z, t)$ under the isomorphism (\ref{eq: chi_S, mathscr Z}). Here we also assume that $\xi_{S,\rho}^\natural$ is in the Kostant slice $\cS_{\fl^\der_S}$, so then it is uniquely determined. It is clear from the above discussion that $\xi_{S,\rho}^\natural$ (resp. $t_\rho^\natural$) is arbitrarily close to $\xi_S$ (resp. $t$) as $|\gamma_{-\Pi\backslash S}(\rho)|\rightarrow 0$. In particular, there exists $\epsilon_{g_S}>0$ (the dependence is only on $g_S$ due to the boundedness of $\xi_S, z,t$) such that for $\rho$ satisfying $|\gamma_{-\Pi\backslash S}(\rho)|<\epsilon_{g_S}$, we can find $Q_\rho=u_{1,\rho}^-u_{2,\rho}\in N_{P_S}^-\cdot N$ with $u_{1,\rho}^-$ (resp. $u_{2,\rho}$) contained in a fixed compact region in (the opposite of the unipotent radical of $P_S$) $N_{P_S}^-$ (resp. arbitrarily close to $I\in N$), such that $\Ad_{Q_\rho}(\Xi_\rho)=\xi_{S,\rho}^\natural+t_\rho^\natural$. More explicitly, we first find $u_{1,\rho}^-\in N_{P_S}^-$ such that $\Ad_{(u_{1,\rho}^-)^{-1}}(\xi_{S,\rho}^\natural+t_\rho^\natural)=(f-f_S)+\xi_{S,\rho}^\natural+t_\rho^\natural$  (this follows from a similar argument as for \cite[Lemma 3.1.44]{CG}). Since $\Xi_\rho$ is arbitrarily close to $(f-f_S)+\xi_{S,\rho}^\natural+t_\rho^\natural$ (and both of them are in $f+\fb$) and they are in the same adjoint orbit, we can find $u_{2,\rho}\in N$ close to $I$ such that $\Ad_{u_{2,\rho}}\Xi_\rho=(f-f_S)+\xi_{S,\rho}^\natural+t_\rho^\natural$.

We must have an equality
\begin{align}
\nonumber&\Ad_{Q_\rho}(u_\rho\phi_Sg_Sz\rho)=g_{S,\rho}^\natural z_\rho^\natural\\
\label{eq: Ad_u2,epsilon}\Leftrightarrow & \Ad_{u_{2,\rho}}(u_\rho\phi_Sg_Sz\rho)(u_{1,\rho}^-)^{-1}=(u_{1,\rho}^-)^{-1}g_{S,\rho}^\natural z_\rho^\natural. 
\end{align}
Now we compare the value of $|b_{\lambda_{\beta^\vee}}|, \beta\in \Pi$ on both sides.

First, we consider the case when $\beta\not \in S$. Let us evaluate $|b_{\lambda_{\beta^\vee}}|$ on the right-hand-side of (\ref{eq: Ad_u2,epsilon}). Recall that
\begin{align}\label{eq: b, u_1,epsilon-}
|b_{\lambda_{\beta^\vee}}((u_{1,\rho}^-)^{-1}g_{S,\rho}^\natural z_\rho^\natural)|=|\lng (u_{1,\rho}^-)^{-1}g_{S,\rho}^\natural z_\rho^\natural (v_{\lambda_{\beta^\vee}}), v_{-w_0(\lambda_{\beta^\vee})}\rng|,
\end{align}
where $v_{\lambda_{\beta^\vee}}$ and $v_{-w_0(\lambda_{\beta^\vee})}$ are highest weight vectors in $V_{\lambda_{\beta^\vee}}$ and $V_{\lambda_{\beta^\vee}}^*\cong V_{-w_0(\lambda_{\beta^\vee})}$ and the right-hand-side is the absolute value of the pairing\footnote{More precisely, we need to take a lift of $(u_{1,\rho}^-)^{-1}g_{S,\rho}^\natural z_\rho^\natural$ in $G_{sc}$ to make $b_{\lambda_{\beta^\vee}}$ well defined. But the value of $|b_{\lambda_{\beta^\vee}}|$ does not depend on the choice of the lifting. Similarly, the line $\bC\cdot(u_{1,\rho}^-)^{-1}g_{S,\rho}^\natural z_\rho^\natural (v_{\lambda_{\beta^\vee}})$ does not depend on the choice of the lifting either. }. If $\beta\not\in S$, then 
\begin{align*}
\bC\cdot (u_{1,\rho}^-)^{-1}g_{S,\rho}^\natural z_\rho^\natural (v_{\lambda_{\beta^\vee}})=\bC\cdot (u_{1,\rho}^-)^{-1}(v_{\lambda_{\beta^\vee}})
\end{align*}
is an invariant line of $\Ad_{u_{2,\rho}}(\Xi_\rho)$. 
Indeed, we have 
\begin{align*}
&\Ad_{u_{2,\rho}}(\Xi_\rho)\cdot (u_{1,\rho}^-)^{-1}v_{\lambda_{\beta^\vee}}=\Ad_{(u_{1,\rho}^-)^{-1}}(\xi_{S,\rho}^\natural+t_\rho^\natural)\cdot (u_{1,\rho}^-)^{-1}v_{\lambda_{\beta^\vee}}\\
&=\lambda_{\beta^\vee}(t_\rho^\natural)(u_{1,\rho}^-)^{-1}v_{\lambda_{\beta^\vee}}.
\end{align*}

By Lemma \ref{lemma: nonzero lowest component} below, (\ref{eq: b, u_1,epsilon-}) is nonzero and we have 
\begin{align}\label{eq: b, RHS, not in S}
&c\cdot |\lambda_{\beta^\vee}(z_\rho^\natural)|\leq |b_{\lambda_{\beta^\vee}}((u_{1,\rho}^-)^{-1}g_{S,\rho}^\natural z_\rho^\natural)|\leq C\cdot |\lambda_{\beta^\vee}(z_\rho^\natural)|,\text{ for } |\gamma_{-\Pi\backslash S}(\rho)|\rightarrow 0, \\
\nonumber&(g_S, \xi_S;z, t\text{ fixed}), \beta\not\in S
\end{align}
for some fixed positive constants $c, C>0$. On the other hand, if we evaluate $|b_{\lambda_{\beta^\vee}}|$ on the left-hand-side of (\ref{eq: Ad_u2,epsilon}), we get 
\begin{align}\label{eq: b, LHS, not in S}
&k\cdot |\lambda_{\beta^\vee}(z\rho)| \leq |b_{\lambda_{\beta^\vee}}(\Ad_{u_{2,\rho}}(u_\rho\phi_Sg_Sz\rho) ( u_{1,\rho}^-)^{-1})|\leq K\cdot |\lambda_{\beta^\vee}(z\rho)|, \\
\nonumber&\text{ as } |\gamma_{-\Pi\backslash S}(\rho)|\rightarrow 0 \ (g_S, \xi_S;z,t) \text{ fixed}, \beta\not\in S
\end{align} 
for some fixed constants $k, K>0$. 
This uses that for a fixed basis $v^{(j)}_\mu$ in the $\mu$-weight space of $V_{\lambda_{\beta^\vee}}$, we have 
\begin{align}\label{eq: u_1,rho, inverse}
(u_{1,\rho}^-)^{-1}v_{\lambda_{\beta^\vee}}=v_{\lambda_{\beta^\vee}}+\sum\limits_{\varpi\in \Sigma(\Delta^+\backslash \Gamma(S))\backslash \{0\}, j}c^{(j)}_{\varpi,\rho} v^{(j)}_{\lambda_{\beta^\vee}-\varpi},
\end{align}
where (i) $\Sigma(\Delta^+\backslash \Gamma(S))\subset X^*(T_{sc})$ is the monoid spanned by $\Delta^+\backslash \Gamma(S)$ over $\bZ_{\geq 0}$; (ii) the summation has only finitely many (possibly) nonzero terms indexed by $\lambda_{\beta^\vee}-\varpi$ (belonging to the convex hull of $W\cdot \lambda_{\beta^\vee}$) and $j$; (iii)
$|c_{\varpi,\rho}^{(j)}|$ are uniformly bounded and $u_{2,\rho}\overset{\text{close}}{\sim}I$ (near the limit).  
Note that we can choose $c, C, k, K$ uniformly for all $(g_S, \xi_S; z,t)$, \emph{but} the range of $\rho$ so that (\ref{eq: b, RHS, not in S}) and (\ref{eq: b, LHS, not in S}) hold depends on $(g_S, \xi_S;z,t)$, which is very important\footnote{In fact, the range of valid $\rho$ only depends on $g_S$, because $\xi_S, z, t$ are bounded.}. 
Comparing (\ref{eq: b, RHS, not in S}) and (\ref{eq: b, LHS, not in S}), we see that there exist uniform constants $m, M>0$ such that 
\begin{align}\label{eq: beta in S, z}
\nonumber&m|\lambda_{\beta^\vee}(z\rho)|\leq |\lambda_{\beta^\vee}(z_\rho^\natural)|\leq  M|\lambda_{\beta^\vee}(z\rho)|, \beta\in \Pi, |\gamma_{-\Pi\backslash S}(\rho)|\rightarrow 0\ (\text{fixing }g_S, \xi_S; z, t)\\
\Leftrightarrow& z\rho (z_\rho^\natural)^{-1}\text{ is contained in a \emph{uniformly} bounded region in }\cZ(L_S)  \text{ near the limit}.
\end{align}
Presumably, the above only works for $\beta\not\in S$, but since $\lambda_{\beta^\vee}, \beta\not\in S$ gives a finite indexed sublattice of $X^*(\cZ(L_S))$ (also technically we should lift everything to $G_{sc}$), the same inequalities hold for all $\beta\in S$ as well.

Now we rewrite the relation (\ref{eq: Ad_u2,epsilon}) as 
\begin{align}\label{eq: Ad_u2,epsilon, 2}
\phi_Sg_Sz\rho Q_\rho^{-1} (z_\rho^\natural)^{-1}=u_\rho^{-1}Q_\rho^{-1}g_{S,\rho}^\natural. 
\end{align}
The left-hand-side can be rewritten as
\begin{align}\label{eq: Ad_u2,epsilon, left}
\phi_Sg_S(z\rho(z_\rho^\natural)^{-1}) \Ad_{z_\rho^\natural}(u_{2,\rho}^{-1}(u_{1,\rho}^-)^{-1})=\phi_Sg_S(z\rho(z_\rho^\natural)^{-1}) \Ad_{z_\rho^\natural}(u_{2,\rho})^{-1}\Ad_{z_\rho^\natural}(u_{1,\rho}^-)^{-1}.
\end{align}
By the assumption that $z, z_\rho^\natural\in \cZ(L_S)_0$ and $|\gamma_{-\Pi\backslash S}(\rho)|\rightarrow 0$, we have  $\Ad_{z_\rho^\natural}(u_{1,\rho}^-)^{-1}\rightarrow I$ and $\Ad_{z_\rho^\natural}(u_{2,\rho})^{-1}\in u_{2,\rho}^{-1}\cdot N_{P_S}$. For any $\beta\in S$, we compare $|b_{\lambda_\beta^\vee}|$ on both sides of (\ref{eq: Ad_u2,epsilon, left}) after multiplying $\Ad_{z_\rho^\natural}(u^-_{1,\rho})$ on the right to each side, and get 
\begin{align*}
|b_{\lambda_{\beta^\vee}}^S(g_S)|\cdot |\lambda_{\beta^\vee}(z\rho(z_\rho^\natural)^{-1})|=|b_{\lambda_{\beta^\vee}}((u_\rho^{-1}Q_\rho^{-1}g_{S,\rho}^\natural)\Ad_{z_\rho^\natural}(u^-_{1,\rho}))|.
\end{align*}
Suppose $g_{S,\rho}^\natural$ is contained in a fixed bounded (i.e. compact) domain $\fQ$ in $L_S^\der$, for $|\gamma_{-\Pi\backslash S}(\rho)|\rightarrow 0$ with $(g_S, \xi_S; z, t)$ fixed, then by (\ref{eq: beta in S, z}) and the uniform boundedness of the right-hand-side, we see that $|b_{\lambda_{\beta^\vee}}^S(g_S)|$ is uniformly bounded.  Hence by Proposition \ref{prop: proper b map}, $g_S$ is contained in a fixed bounded domain (that only depends on $\fQ$) in $L_S^\der$. This implies (i).

For (ii), we use (\ref{eq: Ad_u2,epsilon, 2}) and (\ref{eq: Ad_u2,epsilon, left}) again, and get
\begin{align*}
\phi_Sg_{S,\rho}^\natural \Ad_{z_\rho^\natural}(u_{1,\rho}^-)=\phi_SQ_\rho u_\rho\phi_S g_S(z\rho(z_\rho^\natural)^{-1})\Ad_{z_\rho^\natural}(u_{2,\rho})^{-1}.
\end{align*}
For any $\beta\in S$, we compare $|b_{\lambda_\beta^\vee}|$ on both sides. Using $ \Ad_{z_\rho^\natural}(u_{1,\rho}^-)\in N^-_{P_S}$, we have 
\begin{align}
\nonumber&|b_{\lambda_{\beta^\vee}}^S(g_{S,\rho}^\natural)|=|b_{\lambda_{\beta^\vee}}(\phi_Sg_{S,\rho}^\natural \Ad_{z_\rho^\natural}(u_{1,\rho}^-))|\\
\label{eq: RHS, Ad_u2,epsilon, 2}&=|b_{\lambda_{\beta^\vee}}(\phi_SQ_\rho u_\rho\phi_S g_S(z\rho(z_\rho^\natural)^{-1}))|, (\text{fixing }g_S, \xi_S; z, t).
\end{align}
Since (\ref{eq: RHS, Ad_u2,epsilon, 2}) above is uniformly bounded for $(g_S, \xi_S, z,t)$ in a fixed compact region $\fQ'$ in $\fU_S$  (near the limit of $\rho$), $|b_{\lambda_{\beta^\vee}}^S(g_{S,\rho}^\natural)|$ is uniformly bounded. Hence by Proposition \ref{prop: proper b map} again, $(g_{S,\rho}^\natural, \xi_{S, \rho}^\natural)$ is contained in a fixed compact region in $J_{L_S^\der}$ depending only on $\fQ'$. This proves (ii). 
\end{proof}

\begin{lemma}\label{lemma: nonzero lowest component}
\item[(i)] Let $V_{\lambda}$ be the irreducible highest weight representation of $G_{sc}$ corresponding to $\lambda\in X^*(T_{sc})^+$, and let $v_{\lambda}$ be a fixed highest weight vector. 
Then for any vector $v\in G_{sc}\cdot v_{\lambda}\subset V_{\lambda}$ that generates an invariant line of a Lie algebra element $f+\xi_1\in f+\fb$, it has a nonzero lowest weight component with weight $w_0(\lambda)$.  

\item[(ii)] Let $\fK\subset \fb$ (resp. $Q\subset G$) be a compact subset. Let $V_{\lambda}^{w_0(\lambda), \circ}$ the open subset of $V_\lambda$ consisting of vectors with \emph{nonzero} weight component in $w_0(\lambda)$. 
Then the subset in $V_\lambda$ defined by 
\begin{align*}
V_{\lambda}^{\fK, Q}:=\{v\in Q\cdot v_\lambda: \bC\cdot v\text{ is an invariant line of some element in }f+\fK\}
\end{align*}
is compact in $V_{\lambda}^{w_0(\lambda),\circ}$. In particular, the function $|(-, v_{-w_0(\lambda)})|: V_\lambda\rightarrow \bR_{\geq 0}$, for a fixed highest weight vector $v_{-w_0(\lambda)}$ in $V_\lambda^*\cong V_{-w_0(\lambda)}$, has a strictly positive minimum and a finite maximum on $V_\lambda^{\fK, Q}$, if $V_\lambda^{\fK, Q}\neq \emptyset$. 
\end{lemma}
\begin{proof}
(i) Let $P_{\lambda}$ be the standard parabolic that fixes the line generated by $v_{\lambda}$. First, we have the canonical embedding $\iota: G/P_{\lambda}\hookrightarrow \bP(V_{\lambda})$, that sends every $gP_{\lambda}$ to $\bC\cdot gv_{\lambda}$.  Let $N_{P_\lambda}$ be the unipotent radical of $P_{\lambda}$. The left $N_{P_{\lambda}}$ action on $G/P_{\lambda}$ gives the Bruhat decomposition, indexed by the $T$-fixed points $x_{w(\lambda)}, w(\lambda)\in W\cdot \lambda\cong W/W_{\lambda}$ which correspond to the lines generated by the weight vectors $v_{w(\lambda)}$ (defined unique up to scaling). Since the line generated by $v$ in question is in the image of $\iota$, the lemma is equivalent to saying that the corresponding point $\tilde{v}$ in $G/P_{\lambda}$ for $\bC\cdot v$ must lie in $N_{P_{\lambda}}\cdot x_{w_0(\lambda)}$.

Suppose the contrary that $\tilde{v}$ is not in $N_{P_{\lambda}}\cdot x_{w_0(\lambda)}$. Then 
\begin{align*}
\tilde{v}\in \bigsqcup\limits_{\mu\in W\cdot \lambda\backslash \{w_0(\lambda)\}}N_{P_{\lambda}}\cdot x_\mu. 
\end{align*}
In particular, $v=av_\mu+\sum\limits_{\mu\underset{\neq}{\prec}\mu'}q_{\mu'}$ for some $\mu\in W\cdot \lambda\backslash \{w_0(\lambda)\}$, $a\neq 0$ and some weight vectors $q_{\mu'}$ in the weight spaces of $\mu'$. Now apply $f+\xi_1\in f+\fb$ to $v$. The invariance of $\bC\cdot v$ implies that $v_\mu\in \ker f$, i.e. $f_\alpha\cdot v_\mu=0, \forall\alpha\in \Pi$. However, this contradicts to the assumption that $\mu$ is \emph{not} the lowest weight, so part (i) of the lemma follows.

(ii) First, we have the closed incidence subvariety in $\fb\times \bP(V_\lambda)$
\begin{align*}
\cX_{V_\lambda,f+\fb}:=\{(\xi, [v])\in \fb\times \bP(V_\lambda): [v] \text{ is an invariant line of }f+\xi\},
\end{align*}
Note that the condition that $[v]$ is an invariant line of $f+\xi$ is the same as saying that the vector field on $\bP(V_\lambda)$ corresponding to $f+\xi$ vanishes at $[v]$.
We have the projection (resp. proper projection) $p_{\bP(V_\lambda)}: \cX_{V_\lambda,f+\fb}\rightarrow \bP(V_\lambda)$ (resp. $p_\fb: \cX_{V_\lambda,f+\fb}\rightarrow \fb$). Let $\pi: V_\lambda-\{0\}\rightarrow \bP(V_\lambda)$ be the natural projection. Then for the given compact $\fK\subset \fb$ and $Q\subset G$, we have 
\begin{align*}
V_{\lambda}^{\fK, Q}=\pi^{-1}(p_{\bP(V_\lambda)}p_\fb^{-1}(\fK))\cap (Q\cdot v_\lambda). 
\end{align*}
Since $Q\cdot v_\lambda$ is compact inside $V_\lambda-\{0\}$ and $\pi^{-1}(p_{\bP(V_\lambda)}p_\fb^{-1}(\fK))\subset V_\lambda-\{0\}$ is closed, the intersection $V_{\lambda}^{\fK, Q}$ is compact. 

By part (i), $V_{\lambda}^{\fK, Q}\subset V_\lambda^{w_0(\lambda), \circ}$. The last sentence then follows immediately. 
\end{proof}

\begin{cor}\label{cor: cY_S, R, cover}
Fix the setting as in Question \ref{question: mu_D, K, rho}, and use $\cY_{S,R}$ from Lemma \ref{lemma: cY_S, R}. As we increase $R\uparrow\infty$ and for each $R$ choose $\rho\in \cZ(L_S)_0$ with $|\gamma_{-\Pi\backslash S}(\rho)|$ sufficiently small, the map on Hamiltonian reductions $\overline{\frj}_{S,\rho;\kappa}$ (\ref{eq: j_S,rho, reduction}) gives a symplectic covering map over every fixed pre-compact open region inside $\chi^{-1}(D_S^\dagg/W_S)/\cZ(L_S^\der)_0$ (after restriction to the preimage).
\end{cor}
\begin{proof}
This is a direct consequence of Lemma \ref{lemma: cY_S, R}. Without loss of generality, by enlarging the original $D_S$ to be $\widetilde{D}_S$, we can replace $D_S^\dagg$ by $D_S$. 
It suffices to consider a sequence of pre-compact regions $\cP_{S,K^{(n)}}$ defined by the same equation as for $\cY_{S, K^{(n)}}$ (\ref{eq: def cY_S,R}), with $K^{(n)}\rightarrow\infty$, that are inside $\chi_S^{-1}(D^{(n)}_S/W_S)/\cZ(L_S^\der)_0$ through (\ref{eq: chi_S, mathscr Z}) (contained in the Hamiltonian reduction of $\chi_{\cK'_{S^\perp}}$ at $\kappa$). Here $D^{(n)}_S$ is an increasing sequence of $W_S$-invariant pre-compact open in $D_S$ with $\bigcup\limits_n D^{(n)}_S=D_S$.  Since the image of $\cY_{S,R_0;\kappa, \rho}/\cZ(L_S^\der)_0$ under the map $\overline{\frj}_{S,\rho;\kappa}$, for some fixed $R_0>0$, is contained in a compact region in the target, as $|\gamma_{-\Pi\backslash S}(\rho)|\rightarrow 0$, we can choose $K_1\gg R_0$, such that $\cP_{S,K_1}$ contains the same image (note that $\cP_{S,K_1}$ is connected for $K_1$ sufficiently large). For any $K^{(n)}>K_1$, as we increase $R$ towards $\infty$ and at the same time let $|\gamma_{-\Pi\backslash S}(\rho)|\rightarrow 0$, we have 
$\overline{\frj}_{S,\rho;\kappa}(\cY_{S,R; \kappa, \rho}/\cZ(L_S^\der)_0)\supset \cP_{S,K_1}$
and 
\begin{align*}
\overline{\frj}_{S,\rho;\kappa}(Nb(\partial^h(\cY_{S,R}/\cZ(L_S^\der)_0)))\cap Nb(\partial^h \cP_{S,\widetilde{K}})=\emptyset, \forall\widetilde{K}\in [K_1, K^{(n)}], 
\end{align*}
where $Nb(\partial^h--)$ stands for a fixed tubular neighborhood of the ``horizontal boundary"of $\cY_{S,R; \kappa, \rho}/\cZ(L_S^\der)_0$ and $\cP_{S,K_1}$ respectively, in the same sense as in Lemma \ref{lemma: cY_S, R} (i). On the other hand, a sufficiently thin tubular neighborhood of the ``vertical boundary" of $\cY_{S,R;\kappa, \rho}/\cZ(L_S^\der)_0$, given by the intersection of a thin neighborhood of $\chi_S^{-1}(\partial D_S/W_S)/\cZ(L_S^\der)_0$ with its closure, has image outside the closure of $\chi_S^{-1}(D^{(n)}_S/W_S)/\cZ(L_S^\der)_0$, for the reason that $(\xi_{S, \rho}^\natural, t_\rho^\natural)\rightarrow (\xi_S, t)$ when $|\gamma_{-\Pi\backslash S}(\rho)|\rightarrow 0$ as in the proof Lemma \ref{lemma: cY_S, R}. 
So these imply that $\overline{\frj}_{S,\rho;\kappa}(\cY_{S,R;\kappa, \rho}/\cZ(L_S^\der)_0)\supset \cP_{S,K^{(n)}}$, and it must be a covering map from the preimage of $\overline{\frj}_{S,\rho;\kappa}$ over $\cP_{S,K^{(n)}}$. 
\end{proof}

Using Lemma \ref{lemma: cY_S, R}, we also have direct analogue of Lemma \ref{lemma: proj X_eta, c}, Proposition \ref{prop: C_lhd, j} and 
Corollary \ref{cor: dist star}, for which we only state in the form of the corollary that will be applied later. In the following, we fix a $\cZ(L_S^\der)$-invariant complete metric on $J_{L_S^\der}$, e.g. the complete hyperKahler metric constructed in \cite{Bielawski}. Then it determines a complete $\cZ(L_S)_0$-invariant metric on $\fU_S$ by the Killing form restricted to $\fz_S$. 

\begin{cor}\label{cor: dist star, S}
Fix any open cone $C_{\lhd, S}\subset \mathring{\fz}_S\cap \ft_\bR^+$ such that $\overline{C}_{\lhd, S}-\{0\}\subset \mathring{\fz}_S\cap \ft_\bR^+$. For any $\delta>0$, there exists $\epsilon_{C_{\lhd, S}}>0$ such that for any $(g_S, \xi_S; z, t)\in \cW_{\cY_S, \cV, \cK^\dagg}$, $\rho_1\in \cZ(L_S)_0$ satisfying $|\gamma_{-\Pi\backslash S}(\rho_1)|<\epsilon_{C_{\lhd, S}}$ and $\rho'\in (\cZ(L_S)_0)_{C_{\lhd, S}}$, we have 
\begin{align*}
\rho'\star \frj_{S, \rho_1}(g_S, \xi_S; u_0, t)\in \fU_S, 
\end{align*}
\begin{align}
\label{eq: dist, rho'rho_1, S}&dist(\rho'\star \frj_{S,\rho_1}(g_S, \xi_S; z, t), \frj_{S, \rho'\rho_1}(g_S, \xi_S; z, t))<\delta, 
\end{align}
where the $\cZ(L_S)_0$-action $\star$ is the one on $\chi^{-1}(\chi_\ft(\cQ_{D, \cK'}))\subset J_G$ with respect to the projection $\chi_\ft(\cQ_{D, \cK'})\rightarrow \cK'_{S^\perp}\subset \mathring{\fz}_S$ (cf. Remark \ref{remark: cK in fz_S}).
Moreover, 
\begin{align}
\label{eq: dist, W, S}&dist(w^{-1}(\rho')\star \frj_{S,\rho_1}(u,w(t))), \frj_{S, \rho'\rho_1}(u,w(t)))<\delta, \forall w\in N_W(W_S). 
\end{align}
Here the distance is taken with respect to the fixed $\cZ(L_S)_0$-invariant metric on $\fU_S$. 
\end{cor}
\begin{proof}
This follows essentially from the same proof for Lemma \ref{lemma: proj X_eta, c}, Proposition \ref{prop: C_lhd, j} and 
Corollary \ref{cor: dist star}. Only the part on  
\begin{align*}
\rho'\star \frj_{S, \rho_1}(g_S, \xi_S; u_0, t)\in \fU_S
\end{align*}
needs additional clarification. To show this, we consider $\cY_{S, R_1}\subset \cY_{S, R_2}$ (cf. (\ref{eq: def cY_S,R})) for some $0<R_1\ll R_2$. Then by Corollary \ref{cor: cY_S, R, cover}, for $R_2/R_1$ sufficiently large, there exists $\epsilon_{R_1, R_2}>0$ and a fixed compact region $\cX$ in $\chi_S^{-1}(D^\dagg/W_S)/\cZ(L_S^\der)_0$ contained in the right-hand-side of (\ref{eq: j_S,rho, reduction}), such that for all $\rho\in \cZ(L_S)_0$ satisfying $|\gamma_{-\Pi\backslash S}(\rho)|<\epsilon_{R_1, R_2}$, we have for all $\kappa\in \cK_{S^\perp}^\dagg$ (using $\overline{\frj}_{S, \rho;\kappa}$ from (\ref{eq: j_S,rho, reduction}))
\begin{align}
\label{eq: daggdagg, X, dagg}&\overline{\frj}_{S, \rho;\kappa}(\cY^{\dagg\dagg}_{S, R_1;\kappa, \rho}/\cZ(L_S^\der)_0)\subset \cX\subset \overline{\frj}_{S, \rho;\kappa}(\cY^\dagg_{S, R_2;\kappa, \rho}/\cZ(L_S^\der)_0), 
\end{align}
where (i) $D^{\dagg\dagg}$ is defined in the same way as $D^\dagg$ and satisfies $\overline{D^{\dagg\dagg}}\subset D^\dagg$, $\cY_{S, R_2}^{\dagg}\subset \chi_S^{-1}(D^{\dagg}/W_S)$ (resp. $\cY_{S, R_1}^{\dagg\dagg}\subset \chi_S^{-1}(D^{\dagg\dagg}/W_S)$) is defined by (\ref{eq: def cY_S,R}) using $D^\dagg$ (resp. $D^{\dagg\dagg}$); (ii) in the second inclusion, $\cX$ is disjoint from a tubular neighborhood of the boundary of $\overline{\frj}_{S, \rho;\kappa}(\cY^\dagg_{S, R_2}/\cZ(L_S^\der)_0)$. 

Now by a direct analogue of Lemma \ref{lemma: proj X_eta, c} with $\overline{C}_{\lhd, S}\subset C'_{\lhd, S}$ given and $\epsilon_{C_{\lhd, S}}>0, M>0$ satisfying the corresponding conclusions,  
we claim that for any $(g_S, \xi_S; u_0, t)\in  \cW_{\cY^{\dagg\dagg}_{S, R_1}, \cV,\cK^\dagg}$, we have 
\begin{align}\label{eq: cW_cY_S_dagg, star}
\rho'\star \frj_{S, \rho_1}(g_S, \xi_S; u_0, t)\in \bigcup\limits_{\widetilde{\rho}\in (\cZ(L_S)_0)_{C'_{\lhd, S}}}\frj_{S, \widetilde{\rho}}(\cW_{\cY^{\dagg}_{S, R_2}, \cV,\cK})
\end{align}
for all $\rho_1$ satisfying $|\gamma_{-\Pi\backslash S}(\rho_1)|<\widetilde{\epsilon}_{C_{\lhd, S}}:=\min\{\epsilon_{R_1, R_2}, \epsilon_{C_{\lhd, S}}\}$ and $\rho'\in  (\cZ(L_S)_0)_{C_{\lhd, S}}$.

Suppose the contrary, for some $(g_S, \xi_S;u_0, t)$, $\eta\in C_{\lhd, S}$, $\rho_c'\in (\cZ(L_S)_0)_c$ and the corresponding curve $\Upsilon_{\eta}(s):= \rho_c'\cdot \exp(s\cdot \eta), s\geq 0$, there exists $r>0$ such that $\Upsilon_{\eta}(r)\star \frj_{S, \rho_1}(g_S, \xi_S;u_0, t)$ is \emph{not} in the right-hand-side of (\ref{eq: cW_cY_S_dagg, star}). 
Let 
\begin{align*}
&\overline{\rho}_\eta(s):=u_0^{-1}\proj_{\cZ(L_S)_0/\cZ(L^\der_S)_0}(\Upsilon_{\eta}(s)\star \frj_{S, \rho_1}(g_S, \xi_S;u_0, t))\in \cZ(L_S)_0/\cZ(L^\der_S)_0,\\
&\kappa=\mu_{D, \cK, \rho_1}((g_S, \xi_S;u_0, t)). 
\end{align*}
For any $s\geq 0$ in the (largest connected) interval when $\overline{\rho}_\eta(s)$ is well defined, i.e. when $\Upsilon_{\eta}(s)\star \frj_{S, \rho_1}(g_S, \xi_S;u_0, t))$ is contained in $\fU_S$, we fix a representative $\rho_\eta(s)$ of $\overline{\rho}_\eta(s)$ in $\cZ(L_S)_0$. 

Since 
\begin{align*}
|\gamma_{-\Pi\backslash S}(\rho_\eta(s))|\leq |\gamma_{-\Pi\backslash S}(\rho_1)|<\widetilde{\epsilon}_{C_{\lhd, S}},
\end{align*}
for all $s\geq 0$ in the defining interval of $\overline{\rho}_\eta(s)$,  
we have the minimum of such $r$ satisfies 
\begin{align}\label{eq: boundary cW, dagg}
\Upsilon_{\eta}(r)\star \frj_{S, \rho_1}(g_S, \xi_S;u_0, t)\in \partial \cY^{\dagg}_{S, R_2}\underset{\cZ(L_S^\der)_0}{\times}(\bigcup\limits_{\widetilde{\rho}\in (\cZ(L_S)_0)_{C'_{\lhd, S}}}\frj_{S, \widetilde{\rho}} (\cV_{S^\perp})\times \cK_{S^\perp}).
\end{align}
Here we use that 
\begin{align*}
\proj_{\cZ(L_S)_0/\cZ(L^\der_S)_0}\frj^{-1}_{S, \rho_\eta(r)}(\Upsilon_{\eta}(r)\star \frj_{S, \rho_1}(g_S, \xi_S;u_0, t))=u_0\text{ mod }\cZ(L_S^\der)_0,
\end{align*}
and that whenever 
\begin{align*}
\proj_{J_{L_S^\der}/\cZ(L_S^\der)_0}\Upsilon_{\eta}(s)\star \frj_{S, \rho_1}(g_S, \xi_S;u_0, t)\subset \overline{\cY^\dagg_{S, R_2}}/\cZ(L_S^\der)_0,
\end{align*}
we have 
\begin{align*}
&\proj_{\cK'_{S^\perp}}(\Upsilon_{\eta}(r)\star \frj_{S, \rho_1}(g_S, \xi_S;u_0, t))\overset{\text{close}}{\sim}\\
& \mu_{D,\cK, \Upsilon_{\eta}(r)}(\frj_{\Upsilon_{\eta}(r)}^{-1}(S, \Upsilon_{\eta}(r)\star \frj_{S, \rho_1}(g_S, \xi_S;u_0, t)))=\kappa
\end{align*}
(hence also close to $t$). So we can exclude the other boundaries of the right-hand-side of (\ref{eq: cW_cY_S_dagg, star}) for the minimum $r$. 

However, on one hand, we have 
\begin{align*}
&\overline{\frj}_{S, \rho_\eta(r);\kappa}(\frj^{-1}_{S, \rho_\eta(r)}(\Upsilon_{\eta}(r)\star \frj_{S, \rho_1}(g_S, \xi_S;u_0, t)))\\
&=\overline{\frj}_{S, \rho_1;\kappa}(g_S, \xi_S;u_0, t)\in \overline{\frj}_{S, \rho;\kappa}(\cY^{\dagg\dagg}_{S, R_1;\kappa, \rho}/\cZ(L_S^\der)_0)\subset \cX, 
\end{align*}
while on the other hand,  (\ref{eq: boundary cW, dagg}) and (\ref{eq: daggdagg, X, dagg}) imply that  
\begin{align*}
\overline{\frj}_{S, \rho_\eta(r);\kappa}(\frj^{-1}_{S, \rho_\eta(r)}(\Upsilon_{\eta}(r)\star \frj_{S, \rho_1}(g_S, \xi_S;u_0, t)))\not\in \cX,
\end{align*}
which gives a contradiction. 
\end{proof}

Now we can give an answer to Question \ref{eq: D, K, dagger} (ii). 

\begin{prop}\label{prop: Ham reduction emb}
The covering map in Corollary \ref{cor: cY_S, R, cover} is one-to-one. 
\end{prop}

\begin{proof}
Fix a pre-compact (connected) open region inside $\chi^{-1}(D_S^\dagg/W_S)/\cZ(L_S^\der)_0$. Also fix a sufficiently large $R$ and a sufficiently small $\epsilon_R>0$ so that the conclusion in Corollary \ref{cor: cY_S, R, cover} is satisfied for $\rho\in \cZ(L_S)_0$ with $|\gamma_{-\Pi\backslash S}(\rho)|<\epsilon_R$. 

We apply Corollary \ref{cor: dist star, S}, with $\cY_S=\cY_{S, R}$, a fixed open cone $C_{\lhd, S}$ and an arbitrarily small $\delta>0$ as in the assumption. Let $\epsilon'=\min\{\epsilon_R, \epsilon_{C_{\lhd, S}}\}$. Fixing any  $\rho_1\in \cZ(L_S)_0$ satisfying $|\gamma_{-\Pi\backslash S}(\rho_1)|<\epsilon'$, suppose we have two distinct points $(g_{S}^{(i)}, \xi_{S}^{(i)}; u_0,  t^{(i)})\in \cW_{\cY_{S, R}, \cV, \cK}, i=1, 2$ that are \emph{not} in the same characteristic leaf, but that map to the same point in the fixed pre-compact open region inside $\chi^{-1}(D^\dagg_S/W_S)/\cZ(L_S^\der)_0$ under $\overline{\frj}_{S, \rho_1;\kappa}$.  Then $\frj_{S, \rho_1}(g_{S}^{(i)}, \xi_{S}^{(i)}; u_0, t^{(i)}), i=1,2$ are in the same $\cZ(L_S)_0$-orbit in $\chi^{-1}(\chi_\ft(\cQ_{D,\cK'}))$ with respect to the projection $\chi_\ft(\cQ_{D, \cK'})\rightarrow \cK'_{S^\perp}$. 
Now for all $\widetilde{\rho}\in (\cZ(L_S)_0)_{C_{\lhd, S}}$, $\frj_{S, \widetilde{\rho}\rho_1}^{-1}( \widetilde{\rho}\star\frj_{S,\rho_1}(g_{S}^{(i)}, \xi_ {S}^{(i)}; u_0, t^{(i)}))$ is contained in a $\delta$-neighborhood of $(g_{S}^{(i)}, \xi_ {S}^{(i)};u_0, t^{(i)})$ 
in $\cW_{\cY_{S, R}, \cV,\cK}$, and 
the $\cZ(L_S)_0$-orbit containing both $\frj_{S, \rho_1}(g_{S}^{(i)}, \xi_{S}^{(i)}; u_0, t^{(i)})$, $i=1,2$ will intersect $\frj_{S, \widetilde{\rho}\rho_1}(\cW_{\cY_{S, R}, \cV,\cK})$ in at least two disconnected components containing $\widetilde{\rho}\star  \frj_{S,\rho_1}(g_{S}^{(i)}, \xi_ {S}^{(i)}; u_0, t^{(i)}), i=1,2$ respectively. This is because the set of such $\widetilde{\rho}$ in $(\cZ(L_S)_0)_{C_{\lhd, S}}$ is open, closed and nonempty, hence it is the entire space. 

On the other hand, we have 
\begin{align*}
\widetilde{\rho}\star  \frj_{S,\rho_1}(g_{S}^{(2)}, \xi_ {S}^{(2)}; u_0, t^{(2)})=\rho_{12}\star \widetilde{\rho}\star  \frj_{S,\rho_1}(g_{S}^{(1)}, \xi_ {S}^{(1)}; u_0, t^{(1)})
\end{align*}
for a fixed unique $\rho_{12}\in \cZ(L_S)_0$. 
Without loss of generality, we will assume $u_0=I\in \cV_{S^\perp}\subset \cZ(L_S)_0$. 
Choose a sufficiently large pre-compact open $\widetilde{\cV}_{S^\perp}\subset \cZ(L_S)_0$ (defined in the way described in Question \ref{question: mu_D, K, rho} (i)) that contains $\rho_{12}$. Then there exists $\epsilon_{\widetilde{\cV}}>0$ such that for all $\widetilde{\rho}$ satisfying $|\gamma_{-\Pi\backslash S}(\widetilde{\rho})|<\epsilon_{\widetilde{\cV}}$ (this will be contained in $C_{\lhd, S}$ for $\epsilon_{\widetilde{\cV}}$ sufficiently small), 
\begin{align*}
\mu_{D,\cK', \widetilde{\rho}\rho_1}: \cW_{\cY_{S, R}, \widetilde{\cV}, \cK}\longrightarrow \cK'_{S^\perp}
\end{align*}
is arbitrarily close to the projection map. By Proposition \ref{eq: tilde mu, embed} on the integrability of $\mu_{D,\cK', \widetilde{\rho}\rho_1}$ (on the larger domain $\cW_{\cY_{S, R}, \widetilde{\cV}, \cK}$), we must have 
$\frj_{S, \widetilde{\rho}\rho_1}^{-1}( \widetilde{\rho}\star\frj_{S,\rho_1}(g_{S}^{(i)}, \xi_ {S}^{(i)}; u_0, t^{(i)})$, $i=1,2$ lie in the \emph{same} characteristic leaf. However, since this characteristic leaf,  viewed in the product $\cD_{(g_S^{(1)}, \xi_S^{(1)})}\times \widetilde{\cV}_{S^\perp}\times\cK_{S^\perp}'$ as in (\ref{eq: lemma cS_kappa, S}), projects to $ \widetilde{\cV}_{S^\perp}$ isomorphically, 
its intersection with the original $\cW_{\cY_{S, R}, \cV, \cK}$ \emph{cannot} split into more than one leaves. Thus we reach at a contradiction. 
\end{proof}

We give a sketch of the proof for an analogue of Proposition \ref{lemma: empty, Ham isotopy}. 

\begin{prop}\label{lemma: varphi_s, mu_D,K}
Let $\cY_S, \cY^\dagg_S, \cV_{S^\perp}, \cV'_{S^\perp}, \cK^\dagg_{S^\perp}, \cK_{S^\perp}, \cK_{S^\perp}'$ be as above. For any smooth curve $(c_\beta(s))_\beta\in \bC^{\Pi\backslash S}, s\in (-\epsilon', \epsilon')$ with $(c_\beta(0))_\beta=0$, there exists $\epsilon>0$ and a compactly supported Hamiltonian isotopy $\varphi_s$, $0\leq s\leq \epsilon$, with $\varphi_0=id$, on 
$\cW_{\cY_S, \cV', \cK'}$
such that for every $\kappa\in \cK^\dagg_{S^\perp}$, we have 
\begin{align}\label{eq: varphi_s, mu_D,K}
\varphi_s(\mu_{D,\cK,0}^{-1}(\kappa)\cap \cW_{\cY_S^\dagg, \cV, \cK})\subset \widetilde{\mu}_{D',\cK',(c_\beta(s))_\beta}^{-1}(\kappa)\cap \cW_{\cY_S, \cV', \cK'}.
\end{align}
\end{prop}
\begin{proof}

We assume $|(c_\beta(s))_{\beta}|$ are all sufficient small, so that the conclusions in Lemma \ref{lemma: cS_kappa, S} all hold. 

\noindent\emph{Step 1. Identify the multi-sections $\cY^\dagg_{S, \kappa, (c_\beta(s))_\beta}$ over $s\in (-\epsilon, \epsilon)$ symplectically and $\cZ(L_S^\der)_0$-equivariantly. }

We make the identification between the symplectic quotient spaces\\
$\cY^\dagg_{S, \kappa, (c_\beta(s))_\beta}/\cZ(L_S^\der)_0$ at $\kappa$ for different $s\in (-\epsilon, \epsilon)$ as follows (up to restricting to a slightly smaller domain in $\cY^\dagg_{S, \kappa, (c_\beta(s))_\beta}/\cZ(L_S^\der)_0$ and for $s$ in a smaller interval $(-\epsilon^\dagg,\epsilon^\dagg)$). We put $\cY^\dagg_{S, \kappa, (c_\beta(s))_\beta}$ into a smooth family of symplectic manifold over $(-\epsilon, \epsilon)$ naturally contained inside $\cW_{\cY_S, \cV, \cK}\times (-\epsilon, \epsilon)$, and denote it by $\mathfrak{Y}^\dagg_{S, \kappa, (-\epsilon, \epsilon)}\rightarrow (-\epsilon, \epsilon)$. There is a natural smooth (not necessarily symplectic) identification, a.k.a ``parallel transport", between different fibers (after restricting to a slightly smaller domain), by sending each point in the original multi-section $(\cY_S\times \{u_0\}\times \cK'_{S^\perp})\cap \widetilde{\mu}_{D^\dagg, \cK', (c_\beta(s))_\beta}^{-1}(\kappa)$ to the corresponding point (after averaging) in the modified multi-section.

The $\cZ(L_S^\der)_0$-action on each fiber over $s$ above gives a $\cZ(L_S^\der)_0$-action on $\mathfrak{Y}^\dagg_{S,\kappa, (-\epsilon, \epsilon)}$ that preserves each fiber. Taking the quotient space assembles the symplectic quotient spaces $\cY^\dagg_{S, \kappa, (c_\beta)_\beta}/\cZ(L_S^\der)_0$ into a smooth family over $(-\epsilon, \epsilon)$. Now the ``parallel transport" on the family $\mathfrak{Y}^\dagg_{S, \kappa(-\epsilon, \epsilon)}$ gives a lifting of the unit positive vector field on $(-\epsilon, \epsilon)$ to a smooth vector field $\mathbf{v}$ on it, the average of the projection of $\bv$ to the quotient $\mathfrak{Y}^\dagg_{S, \kappa, (-\epsilon, \epsilon)}/\cZ(L_S^\der)_0$ is a smooth vector field that is a lifting of the unit positive vector field on the base $(-\epsilon, \epsilon)$. Integrating the vector field gives a smooth identification $\widetilde{\varphi}_{s,\cY,\kappa}$ between the fiber at $0$ and that at $s$. Since  the symplectic manifolds are exact and the diffeomorphisms are close to be symplectic (in fact, $\widetilde{\varphi}_{s,\cY,\kappa}^*\vartheta-\vartheta$ is close to zero), using Moser's argument, we can modify the smooth identifications to be symplectic, after restricting to a slightly smaller subdomain on each fiber.

Lastly, we lift the identification on the quotient spaces uniquely to a $\cZ(L_S^\der)_0$-equivariant symplectic identification $\varphi_{s,\cY,\kappa}: \cY^{\dagg\dagg}_{S, \kappa,0}\rightarrow \cY^{\dagg\dagg}_{S, \kappa, (c_\beta(s))_\beta}$, subject to the condition that the distance between $\varphi_{s,\cY,\kappa}(g_S, \xi_S, u_0, \kappa)$ and $(g_S, \xi_S, u_0, \kappa)$ is small. Here the double $\dagg$ superscript means we are taking some slightly smaller subdomain. Note that the identifications $\varphi_{s,\cY,\kappa}$ are smoothly depending on $\kappa$.

\noindent \emph{Step 2. Construction of the Hamiltonian isotopy $\varphi_s$ on $\cW_{\cY_S^\dagg, \cV', \cK'}$.}

We fix some real linear coordinates $(p_c^j; p_\bR^j)$ on the base $\fz_S^*\cong \fz_S\cong \fz_{S, c}\oplus \fz_{S, \bR}$. For each $\kappa, s$, applying Proposition \ref{eq: tilde mu, embed} for $(c_\beta(s))_{\beta\in \Pi\backslash S}$, 
the $\cZ(L_S^\der)_0$-equivariant multi-section $\cY_{S, \kappa, (c_\beta(s))_\beta}^\dagg$ determines an embedding
\begin{align*}
\widetilde{\mu}_{D^{\dagg}, \cK, (c_\beta(s))_\beta}^{-1}(\kappa)\longhookrightarrow \cY^{\dagg}_{S, \kappa, (c_\beta(s))_\beta}\underset{\cZ(L_S^\der)_0}{\times}\cZ(L_S)_0
\end{align*}

Similarly to the proof of Lemma \ref{lemma: empty, Ham isotopy}, using $\varphi_{s,\cY,\kappa}: \cY^{\dagg\dagg}_{S, \kappa,0}\rightarrow \cY^{\dagg\dagg}_{S, \kappa, (c_\beta(s))_\beta}$ from the previous step and Proposition \ref{eq: tilde mu, embed}, we have a uniquely defined map 
\begin{align*}
\widetilde{\varphi}_s: &\cY^{\dagg\dagg}_S\underset{\cZ(L_S^\der)_0}{\times} (\cV^\dagg_{S^\perp}\times \cK^\dagg_{S^\perp})\longrightarrow \widetilde{\mu}_{D',\cK',(c_\beta(s))}^{-1}(\cK'_{S^\perp})\cap \cW_{\cY_S, \cV', \cK'}
\end{align*}
which sends $\cY^{\dagg\dagg}_S\times \{(u_0, \kappa)\}$ to $\cY^{\dagg\dagg}_{S,\kappa, (c_{\beta}(s))_\beta}$ through $\varphi_{s, \cY,\kappa}$, and which respects the canonical (locally defined) real affine coordinates $(q_{c, s, (g_S, \xi_S)}^j; q_{\bR, s, (g_S, \xi_S)}^j), j=1,\cdots, n-|S|$ and $(q_{c, s, \varphi_{s, \cY,\kappa}(g_S, \xi_S;u_0, \kappa)}^j; q_{\bR, s, \varphi_{s, \cY,\kappa}(g_S, \xi_S;u_0, \kappa)}^j), j=1,\cdots, n-|S|$ on each characteristic leaf that are dual to $(p_c^j;p_\bR^j)$ and that are relative to the respective $\cZ(L_S^\der)_0$-equivariant multi-sections.

It is clear that 
\begin{align*}
\widetilde{\varphi}_s^*\omega-\omega=\sum\limits_j \alpha_{c,s,j}(p)\wedge dp_{c}^j+\alpha_{\bR,s,j}(p)\wedge dp_{\bR}^j,
\end{align*}
where $\alpha_{c,s, j}(p)$ and $\alpha_{\bR,s,  j}(p)$ are $\cZ(L_S^\der)_0$-equivariant $1$-forms on $\cY^{\dagg\dagg}_S$ depending smoothly on $p\in \cK^\dagg_{S^\perp}$. Since $\widetilde{\varphi}_s^*\omega-\omega$ is closed, we get both $\alpha_{c,s,j}(p)$ and $\alpha_{\bR,s, j}(p)$ are closed 1-forms on $\cY^{\dagg\dagg}_S$ (with $p$ fixed). By Lemma \ref{lemma: pi_1, J_G} below, $H^1(\cY^{\dagg\dagg}_S;\bR)=0$, so we can choose $f_{c,s,j}(p)$ and $f_{\bR,s, j}(p)$ to be primitives of $\alpha_{c,s,j}(p)$ and $\alpha_{\bR,s, j}(p)$ on $\cY_S^{\dagg\dagg}$, respectively, such that they all vanish at  a fixed point in $\cY_S^{\dagg\dagg}$.  Then we have 
\begin{align*}
d(\sum\limits_jf_{c,s,j}(p)dp_{c}^j+f_{\bR,s,j}(p)dp_{\bR}^j)-\sum\limits_j (\alpha_{c,s,j}(p)\wedge dp_{c}^j+\alpha_{\bR,s,j}(p)\wedge dp_{\bR}^j)
\end{align*}
a \emph{closed} 2-form that is a combination of wedges of $dp_c^j, dp_\bR^k, j,k=1, \cdots, n-|S|$, which by the assumptions on $f_{c,s,j}(p)$ and $f_{\bR, s,j}(p)$ must be $0$.

Now we can apply Moser's argument in the specific form of \cite[Section 3.2]{McSa} with 
\begin{align*}
\sigma_s=\frac{d}{ds}(\sum\limits_jf_{c,s,j}(p)dp_c^j+f_{\bR,s,j}(p)dp_\bR^j),
\end{align*}
in (3.2.1) of \emph{loc. cit.}
Then for $\epsilon>0$ small, we can compose $\widetilde{\varphi}_s$ with the isotopy to define $\varphi_s$ that  preserves the symplectic form. Moreover, $\varphi_s^*\vartheta-\vartheta$ must be exact, because its integral along the 1-cycles in $\cV_{S^\perp}$ are all zero (cf. Lemma \ref{lemma: cS_kappa, S}), which implies that $\varphi_s$ is a Hamiltonian isotopy satisfying (\ref{eq: varphi_s, mu_D,K}). It is then easy to extend $\varphi_s$ to be a compactly supported Hamiltonian isotopy on $\cW_{\cY, \cV', \cK'}$. 
\end{proof}

\begin{lemma}\label{lemma: pi_1, J_G}
For any complex connected semisimple Lie group $G$, we have a natural isomorphism $\pi_1(J_G)\cong \pi_1(G)$. 
\end{lemma}
\begin{proof}
First, from the centralizer description of $J_{G}$ (\ref{eq: centralizer cS}), we have a natural morphism $p: \pi_1(J_G)\rightarrow \pi_1(G)$. By the Cartesian square with vertical morphisms regular connected coverings (of deck transformations by $\pi_1(G)$),
\begin{align*}
\xymatrix{J_{G_{sc}}\ar[r]\ar[d]&G_{sc}\ar[d]\\
J_{G}\ar[r]&G
}
\end{align*}
we see that $p$ is surjective. Now we need to show that $\ker p$ is trivial. For this it suffices to work with $G_{sc}$ for which $\ker p\cong \pi_1(J_{G_{sc}})$, and let $\widetilde{J}_{G_{sc}}$ be the universal cover. By the Weinstein handle attachment structure of $J_{G_{sc}}$ from Proposition \ref{prop: split generation} (i), especially its inductive pattern, and the isomorphism between $\pi_1(T\times\{\xi\})\cong \pi_1(\chi^{-1}([\xi]))$ for a torus fiber in $\cB_{w_0}$ over $\xi\in \ft^\reg$ by Lemma \ref{eq: lemma L_t, epsilon}, it is clear that $\pi_1(J_{G_{sc}})$ is finite, and the natural map $\pi_1(\chi^{-1}([\xi]))\rightarrow \pi_1(J_{G_{sc}})$ is surjective.  Then $\widetilde{\chi}: \widetilde{J}_{G_{sc}}\rightarrow \fc$
is an abelian group scheme with generic fibers connected complex tori (of the same rank). 

Now we look at the commutative diagram
\begin{align*}
\xymatrix{
\widetilde{J}_{G_{sc}}\underset{J_{G_{sc}}}{\times}(T\times \ft^{\reg})/W\ar[r]\ar[dr]& (T\times \ft^{\reg})/W\ar[d]&\\
&\fc^{\reg}\ar@{^{(}->}[r]&\fc
}
\end{align*}
where the fiber of the right-downward arrow (on the left) is a $\pi_1(J_{G_{sc}})$-cover of $T$, denoted by $\widetilde{T}$. This induces a $\pi_1(\fc^\reg)=Br_W$-action on $\pi_1(\widetilde{T})$, and a $Br_W$-equivariant embedding $\pi_1(\widetilde{T})\hookrightarrow \pi_1(T)$. Since the pure braid group acts trivially on $\pi_1(T)$, the embedding $\pi_1(\widetilde{T})\hookrightarrow \pi_1(T)$ is $W$-equivariant. In particular, the image of $\pi_1(\widetilde{T})$ in $\pi_1(T)\cong X_{*}(T)$ is a finite indexed $W$-invariant sublattice, and we have a $W$-action on $\widetilde{T}$ together with the isomorphism 
\begin{align*}
\widetilde{J}_{G_{sc}}\underset{J_{G_{sc}}}{\times}(T\times \ft^{\reg})/W\cong (\widetilde{T}\times \ft^{\reg})/W.
\end{align*} 

The Kostant sections over the contractible base $\fc$ are lifted to $|\cZ(G_{sc})|\times |\pi_1(J_{G_{sc}})|$ many disjoint sections of $\widetilde{\chi}$. On the other hand, if $\widetilde{T}$ is a non-trivial $W$-equivariant covering of $T$, then there exists a simple coroot $\alpha^\vee$ that is not in $X_*(\widetilde{T})$. 
Then $\lambda_{\alpha}^{\vee}\in (\Lambda^\vee/X_{*}(T))\cong \cZ(G_{sc})$ has 
\begin{align*}
s_{\alpha^\vee}(\lambda_{\alpha}^{\vee})=\lambda_{\alpha}^{\vee}-\alpha^\vee\neq \lambda_\alpha^\vee\text{ mod }X_*(\widetilde{T})
\end{align*}
This means the lifting of the Kostant section corresponding to $\lambda_{\alpha}^{\vee}$ to $\widetilde{J}_{G_{sc}}$ cannot be a collection of disjoint sections, for the lifting of its restriction inside $(T\times \ft^{\reg})/W$ to $(\widetilde{T}\times \ft^{\reg})/W$ already has a connected component that is a multi-section over $\fc^{\reg}$.  The proof is complete. 
\end{proof}

\subsection{Discussions around walls beyond $S^\perp, \emptyset\neq S\subsetneq \Pi$}\label{subsec: walls}
In this subsection, we develop some analysis around an arbitrary ``wall" $w(S^\perp), w\in W/W_S, \emptyset\neq S\subsetneq \Pi$ in $\ft$, that is needed for the proof of Proposition \ref{prop: L_0, part 2} and \ref{prop: L_0, W}. The main result is Proposition \ref{prop: g_S, natural, w}. We remark that there is no direct generalization of the analysis done in Subsection \ref{subsec: analysis fU_S} to the current setting for an arbitrary $w(\cZ(L_S)_0)$ (here all the subtori $w(\cZ(L_S)_0), w\in W/W_S$ are  \emph{relative to the same Borel $B$}). 

For any $\emptyset\neq S\subsetneq \Pi$, let $W^S_{\min}\cong W/W_S$ be the set of elements in $W$ consisting of the unique shortest representative of each coset. Recall that $w\in W^S_{\min}$ if and only if $w(S)\subset \Delta^+$. 
In $\cB_{w_0}\cong T^*T$, we look at $\cU_{\cQ', \cV}^{w(S)}:=
\cV\times w(D_S'+\cK_{S^\perp}'), w\in W_{\min}^S$, where $\cQ'_{D', \cK'}=D_S'+\cK_{S^\perp}'$ 
is a tubular neighborhood of $\cK_{S^\perp}'\subset \mathring{\fz}_{S}$ and $\cV\subset T$ is as in the setting of Section \ref{subsec: analysis cB_w0}. We fix a representative $\overline{w}\in N_G(T)$ for any $w\in W_{\min}^S$. Let $D_S\subset\overline{D}_S\subset D_S'$, $\cK_{S^\perp}\subset \overline{\cK}_{S^\perp}\subset\cK_{S^\perp}'$ be slightly smaller open subsets. Define $\cQ_{D, \cK}=D_S+\cK_{S^\perp}$ and $\cU_{\cQ, \cV}^{w(S)}$ similarly as for $\cQ'_{D', \cK'}$ and $\cU_{\cQ', \cV}^{w(S)}$. 
For any $(u\overline{w}_0^{-1}h, \xi=f+w(t)+\Ad_{(\overline{w}_0^{-1}h)^{-1}}f)\in \cB_{w_0}$ with $(h, w(t))\in \cU_{\cQ, \cV}^{w(S)}$ and $u\in N$ uniquely determined making the pair in $\mathscr{Z}_G$, and for any $\rho\in T$ with $|\gamma_{-\Pi}(\rho)|<\epsilon\ll 1$, we have 
\begin{align*}
&\frj_\rho(u\overline{w}_0^{-1}h, \xi)=(u_\rho\overline{w}_0^{-1}h\rho, f+w(t)+\Ad_{(\overline{w}_0^{-1}h\rho)^{-1}}f)\\
&=:(u_\rho\overline{w}_0^{-1}h\rho, \xi_\rho)\in \mathscr{Z}_G\cap G\times (f+\fb)
\end{align*}
for some (unique) $u_\rho\in N$ (contained in a bounded region, i.e. in a compact region, from Lemma \ref{lemma: u_rho bounded} below) and the commutative diagram (where the items with $\{\}$ are one-point sets) 
\begin{align}\label{diagram: xi_rho, w, cS, S}
\xymatrix{\{\xi_\rho\}\ar[r]^{\Ad_{\nu_\rho}\ \ \ }& \{f+w(t'_\rho)\}\ar[d]_{\Ad_{\widetilde{u}_\rho}}&&\{f_S+t_{\rho}'\}\ar[ll]_{\Ad_{b_{1,\rho}^-\overline{w}}}\ar[d]_{\Ad_{u_{S,\rho}}}\\
&\cS&&(\cS_{\fl_S^\der}+\mathring\fz_S)\cap\fg^{\reg}\ar[ll]^{\Ad{u''(\widetilde{\varsigma})u(\varsigma)^-}\ \ \ \ }
}
\end{align}
where (i) $t_\rho'\in \cQ'_{D', \cK'}$, $N\ni\nu_\rho\overset{\text{close to}}{\sim} I$,  $\widetilde{u}_\rho\in N$ and  $u_{S,\rho}\in N_S$ are uniquely determined elements, $\widetilde{u}_\rho$ and $u_{S,\rho}$ are clearly uniformly bounded; 
(ii) $u(\varsigma)^-\in N_{P_S}^-$ and $u''(\widetilde{\varsigma})\in N$ are uniquely associated to each $\varsigma\in (\cS_{\fl_S^\der}+\mathring\fz_S)\cap \fg^{\reg}$ (note that in general $\cS_{\fl_S^\der}+\mathring\fz_S\not\subset \fg^\reg$) so that $\widetilde{\varsigma}=\Ad_{u(\varsigma)^-}(\varsigma)\in f+\fb$; (iii) one can assign a unique $b^-_{1,\rho}\in B^-$ so that the product of elements inducing the adjoint action following the two different paths from $\{f_S+t_\rho'\}$ to $\cS$ coincide (see Lemma \ref{lemma: u_rho bounded} (b) below); in particular, such a $b^-_{1,\rho}\in B^-$ is uniformly bounded. 
If we use $(g_{S,\rho}^\natural, \xi_{S,\rho}^\natural; z_\rho^\natural, t_\rho^\natural)$ as in the proof of Lemma \ref{lemma: cY_S, R} to present the equivalent point $(u_\rho\overline{w}_0^{-1}h\rho, \xi_\rho)$, through the isomorphism in (\ref{eq: chi_S, mathscr Z}), then we have 
\begin{align}
\label{eq: u_rho, natural, gz}&\Ad_{\widetilde{u}_\rho b_{1, \rho}^-\overline{w}u_{S,\rho}^{-1}}(g_{S,\rho}^\natural z_\rho^\natural, \xi_{S,\rho}^\natural+t_{\rho}^\natural)= \Ad_{\widetilde{u}_\rho\nu_\rho}(u_\rho\overline{w}_0^{-1}h\rho, \xi_\rho)\in G\times \cS
\end{align}

\begin{lemma}\label{lemma: u_rho bounded}
\item[(a)] Under the above setting, $u_\rho$ is contained in a bounded region in $N$. Moreover, for a fixed $h,t$, $\lim\limits_{|\gamma_{-\Pi}(\rho)|\rightarrow 0}u_\rho$ exists and it is the unique element in $N$ that sends $f+w_0(t)$ to $f+t$ through the adjoint action (so in fact only depends on $t$). 

\item[(b)] There exists a unique $b_{1,\rho}^-\in B^-$ making 
\begin{align*}
\widetilde{u}_\rho b_{1,\rho}^-\overline{w}=u''(\widetilde{\varsigma})u(\varsigma)^-u_{S,\rho},
\end{align*}
where $\varsigma=\Ad_{u_{S,\rho}}(f_S+t_\rho')$ and $\widetilde{\varsigma}=\Ad_{u(\varsigma)^-}(\varsigma)$. 
In particular, the elements $b^-_{1,\rho}\in B^-$ can be chosen to be uniformly bounded (i.e. contained in a fixed compact region) for $t\in \overline{D'}_S+\overline{\cK'}_{S^\perp}$. 
\end{lemma}
\begin{proof}
(a) By assumption, $u_\rho$ is determined by the property
\begin{align*}
&\Ad_{u_\rho\overline{w}_0^{-1}h\rho}(f+t+\Ad_{(\overline{w}_0^{-1}h\rho)^{-1}}f)=f+t+\Ad_{(\overline{w}_0^{-1}h\rho)^{-1}}f\\
\Leftrightarrow&\Ad_{u_\rho}(f+w_0(t)+\Ad_{\overline{w}_0^{-1}h\rho}f)=f+t+\Ad_{(\overline{w}_0^{-1}h\rho)^{-1}}f
\end{align*}
Since as $|\gamma_{-\Pi}(\rho)|\rightarrow 0$, 
\begin{align*}
&f+w_0(t)+\Ad_{\overline{w}_0^{-1}h\rho}f\rightarrow f+w_0(t)\in f+\ft,\\
&\text{ and } f+t+\Ad_{(\overline{w}_0^{-1}h\rho)^{-1}}f\rightarrow f+t\in f+\ft,
\end{align*} 
we have $u_\rho$ is bounded and $\lim\limits_{|\gamma_{-\Pi}(\rho)|\rightarrow 0}u_\rho$ is the unique element in $N$ that sends $f+w_0(t)$ to $f+t$ through the adjoint action. 

(b) The uniqueness of $b_{1,\rho}^-$ is clear. For existence, we observe that $b_{1,\rho}^-$ is an element in $G$ that takes $\Ad_{\overline{w}}(f_S+t_\rho')$ to $f+w(t_\rho')$. Since both elements are in $(\fb^-)^{\reg}$ and have the same image in $\fb^-/[\fb^-,\fb^-]$, any conjugation between them must be induced from an element in $B^-$. Then the claim follows. 

We remark that we don't really need the first claim in (b) to deduce the second claim. Here we include a slightly different proof of the second claim independent of the first, which is more natural. 
First, we have $\ft+f\subset (\fb^-)^{\reg}\rightarrow \ft$ a transverse slice to the $B^-$-orbits in $(\fb^-)^{\reg}$. Second, 
$\Ad_{\overline{w}}(f_S+t_\rho')$ gives a local transverse slice of the $B^-$-orbits over $w(\overline{D}_S'+\overline{\cK}_{S^\perp}')$. Now for each $\widetilde{t}\in w(\overline{D}_S'+\overline{\cK}_{S^\perp}')$, choose any $\widetilde{b}^-$ such that $\Ad_{\widetilde{b}^-}(\widetilde{t}+\Ad_{\overline{w}}(f_S))=\widetilde{t}+f$. Then there exists a small neighborhood $\cU_{\widetilde{t}}$ around $\widetilde{t}$ and a neighborhood $\cV_{\widetilde{b}^-}$ of $\widetilde{b}^-$ in $B^-$ such that 
\begin{align*}
\Ad_{\cV_{\widetilde{b}^-}}(\cU_{\widetilde{t}}+\Ad_{\overline{w}}f_S)\supset \cU_{\widetilde{t}}+f.
\end{align*}
Lastly, by the compactness of $\overline{D}_S'+\overline{\cK}'_{S^\perp}$, the claim follows. 
\end{proof}

\begin{lemma}\label{lemma: gh_2varphi}
Let $\cK\subset G$ be a fixed compact region. For any two $h_1, h_2\in T$ with $\log_\bR (h_j)\in \ft_{\bR}^+$, if $h_1=g h_2 \varphi$ for some $g, \varphi\in \cK$, then 
\begin{align*}
c|\lambda_{\beta^\vee}(h_2)|\leq |\lambda_{\beta^\vee}(h_1)|\leq C|\lambda_{\beta^\vee}(h_2)|, \beta\in \Pi
\end{align*}
for some constants $c,C>0$ that only depend on $\cK$. 
\end{lemma}
\begin{proof}
By symmetry, it suffices to prove $|\lambda_{\beta^\vee}(h_1)|\leq C|\lambda_{\beta^\vee}(h_2)|$. Using
$\log_\bR (h_2)\in \ft_{\bR}^+$ and $g, \varphi$ are bounded, we see that 
\begin{align*}
&|\lambda_{\beta^\vee}(h_1)|=|b_{\lambda_{\beta^\vee}}(\overline{w}_0h_1)|=|b_{\lambda_{\beta^\vee}}(\overline{w}_0gh_2\varphi)|=|\lng h_2\varphi (v_{\lambda_{\beta^\vee}}), g^{-1}v_{(-\lambda_{\beta^\vee})}\rng|\\
&\leq C|\lambda_{\beta^\vee}(h_2)|. 
\end{align*}
for some uniform constant $C>0$. 
\end{proof}

\begin{prop}\label{prop: g_S, natural, w}
Under the above settings, given any fixed compact region in $J_{L_S^\der}$, there exists $\epsilon>0$ such that for all $(h,t)\in \overline{\cU}_{\cQ, \cV}^{w(S)}$ and $|\gamma_{-\Pi}(\rho)|<\epsilon$, the corresponding $(g_{S,\rho}^\natural, \xi_{S,\rho}^\natural)$ for $\frj_\rho(u\overline{w}_0^{-1}h, \xi)$ from (\ref{eq: u_rho, natural, gz})  is outside the compact region. 
\end{prop}

\begin{proof}
It is clear from (\ref{eq: u_rho, natural, gz}) that $\xi_{S,\rho}^\natural$ is bounded, for $\xi_\rho$ and the group elements after $\Ad$ are all bounded. We only need to prove that $g_{S,\rho}^\natural\in L_S^\der$ is outside any bounded region in $L_S^\der$ near the limit of $\rho$.

Suppose the contrary, there exists a sequence $(h_n, w(t_n))\in \overline{\cU}_{\cQ, \cV}^{w(S)}$, and $\rho_n$ with $|\gamma_{-\Pi}(\rho_n)|\rightarrow 0$, such that the corresponding $g_{S,\rho_n}^\natural$ is contained in some fixed compact region $\cD^\natural\subset L_S^\der$ for all $n$. Then for each $n$, choose $w_n\in W$ such that $w_n(\log_\bR z_{\rho_n}^\natural)\in \ft_\bR^+$. Since $|W|$ is finite, by restricting to a subsequence, we may assume that $w_n=\widetilde{w}$ for a fixed $\widetilde{w}$. Fix a representative of $\widetilde{w}\in N_G(T)$ and denote it by the same notation. Then we apply Lemma \ref{lemma: gh_2varphi} to the identity on the $G$-factors in (\ref{eq: u_rho, natural, gz}), where aside from $z_{\rho_n}^\natural=\Ad_{\widetilde{w}^{-1}}(\widetilde{w}(z^\natural_{\rho_n}))$ and $\rho_n$, every other element in the products is uniformly bounded. Hence we get there are uniform constants $c, C>0$ such that 
\begin{align*}
c\leq |\lambda_{\beta^\vee}(\widetilde{w}(z^\natural_{\rho_n})\rho_n^{-1})|\leq C,\forall \beta\in \Pi
\end{align*}
which means $\widetilde{w}(z^\natural_{\rho_n})\rho_n^{-1}$ is contained in a fixed compact region in $T$. However, this is impossible for $n\gg 1$ under the assumption that $S\neq \emptyset$. 
\end{proof}

\subsection{A construction of $L_0$ and $\cL_\zeta$ for any $\zeta\in \ft_c^\reg$.}\label{subsec: construct L_zeta}
In this subsection, we give a construction of $L_0$ and $\cL_\zeta$ for any $\zeta\in \ft_c^\reg$ that are used in the main propositions in Section \ref{subsec: key propositions}. For any $R\gg 1$, consider the conormal bundle 
\begin{align}\label{eq: def Lambda_R}
\Lambda_{R}:= \Lambda_{T_{\cpt, R}}
\end{align}
of the compact torus (more precisely, an orbit of it)
\begin{align}\label{eq: T_cpt, R def}
T_{\cpt, R}=
:\{|b_{\lambda_{\beta^\vee}}|^{1/\lambda_{\beta^\vee}(\sfh_0)}=R^2/n: \beta\in \Pi\}\subset T
\end{align}
 in $\cB_{w_0}\cong T^*T$. For any $\zeta\in \ft_c^\reg$, we can form the closed \emph{non-exact} Lagrangian $\zeta+\Lambda_{R}$. We will perform two modifications for the shifted conormal bundle $\zeta+\Lambda_R$: 
 \begin{itemize}
 \item[(1)] In Subsection \ref{subsubsec: conic}, we will do a cylindrical modification of $\Lambda_R$  and get a cylindrical Lagrangian $L_0$ contained in a Liouville subsector $\cB_{w_0}^\dagg\subset J_G$ define in Subsection \ref{subsubsec: subsector}. The upshot is $L_0^\zeta:= \zeta+L_0$ will be tautologically unobstructed, so that 
 $(L_0^\zeta,\check{\rho}), \check{\rho}\in \Hom(\pi_1(T), \bC^\times)$ is a well defined object in the  wrapped Fukaya category $\cW(J_G;\sfLambda)$ over the Novikov field $\sfLambda$.
 
 \item[(2)] In Subsection \ref{subsubsec: def Lambda_R}, we will perform a compactly supported Hamiltonian deformation of $L_0^\zeta$ based on the analysis in Subsection \ref{subsubsec: B_w_0}. The resulting Lagrangian is the desired $\cL_\zeta$.

 \end{itemize}
 
 \subsubsection{A Liouville subsector}\label{subsubsec: subsector}

 Using the Weinstein handle decomposition in Proposition \ref{prop: hypersurface F} (2) and its proof for model (A), let  $\fF_0\subset \fF$ be the portion of Liouville hypersurface defined by $I=0, \widetilde{\norm}=1$, whose projection to $H^{sm}$ is contained in the stratum corresponding to $S=\emptyset$, and we will denote the projection by $\Omega_\emptyset$ (cf. Figure \ref{figure: proj L_xi, R}). Let $c_\emptyset$ denote the point that radially projects to the barycenter of the interior of $\fC^{n-1}$. 
 We assume the functions $I$ and $\widetilde{\norm}$ are of the form (\ref{eq: tildeN, B_w0}) and (\ref{eq: I, B_w0}), respectively, in a sufficiently large conic open subset (with respect to the Euler vector field) in $\cB_{w_0}$. Since the projection of $Z_{\fF_0}$ is zero in $\Omega_\emptyset$ (cf. Figure \ref{figure: proj L_xi, R}), $\fF_0$ is a itself a Liouville sector. Using the Darboux coordinates listed in (\ref{eq: q, p}), we have a natural sector splitting 
 \begin{align}\label{eq: fF_0, splitting}
 &\fF_0\cong T^*\Omega_\emptyset\times T^*T_{\cpt, 1}, \text{ where } T^*\Omega_\emptyset\cong\Omega_\emptyset\times \{(\Re p_{\beta^\vee}): \sum\limits_{\beta\in \Pi}\Re p_{\beta^\vee}=0\}, 
 \end{align}
 where $T^*\Omega_\emptyset$ and $T^*T_{\cpt, 1}$ are equipped with the standard Liouville sector structure.

Let $\cP\subset \bC_{\Re z\leq 0}$ be a Liouville subsector constructed as follows. Pick any real codimension $1$ sphere $S$ in $\chi^{-1}([0])$ surrounding $(g=I, \xi=f)$. The projection $S\subset \fF\times \bC_{\Re z<0}\rightarrow \bC_{\Re z<0}$ is contained in some proper open cone 
  \begin{align}\label{eq: Q, theta+-}
 Q=\{z=r e^{i\theta}: \theta\in (\theta_-,\theta_+)\}, \text{ for some }[\theta_-,\theta_+]\subset (\frac{\pi}{2},\frac{3\pi}{2}) 
 \end{align}
then the same holds for $\chi^{-1}([0])\cap \fF\times \bC_{\Re z<0}$. We assume that the subsector $\cP$ is of the form
 \begin{align}\label{eq: cP def}
 \cP=\{z=r e^{i\theta}: \theta\in [\frac{\pi}{2}, \frac{3\pi}{2}]\backslash (\theta_-,\theta_+), r\geq 0\}\cup \{\Re z\geq -A\},
 \end{align} 
for some fixed sufficiently large positive number $A$, shown in Figure \ref{figure: proj L_xi, R}. 

Let $\widetilde{\cH}_{K_0}\subset T^*(\Omega_\emptyset\times T_{\cpt, 1})$ (resp. $\widetilde{\cH}_{\leq K_0}$) be the contact hypersurface (resp. Liouville domain) defined by 
\begin{align}\label{eq: cH_K_0}
\sum\limits_{\beta\in \Pi}(\Re p_{\beta^\vee})^2+(\Im p_{\beta^\vee})^2= K_0^2 (\text{resp. }\leq K_0^2)\ (\text{recall }\sum\limits_{\beta\in \Pi}\Re p_{\beta^\vee}=0\text{ in }T^*\Omega_\emptyset)
\end{align}
for a sufficiently large $K_0>1$.  

\begin{lemma}\label{lemma: cH_K x P}
For sufficiently large $K_0>1$, $\chi^{-1}([0])\cap (\widetilde{\cH}_{K_0}\times \cP)=\emptyset$.
\end{lemma}
\begin{proof}
By assumption, the projection of $\chi^{-1}([0])\cap \fF\times \cP$ to $\bC_{\Re z<0}$ is contained in the pre-compact region $Q\cap \{\Re z\geq -A\}$, so it suffices to prove that the intersection $\chi^{-1}([0])\cap (\widetilde{\cH}_{K_0}\times (Q\cap \{\Re z\geq -A\}))$ is empty. Since $\overline{Q}\cap \{\Re z\geq -A\}$ is compact, for any small $\epsilon>0$, there exists $M_\epsilon>1$ such that $\varphi_{Z_\bC}^{-M_\epsilon}(\overline{Q}\cap \{\Re z\geq -A\})\subset \bC_{\Re z<0, |z|^2\leq \epsilon^2}$. 
Now apply Lemma \ref{lemma: hat chi_epsilon, gamma} (i) with $\fK=\{[0]\}$, $I\in \cV$ and any $0<\delta\ll 1$. Choose $\epsilon>0$ such that $\{\|b_\lambda(\rho)\|\leq \epsilon\}\subset \{|\gamma_{-\Pi}(\rho)|<r_{\cV, \fK, \delta}\}$. Let $K_0=e^{M_\epsilon}$, then $\chi^{-1}([0])\cap (\widetilde{\cH}_{K_0}\times (\overline{Q}\cap \{\Re z\geq -A\}))=\emptyset$ as desired. 
\end{proof}

Using Lemma \ref{lemma: cH_K x P}, we can form the ``cylindricalization" of $\fF_0\times \cP$ as 
\begin{align*}
(\widetilde{\cH}_{\leq K_0}\times \cP)\cup \bigcup\limits_{s\geq 0}\varphi_Z^{s}(\widetilde{\cH}_{K_0}\times \cP).
\end{align*}
After a standard smoothing of the corners as in \cite{GPS1}, we get a Liouville subsector of $J_G$, denoted by $\fF_0\overset{\triangle}{\times}\cP$ or $\cB_{w_0}^\dagg$ (similarly, we can also define the subsector $\fF\overset{\triangle}{\times}\cP$).  To simplify notations, to represent a Liouville sector, we will usually just write its interior. The boundary of such a Liouville sector either has been introduced or is clear from the context.

  \subsubsection{A cylindrical modification of the conormal bundle $\Lambda_R$} \label{subsubsec: conic}
 
 Let  $\Lambda_{T_{\cpt}, 1}^0$ denote for the zero-section of $T^*T_{\cpt, 1}$. 
With respect to the splitting (\ref{eq: fF_0, splitting}), we have a splitting for the conormal bundle $\Lambda_R$ as 
 \begin{align}\label{eq: Lambda_R, split}
 \Lambda_R&=T^*_{c_\emptyset}\Omega_\emptyset\times\Lambda_{T_{\cpt}, 1}^0 \times \{\Re z=-\frac{1}{R}\}. 
 \end{align}

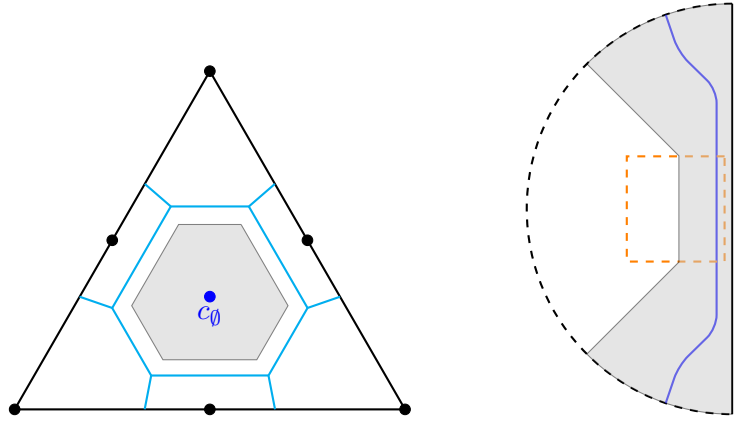
\begin{figure}[htbp]
\centering
\begin{tikzpicture}
\draw[thick] (1.73*1.5,-1.5) to (0,3);
\draw[thick] (-1.73*1.5, -1.5) to (0,3);
\draw[thick] (-1.73*1.5,-1.5) to  (1.73*1.5, -1.5);
\draw[thick, cyan] (-1.73*0.5, 1.5)--(-1.73*0.3, 1.2) coordinate (P)-- (1.73*0.3, 1.2) coordinate(Q)--(1.73*0.5, 1.5);
\draw[thick, cyan, rotate=120] (-1.73*0.5, 1.5)--(-1.73*0.3, 1.2) coordinate(PP) -- (1.73*0.3, 1.2) coordinate(QQ)--(1.73*0.5, 1.5);
\draw[thick, cyan, rotate=240] (-1.73*0.5, 1.5)--(-1.73*0.3, 1.2)coordinate(PPP) -- (1.73*0.3, 1.2) coordinate(QQQ)--(1.73*0.5, 1.5);
\draw[thick, cyan] (-1.73*0.3, 1.2)--(-1.73*0.75, 1.2-1.35);
\draw[thick, cyan, rotate=120] (-1.73*0.3, 1.2)--(-1.73*0.75, 1.2-1.35);
\draw[thick, cyan, rotate=240] (-1.73*0.3, 1.2)--(-1.73*0.75, 1.2-1.35);
\fill[fill=gray!40, draw=black, opacity=0.5] ($0.8*(Q)$)--($0.8*(P)$)--($0.8*(QQ)$)--($0.8*(PP)$)--($0.8*(QQQ)$)--($0.8*(PPP)$)--($0.8*(Q)$);
\filldraw (1.73*1.5,-1.5) circle (2pt);
\filldraw (0,3) circle (2pt);
\filldraw (-1.73*1.5,-1.5) circle (2pt);
\filldraw (1.73*0.75, 0.75) circle (2pt);
\filldraw (-1.73*0.75, 0.75) circle (2pt);
\filldraw (0, -1.5) circle (2pt);
\filldraw[blue] (0,0) circle (2pt) node[below] {$c_\emptyset$};
\end{tikzpicture}
\hspace{0.5in}
\begin{tikzpicture}
\draw[thick] (8,-1)--(8,1.73*2+1);
\draw[thick, dashed] (8,1.73*2+1) arc (90:270:1.73+1);
\draw[thick, blue, rounded corners=5pt] ({8+ cos (1.9 r)*2.73 }, {1.73+ sin (1.9 r)*2.73})--({8+ cos (1.9 r)*2.2 }, {1.73+ sin (1.9 r)*2.2})--({8+ cos (1.7 r)*1.6}, {1.73+ sin (1.7 r)*1.6})-- ({8+ cos (-1.7 r)*1.6 }, {1.73+ sin (-1.7 r)*1.6})--({8+ cos (-1.9 r)*2.2 }, {1.73+ sin (-1.9 r)*2.2})--({8+ cos (1.9 r)*2.73 }, {1.73+ sin (-1.9 r)*2.73}); 
\draw[thick, dashed, orange] (8-0.1, 1.73+0.7)--(8-1.4, 1.73+0.7)--(8-1.4, 1.73-0.7)--(8-0.1, 1.73-0.7)--(8-0.1, 1.73+0.7);
\filldraw[fill=gray!40, opacity=0.5, draw=black] ({8+cos (135)*2.73}, {1.73 +sin (135)*2.73})--({8+cos (135)}, {1.73 +sin (135)})--({8+cos (135)}, {-1.73 -sin (135)+1.73*2})-- ({8+cos (135)*2.73}, {-1.73 -sin (135)*2.73+1.73*2}) arc (225: 270: 1.73+1)--(8,1.73*2+1) arc (90: 135: 1.73+1);
\end{tikzpicture}
\caption{(1) The gray region inside $\fC^{n-1}$ (introduced in Subsection \ref{subsec: tilde N, I}) on the left represents the projection of $\fF_0$ to $H^{sm}$ composed with the radial projection $H^{sm}\rightarrow \fC^{n-1}$  (to be more precise, one needs to smooth the corners in the picture, but this is not essential as explained in \cite{GPS1}). The projections of $Z_{\fF_0}$ are all zero. The blue dot in the center of $\fC^{n-1}$, denoted by $c_\emptyset$, represents the projection of $\Lambda_R$ to $H^{sm}$ (composed with the radial projection $H^{sm}\rightarrow \fC^{n-1}$). 
(2) The gray region in the right half-disc is the Liouville subsector $\cP$. The blue curve in the right half-disc shows a cylindrical modification of the projection of $\Lambda_R$ to $\bC_{\Re z\leq 0}$. 
}
\label{figure: proj L_xi, R}
\end{figure}

Let $C_1$ be a cylindrical modification of the projection of $\Lambda_R$ in $\bC_{\Re z<0}$, which is initially $\{\Re z=-1/R\}$. Let $q=\Re z$ and $p=\Im z$. We assume that $C_1$ is contained in $\{\Re z\leq -1/R\}$ inside the region $\{|\Im z|\leq R\}$, and it is conic outside the region $\{|\Im z|\leq R\}$. Without loss of generality, we may assume that $C_1$ has compactly supported primitive $f_{C_1}$, and we can choose a compactly supported extension of $f_{C_1}$ on $\bC_{\Re z<0}$.

We make the following additional assumptions on $C_1$ and $f_{C_1}: \bC_{\Re z<0}\rightarrow \bR$:
\begin{assumption}\label{assumption: f_C1}
The curve $C_1$ is defined as the graph of a function $\varphi: \bR_p\rightarrow \bR_{q}$, where $p=\Im z$ and $q=\Re z$, which is symmetric about $p=0$ and satisfies
\begin{align*}
&\varphi(p)=\begin{cases}&-\frac{1}{R},\ -R\leq p\leq R\\
&-\frac{1}{R^2}p, \ 3R\leq p<\infty\\
\end{cases},\\
&\varphi'(p)\in (-\frac{2}{R^2},0], \text{ for } R<p<3R,\\
&\varphi(p)-p\varphi'(p)\geq -\frac{1}{R}. 
\end{align*}
Let $f_{C_1}: \bC_{\Re z<0}\rightarrow\bR$ be a compactly supported extension of $\int_0^p\frac{1}{2}(s(-\varphi'(s))+\varphi(s))ds$ on $C_1$ with support contained in $\{p^2+q^2\leq 12R^2, q\leq -\frac{\epsilon}{R}\}$ for some $\epsilon>0$. We assume that $R$ and $A$ are sufficiently large so that $\cP\supset C_1\cup \{p^2+q^2\leq 24 R^2\}\supset C_1\cup \supp(f_{C_1})$. 
\end{assumption}

Now via a similar construction as in \cite[6.2]{GPS2}, we can deform the product Lagrangian $F_{\bR, c_\emptyset}\times\Lambda_{T_{\cpt, 1}}^0\times C_1$ into a cylindrical Lagrangian. 
Since the factor $\Lambda_{T_{\cpt, 1}}^0$ is compact and conic, we only need to do a cylindrical modification of the other two factors $F_{\bR, c_\emptyset}\times C_1$ inside $T^*\Omega_\emptyset\times \bC_{\Re z<0}$. 

First, intersecting the Lagrangian $F_{\bR, c_\emptyset}\times C_1$ with a fixed contact hypersurface $\cH_K\times \bC_{\Re z<0}\subset T^*\Omega_\emptyset\times \bC_{\Re z<0}, K\gg R$ given by
\begin{align*}
&\cH_K=\{\|(\Re p_{\beta^\vee})\|=(\sum\limits_{\beta\in\Pi}(\Re p_{\beta^\vee})^2)^{\frac{1}{2}}=K\}\subset T^*\Omega_\emptyset\ (\text{recall }\sum\limits_{\beta\in\Pi}\Re p_{\beta^\vee}=0\text{ in }T^*\Omega_\emptyset), \\
&(\text{similarly, set }\cH_I=\{\|(\Re p_{\beta^\vee})\|\in I\}\text{ for any connected interval }I\subset [0,\infty))
\end{align*} 
we get a submanifold of (real) dimension $n-1$, over which 
\begin{align}\label{eq: lambda_R cap Sigma_K}
\alpha_{T^*\Omega_\emptyset}+\alpha_{\bC_{\Re z}<0}|_{(F_{\bR, c_\emptyset}\cap \cH_K)\times C_1}=df_{C_1}. 
\end{align}
In the following, we use $\alpha_{\cH_K}$ to denote for $\alpha_{T^*\Omega_\emptyset}|_{\cH_K}$, and use $\alpha_{\bC}$ to denote for $\alpha_{\bC_{\Re z}<0}$. 

Consider the 1-parameter family of contact 1-forms on $\cH_K\times \bC_{\Re z<0}$, 
\begin{align*}
\alpha_t=\alpha_{\cH_K}+\alpha_{\bC}-tdf_{C_1}, 0\leq t\leq 1. 
\end{align*}
Let $V_{\alpha}$ be the direct sum of the Reeb vector field on $(\cH_K, \alpha_{\cH_K})$ and the zero vector field on $\bC_{\Re z<0}$, and let $\varphi_{-f_{C_1}V_{\alpha}}^t$ be the flow of $-f_{C_1}V_{\alpha}$, then we have 
\begin{align}\label{eq: varphi_fV}
(\varphi_{-f_{C_1}V_{\alpha}}^t)^* \alpha_{1-t}=\alpha_1=\alpha_{\cH_K}+\alpha_{\bC}-df_{C_1}.
\end{align}
By (\ref{eq: lambda_R cap Sigma_K}), the intersection
\begin{align*}
\Gamma_{R,K}:=(F_{\bR, c_\emptyset}\cap \cH_K)\times C_1
\end{align*}
is a Legendrian submanifold in $\cH_K\times \bC_{\Re z<0}$ with respect to $\alpha_1$. By (\ref{eq: varphi_fV}), for any $0\leq t\leq 1$, $\varphi_{-f_{C_1}V_{\alpha}}^t(\Gamma_{R,K})$ is a Legendrian submanifold with respect to $\alpha_{1-t}$.

Second, since $f_{C_1}$ is compactly supported, by choosing $K$ sufficiently large, the flow $\varphi_{-f_{C_1}V_\alpha}^t(\Gamma_{R,K})$ is defined for all $0\leq t\leq 1$ inside $T^*\Omega_\emptyset\times \cP$.  
Now take the union of flow lines of $\varphi_{-f_{C_1}V_\alpha}^t(\Gamma_{R,K})$ under the Liouville vector field 
\begin{align*}
Z_{1-t}=Z_{T^*\Omega_\emptyset}+Z_{\bC_{\Re z<0}}+(1-t)X_{f_{C_1}}
\end{align*}
 of $\alpha_{1-t}$ (on the symplectization), i.e. 
 \begin{align}\label{eq: L_1-t, cyl}
 L_{1-t}^{cyl, \alpha_{1-t}}:=\bigcup\limits_{s\geq 0}\varphi_{Z_{1-t}}^s(\varphi_{-f_{C_1}V_\alpha}^t(\Gamma_{R,K})).
 \end{align}

Lastly, let 
\begin{align*}
&\phi_{t}: \cH_{[K,\infty)}\times \bC_{\Re z<0}\longrightarrow  \cH_{[K,\infty)}\times \bC_{\Re z<0}
\end{align*}
be the diffeomorphism defined by 
\begin{align*}
&\phi_{t}|_{\cH_K\times \bC_{\Re z<0}}=\varphi_{-f_{C_1}V_\alpha}^t\\
&\phi_{t}\circ\varphi^s_{Z_1}=\varphi^s_{Z_{1-t}}\circ \phi_{t}, \forall s\geq 0.
\end{align*}
Since by definition $\phi_{t}^*\alpha_{1-t}=\alpha_1$, $\phi_t$ is the Hamiltonian flow of a time-dependent family of Hamiltonian functions 
\begin{align*}
H_t=\iota_{X_t}\alpha_{1-t}+f_{C_1}, 0\leq t\leq 1.
\end{align*} 
The Hamiltonian vector field is $X_t=-f_{C_1}V_\alpha+Y_t$, where $V_\alpha$ is Hamiltonian vector field of $\frac{\|(\Re p_{\beta^\vee})\|}{K}$, and $Y_t$ is the component tangent to the factor $\bC_{\Re z<0}$ (depending on the level $\|(\Re p_{\beta^\vee})\|$). Then 
\begin{align*}
H_t=(1-\frac{\|(\Re p_{\beta^\vee})\|}{K})f_{C_1}+\iota_{Y_t}(\alpha_{\bC}-(1-t)df_{C_1}). 
\end{align*}
In particular, for any $K'>K$, on $\cH_{[K, K']}\times \bC_{\Re z<0}$, 
we have 
\begin{align*}
\supp H_t,\ \supp Y_t\subset \cH_{[K,K']}\times(\bigcup\limits_{0\leq s\leq\log (K'/K)}\varphi^s_{Z_{\bC_{\Re z<0}}+(1-t)X_{f_{C_1}}}(\supp f_{C_1}))
\end{align*}
and $|Y_t|$ is bounded from above. 

To simplify the notations, we will denote $Z_{\bC_{\Re z<0}}$ (resp. $X_{f_{C_1}}$) simply as $Z_{\bC}$ (resp. $X$), when there is no cause of confusion. Note that on any level $K\cdot e^s$, i.e. $\cH_{K\cdot e^s}\times \bC_{\Re z<0}$, 
\begin{align}\label{eq: Y_t def}
Y_t=\frac{d}{dt}\varphi_{Z_{\bC}+(1-t)X}^s\circ \varphi^{-s}_{Z_\bC+X}. 
\end{align}
Since $f_{C_1}$ has compact support, we directly see that on $\cH_{[K,K']}\times \bC_{\Re z<0}$,
\begin{align}\label{eq: bound Y_t}
|Y_t|, |H_t|\leq Q_{R, f_{C_1}}(K'/K)
\end{align}
for some constant $Q_{R, f_{C_1}}>0$ (depending only on $R$ and $f_{C_1}$).

Choose a smooth cut-off function 
\begin{align*}
&b_K: [K,\infty)\rightarrow [0,1],\\
&b_K|_{[K+2,\infty)}=1,\ b_K|_{[K, K+(1/R)]}=0,\ 0\leq b_K'(x)\leq 1\text{ for all }x.  
\end{align*}
Consider the Hamiltonian function 
\begin{align*}
\widetilde{H}_t:=b_K(\|(\Re p_{\beta^\vee})\|)H_t,
\end{align*}
where as before $\|(\Re p_{\beta^\vee})\|$ is considered as a function on $T^*\Omega_\emptyset$. 
We can extend $\widetilde{H}$ to be homogeneous near infinity and to have support contained in $\cH'_{[K,\infty)}\overset{\triangle}{\times} \cP\subset T^*\Omega'_\emptyset\overset{\triangle}{\times} \cP$, for a slight larger $\Omega'_\emptyset\supset \overline{\Omega}_\emptyset$, where $\cH'_{[K,\infty)}$ is defined similarly as $\cH_{[K,\infty)}$. 

We have  
\begin{align*}
X_{\widetilde{H}_t}=&b_K(\|(\Re p_{\beta^\vee})\|)\cdot X_t+b'_K(\|(\Re p_{\beta^\vee})\|)H_t\cdot V_\alpha\\
=& (b'_K(\|(\Re p_{\beta^\vee})\|)H_t-b_K(\|(\Re p_{\beta^\vee})\|)f_{C_1})V_\alpha+b_K(\|(\Re p_{\beta^\vee})\|)\cdot Y_t.
\end{align*}
Now set 
\begin{align}\label{eq: def L_1-t}
L_{1-t}:=\varphi_{X_{\widetilde{H}_t}}^t(F_{\bR, c_\emptyset}\times C_1)\times \Lambda_{T_{\cpt,1}}^0. 
\end{align}
Using Assumption \ref{assumption: f_C1}, it is clear from the estimate (\ref{eq: bound Y_t}) and the description of the ``cylindrical" part (\ref{eq: L_1-t, cyl}) that by choosing $K\gg R\gg K_0\gg1$ (where $K_0$ is from (\ref{eq: cH_K_0})), we have $L_{1-t}\subset \fF_0\overset{\triangle}{\times} \cP$.

The Lagrangian $L_0$, i.e. for $t=1$, gives a cylindrical modification of $\Lambda_R$ that is used in  Proposition \ref{prop: L_0, part 2} and \ref{prop: L_0, W}. Note that  
\begin{align*}
L_0\cap (\cH_{[0,K]}\times \Lambda_{T_{\cpt, 1}}^0\times \bC_{\Re z<0})= (F_{\bR, c_\emptyset}\cap \cH_{[0,K]})\times \Lambda_{T_{\cpt, 1}}^0\times C_1.
\end{align*}
For any $\zeta\in \ft_c$, let $L_0^\zeta=\zeta+L_0$. By choosing $K\gg K_0\gg R\gg |\zeta|$, we have $L_0^\zeta\subset \fF_0\overset{\triangle}{\times} \cP$. 
Let
\begin{align*}
L_0^{\zeta; 1}=(L_0^\zeta\cap (\cH_{[0,1]}\times T^*T_{\cpt, 1}\times \{\Re z\geq -\frac{3}{R}\})). 
\end{align*}
Fix sufficiently small $0<\delta<\delta'$. For $R\gg 1$, we can choose a constant $C_0\geq 1$ so that 
\begin{align}\label{eq: K_zeta, delta}
\proj_\ft(L_0^{\zeta;1})\subset \cK^{\delta,C_0}_\zeta:=\{\|\proj_{\ft_c} t-\zeta\|\leq \delta\}\cap \{\sum\limits_{\beta\in \Pi}(\Re p_{\beta^\vee})^2\leq C_0\}\subset \ft^{\reg}.
\end{align} 
In this case, we have 
\begin{itemize}
\item[(1)] the composition
\begin{align*}
L_0^{\zeta; 1}\hookrightarrow J_G\overset{\chi}{\longrightarrow}\fc
\end{align*}
is $C^1$-close to the composition
\begin{align*}
L_0^{\zeta; 1}\hookrightarrow \cB_{w_0}\cong T^*T\rightarrow \ft \overset{\chi_\ft}{\longrightarrow}\fc. 
\end{align*}
In particular, the images of both maps lie in a compact region in $\fc^\reg$. 

\item[(2)] There is a canonical $W$-equivariant identification (with respect to the standard Borel $B$ determined by the principal $\fsl_2$-triple $(e, f, \sfh_0)$)
\begin{align}\label{eq: trivial W zeta}
L_0^{\zeta; 1}\underset{\fc}{\times}\ft\overset{\sim}{\longrightarrow} L_0^{\zeta;1}\times W
\end{align}
given by the trivialization of the left-hand-side principal $W$-bundle 
\begin{align}\label{eq: trivial W, section}
\{(g,\xi,B_1): \overline{\xi}\in \fb_1/[\fb_1, \fb_1]\overset{\text{canonical}}{\cong}\fb/[\fb,\fb]\cong \ft\text{ satisfies }\overline{\xi}\in \cK^{\delta', C_0'}_\zeta\}\subset L_0^{\zeta; 1}\underset{\fc}{\times}\ft.
\end{align}
for some slightly larger $\delta'>\delta$ and $C_0'>C_0$. 
\end{itemize}

For any $\xi\in \ft$, let $\xi_\bR+\xi_c$ be the decomposition of $\xi$ with respect to $\ft\cong \ft_\bR\oplus \ft_c$. 

\begin{lemma}\label{lemma: Omega}
For $K\gg R\gg 1$, consider the projection to the second factor in the fiber product 
\begin{align*}
\varpi_\zeta: L_0^{\zeta}\underset{\fc}{\times}\ft\longrightarrow \ft.
\end{align*}
\item[(a)] If $\zeta\in \ft_c^{\reg}$, then the map $\varpi_\zeta|_{L_0^{\zeta;1}\underset{\fc}{\times}\ft}$ is arbitrarily $C^1$-close to the composition
\begin{align*}
L_0^{\zeta;1}\times W\hookrightarrow \cB_{w_0}\times W\cong T^*T\times W&\longrightarrow \ft\\
(h,t;w)&\mapsto w(t),
\end{align*}
under the canonical identification (\ref{eq: trivial W zeta}), as $R\rightarrow\infty$. 

\item[(b)] For general $\zeta\in \ft_c$, there exist $\delta_0, \widetilde{\delta}_0>0$, which are independent of (sufficiently large) $R$ and $K$, such that for $x\in (L_0^\zeta\backslash L_0^{\zeta;1})\underset{\fc}{\times}\ft$, 
\begin{align}
\label{eq: lemma: Omega (b)}&\|(\varpi_\zeta(x))_\bR\|^2\geq  \delta_0^2 \max\{1, \sum\limits_{\beta\in \Pi} (\Re p_{\beta^\vee}(x))^2\},\\
\label{eq: lemma: Omega (b),1}&\|(\varpi_\zeta(x))_\bR\|^2\geq  \widetilde{\delta}_0^2 \max\{1, \|(\varpi_\zeta(x))\|^2\}. 
\end{align}

\item[(c)] For $\zeta=0$, fix any standard ball $\bD_\bR$ centered at $0$ in $\ft_\bR$ of radius $r_0>0$. Given any $\delta>0$, for all $K\gg R\gg r_0$, we have 
\begin{align}\label{eq: lemma: Omega (c)}
\varpi_0(L_0\underset{\fc}{\times}\ft)\subset (\bD_\bR\times D_{c,\delta})\cup \bR_{\geq 1}\cdot (\partial\bD_\bR\times D_{c,\delta}),
\end{align}
where $D_{c,\delta}$ is the standard ball in $\ft_c$ centered at $0$ of radius $\delta$. 
\end{lemma}

\begin{proof}
(a) is straightforward from the above comments. 

(b) We write (the closure of) the complement of $L_0^{\zeta;1}$ in $L_0^\zeta$ as the union of three parts:
\begin{align}
\label{eq: proof L part 1}&L_0^{\zeta; [1, K+2]; 1}:= L_0^\zeta\cap (\cH_{[1,K+2]}\times T^*T_{\cpt, 1}\times \{\Re z\geq -\frac{3}{R}\})\\
\label{eq: proof L part 2}&L_0^{\zeta;[0, K+2]; \text{cone}}:= L_0^\zeta\cap (\cH_{[0,K+2]}\times T^*T_{\cpt, 1}\times \{\Re z\leq -\frac{3}{R}\})\\
\label{eq: proof L part 3}&L_0^{\zeta; [K+2,\infty)}:=L_0^\zeta\cap (\cH_{[K+2,\infty)}\times T^*T_{\cpt, 1}\times \bC_{\Re z<0}\}). 
\end{align}
Let 
\begin{align}\label{eq: cT_bR}
\cT_\bR:=\bigcup\limits_{(\nu_\beta)_\beta\in \Omega_\emptyset}\partial \{t\in \ft_\bR: \sum\limits_{\beta\in \Pi}(p_{\beta^\vee}-\sum\limits_{\beta\in \Pi}\nu_\beta\cdot p_{\beta^\vee})^2\leq 1, |\sum\limits_{\beta\in \Pi}p_{\beta^\vee}|\leq 9\},
\end{align}
where $\Omega_\emptyset\subset \{\sum\limits_{\beta\in \Pi}r_\beta=1\}\subset \bR^\Pi_{\geq 0}$, and $(\nu_\beta)_\beta$ are viewed as weights contained in $\Omega_\emptyset$. Since each $\nu_\beta$ has a strictly positive lower bound, there exists $\delta_1>0$, such that $\cT_\bR\subset \{\sum\limits_{\beta\in \Pi}(p_{\beta^\vee})^2\geq 2\delta_1^2\}$.  

Let $K\gg R$. For any point $x$ in any of the three parts (\ref{eq: proof L part 1})-(\ref{eq: proof L part 3}), there exists $t_x\geq 0$ such that $\varphi_{-Z}^{t_x}(x)$ is contained in the product 
\begin{align}\label{eq: L_1-t, 1, K+2, 1}
(\Omega_\emptyset\times T_{\cpt, 1}\times \bR_{q\in [-\frac{3}{R}, -\epsilon]}) \times\{t\in \ft: \proj_{\ft_\bR} t\in \cT_\bR, |\proj_{\ft_c} t|\leq |\zeta|\},
\end{align}
for some uniform $0<\epsilon<\frac{3}{R}$ independent of $x$, where $q=\Re z=-1/\widetilde{\norm}$ as in Subsection \ref{subsec: tilde N, I}. 
Since $\varphi_{-Z}^{t}, t\geq 0$ scales $p_{\beta^\vee}$ with weight $-1$, we have $t_x\geq  \log (\sum\limits_{\beta\in \Pi}(\Re p_{\beta^\vee}(x))^2)^{\frac{1}{2}}-C$ for some uniform constant $C\geq 0$.

It is clear that for $x$ in (\ref{eq: L_1-t, 1, K+2, 1}), $\|(\varpi_\zeta(x))_\bR\|^2\geq 1.5\delta_1^2$. Since $\varphi_{-Z}^{t}, t\geq 0$ scales $\varpi_\zeta$ with weight $-1$, we get (\ref{eq: lemma: Omega (b)}) with $\delta_0=\delta_1e^{-C}$as desired. (\ref{eq: lemma: Omega (b),1}) can be obtained similarly.  

(c) We follow essentially the same argument as for (b). In the current case, (\ref{eq: L_1-t, 1, K+2, 1}) is the same as  
\begin{align}\label{eq: L_0, 1, K+2, 1}
(\Omega_\emptyset\times T_{\cpt, 1}\times \bR_{q\in [-\frac{3}{R}, 0)}) \times \cT_\bR.
\end{align}
Then for any $x$ in (\ref{eq: L_0, 1, K+2, 1}), we have $\varpi_0(x)$ contained in the right-hand-side of (\ref{eq: lemma: Omega (c)}), then so is  $\varpi_0((L_0\backslash L_{0}^{0;1})\underset{\fc}{\times}\ft)$. 
Clearly $\varpi_0(L_{0}^{0;1}\underset{\fc}{\times}\ft)$ is contained in there too. So the proof is complete. 
\end{proof}

Let $\cJ$ be any (regular) compatible cylindrical almost complex structure on $J_G$.

\begin{prop}\label{prop: L_t, zeta, no disc}
Let $\zeta\in \ft_c^{\reg}$. For any homology class $\ell\in H_1(T,\bZ)$ with $\lng -i\zeta, \ell\rng>0$, the moduli space of $\cJ$-holomorphic discs $f: (\cD, \partial \cD)\rightarrow (J_G, L_0^\zeta)$ satifying $[f(\partial \cD)]=\ell$ is compact, and it is cobordant to $\emptyset$.
\end{prop}
\begin{proof}
First, using the same argument as in \cite[Lemma 2.42]{GPS1}, $(\cB_{w_0}^\dagg, \omega, \cJ, L_0^\zeta)$ has bounded geometry, which yields the compactness of moduli space. 
Second, by choosing $\cJ$ so that the assumption in Lemma 2.41 \emph{loc. cit.} holds for $X=\cB_{w_0}^\dagg$, we have $f(\cD)\subset \cB_{w_0}^\dagg$. Since $H_1(\cB_{w_0}^\dagg,\bZ)\cong H_1(T,\bZ)$, and $[f(\partial \cD)]=\ell\neq 0$, we conclude that the space of such $\cJ$-holomorphic discs is empty. 
\end{proof}

\subsubsection{Hamiltonian deformation of $L_0^\zeta, \zeta\in \ft_c^\reg$}\label{subsubsec: def Lambda_R}

Applying Proposition \ref{lemma: empty, Ham isotopy} for $\cK=\cK_{\zeta,C_0}^\delta$ from (\ref{eq: K_zeta, delta}), $T_{\cpt, 1}\subset \cV^\dagg$, $(c_\beta(s))_\beta=\gamma_{-\Pi}(s^{-\sfh_0})$, we get a compactly supported Hamiltonian isotopy $\varphi_s$ on $\cW_{\cV, \cK}$. 
For $L_0^\zeta=\zeta+L_0$ with $K\gg R\gg 1$ in the definition, let 
\begin{align}\label{eq: cL_zeta, def}
\cL_\zeta:=\frj_{R^{\sfh_0}}\circ\varphi_{\frac{1}{R}}\circ\frj_{R^{-\sfh_0}}(L_0^\zeta). 
\end{align}
It follows from Proposition \ref{prop: L_t, zeta, no disc} that $\cL_\zeta$ is tautologically unobstructed. 

Recall that for any Lagrangian $L\subset J_G$, we use $\widehat{L}\subset T^*T$ to denote for its transformation under the canonical Lagrangian correspondence (\ref{eq: Lag corresp}). 
Since by construction, $\varphi_s$ is the restriction of a $T$-equivariant symplectomorphism (\ref{eq: tilde varphi_s}), combining with Lemma \ref{lemma: Omega} (a), (b), we directly get the following:
\begin{lemma}\label{lemma: L_xi, R, clean}
For any $\zeta\in \ft_c^\reg$, there exists $\delta_0>0$, such that the transformed Lagrangian $\widehat{\cL}_{\zeta}\subset T^*T$ satisfies that 
\begin{align*}
\widehat{\cL}_\zeta\cap (T\times \{\|\xi_\bR\|\leq \delta_0\})=\coprod\limits_{w\in W}(T_{\cpt}\times\{w(\zeta)\})\times w(\Gamma)\subset T^*T_{\cpt}\times T^*\bR_{>0}^n 
\end{align*}
where $\Gamma\subset T^*\bR_{>0}^n$ is a Lagrangian graph over $\{\|\xi_\bR\|\leq \delta_0\}\subset \ft^*_\bR\cong \ft_\bR$.
\end{lemma}

\section{Proof of the main propositions}\label{subsec: proof prop}

Let $L_0$ and $\cL_\zeta,\zeta\in \ft_c^\reg$ be the Lagrangians defined in Section \ref{subsec: construct L_zeta}. We fix the grading on $L_0$ and $\cL_{\zeta}$ induced from the constant grading $\frac{1}{2}\dim_\bC T=\frac{1}{2}n$ on $\Lambda_{R}$ (with respect to the Sasaki almost complex structure and the canonical trivialization of $\kappa^{\otimes 2}$ as in \cite{NaZa}). Fix the trivial Pin-structure on the base $T$, i.e. the one induced from an open embedding $T\hookrightarrow \bC^n$. Then using the homotopy equivalence $L_0\rightarrow T$ (resp. $\cL_\zeta\rightarrow T$), the projection to the base of $T^*T\cong \cB_{w_0}$, $\cL_{\zeta}$ is equipped with the trivial relative Pin-structure.

\subsection{Proof of Proposition \ref{prop: L_xi, S_e, part 1}}\label{subsec: proof non-exact}

Proposition \ref{prop: L_xi, S_e, part 1} (i) is well known and (ii) is a direct consequence of the following lemma. 

\begin{lemma}\label{lemma: step 1,2,3}
By appropriate cofinal sequence of positive (resp. negative) wrappings of $\Sigma_I\rightarrow \Sigma_I^{+, (j)}$ (resp. $\Sigma_I\rightarrow \Sigma_I^{-, (j)}$), we can make $\cL_\zeta$ intersect $\Sigma_I^{+,(j)}$ (resp. $\Sigma_I^{-,(j)}$) transversely at exactly one point for all $j\gg 1$ with grading $0$ (resp. grading $n$).
\end{lemma}
\begin{proof}

Using the Lagrangian correspondence (\ref{eq: Lag corresp}), it suffices to show that $\widehat{\cL}_{\zeta}$ intersects $\widehat{\Sigma}_I^{+, (j)}$ (resp. $\widehat{\Sigma}_I^{+- (j)}$) transversely at exactly $|W|$ many points (that constitute a $W$-orbit), where $\widehat{\Sigma}_I$ is just the cotangent fiber at $I\in T$. We will define a positive linear Hamiltonian $H_1: J_G\rightarrow\bR_{\geq 0}$ (which will be a modification of (\ref{eq: tilde H_R def})), the image of $\Sigma_I$ under whose positive/negative Hamiltonian flow at time $\pm s_j, s_j\rightarrow\infty$ will give $\Sigma_I^{+, (j)}$ and $\Sigma_I^{-, (j)}$, respectively.

\emph{Step 1.} Some key features about $\widehat{\cL}_\zeta$

First, it is clear from the construction of $\cL_{\zeta}$ that $\widehat{\cL}_\zeta$ is asymptotically conic (note that $\widehat{\cL}_\zeta$ could be singular). By the proof of Lemma \ref{lemma: Omega} (b), specifically the fact that every point in $L_0^\zeta\backslash L_0^{\zeta;1}$ can be flowed into the compact region \ref{eq: L_1-t, 1, K+2, 1} under $\varphi_{-Z}^t$, we see the projection of $\widehat{\cL}_\zeta$ to $T$ is compact. Second, we recall the property of $\widehat{\cL}_\zeta$ from Lemma \ref{lemma: L_xi, R, clean}.

\emph{Step 2}. Definition of $H_1$

Without loss of generality, we may assume $\|\zeta\|>3$. We start with the Hamiltonian function (more precisely, the pullback function to $J_G$) $H_1:\fc\rightarrow \bR_{\geq 0}$ from (\ref{eq: H_R def}), whose pullback to $\ft^*$ is $\widetilde{H}_1: \ft^*\rightarrow\bR_{\geq 0}$ (\ref{eq: tilde H_R def}). 

Let 
\begin{align}\label{eq: cotangent T, splitting}
&T^*T\cong T^*T_\cpt\times T^*\ft_\bR\\
\nonumber&(h,\xi)\mapsto (\vec{\theta}, \xi_c), (\log_\bR h, \xi_\bR)
\end{align}
be the canonical splitting of $T^*T$. 
For any $\eta_0>0$, let 
\begin{align}\label{eq: cQ_eta_0}
\cQ_{\eta_0}=\{\xi=\xi_c+\xi_\bR: \|\xi_\bR\|\geq \eta_0\cdot \max\{1,\|\xi\|\}\}\subset T^*T. 
\end{align}
By Lemma \ref{lemma: Omega}, for $R, K\gg 1$ and sufficiently small $\eta_0>0$, there exists a small neighborhood $\cU_{\zeta}$ of $\zeta$ inside $\ft^*_c$, such that
\begin{align}\label{eq: proj t L_zeta}
\proj_{\ft^*}(\widehat{\cL}_\zeta)\subset \coprod\limits_{w_1\in W}w_1(\cU_{\zeta}+\{\xi_\bR\in \ft_\bR^*: \|\xi_\bR\|\leq \eta_0\})\cup \cQ_{\eta_0}.
\end{align}
Since 
\begin{align}\label{eq: cU_zeta, eta_0}
\cU_{\zeta,\eta_0}:=\cU_\zeta+\{\xi_\bR\in \ft^*_\bR: \|\xi_\bR\|\leq \eta_0\}
\end{align}
is pre-compact and its closure is away from $\ft^\sing$, we can modify the Hamiltonian $\widetilde{H}_1$ (in a $W$-invariant way) so that 
\begin{align}\label{eq: tilde H_1 split 1}
\widetilde{H}_1|_{\cU_{\zeta,\eta_0}}=\frac{1}{2}\|\xi_c\|+\frac{1}{2}\|\xi_\bR\|^2,
\end{align}
and by choosing the sequence $(\epsilon_j^{(i)})_{1\leq j\leq n}$ in the induction steps in Subsection \ref{subsec: Hamiltonians} to be much smaller than $\eta_0$, we can make sure that 
\begin{align}\label{eq: D_xi_R tilde H}
\|D_{\xi_\bR}\widetilde{H}_1|_{\cQ_{\eta_0}}\|\geq \frac{1}{2}\eta_0.
\end{align}
Note that (\ref{eq: D_xi_R tilde H}) implies that 
\begin{align}\label{eq: compact escape}
\text{for any compact region } \cK\subset T,\ \varphi_{\widetilde{H}_1}^{s}(\cQ_{\eta_0})\cap (\cK\times \ft^*)=\emptyset, \text{ for }|s|\gg 1. 
\end{align}  
Let 
\begin{align*}
\widetilde{H}_{1,c}:=y_1(\|\xi_c\|),\ \widetilde{H}_{1,\bR}:= \frac{1}{2}\|\xi_\bR\|^2,
\end{align*}
be functions on $\ft^*$, then (\ref{eq: tilde H_1 split 1}) becomes 
\begin{align*}
\widetilde{H}_1|_{\cU_{\zeta,\eta_0}}=(\widetilde{H}_{1,c}+\widetilde{H}_{1,\bR})|_{\cU_{\zeta,\eta_0}}. 
\end{align*}

{\it Step 3.} The intersection $\widehat{\cL}_\zeta\cap \varphi_{\widetilde{H}_1}^{s}(\widehat{\Sigma}_I)$ for $|s|\gg 1$.

 First, by (\ref{eq: compact escape}) and (\ref{eq: proj t L_zeta}), we must have 
 \begin{align}\label{eq: proj L_zeta cap Sigma}
 \proj_{\ft^*}(\widehat{\cL}_\zeta\cap \varphi_{\widetilde{H}_1}^{s}(\widehat{\Sigma}_I))\subset  \coprod\limits_{w_1\in W}w_1(\cU_{\zeta}+\{\xi_\bR\in \ft_\bR^*: \|\xi_\bR\|\leq \eta_0\}), |s|\gg 1.
 \end{align}
 Applying Lemma \ref{lemma: L_xi, R, clean} with $\eta_0\leq \delta_0$, we have 
\begin{align}\label{eq: L_zeta, H, H_1}
\widehat{\cL}_\zeta\cap \varphi_{\widetilde{H}_1}^{s}(\widehat{\Sigma}_I)=\widehat{\cL}_\zeta\cap \varphi_{\widetilde{H}_{1, c}+\widetilde{H}_{1, \bR}}^{s}(\widehat{\Sigma}_I)\subset \{\|\xi_\bR\|\leq \eta_0\}, |s|\gg 1.
\end{align}
Since the wrapping is under a product Hamiltonian function, it is clear that the intersection (\ref{eq: L_zeta, H, H_1}) is transverse and consists of exactly one point (cf. Figure \ref{figure: wrapping} for the negative wrapping).

Transferring back the geometry to $J_G$, it is straightforward to identify the grading for the (only) intersection point $\cL_\zeta\cap \varphi_{H_1}^{-s}(\Sigma_I), s\gg 1$ (resp. $\varphi_{H_1}^{s}(\Sigma_I)\cap \cL_{\zeta}, s\gg 1$) as $\dim_\bC T=\dim_\bC T^\vee=n$ (resp. $0$). So for any sequence $0\leq s_j\uparrow \infty$, the sequence $\Sigma_I^{-, (j)}:=\varphi_{H_1}^{-s_j}(\Sigma_I)$ (resp. $\Sigma_I^{+, (j)}:=\varphi_{H_1}^{s_j}(\Sigma_I)$) gives a desired cofinal sequence of negative (resp. positive) wrappings of $\Sigma_I$.

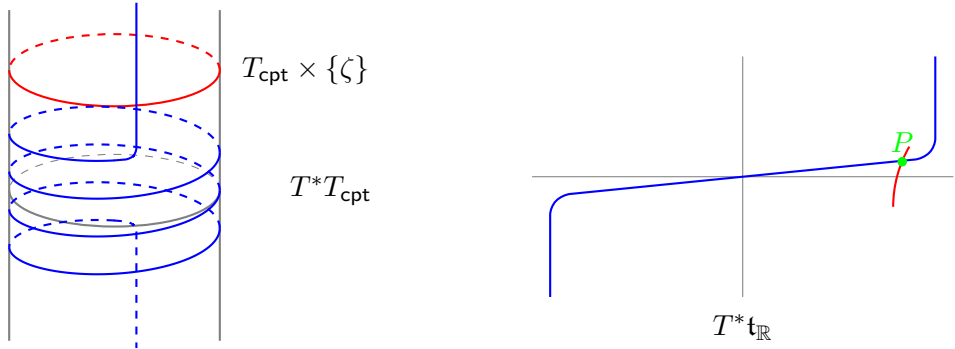
\begin{figure}[!tbp]
\centering
\begin{minipage}[b]{0.45 \textwidth}
\begin{tikzpicture}[scale=0.8]
\draw[dashed, gray] (0,0) arc (0:180:1.75 and 0.6);
\draw[gray, thick](0,0) arc (0:-180:1.75 and 0.6); 
\draw[gray, thick] (0, 3)--(0,-2.5);
\draw[gray, thick] (-3.5,3)--(-3.5, -2.5);
\draw[dashed, thick, red] (0,2) arc (0:180:1.75 and 0.6);
\draw[thick, red] (0,2) arc (0:-180:1.75 and 0.6); 
\node[right](zeta) at (0.2, 2){$T_{\cpt}\times \{\zeta\}$};
\draw[blue, thick, rounded corners=2pt] ({1.75*cos (3.5*3.14 r)-1.75}, {0.6*sin (3.5*3.14 r)+0.1*3.5*3.14})--({1.75*cos (3.55*3.14 r)-1.75}, {0.6*sin (3.55*3.14 r)+0.1*3.55*3.14})--({1.75*cos (3.55*3.14 r)-1.75+0.1}, {0.6*sin (3.55*3.14 r)+0.1*3.55*3.14+0.1})--({1.75*cos (3.55*3.14 r)-1.75+0.1}, {0.6*sin (3.55*3.14 r)+0.1*3.55*3.14+0.1+2.5)});
\draw[blue, thick, domain=3*pi: 3.5*pi, samples=50] plot ({1.75*cos (\x r)-1.75},{0.6*sin(\x r)+0.1*\x});
\draw[blue, thick, dashed, domain=2*pi: 3*pi, samples=50] plot ({1.75*cos (\x r)-1.75},{0.6*sin(\x r)+0.1*\x});
\draw[blue, thick, domain=pi:2*pi, samples=50] plot ({1.75*cos (\x r)-1.75},{0.6*sin(\x r)+0.1*\x});
\draw[blue, thick, dashed, domain=-0:pi, samples=50] plot ({1.75*cos (\x r)-1.75},{0.6*sin(\x r)+0.1*\x});
\draw[blue, thick, domain=-pi:0, samples=50] plot ({1.75*cos (\x r)-1.75},{0.6*sin(\x r)+0.1*\x});
\draw[blue, thick, dashed, domain=-2*pi:-pi, samples=50] plot ({1.75*cos (\x r)-1.75},{0.6*sin(\x r)+0.1*\x});
\draw[blue, thick, domain=-3*pi:-2*pi, samples=50] plot ({1.75*cos (\x r)-1.75},{0.6*sin(\x r)+0.1*\x});
\draw[blue, thick, dashed, domain=-3.5*pi:-3*pi, samples=50] plot ({1.75*cos (\x r)-1.75},{0.6*sin(\x r)+0.1*\x});
\draw[blue, thick, dashed, rounded corners=2pt] ({1.75*cos (-3.5*3.14 r)-1.75}, {0.6*sin (-3.5*3.14 r)-0.1*3.5*3.14})--({1.75*cos (-3.55*3.14 r)-1.75}, {0.6*sin (-3.55*3.14 r)-0.1*3.55*3.14})--({1.75*cos (-3.55*3.14 r)-1.75+0.1}, {0.6*sin (-3.55*3.14 r)-0.1*3.55*3.14-0.1})--({1.75*cos (-3.55*3.14 r)-1.75+0.1}, {0.6*sin (-3.55*3.14 r)-0.1*3.55*3.14-0.1-2)});
\node[right](cotangent) at (1, 0){$T^*T_\cpt$};
\end{tikzpicture}
\end{minipage}
\hfill
\begin{minipage}[b]{0.45\textwidth}
\begin{tikzpicture}[scale=0.8]
\draw[gray] (-3.5,0)--(3.5,0);
\draw[gray] (0,-2)--(0,2);
\draw[thick, red] (2.5,-0.5) arc (180: 150: 2 and 2);
\draw[thick, blue,  rounded corners=3pt] (-3.2, -2)--(-3.2,-0.5)--(-3,-0.3)--(3, 0.3)--(3.2, 0.5)--(3.2, 2);
\filldraw[green] 
(2.65, 0.25) circle (2pt) node[above] {$P$};
\draw (0,-2.5) node {$T^*\ft_\bR$};
\end{tikzpicture}
\end{minipage}
\caption{The intersection of the transformed Lagrangians $\widehat{\cL}_\zeta$ (red) and $\varphi_{\widetilde{H}_{1,c}+\widetilde{H}_{1, \bR}}^{-s}(\widehat{\Sigma}_I)$ (blue) in $T^*T$, which has the same intersection as $\widehat{\cL}_\zeta\cap \varphi_{\widetilde{H}_1}^{-s}(\widehat{\Sigma}_I), s\gg 1$. }\label{figure: wrapping}
\end{figure}
\end{proof}

\begin{proof}[Proof of Proposition \ref{prop: L_xi, S_e, part 1} (iii)]

We use the same $\widetilde{H}_1$ as in the proof of Lemma \ref{lemma: step 1,2,3} \emph{Step 2}. 
We look at $\widehat{\cL}_\zeta$ and $\widehat{\cL}_{w(\zeta)}$ in $T^*T$, and relate the intersections and discs for $\widehat{\cL}_\zeta$ and $\varphi_{\widetilde{H}_1}^{-s}(\widehat{\cL}_{w(\zeta)})$ in $T^*T$ to those of $\cL_{\zeta}$ and the negative wrapping of $\cL_{w(\zeta)}$ in $J_G$.

Following a similar argument as in the proof of Lemma \ref{lemma: step 1,2,3} \emph{Step 3}, we have 
\begin{align*}
\widehat{\cL}_\zeta\cap \varphi_{\widetilde{H}_1}^{-s}(\widehat{\cL}_{w(\zeta)})\subset \{\|\xi_\bR\|\leq \eta_0\}, s\gg 1.
\end{align*}
The intersection is clean along the $W$-orbit of $T_{\cpt}\times \{\zeta\}\times \{Q\}$ for some $Q\in T^*\ft_\bR$ (cf. Figure \ref{figure: wrapping, w_1}).

Transferring the geometry back to $J_G$,  
we get $\cL_\zeta\cap \varphi_{H_1}^{-s}(\cL_{w(\zeta)})$ intersects cleanly along a single $T_\cpt$-orbit in $\chi^{-1}([\zeta])\cong C_G(\zeta)\cong T$, where the identifications are using $\xi=\zeta, B_1=B$ in (\ref{eq: J_G, fc, ft}).
The restriction of the rank 1 local system $\check{\rho}_1$ on $\cL_{\zeta}$ (resp. $w(\check{\rho}_2)$ on $\cL_{w(\zeta)}$) to the $T_\cpt$-orbit is $\check{\rho}_1$ (resp. $\check{\rho}_2$) under the above identification. Now (iii) in the proposition follows.

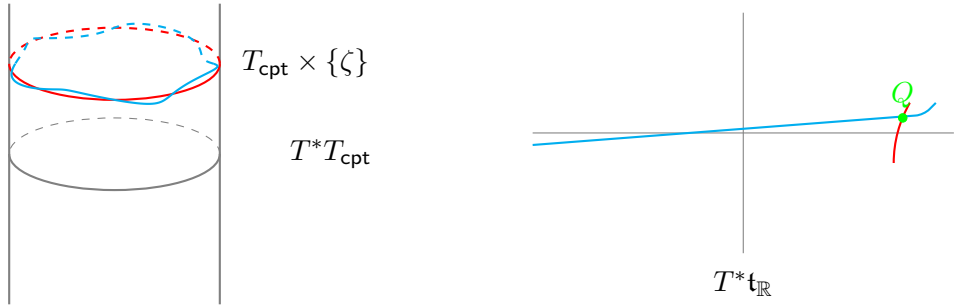
\begin{figure}[!tbp]
\centering
\begin{minipage}[b]{0.45 \textwidth}
\begin{tikzpicture}[scale=0.8]
\draw[dashed, gray] (0,0) arc (0:180:1.75 and 0.6);
\draw[gray, thick](0,0) arc (0:-180:1.75 and 0.6); 
\draw[gray, thick] (0, 2.5)--(0,-2.5);
\draw[gray, thick] (-3.5,2.5)--(-3.5, -2.5);
\draw[dashed, thick, red] (0,1.5) arc (0:180:1.75 and 0.6);
\draw[thick, red] (0,1.5) arc (0:-180:1.75 and 0.6); 
\node[right](zeta) at (0.2, 1.5){$T_{\cpt}\times \{\zeta\}$};
\draw[cyan, thick, domain=pi: 2*pi, samples=100] plot  ({1.65*cos (\x r)-0.1*sin (6*\x r)-1.75},{0.6*sin(\x r)-0.1*sin (4*\x r)+1.5}); 
\draw[cyan, thick, dashed, domain=0: pi, samples=100] plot  ({1.65*cos (\x r)-0.1*sin (6*\x r)-1.75},{0.6*sin(\x r)-0.1*sin (4*\x r)+1.5}); 
\node[right](cotangent) at (1, 0){$T^*T_\cpt$};
\end{tikzpicture}
\end{minipage}
\hfill
\begin{minipage}[b]{0.45\textwidth}
\begin{tikzpicture}[scale=0.8]
\draw[gray] (-3.5,0)--(3.5,0);
\draw[gray] (0,-2)--(0,2);
\draw[thick, red] (2.5,-0.5) arc (180: 150: 2 and 2);
\draw[thick, cyan,  rounded corners=3pt] (-3.5,-0.2)--(3, 0.3)--(3.2, 0.5);
\filldraw[green] 
(2.65, 0.25) circle (2pt) node[align=left, above] {$Q$};
\draw (0,-2.5) node {$T^*\ft_\bR$};
\end{tikzpicture}
\end{minipage}
\caption{One portion of the intersections of the transformed Lagrangians $\widehat{\cL}_\zeta$ (red) and Hamiltonian perturbed $\varphi^{-s}_{\widetilde{H}_1}(\widehat{\cL}_{w(\zeta)})$ (cyan) in $T^*T$}\label{figure: wrapping, w_1}
\end{figure}
\end{proof}

\subsection{Proof of Proposition \ref{prop: L_0, part 2} and Proposition \ref{prop: L_0, W}}\label{subsec: proof 5.6, 5.7}

Before giving the actual proofs, we give an overview of the main ideas and fix some notations. Let $L_0$ be constructed as in Subsection \ref{subsubsec: conic}. We will use the notations and results from Subsections \ref{subsec: analysis cB_w0} and \ref{subsec: walls}. 

Fix some standard open balls $\bD\subsetneq \bD'$ in $\ft$ centered at $0$, whose closures are contained in the open region $\{|\sum\limits_{\beta\in \Pi} \Re p_{\beta^\vee}|<1\}$. For sufficiently large $K\gg R\gg 1$ in the construction of $L_0$, the projection
\begin{align}\label{eq: p_D}
L_0\cap (T\times \bD')=\Lambda_R\cap (T\times \bD')\overset{p_{\bD'}}{\longrightarrow} \bD'\subset \ft\overset{\chi_\ft}{\longrightarrow} \fc
\end{align}
is very close to the restriction of $\chi$, and outside the region $T\times \bD'$, $\proj_{\ft_\bR}((L_0\cap (T\times (\ft-\bD)))\underset{\fc}{\times}\ft))$ is outside $\bD'\cap \ft_\bR$. Since the image of $p_{\bD'}$ in (\ref{eq: p_D}) is contained in $\ft_\bR$, we get $\chi(L_0\cap (T\times \bD'))$ is contained in a thin neighborhood $Nb(\fc_\bR)$ of the real locus $\fc_\bR:=\chi_\ft(\ft_\bR)$ in $\fc$. Denote the preimage of $Nb(\fc_\bR)$ in $\ft$ by $Nb(\ft_\bR)$. Without loss of generality, we may assume that $Nb(\ft_\bR)=\ft_\bR\times D_{c, \delta}$, where $D_{c,\delta}\subset \ft_c$ is a small $W$-invariant ball centered at $0$ of radius $0<\delta\ll 1$. We set $\bD'_{\bR}=\bD'\cap \ft_\bR$ and \emph{reset} $\bD'=\bD'_{\bR}\times D_{c,\delta}$ (and similarly for $\bD_{\bR}$ and $\bD$ respectively). 

Second, let $\bD^\circ_{\bR}$ be the complement of a $W$-invariant tubular neighborhood of $\ft^\sing_\bR$ in $\bD'\cap \ft_\bR$, and let $\bD^\circ=\bD^\circ_{\bR}\times D_{c, \delta}\subset Nb(\ft_\bR)$. By Proposition \ref{lemma: empty, Ham isotopy} and Corollary \ref{cor: dist star}, we have a good understanding of $L_0\cap (T\times \bD^\circ)$ inside the integrable system picture $J_G\rightarrow \fc$. 
Namely, if we do the identification 
\begin{align}\label{eq: bD, circ, T_cpt}
\chi^{-1}(\bD^\circ/W)\cong (T_{\cpt}\times D_{c,\delta})\times (\bR^n_{>0}\times (\bD^\circ_\bR\cap \ft_\bR^+))\cong T\times (\bD^\circ\cap (\ft_\bR^+\times D_{c,\delta}))
\end{align}
using $\bD^\circ/W\cong \bD^\circ\cap (\ft_\bR^+\times D_{c,\delta})$, then  
after a small Hamiltonian isotopy, there exists a pre-compact open region $\Omega\subset \bR_{>0}^n$, a Lagrangian $\Gamma_w\subset \Omega\times (\bD^\circ_{\bR}\cap \ft_\bR^+)$ that is a graph over $\bD^\circ_{\bR}\cap \ft_\bR^+$ for each $w\in W$, and some $\epsilon>0$ very small, 
such that $L_0\cap (T\times \bD^\circ)$ is identified with
\begin{align}\label{eq: L_w'}
\coprod\limits_{w\in W}(T_{\cpt}\times\{0\})\times w^{-1}(\epsilon^{-\sfh_0})\cdot \Gamma_w. 
\end{align}
So over this region, the intersections of the wrapping of $\Sigma_I$ and $L_0\cap (T\times \bD^\circ)$ can be well understood. For the portion of $L_0$ outside $T\times \bD$, we can do similar things as in the proof of Lemma \ref{lemma: step 1,2,3}, so that the wrapping of $\Sigma_I$ after sufficiently long time will have no intersection with $L_0$ over there. 

The subtle part is about $L_0\cap (T\times (\bD'-\bD^\circ))$, for which we do not have a concrete description inside the integrable system $J_G\rightarrow\fc$. Note that it is not helpful to transform the Lagrangians to $T^*T$ using the correspondence (\ref{eq: Lag corresp}), exactly by the remarks in the end of Subsection \ref{subsec: def of J_G}. However, by appropriately defining the wrapping Hamiltonian near the ``walls" in $\fc_{\bR}^\sing$, and using results from Section \ref{subsec: walls} (in particular Proposition \ref{prop: g_S, natural, w}), we can show that if the $R$ in (\ref{eq: p_D}) is sufficiently large, then the wrapping of $\Sigma_I$ will never intersect $L_0$ inside $T\times (\bD-\bD^\circ)$.

\begin{proof}[Proof of Proposition \ref{prop: L_0, part 2}]
As explained above, we are going to define an appropriate positive wrapping Hamiltonian $H$, and choose $L_0$ with $R$ sufficiently large so that the intersections $\varphi_H^{s}(\Sigma_I)\cap L_0$ are contained in $L_0\cap (T\times \bD^\circ)$ as above, for all $|s|\gg 1$. 

\emph{Step 1. Definition of a positive linear Hamiltonian $H$ on $\fc$.}

The space $\ft_\bR$ is stratified by open cones $w(\mathring{\fz}_{S,\bR}^+)$, indexed by $(S, w)$ with $S\subset \Pi$ and $w\in W^S_{\min}$, which can be viewed as a fan. Let $\bfP\subset \ft_\bR$ be the $W$-invariant dual convex polytope defined by 
\begin{align*}
\{t\in \ft_\bR^*\cong \ft_\bR: \lng w(\lambda_{\beta^\vee}), t\rng\leq 1, \beta\in \Pi, w\in W/W_{\Pi-\{\beta\}}\}.
\end{align*}
Note that on the dominant cone $\ft_\bR^+$, the polytope is cut out by $ \lng \lambda_{\beta^\vee}, t\rng\leq 1, \beta\in \Pi$. We do a $W$-invariant smoothing of $\partial\bfP$, denoted by $\partial\bfP_{sm}$, in a similar way as we did in Subsection \ref{subsubsec: H, sm}, such that (1) for any $(S,w), S\subsetneq \Pi$, there is an open neighborhood $\cU_{(S,w)}$ of $\partial\bfP_{sm}\cap w(\mathring{\fz}_{S,\bR}^+)$ in $\ft_\bR$ for which 
\begin{align*}
\partial\bfP_{sm}\cap \cU_{(S,w)}\subset (\partial\bfP_{sm}\cap w(\mathring{\fz}_{S,\bR}^+))+\ft_{S,\bR}; 
\end{align*}
in other words, $\partial\bfP_{sm}\cap \cU_{(S,w)}$ is contained in the union of normal slices of $w(\mathring{\fz}_{S,\bR}^+)$ along the intersection $\partial\bfP_{sm}\cap w(\mathring{\fz}_{S,\bR}^+)$; 
(2) in an open neighborhood of $\partial \bfP_{sm}\cap \bR_{>0}\cdot \sfh_0$, $\partial \bfP_{sm}$ is defined by $\|\xi\|=c$ for some constant $c>0$; (3) the domain $\bfP_{sm}$ enclosed by $\partial\bfP_{sm}$ is convex. 
Since the smoothing process (by induction) is very similar to that in Subsection \ref{subsubsec: H, sm}, we omit the details. Up to radial scaling we may assume that $\bfP_{sm}$ is contained in $\bD_\bR$ (cf. Figure \ref{figure: bD_R, ft_R,+}). Let $\widetilde{\cU}_{S,w}=\bR_{>0}\cdot (\partial \bfP_{sm}\cap \cU_{S,w})$ for $S\subsetneq \Pi$ and let $\widetilde{\cU}_{\Pi,1}=[0,\frac{1}{4})\cdot \partial \bfP_{sm}$. Let $\bD^\circ_\bR$ be a $W$-invariant open neighborhood (not too large so that it avoids a tubular neighborhood of $\ft_\bR^\sing$) of the complement of the union of $\widetilde{\cU}_{S,w}$ over all $(S,w)$ with $\emptyset\neq S\subset \Pi$ 
in $\bD'_\bR$, and let $\bD^\circ=\bD^\circ_\bR\times D_{c,\delta}$. 

First, we define a $W$-invariant function $\widetilde{H}$ on (part of) $\ft$ as follows. First, choose a smooth function $\sfb: \bR_{\geq 0}\rightarrow \bR_{\geq 0}$, such that 
\begin{align*}
\sfb(r)=0, r\in [0, \frac{1}{4}];\ \sfb(r)=r, r\geq\frac{3}{4};\ \sfb''(r)>0, r\in (\frac{1}{4},\frac{3}{4}).
\end{align*}
Second, define $\widetilde{H}|_{\bfP_{sm}}(r\cdot \xi)=\sfb(r)$, for $\xi\in \partial \bfP_{sm}, r\in[0,1]$, and extend it to $\bfP_{sm}\times D_{c,\delta}$ by pulling back under the obvious projection to $\bfP_{sm}$. Then extend $\widetilde{H}|_{\bfP_{sm}\times D_{c,\delta}}$ homogeneously to $(\bfP_{sm}\times D_{c,\delta})\cup \bR_{\geq 1}\cdot (\partial \bfP_{sm}\times D_{c,\delta})$. 

Now it is clear that $\widetilde{H}$ descends to a smooth function on the quotient of its defining domain in $\fc$. Then extend this to a nonnegative $H$ on $\fc$ that is homogeneous (and strictly positive) outside a compact region. We will also use $H$ to denote its pullback to $J_G$.

\emph{Step 2. Some key facts}.

If we choose $L_0$ with $K\gg R$ both sufficiently large, then we have the following: 
\begin{itemize}
\item[(i)] Let $\bD^\circ_1\subsetneq \bD^\circ_2$ be slight enlargements of $\bD^\circ$. By Proposition \ref{lemma: empty, Ham isotopy}, after a small compactly supported Hamiltonian isotopy inside $T\times\bD^\circ_2$, we can make 
\begin{align*}
&L_0\cap (T\times \bD^\circ)\subset L_0\cap \chi^{-1}(\chi_\ft(\bD^\circ_1))\overset{(\ref{eq: bD, circ, T_cpt})}{\cong} (\ref{eq: L_w'}),\\
&\chi(L_0\cap (T\times (\bD'-\bD^\circ)))\subset \chi_\ft(\bigcup\limits_{\emptyset\neq S\subset \Pi}\widetilde{\cU}_{(S,w)}). 
\end{align*}

\item[(ii)] $\chi(L_0)\subset ((\bfP_{sm}\times D_{c,\delta})\cup \bR_{\geq 1}\cdot (\partial \bfP_{sm}\times D_{c,\delta}))/W\subset \fc$ (cf. Lemma \ref{lemma: Omega} (c)). In particular, to calculate $\varphi_H^s(\Sigma_I)\cap L_0$ for any $s\in \bR$, we only use the portion of $H$ descended from $\widetilde{H}$. 

\item[(iii)] By the definition of $H$, for any $(S, w)$
\begin{align*}
\varphi_H^s(\Sigma_I)\cap \chi^{-1}(\chi_\ft(\widetilde{\cU}_{S, w}\times D_{c,\delta}\cap \bD')), s\in \bR
\end{align*}
is contained in the $\cZ(L_S)_0$-orbit of the portion of $\Sigma_I$ under the isomorphism (\ref{eq: chi_S, mathscr Z}). Therefore, by Proposition \ref{prop: g_S, natural, w}, 
\begin{align*}
\varphi_H^s(\Sigma_I)\cap (L_0\cap T\times (\bD'-\bD^\circ))=\emptyset, s\in \bR. 
\end{align*}

\item[(iv)] Using a similar argument as in the proof of Lemma \ref{lemma: step 1,2,3}, we have 
\begin{align*}
\varphi_H^s(\Sigma_I)\cap (L_0\cap T\times (\ft-\bD))\subset \varphi_H^s(\Sigma_I)\cap L_0\cap \chi^{-1}(\chi_\ft(\bR_{\geq 1}\cdot (\partial \bfP_{sm}\times D_{c,\delta}))=\emptyset,
\end{align*}
for $|s|\gg 1.$ This is due to the fact that the transformed Lagrangian $\widehat{L}_0\subset T^*T$ projects to a compact domain in $T$, while the projection of $\varphi_{\widetilde{H}}^s(\widehat{\Sigma}_I)\cap (T\times \bR_{\geq 1}\cdot (\partial \bfP_{sm}\times D_{c,\delta}))$ to $T$ is disjoint from the compact region for $|s|\gg 1$. 
\end{itemize}
 
\emph{Step 3. Calculation of wrapped Floer complexes}

By \emph{Step 2}, using the identification (\ref{eq: bD, circ, T_cpt}), the intersection(s) $\varphi_H^{s}(\Sigma_I)\cap L_0$ for $|s|\gg 1$ can be calculated in a standard way inside (\ref{eq: bD, circ, T_cpt}) with $\bD^\circ$ replaced by $\bD^\circ_1$,  as 
\begin{align}\label{eq: intersection, Gamma_w}
\varphi_H^{s}(\{I\}\times (\bD_{1}^\circ\cap (\ft_{\bR}^+\times D_{c,\delta})))\cap \coprod\limits_{w\in W}(T_{\cpt}\times\{0\})\times w^{-1}(\epsilon^{-\sfh_0})\cdot \Gamma_w,
\end{align}
where $\{I\}\times (\bD_{1}^\circ\cap (\ft_{\bR}^+\times D_{c,\delta}))$ is just the portion of the cotangent fiber at $I$ contained in (\ref{eq: bD, circ, T_cpt}). It is clear from Figure \ref{figure: bD_R, ft_R,+} that for $s\gg 1$ (resp. $s\ll -1$), 
\begin{align*}
\varphi_{\widetilde{H}|_{\bD_{\bR}^\circ\cap \ft_\bR^+}}^s(\{I\}\times (\bD_{1, \bR}^\circ\cap \ft_\bR^+))\text{ intersects } \coprod\limits_{w\in W}w^{-1}(\epsilon^{-\sfh_0})\cdot \Gamma_w
\end{align*}
transversely at exactly one point in $w_0(\epsilon^{-\sfh_0})\cdot \Gamma_{w_0}$ (resp. $\epsilon^{-\sfh_0}\cdot \Gamma_1$), for the former covers the ``strip" $\bigcup\limits_{0<\epsilon<\epsilon_0}w_0(\epsilon^{-\sfh_0})\cdot \overline{\Omega}\subset \bR^n_{>0}$ (resp. $\bigcup\limits_{0<\epsilon< \epsilon_0}\epsilon^{-\sfh_0}\cdot \overline{\Omega}\subset \bR^n_{>0}$), for some fixed $0<\epsilon_0\ll 1$, in a one-to-one manner, and approaches to the zero-section over any compact region in the ``strip" as $s\rightarrow \infty$ (resp. $s\rightarrow -\infty$). 
Now the isomorphisms (\ref{eq: skyscraper 1 no q}) and (\ref{eq: skyscraper 2 no q}) in the proposition follow directly. 

\begin{figure}[h]
\begin{tikzpicture}
\draw (0,2.5)--(0,-2.5);
\draw ({2.5*cos(30)},{2.5*sin(30)})--({-2.5*cos(30)},{-2.5*sin(30)});
\draw ({2.5*cos(150)},{2.5*sin(150)})--({-2.5*cos(150)},{-2.5*sin(150)});
\draw[thick, rounded corners=10pt] ({2.3*(1+cos (60))/2},{2.3*sin(60)/2})--({2.3*cos(60)}, {2.3*sin(60)})--({2.3*cos(120)}, {2.3*sin(120)})--({2.3*cos(180)}, {2.3*sin(180)})--({2.3*cos(240)}, {2.3*sin(240)})--({2.3*cos(300)}, {2.3*sin(300)})--(2.3,0)--({1.99*cos(30)}, {1.99*sin(30)});
\fill[cyan, draw, opacity=0.5] ({2.3*(0.6*1+0.4*cos (60))+0.3*cos(30)},{2.3*0.4*sin(60)+0.3*sin(30)})--
({2.3*(0.6*1+0.4*cos (60))-1.3*cos(30)},{2.3*0.4*sin(60)-1.3*sin(30)})--
({2.3*(0.6*1+0.4*cos (60))-1.3*cos(30)},{-(2.3*0.4*sin(60)-1.3*sin(30))})--
({2.3*(0.6*1+0.4*cos (60))+0.5*cos(30)},{-2.3*0.4*sin(60)-0.5*sin(30)}) arc (-24:24: 2.5) -- ({2.3*(0.6*1+0.4*cos (60))+0.5*cos(30)},{2.3*0.4*sin(60)+0.5*sin(30)});
\draw[thick, dashed, rounded corners=6pt, orange, scale=0.4] ({2.3*(1+cos (60))/2},{2.3*sin(60)/2})--({2.3*cos(60)}, {2.3*sin(60)})--({2.3*cos(120)}, {2.3*sin(120)})--({2.3*cos(180)}, {2.3*sin(180)})--({2.3*cos(240)}, {2.3*sin(240)})--({2.3*cos(300)}, {2.3*sin(300)})--(2.3,0)--({1.99*cos(30)}, {1.99*sin(30)});
\draw[cyan] (2.4,0) node[right] {$\bD_{1,\bR}^\circ\cap \ft_\bR^+$};
\end{tikzpicture}
\caption{The $W$-invariant region enclosed by the outer thick black curve represents $\bfP_{sm}$. The cyan region represents $\bD^\circ_{1,\bR}\cap \ft_{\bR}^+$. The middle orange dashed curve encloses $\widetilde{\cU}_{\Pi, 1}$.}\label{figure: bD_R, ft_R,+}
\end{figure}
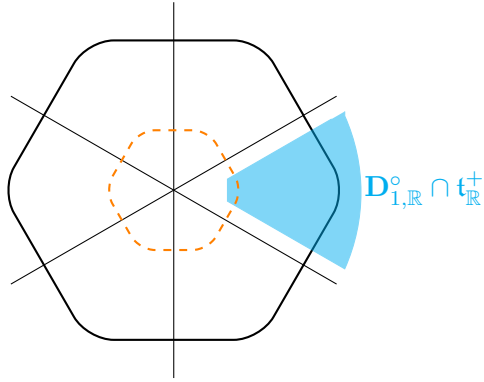
\end{proof}

\begin{proof}[Proof of Proposition \ref{prop: L_0, W}]
The proof goes similarly as the previous one for Proposition \ref{prop: L_0, part 2}, and we will use the same notations from there. To show (\ref{eq: prop L_0, W}), we pick $L_0^{(1)}$ and $L_0^{(2)}$ with respective $K^{(j)}\gg R^{(j)}\gg 1$ satisfying $R^{(2)}\gg R^{(1)}$. Then by the same reasoning in \emph{Step 2} of the previous proof, we have for $s\gg 1$, 
\begin{align*}
&\varphi_{H}^s(L_0^{(1)})\cap L_0^{(2)}\\
=&\coprod\limits_{w\in W}(T_{\cpt}\times\{0\})\times \varphi_{\widetilde{H}|_{\bD_{\bR}^\circ\cap \ft_\bR^+}}^s(w^{-1}(\epsilon_1^{-\sfh_0})\cdot \Gamma_w^{(1)})\cap \coprod\limits_{w\in W}(T_{\cpt}\times\{0\})\times w^{-1}(\epsilon_2^{-\sfh_0})\cdot \Gamma_w^{(2)}\\
=&\coprod\limits_{w\in W}(T_{\cpt}\times\{0\})\times \varphi_{\widetilde{H}|_{\bD_{\bR}^\circ\cap \ft_\bR^+}}^s(w^{-1}(\epsilon_1^{-\sfh_0})\cdot \Gamma_w^{(1)})\cap ((T_{\cpt}\times\{0\})\times w_0(\epsilon_2^{-\sfh_0})\cdot \Gamma_{w_0}^{(2)})
\end{align*}
for some $0<\epsilon_2\ll \epsilon_1\ll 1$. The intersection is clean and has $|W|$ many connected components $C_w$, each isomorphic to $T_{\cpt}$. For $(\varphi_H^s(L_0^{(1)}), \check{\rho})$ and $(L_0^{(2)}, w_1(\check{\rho}))$, $s\gg 1$, there is an indexing $p: W\rightarrow \{1,\cdots, |W|\}$ and a spectral sequence converging to their Floer cohomology (cf.  \cite{Seidel2}, \cite{Pozniak}, \cite{Schmaschke}) whose $E_1$-page is given by 
\begin{align}\label{eq: E_1, p, q}
E_1^{p(w),q}=H^{p(w)+q+i'(C_{p(w)})-n}(C_{p(w)};  w^{-1}(\check{\rho}^{-1})\otimes w_0w_1(\check{\rho})), 
\end{align}
for some coherent index $i'(C_{p(w)})\in \bZ$. If $\rho\in (T^\vee)^\reg$, then (\ref{eq: E_1, p, q}) is zero unless $w=w_1^{-1}w_0$, in which case, (\ref{eq: E_1, p, q}) is the cohomology $H^*(T,\bC)$ up to some grading shift. So the $E_1$-page converges to $\bigoplus_q E_1^{p(w_1^{-1}w_0), q}[-p(w_1^{-1}w_0)-q]=H^*(T,\bC)[d]$, for some $d\in \bZ$. On the other hand, it is clear that by local Hamiltonian perturbation of $L_0^{(2)}$ near $C_w, w\in W$, we can achieve transverse intersections with gradings ranging between $0$ and $n$, therefore $d$ must be $0$ and we obtain (\ref{eq: prop L_0, W}) as desired.    

Lastly, by Proposition \ref{prop: L_0, part 2}, both $(L_0,\check{\rho})$ and $(L_0, w_1(\check{\rho}))$ correspond to simple left $\cA_G$-modules in the abelian category of (finitely generated) $\cA_G$-modules. Hence 
\begin{align*}
H^0\Hom_{\cW(J_G)}((L_0,\check{\rho}), (L_0, w_1(\check{\rho})))\cong \bC
\end{align*}
implies that they are isomorphic. 
\end{proof}

\begin{remark}
It will follow from Proposition \ref{prop: A_G, W-inv} that for non-regular $\check{\rho}$, we also have 
\begin{align*}
\Hom_{\cW(J_G)}((L_0,\check{\rho}), (L_0, w_1(\check{\rho})))\cong H^*(T,\bC).
\end{align*} 
The above proof gives the $E_1$-page of the spectral sequence (\ref{eq: E_1, p, q}) to compute the Floer cohomology. However, there are multiple columns having nonzero entries $H^{p(w)+q}(T,\bC), 0\leq p(w)+q\leq n$. Hence the differentials $d_r^{p,q}, r\geq 1$ are not all zero, which means there are non-trivial counts of holomorphic discs entering into the calculation. 
\end{remark}

\subsection{Proof of Proposition \ref{prop: A_G commutative} and Proposition \ref{prop: A_G, W-inv}}\label{sec: proof_prop5.2.3}

Before giving the actual proof, we prove a couple of lemmas in algebra. 
Let $U\subset \bA^n=\Spec\bC[z_1,\cdots, z_n]$ be any (nonempty) affine Zariski open subvariety. 

\begin{lemma}\label{lemma: countable, 0}
Let $M^\bullet$ be any $\bC[U]$-module whose cohomology modules are countably generated over $\bC[U]$. Assume  
\begin{itemize}
\item[] $M^\bullet\otimes_{\bC[U]} (\bC[U]/\fM)\cong 0$ (here as always the tensor product and all functors have been derived unless otherwise specified) for every maximal ideal $\fM\subset \bC[U]$.
\end{itemize}
Then $M^\bullet\cong0$. 
\end{lemma}
\begin{proof}
We prove by induction on the dimension of the support of $M$ (inside the proof, we denote $M^\bullet$ simply by $M$). First, if the support of $M$ has dimension $0$ (i.e. it is at most a countable union of closed points), then $M$ is a direct sum of skyscraper sheaves, which cannot satisfy the condition unless it is $0$. 

Suppose we have proved the case for $\dim \supp(M)<k$. Now assume $\dim \supp(M)=k$. Let $\fp\subset \bC[U]$ be a prime ideal of dimension $<k$, and let $Z$ be the corresponding closed subscheme. There is an open subset $W\subset \bA^{\dim Z}$ and an open subset $Z^\dagg\subset Z$ (both nonempty), together with a proper $\acute{\text{e}}$tale map $Z^\dagg\rightarrow W$ (cf. \cite[\href{https://stacks.math.columbia.edu/tag/054L}{Tag 054L}]{stacks-project}). The pushforward of $M|_{Z^\dagg}$ along $Z^\dagg\rightarrow W$ satisfies that it has zero fiber at every closed point, hence by induction, it is zero. This shows that $M$ has zero fiber at every point of dimension $<k$. 

Now for any $\fp\subset \bC[U]$ of dimension $k$, consider  
\begin{align*}
M\otimes_{\bC[U]} (\bC[U]/\fp)\longrightarrow M\otimes_{\bC[U]}\mathbf{K}_\fp,
\end{align*}
where $\mathbf{K}_\fp$ is the fraction field of $\bC[U]/\fp$. This is an isomorphism because it induces isomorphisms on all fibers. In particular, $M\otimes_{\bC[U]} (\bC[U]/\fp)$ is a free $\mathbf{K}_\fp$-module. Thus it cannot be countably generated over $\bC[U]/\fp$ unless it is zero. This finishes the inductive step. 
\end{proof}

\begin{lemma}\label{lemma: M, countable, lb}
Let $(M^\bullet, d^\bullet)$ be any complex of $\bC[U]$-modules concentrated in degree $\leq 0$  satisfying:
\begin{itemize}
\item[(i)] each $H^i(M^\bullet)$ is countably generated over $\bC[U]$;
\item[(ii)] for any maximal ideal $\fM\subset \bC[U]$, $M^\bullet\otimes_{\bC[U]} (\bC[U]/\fM)\cong \bC[U]/\fM$;
\item[(iii)] there exists a finite collection $\{\overline{e}_1,\cdots, \overline{e}_K\}\subset H^0(M^\bullet)$ with any choice of lifts $e_1,\cdots, e_K\in \ker d^0$, such that they generate $H^0(M^\bullet\otimes_{\bC[U]} (\bC[U]/\fM))\overset{(ii)}{\cong} M^\bullet\otimes_{\bC[U]} (\bC[U]/\fM)$ for every $\fM$, under the natural surjective map $\ker d^0\rightarrow H^0(M^\bullet)\rightarrow H^0(M^\bullet\otimes_{\bC[U]} (\bC[U]/\fM))$.
\end{itemize} 
Then $M^\bullet$ is concentrated in degree $0$, and it is a locally free rank $1$ module of $\bC[U]$ (actually free of rank $1$). 
\end{lemma}

\begin{proof}
In the proof, we usually denote $M^\bullet$ simply by $M$, and we assume $M^i=0, i>0$. 
We prove by induction on the dimension $n$ and the number $K$ in (iii). If $n=0$, there is nothing to prove. If $K=1$, then $e_1\in M^0$ gives a non-vanishing global section, i.e. non-vanishing at every closed point, and clearly $e_1\cdot \bC[U]\cong \bC[U]$. So doing $\otimes (\bC[U]/\fM)$ to the exact triangle 
\begin{align*}
 e_1\cdot \bC[U]\overset{\iota_1}{\rightarrow} M\rightarrow \Cone(\iota_1)\overset{+1}{\rightarrow},
\end{align*}
we get the exact triangle 
\begin{align*}
\bC[U]/\fM\rightarrow M\underset{\bC[U]}{\otimes}(\bC[U]/\fM)\rightarrow\Cone(\iota_1)\underset{\bC[U]}{\otimes} (\bC[U]/\fM)\overset{+1}{\rightarrow}. 
\end{align*}
By the assumptions (ii) and (iii) with $K=1$, we know the first arrow is an isomorphism for all $\fM$. Hence we have 
\begin{align*}
\Cone(\iota_1)\underset{\bC[U]}{\otimes} (\bC[U]/\fM)\cong 0, \text{ for all }\fM. 
\end{align*}
Then by Lemma \ref{lemma: countable, 0}, we have $\Cone(\iota_1)\cong 0$, i.e. $M\cong \bC[U]$. We also note that all assumptions automatically hold for any open affine $\widetilde{U}\subset U$ and pullback of $M$ there.

Suppose we have proved the statement for all $U'$ of dimension strictly less than $n\geq 1$ together with all $K'$, and we have proved for $U'$ of dimension $n$ with $1\leq K'<K$. \\

\noindent\emph{Step 1. Showing that $H^0(M)$ is generated by $\overline{e}_1,\cdots, \overline{e}_K$ after localizing $(z_n-a), \forall a\in \bC$. } 

Let $\widetilde{M}^\bullet$ defined as follows. First, $\widetilde{M}^i=M^i, i<0$, and 
\begin{align*}
\widetilde{M}^0=\{&v\in M^0: \text{there exist }a_1,a_2,\cdots, a_s\in \bC,\text{ such that } \overline{v}\cdot \prod\limits_{j=1}^s(z_n-a_j)\\
&\in \sum\limits_{j=1}^K \overline{e}_j\cdot \bC[U]\},
\end{align*}
where $\overline{v}$ means the image of $v$ in $H^0(M)=M^0/\text{Im}\ d^{-1}$. Then we have a natural map $\widetilde{M}\rightarrow M$, which induces $H^i(\widetilde{M})\cong H^i(M), i<0,$ and $H^0(\widetilde{M})\hookrightarrow H^0(M)$. Moreover, for any $a\in \bC$, using the Koszul resolution of $\bC[U]/(z_n-a)$, we see that 
\begin{align}
\label{eq: H, i, below 0}&H^i(\widetilde{M}\otimes \bC[U]/(z_n-a))\cong H^i(M\otimes \bC[U]/(z_n-a)), i<0\\
\label{eq: H, 0, M}&H^0(\widetilde{M}\otimes \bC[U]/(z_n-a))\hookrightarrow H^0(M\otimes \bC[U]/(z_n-a)).
\end{align}
Since $M\otimes \bC[U]/(z_n-a)$ satisfies (i)-(iii) with dimension $n-1$, from induction we get (\ref{eq: H, i, below 0}) is zero for all $i<0$ and (\ref{eq: H, 0, M}) is an isomorphism of line bundles. In particular, we have $\widetilde{M}\cong M$, by Lemma \ref{lemma: countable, 0}. \\

\noindent\emph{Step 2. Reduction of the number of generators $K$ on certain Zariski open subsets.}

Let $\{q_1,q_2,\cdots\}$ be a countable set of generators $H^0(M)=H^0(\widetilde{M})$ such that for each $q_\ell$ there exists $\tau_\ell\in \bZ_{\geq 0}$, $a_i^{(\ell)}\in \bC, i=1,\cdots,\tau_\ell$ and $f_j^{(\ell)}(z)\in \bC[U], j=1,\cdots K$ with 
\begin{align*}
q_\ell\cdot \prod\limits_{i=1}^{\tau_\ell}(z_n-a_{i}^{(\ell)})=\sum\limits_{j=1}^K \overline{e}_j\cdot f_j^{(\ell)}(z). 
\end{align*}
For any $b\in \bC-\{a_i^{(\ell)}:1\leq i\leq \tau_\ell, \ell\in \bZ_{\geq 1} \}$ with $U\cap \{z_n=b\}\neq\emptyset$, and any closed point $P\in U\cap \{z_n=b\}$, by induction, there is a Zariski open affine subset $V:=\{g_1(z_1,\cdots, z_{n-1})\neq 0, \cdots, g_k(z_1,\cdots, z_{n-1})\neq 0\}\subset U\cap \{z_n=b\}$ containing $P$ such that for some $1\leq j_P\leq K$, $e_{j_P}$ generates $M|_{V}$. Let $\widetilde{U}=(V\times \bA^1_{z_n})\cap U$ which is again affine open in $\bA^n$. For any fixed $s\neq j_P$, we have $\overline{e}_s|_{V}=\overline{e}_{j_P}|_V\cdot h_1(z_1, \cdots, z_{n-1})$, for some $h_1\in \bC[V]$. Viewing $h_1$ as the pullback function on $\widetilde{U}$, we have $\overline{e}_s-\overline{e}_{j_P}\cdot h_1\in H^0(M|_{\widetilde{U}})\cdot (z_n-b)$, and by the generation assumption we have that there exists a finite indexing set $C$ and $r_c(z)\in \bC[\widetilde{U}], c\in C$ such that 
\begin{align*}
&\overline{e}_s-\overline{e}_{j_P}\cdot h_1=\big(\sum\limits_{c\in C} q_c\cdot r_{c}(z)\big) (z_n-b)\\
\Rightarrow&(\overline{e}_s-\overline{e}_{j_P}\cdot h_1)\cdot \prod\limits_{c\in C}\prod\limits_{i=1}^{\tau_c}(z_n-a_{i}^{(c)})= \big(\sum\limits_{j=1}^K \overline{e}_j\cdot \widetilde{r}_j(z)\big)(z_n-b),\\
&\text{ for some }\widetilde{r}_j(z)\in \bC[\widetilde{U}], \\
\Rightarrow& \overline{e}_s\cdot\big(\prod\limits_{c\in C}\prod\limits_{i=1}^{\tau_c}(z_n-a_{i}^{(c)})-\widetilde{r}_s(z)(z_n-b)\big)\in \sum\limits_{j\neq s}\overline{e}_j\cdot \bC[\widetilde{U}]. 
\end{align*}
Let 
\begin{align*}
F_s(z)=\prod\limits_{c\in C}\prod\limits_{i=1}^{\tau_c}(z_n-a_{i}^{(c)})-\widetilde{r}_s(z)(z_n-b)\big.
\end{align*}
and let $\widetilde{U}_{P}=\widetilde{U}\cap \{F_s(z)\neq 0\}$. Clearly, $\widetilde{U}_{P}\cap \{z_n=b\}=\widetilde{U}\cap \{z_n=b\}=V$. Now $M|_{\widetilde{U}_{P}}$ satisfies (i)-(iii) with $K$ replaced by $K-1$. Hence by induction, $M|_{\widetilde{U}_{P}}$ is in degree $0$ and is locally free of rank $1$. \\

\noindent\emph{Step 3. Completing the inductive step.}

Since $P\in U\cap \{z_n=b\}$ is arbitrary, we conclude with the following:
\begin{itemize}
\item[]\emph{There is a Zariski open subset $U_n\subset U$ such that $M|_{U_n}$ is locally free of rank $1$, and $U_n\supset \bigcup\limits_{b\not\in \{a_i^{(\ell)}: 1\leq i\leq \tau_\ell, \ell\in\bZ_{\geq 1}\}}\{z_n=b\}\cap U$.}
\end{itemize}
This implies that the complement $U-U_n$ projects to $\bA^1_{z_n}$ in a \emph{finite} subset. Repeating this with $z_n$ replaced by $z_j$ for each $j<n$, we get: 
\begin{itemize}
\item[]\emph{There is a finite collection of closed points $\{P_1,\cdots, P_m\}\subset U$ (with corresponding maximal ideal $\fM_{P_j}$) such that the restriction of $M$ to $U^\dagg:=U-\{P_1,\cdots, P_m\}$ is locally free of rank 1.}
\end{itemize}

Now we show that $M$ is locally free of rank $1$ on $U$.  First assume $n=1$. For each $1\leq s\leq m$, we choose any $e_{j_s}$ that generates $M\otimes\bC[U]/\fM_{P_s}$. There are two cases:
\begin{itemize}
\item[Case 1:] $\overline{e}_{j_s}|_{U^\dagg}\neq 0$.  Then since the zeros of $e_{j_s}|_{U^\dagg}\neq 0$ are finite, $e_{j_s}$ gives a non-vanishing section on a Zariski open subset containing $P_s$;

\item[Case 2:] $\overline{e}_{j_s}|_{U^\dagg}=0$. Then there exists $e_{j_s'}$ such that $\overline{e}_{j_s'}|_{U^\dagg}\neq 0$. Consider the linear combination $R\cdot \overline{e}_{j_s}+\overline{e}_{j_s'}$, for some $R\in \bR$. For $R$ sufficiently large, this gives a non-vanishing section on a Zariski open subset containing $P_s$.
\end{itemize}
So we have proved the case for $n=1$. Note that by projecting any affine smooth curve to $\bA^1$ (and using \cite[\href{https://stacks.math.columbia.edu/tag/054L}{Tag 054L}]{stacks-project} as before) and running the same argument, we get the same statement for $U$ being any affine smooth curve (without the assertion about freeness of $M$).  

For $n\geq 2$, choose $e_{j_s}$ for every $P_s$ such that the restriction of $\overline{e}_{j_s}$ at $P_s$ is nonzero. We claim that there is an open neighborhood of $P_s$ on which $\overline{e}_{j_s}$ is nonvanishing everywhere. Suppose the contrary, there exists an irreducible and reduced closed affine curve $C_s$ in a neighborhood $U_s$ of $P_s$, such that $P_s\in C_s$ and $\overline{e}_{j_s}|_{C_s\backslash \{P_s\}}=0$. Let $\pi:\widetilde{C}_s\rightarrow C_s$ be the normalization, and let $\varphi_s: \widetilde{C}_s\rightarrow U_s$ be the composition of $\pi$ with the embedding of $C_s$. Then $\varphi_s^*M$ satisfies all the requirements in the lemma on $\widetilde{C}_s$, so $\varphi_s^*M$ is a line bundle on $\widetilde{C}_s$. Since the restriction of  $H^0(\varphi_s^*)(\overline{e}_{j_s})\in H^0(\varphi_s^*M)\cong \varphi_s^*M$ at $(\pi^{-1}(P_s))_{red}$ is nonzero, we get that $H^0(\varphi_s^*)(\overline{e}_{j_s})$ is nonvanishing on an open dense subset of $\widetilde{C}_s$ containing $(\pi^{-1}(P_s))_{red}$. But this contradicts to the assumption that $\overline{e}_{j_s}$ is vanishing along $C_s\backslash \{P_s\}$. 
Now for each $1\leq s\leq m$, we are reduced to the case of $K=1$ on some open neighborhood of $P_s$, thus the case for $n\geq 2$ follows. 
\end{proof}

\begin{remark}
There are direct generalizations of Lemma \ref{lemma: countable, 0} and Lemma \ref{lemma: M, countable, lb} for any smooth (affine) variety. Using \cite[\href{https://stacks.math.columbia.edu/tag/054L}{Tag 054L}]{stacks-project}, Lemma \ref{lemma: countable, 0} easily extends. For Lemma \ref{lemma: M, countable, lb}, first one can generalize to the case where (ii) becomes that the stalk of $M$ at each closed point is concentrated in degree $0$ and has constant rank $r$, and the conclusion becomes that $M$ is locally free of rank $r$. The proof extends easily with only the following two changes:
 \begin{itemize}
 \item $K\geq r$ so the base case in the induction for $K$ is $r$;
 
 \item in Step 3, for $n=1$ case, one runs an induction on the rank $r$, i.e. for $r\geq 2$, one can find a nonvanishing section $\overline{\sigma}\in H^0(M)$ with a lifting $\sigma\in M^0$ in an affine open neighborhood $U_s$ of each $P_s$, then the cone of the natural map $\sigma\cdot \bC[U_s]\rightarrow M|_{U_s}$ satisfies the conditions in the lemma with stalks having constant rank $r-1$, so by induction it is locally free and so is $M|_{U_s}$. 
 \end{itemize} 
 Second, one can further generalize to the case that $U$ is a smooth affine variety (then for any smooth variety one can formulate the conditions in the lemma using an open affine cover), for which the proof verbatim extends (using again \cite[\href{https://stacks.math.columbia.edu/tag/054L}{Tag 054L}]{stacks-project}). 

\end{remark}

\begin{proof}[Proof of Proposition \ref{prop: A_G commutative}]

Let $\cM$ be the $\cA_G-\bC[T^\vee]$-bimodule corresponding to the co-restriction functor 
 in (\ref{eq: res, co-res}).
 
(i) For a generic cotangent fiber\footnote{Strictly speaking, we need to take a cylindricalization of $F_h$ as done in Subsection \ref{subsec: construct L_zeta} for $L_0$. Using a similar construction there, the resulting cylindrical $F_h$ satisfies the same properties.} $F_h\subset \cB_{w_0}\cong T^*T$ (equipped with constant grading $n$), $F_h\cap \chi^{-1}([0])$ transversely in $|W|$ many points and they are in the same degree (note that $F_h$ and $\chi^{-1}([0])$ are both holomorphic Lagrangians), 
 where $\chi^{-1}([0])$ is the critical handle whose cocore $\Sigma_I$ generates $\cW(J_G)$. This is due to (1) the map $\chi|_{F_h}: F_h\rightarrow\fc$ is proper by Proposition \ref{prop: proper b map}, and (2) the intersection of $F_h$ and $\chi^{-1}([\xi])$ for $[\xi]\in \fc^{\reg}$ (in a compact region) and $|\gamma_{-\Pi}(h)|\ll 1$ is transverse at $|W|$ many points (cf. Lemma \ref{lemma: hat chi_epsilon, gamma}). 
 Since $\End(\Sigma_I)$ is concentrated in degree $0$, by the wrapping exact triangle from \cite{GPS2}, we have $co\text{-}res(\bC[T^\vee])\cong 
 \cA_G^{\oplus |W|}$. This is explained in more details below. 
 
Following the same ideas as in Section 5, 6 and 9 in \emph{loc. cit.}, consider $\cW(\overline{J}_G\times T^*[0,1])\simeq \cW(\widehat{J}_G\times \bC;\ff)$, where $\ff$ is an isotropic stop in $\partial_\infty(\widehat{J}_G\times \bC)$ determined by the Lagrangian core of $\fF$ and $\{\pm\infty\}\subset \partial_\infty\bC$, which is equivalent to $\cW(J_G)$ under the Kunneth formula. It is generated by the product of the cocores $D_p:=\Sigma_I\times T^*_{1/4}(0,1)$, which is a ``small" Lagrangian linking disc of the Legendrian stop (a portion of the $\ff$) $\ff_1:= \chi^{-1}([0])\times \{-\infty\}$ at $p=(g=I, \xi=f; -\infty)$ in the contact $\infty$-boundary of the Liouville completion $\widehat{J}_G\times \bC$ (here we assume the Liouville 1-form on $\bC$ is deformed from the standard one by an exact 1-form supported in $\Re z\in (-1, 2)$ union with $\arg z\in (\frac{\pi}{3}, \frac{2\pi}{3})\cup (-\frac{2\pi}{3}, -\frac{\pi}{3})$). Now we have $L_h:=F_h\times T^*_{1/2}(0,1)$ an object in $\cW(\overline{J}_G\times T^*[0,1])$, after a cylindricalization.  We will not deal with cylindricalization explicitly since the process can be chosen to only affect things in an arbitrarily small neighborhood of the infinity end of the product Lagrangian (cf. the last paragraph of Section 7.2 in \emph{loc. cit.} for the precise statement). In particular, such a process will not affect any of the discussions below.  
 The way that $L_h$ can be generated by $D_p$ is through the wrapping exact triangles associated with a positive wrapping near $p$, which only involves a positive wrapping on the factor $\bC$ that takes the upper $\infty$-end $1/2+i\infty$ of $T^*_{1/2}(0,1)=\{\Re z=1/2\}$ passing through $-\infty$ counterclockwise. We fix the grading on $F_h$ and $\Sigma_I$ to be constantly $0$, and the grading on $\chi^{-1}([0])$ (not as an object in $\cW(J_G)$) to be $2n$. We also fix the grading (with respect to the obvious complex structure on $\bC$) on $\{\Re z=x\}, \forall x\in \bR$ to be constantly $\frac{1}{2}$, and the grading on the skeleton $\bR$ (for $\bC$ with stops $\{\pm\infty\}$) to be constantly $0$. This induces canonical gradings on $D_p$, $L_h$ and the non-closed conic Lagrangian $\cC:=\Cone(\ff_1)$.
  
Label the intersection points $F_h\cap \chi^{-1}([0])$ by $p_1', \cdots, p_{|W|}'$. Then $L_h\cap \cC=\{\widetilde{p}_1,\cdots, \widetilde{p}_{|W|}\}$, where $\widetilde{p}_j=(p_j';1/2)\in \widehat{J}_G\times \bC$. Let $p_j=(p_j';-\infty)\in (F_h\cap \chi^{-1}([0]))\times \partial_\infty\bC$. There is a sequence of wrapping exact triangles  
\begin{align*}
L_h^{(i+1)}=(L_h^{(i)})^w\longrightarrow L_h^{(i)}\longrightarrow D_{p_{i+1}}[d_i]\overset{+1}{\longrightarrow},\ 0\leq i\leq |W|-1
\end{align*}
where $L_h^{(0)}=L_h$, $L_h^{(|W|)}\cong 0$, $(L_h^{(i)})^w$ is the wrapping of $L_h^{(i)}$ across $p_{i+1}$, $D_{p_{j}}\cong D_p$ and $L_h^{(i)}\cap \Cone(\ff_1)$ transversely in $|W|-i$ many points, and $d_i\in \bZ$ is some grading shift. To determine each grading shift $d_i$, we first appeal to the local model of the surgery $L_h^{(i)}\sharp_\gamma D_p=(L_h^{(i)})^w$ as in Section 5.5 of \cite{GPS2}. Using Figure 14 of \emph{loc. cit.} and forget about the current situation for a moment, if we assume that the (piece of) cone over the portion of $\partial_\infty L$ in $L$ in that picture has the same grading as the (piece of) cone of the bottom portion of $\partial_\infty D_p$ in $D_p$ as well as the cone of $\Lambda$, then the wrapping exact triangle is $L^w\rightarrow L\rightarrow D_p\overset{+1}{\rightarrow}$ without any grading shift on $D_p$. This can be equivalently characterized as follows. Introduce $\Lambda'$ to be a shift of $\Lambda$ in the negative vertical direction (i.e. under a negative Reeb flow), that is placed in the middle of $\Lambda$ and $\partial_\infty L$ but above the bottom of $\partial_\infty D_p$. Do a wrapping of $L$ in a small neighborhood of the portion of $\partial_\infty L$ in the positive direction that passes through $\Lambda'$ but not $\Lambda$ (i.e. the rightmost picture with $\Lambda$ replaced by $\Lambda'$; the resulting $L$ is equivalent to the initial $L$ because the wrapping does not across a stop). Denote the resulting Lagrangian by $L'$. Then $D_p$ intersects $\Cone(\Lambda')$ in exactly one point, and $L'$ also intersects $\Cone(\Lambda')$ at exactly one point in a small neighborhood of that local picture. Then the degree of the former intersection point is one higher than the degree of the latter. Equivalently, this can be characterized by introducing a linking disc $D$ that links both $D_p$ and $L$ as in Figure 17 of \emph{loc. cit.}, then $HF(D_p, D)=HF(L, D)$ (usual Floer complex but not wrapped) are both 1-dimensional and of the same degree. As explained in \emph{loc. cit.}, the higher dimensional local models are just from spinning the above pictures around the middle vertical axis. 

Now return to our current situation. The grading of the intersection points $L_h\cap \cC$ is all $0$, and the same holds for $D_p\cap \cC$ (cf. \cite[Proposition 5.2]{Jin1}). Let $D'_{p_j'}$ be a small disc around $p'_j$ in $F_h$, and let $D''_{p_j'}$ be a smaller one. Then the wrapping of $L_h$ across $p_j$ only involves $D'_{p_j'}\times \{1/2+i\bR_{\geq 1000}\}$. Apply a positive wrapping supported in a neighborhood $\cU_{p_j}$ of $D'_{p_j'}\times (\partial_\infty \bC\backslash \{\pm \infty\})$ inside $J_G\times (\bC\backslash \bR)$, whose restriction to a neighborhood $\cU_{p_j}'$ of $D''_{p_j'}\times (\partial_\infty \bC\backslash \{\pm \infty\})$ is just a positive wrapping on the $\bC$-factor. The resulting $L_h$ (which is equivalent to the original $L_h$ in the wrapped Fukaya category of $\widehat{J}_G\times \bC$ with stop $\ff$) intersects a small negative push-off (in the $\bC$-factor) of $\cC$ inside $\cU_{p_j}'$ at exactly one point, and the grading is $-1$. Using the above local model where $L_h$ plays the role of $L$ there, we see that $d_1=0$, and consequently $d_j=0$, since everything is local around $p_j'$ in the $J_G$-factor and then the same argument works for $L_h^{(j-1)}$ wrapping across $p_{j}$. 

Lastly using that $\Hom(D_p, D_p)$ is concentrated in degree $0$, we get $L_h\cong D_p^{\oplus |W|}$, and so $F_h\cong \Sigma_I^{\oplus |W|}$ as desired.

We now show that $\cM\cong \bC[T^\vee]$ as a $\bC[T^\vee]$-module. Fix an identification $\cM\cong \cA_G^{\oplus|W|}$ as left $\cA_G$-modules from above. 
Let $e_j=[0,\cdots, 0, 1,0,\cdots,0], 1\leq j\leq |W|$ be the element in $\cA_G^{\oplus |W|}$ that has $1\in\cA_G$ in the $j$-th component and $0$ otherwise. Then $\cM$ as a right $\bC[T^\vee]$-module (in degree $0$) satisfies (i)-(iii) in Lemma \ref{lemma: M, countable, lb}. Indeed, Proposition \ref{prop: L_0, part 2} (\ref{eq: skyscraper 2 no q}) says $co\text{-}res(L_0,\check{\rho})\in \cW(J_G)$, for any $(L_0,\check{\rho})\in \cW(\cB_{w_0})$\footnote{In that proposition, $(L_0,\check{\rho})$ is regarded as an object in $\cW(J_G)$, which is the same as $co\text{-}res(L_0,\check{\rho})$, when $(L_0,\check{\rho})$ is thought as an object in $\cW(\cB_{w_0})$.}, is an $\cA_G$-module with underling vector space just $\bC$. Using the correspondence between $(L_0, \check{\rho})\in \cW(\cB_0)$ and simple skyscraper sheaves on $T^\vee$, this means $\cM\otimes_{\bC[T^\vee]}\bC[T^\vee]/\fM\cong \bC$ as a vector space, for every $\fM$. 
Also $e_j,j=1,\cdots, |W|$ serve as those in (iii), since otherwise if there exists $\fM$ such that $e_1,\cdots, e_{|W|}\in \cM\fM$, then $\cM=\cA_Ge_1\oplus\cdots \oplus \cA_Ge_{|W|}\subset \cM\fM$, contradicting to property (ii). 
Hence $\cM$ is locally free of rank $1$ over $\bC[T^\vee]$. But this means $\cM\cong \bC[T^\vee]$.

(ii)
By (i), we have  
 \begin{align*}
 &res(\cF)\cong \Hom_{\cA_G\text{-}\Mod}(\cM, \cF)\cong \Hom_{\cA_G\text{-}\Mod}(\cA_G^{\oplus |W|}, \cF)\\
 &\cong  (\cA_G^{\oplus |W|})^\vee\underset{\cA_G}{\otimes} \cF,
 \end{align*}
where $\cM^\vee\cong (\cA_G^{\oplus |W|})^\vee:=\Hom_{\cA_G-\text{Mod}}(\cA_G^{\oplus |W|}, \cA_G)$ with the right $\cA_G$-module structure from that on the target. Using the same method as (i), and Proposition \ref{prop: L_0, part 2}, (\ref{eq: skyscraper 1 no q}), we deduce that $\cM^\vee$ as a left $\bC[T^\vee]$-module is free of rank 1. In more details, by adjunction, we have 
\begin{align*}
\Hom_{\cW(\cB_{w_0})}((L_0, \check{\rho}), res(\Sigma_I))\cong \Hom_{\cW(J_G)}((L_0, \check{\rho}), \Sigma_I)\cong \bC[-n]. 
\end{align*} 
where we use that the co-restriction functor takes $(L_0, \check{\rho})\in \cW(\cB_{w_0})$ to $(L_0, \check{\rho})\in\cW(J_G)$. Now using the correspondence between $(L_0, \check{\rho})\in \cW(\cB_0)$ and simple skyscraper sheaves on $T^\vee$ again, the above equation says $\bC[T^\vee]/\fM\otimes_{\bC[T^\vee]}\cM^\vee\cong \bC$ for all maximal ideals $\fM$, which verifies property (ii) in  Lemma \ref{lemma: M, countable, lb}. The rest steps go verbatim as we did for part (i).

(iii)
  
By (ii) we have an isomorphism of $\bC[T^\vee]-\cA_G$-bimodules 
\begin{align*}
\cM^\vee\cong (\cA_G^{\oplus |W|})^\vee\cong \bC[T^\vee]
\end{align*}
that represent the restriction functor. This implies that the natural algebra map $\cA_G\rightarrow \bC[T^\vee]\cong \End_{\bC[T^\vee]}(\bC[T^\vee])$ is injective, which forces $\cA_G$ to be commutative. Alternatively, we can use (i) to deduce the injective algebra map $\cA_G\rightarrow \bC[T^\vee]\cong \End_{\bC[T^\vee]}(\cM)\cong\bC[T^\vee]$.

(iv) 
We view the $\bC[T^\vee]$-module structure on $\cM$ in terms as an embedding into matrix algebras over $\cA_G$: 
\begin{align}\label{eq: C[T], matrix}
\bC[T^\vee]\hookrightarrow\End_{\cA_G}(\cA_G^{\oplus|W|}).
\end{align}
Let $\{x^{\pm\lambda^\vee_{\alpha}}, \alpha\in \Pi\}$ be the standard algebra generators of $\bC[T^\vee]$, and let $c^{\alpha,\pm}_{ij},1\leq i,j\leq |W|$ be the entries of the matrix image of $x^{\pm\lambda^\vee_{\alpha}}$ from (\ref{eq: C[T], matrix}). Let $v=(v_1,\cdots, v_{|W|})$ be a generator of $\cA_G^{\oplus|W|}$ as a rank $1$ $\bC[T^\vee]$-module. Then it is clear that $\cA_G$ as an algebra is generated by $c^{\alpha,\pm}_{ij},1\leq i,j\leq |W|, \alpha\in\Pi$, and $v_1,\cdots, v_{|W|}$. 
\end{proof}

\begin{proof}[Proof of Proposition \ref{prop: A_G, W-inv}]

(i) It follows directly from Proposition \ref{prop: A_G commutative} that the co-restriction functor (resp. the restriction functor) is isomorphic to $\mathsf{f}_*$ (resp.  $\mathsf{f}^!$) on coherent sheaves. Here we also use the general result about partially wrapped Fukaya categories that $\cA_G$ is smooth\footnote{Alternatively, one can use $\Perf(\cA_G)$ instead of coherent sheaves on $\Spec\cA_G$ in the statement of the proposition, which will not affect the proof of Theorem \ref{thm: sec G adjoint}.}. 

(ii) follows from Proposition  \ref{prop: L_0, W}. More explicitly, since $(L_0, \check{\rho})\in \cW(\cB_{w_0})\simeq \Coh(T^\vee)$ represents the (simple) skyscraper sheaf at $\check{\rho}\in T^\vee$, Proposition \ref{prop: L_0, W} implies that for any $\check{\rho}\in (T^\vee)^{\reg}$, the $W$-orbit of the corresponding skyscraper sheaves are sent to the same skyscraper sheaf on $\Spec \cA_G$ via $\mathsf{f}_*$. Since $T^\vee$ is a smooth affine variety, the map $\mathsf{f}$ is $W$-invariant. This finishes the proof.  
\end{proof}

\section{HMS of $J_G$ for reductive groups}\label{sec: HMS reductive}

 In this section, we prove HMS for $J_G$ when $G$ is any complex reductive group, based on the result for adjoint groups. The main result is stated in Theorem \ref{thm: HMS for reductive}. We give two proofs of the main theorem: one is using microlocal sheaf category and the other is using wrapped Fukaya category. Although arguing in different languages, the underlying principles of the proofs are essentially the same: one combines the HMS for adjoint $G$ and monadicity properties of functors associated with finite central quotients for a reductive group. Lastly, we describe the diagram of co-restriction functors among wrapped Fukaya categories associated with the sector covering of $J_G$ by $J_{L_S}$ of the standard Levi subgroups $L_S$. Each co-restriction functor serves as an induction functor from a smaller Levi to a bigger Levi.

\subsection{Basic set-up}

\subsubsection{Regular coverings}
Let $G$ be any connected complex reductive Lie group, with $G^\der=[G, G]$, and let $\cZ(G)$ (resp. $\cZ(G)_0$, $\cZ(G^\der)_0$) denote its center (resp. the identity component, $\cZ(G^\der)\cap \cZ(G)_0$). Let $G_\flat=G/\cZ(G^\der)=G_\ad\times (\cZ(G)/\cZ(G^\der))$. Then 
\begin{align}
\nonumber&G=G^\der\underset{\cZ(G^\der)_0}{\times}\cZ(G)_0,\ \cZ(G)/\cZ(G^\der)=\cZ(G)_0/\cZ(G^\der)_0,\text{ and the projection}\\
\label{eq: fq}&\fq: J_G=(J_{G^\der}\times T^*\cZ(G)_0)/\cZ(G^\der)_0\longrightarrow J_{G_\flat}=J_{G_\ad}\times T^*(\cZ(G)_0/\cZ(G^\der)_0)
\end{align}
is a regular covering with covering group $\cZ(G^\der)$. By Section \ref{sec: skeleton, sector}, $J_{G^\der}$ can be partially compactified to be a Liouville (or Weinstein) sector, which is $\cZ(G^\der)$-equivariant. Then $J_{G^\der}\times T^*\cZ(G)_0$ can be equipped with the product sector structure that is $\cZ(G^\der)\times \cZ(G^\der)_0$-equivariant, and the quotient by the diagonal $\cZ(G^\der)_0$ gives a Liouville sector structure on a partial compactification of $J_G$. 

Let $\Lambda_G$ denote for the Lagrangian skeleton of $J_G$ in the Liouville completion. Then we have the restriction
\begin{align}\label{eq: q_LambdaG}
\fq|_{\Lambda_G}: \Lambda_G=(\Lambda_{G^\der}\times \cZ(G)_0)/\cZ(G^\der)_0\longrightarrow \Lambda_{G_\flat} =\Lambda_{G_\ad}\times (\cZ(G)_0/\cZ(G^\der)_0)
\end{align}
a regular covering as well. Here and after, for notations regarding Lie groups, their Langlands dual, maximal tori and their different forms (e.g. adjoint, simply connected etc.), we follow the notations in Section \ref{sec: introduction}.

\subsubsection{Sector inclusions}\label{ss: sectorinclusion}

Assume first that $G$ is semisimple. Using a similar argument as in Proposition \ref{prop: hypersurface F} and in Subsection \ref{subsubsec: subsector}, we have natural sector inclusions up to Liouville homotopies (that do \emph{not} affect the wrapped Fukaya category), for each $S\subsetneq \Pi$, 
\begin{align}\label{eq: sector_inclusion}
\overline{J}_{L_S}\hookrightarrow \overline{J}_G,
\end{align}
compatible with $\fq$, which gives a covering of $\Lambda_G$ by $\Lambda_{L_S}$. These can be constructed as follows. 

For any $S\subsetneq \Pi$, 
 we have a canonical isomorphism 
\begin{align*}
\pi_r^{-1}(\bR^S_{\geq 0}\times \bR^{\Pi\backslash S}_{>0})=\fU_S\cong J_{L_S^\der}\underset{\cZ(L_S^\der)_0}{\times}T^*\cZ(L_S)_{0,\cpt}\times T^*\bR_{>0}^{\Pi\backslash S}. 
\end{align*}
Now we choose subsectors in $J_{L_S^\der}$ and $T^*\bR_{>0}^{\Pi\backslash S}$ as follows. 

First, applying results from Subsection \ref{subsec: compactify, J_G} for $L_S^\der$ replacing $G$, by choosing a Liouville hypersurface $\fF^S$, we have the partial compactification 
\begin{align*}
\overline{J}_{L_S^\der}=J_{L_S^\der}\coprod_{J_{L_S^\der}-\cB_1^S} (\fF^S\times \bC_{\Re z\leq 0}),
\end{align*}
where $\cB_1^S$ is the union of Kostant sections in $J_{L_S^\der}$. 
Let $\theta_-, \theta_+$ be as in (\ref{eq: Q, theta+-}) for $L_S^\der$, and let 
\begin{align*}
\cQ'_S=\{z=re^{i\theta}: \theta\in [\theta_-, \theta_+], r\geq 0\}\cap \{\Re z\leq -A\} 
\end{align*}
for any fixed $A>0$. Then 
\begin{align*}
\overline{J}'_{L_S^\der}:=\cB_{1}^S\cup (\fF^S\times \cQ_S')
\end{align*}
is a Liouville subsector (up to deformations on the factor $\cQ_S'$; cf. \cite[Proposition 2.27]{GPS1}) equivalent to $\overline{J}_{L_S^\der}$. In particular, the embedding $\overline{J}'_{L_S^\der}\hookrightarrow \overline{J}_{L_S^\der}$ induces an equivalence of the wrapped Fukaya categories. 

Second, the Liouville form on the factor $T^*\bR_{>0}^{\Pi\backslash S}\cong \prod_{\beta\not\in S}T^*\bR_{>0}^{\{\beta\}}$ is given by 
\begin{align*}
-\sum_{\beta\not\in S} (\Re p_{\beta_{S^\perp}^\vee}d\Re q_{\lambda_{\beta^\vee}}+d\Re p_{\beta_{S^\perp}^\vee}).
\end{align*}
using the Darboux coordinates (for the logarithmic coordinates on the base) at the end of Section \ref{subsec: algebraic setup}. 
The Liouville vector field is the sum of $-\partial_{\Re q_{\lambda_{\beta^\vee}}}+\Re p_{\beta_{S^\perp}^\vee}\partial_{\Re p_{\beta_{S^\perp}^\vee}}$ in each factor $T^*\bR_{>0}^{\{\beta\}}$. In particular, any integral curve of the Liouville vector field is of the form $C_{a}:=\{\Re p_{\beta_{S^\perp}^\vee}=a e^{-\Re q_{\lambda_{\beta^\vee}}}\}$, for $a\in\bR$. Fixing $a\gg k>1$, then for each factor $T^*\bR_{>0}^{\{\beta\}}$,  we can define a Liouville subsector as the closed region bounded by 
\begin{align*}
&C_{\pm a}\cap \{\Re q_{\lambda_{\beta^\vee}}\leq k\}
, C_{\pm \frac{1}{a}}\cap \{\Re q_{\lambda_{\beta^\vee}}\leq -k\}, \{\Re q_{\lambda_{\beta^\vee}}=k, \Re p_{\beta_{S^\perp}^\vee}\in [-ae^{-k},ae^{-k}]\},\\
&\{\Re q_{\lambda_{\beta^\vee}}=-k, \Re p_{\beta_{S^\perp}^\vee}\in [-\frac{e^k}{a}, \frac{e^k}{a}]\}.
\end{align*} 
Let $\cP_{S^\perp}$ be the product of these regions in $T^*\bR_{>0}^{\Pi\backslash S}$. Then the product sector 
\begin{align*}
\overline{J}'_{L_S}:=\overline{J}'_{L_S^\der}\underset{\cZ(L_S^\der)_0}{\times}T^*\cZ(L_S)_{0,\cpt}\times \cP_{S^\perp}\subset \fU_S\subset J_G
\end{align*}
is equivalent to $\overline{J}_{L_S}$ (up to deformations on the factor $\cP_{S^\perp}$ and on $\overline{J}'_{L_S^\der}$ as above). 

Clearly, we can make (a contractible space of) choices for each $\overline{J}'_{L_S}, S\subsetneq \Pi$, so that we have sector inclusions $\overline{J}'_{L_{S_1}}\subset \overline{J}'_{L_{S_2}}$ for any $S_1\subset S_2$. Now assume $G$ is reductive, then we have the regular covering 
\begin{align*}
J_G\rightarrow J_{G_\flat}=J_{G_\ad}\times T^*(\cZ(G)_{\cpt}/\cZ(G^\der))\times T^*\cZ(G)_{\bR}. 
\end{align*}
Choose a sufficiently large open ball $\Omega_\bR\subset \cZ(G)_{\bR}$, then 
\begin{align*}
\overline{J}^\flat_{L_{S}}:= \overline{J}'_{L'_{S;\ad}}\times T^*(\cZ(G)_{\cpt}/\cZ(G^\der))\times T^*\overline{\Omega}_\bR
\end{align*}
is a subsector of $\overline{J}_{G_\flat}$, where $L'_{S;\ad}$ denote the standard Levi subgroup in $G_\ad$. Its preimage in $\overline{J}_G$, denoted by $\overline{J}'_{L_{S}}$, is a subsector that is equivalent to $\overline{J}_{L_S}$.  We can endow the set of sector inclusions $\overline{J}'_{L_{S}}\subset \overline{J}_G, S\subset \Pi$, which depends on a contractible space of choices, with an obvious filtered category structure:   
$\{\overline{J}'_{L_{S}}\subset \overline{J}_G, S\subset \Pi\}\rightarrow\{\overline{J}''_{L_{S}}\subset \overline{J}_G, S\subset\Pi\} $ if and only if $\overline{J}'_{L_{S}}\subset \overline{J}''_{L_{S}}$ for all $S$. Let $\OneCat_\bC^L$ (resp. $\OneCat_\bC^R$) is the $\infty$-category of presentable stable $\bC$-linear categories with right adjointable (resp. left adjointable) functors, i.e. those admit a right adjoint (resp. left adjoint).  
Then we have the following:

\begin{prop}\label{prop: sector_cover}
For any complex reductive group $G$, there is a filtered category of sector inclusions $\overline{J}'_{L_S}\subset \overline{J}_G$, for any $S\subset\Pi$, so that 
\begin{itemize}
\item each $\overline{J}'_{L_S}$ is canonically a subsector of $\overline{J}_{L_S}$ that induces an equivalence of wrapped Fukaya categories; \\

\item for any $S$, we have $\overline{J}'_{L_{S^\dagg}}\subset \overline{J}'_{L_{S}}, \forall S^\dagg\subset S$, are sector inclusions, whose composition with the inclusion $\overline{J}'_{L_{S}}\subset \overline{J}_{L_{S}}$ give an object in the filtered category of sector inclusions associated with $ \overline{J}_{L_{S}}$. 

\end{itemize}
In particular, there is a canonically defined functor 
\begin{align}\label{eq: functor, sector-cover}
(\{S\subset \Pi\}, \subset)&\longrightarrow \OneCat_\bC^L \\
\nonumber S&\mapsto \cW(\overline{J}_{L_S}).
\end{align}
\end{prop}

\subsection{Proof of HMS using microlocal sheaf categories}

By Lemma \ref{lemma: pi_1, J_G}, we have 
\begin{align*}
\pi_1(\Lambda_{G_\flat})=\pi_1(J_{G_{\ad}})\times \pi_1(\cZ(G)/\cZ(G^\der))\cong \pi_1(G_\ad)\times \pi_1(\cZ(G)/\cZ(G^\der)).
\end{align*}
The covering $\fq_\Lambda:=\fq|_{\Lambda_G}$ corresponds to 
\begin{align}\label{eq: fqcorresppi_1}
\pi_1(G_\ad)\times \pi_1(\cZ(G)/\cZ(G^\der)) \cong \cZ(G^\der_{sc})\times \pi_1(\cZ(G)/\cZ(G^\der))\twoheadrightarrow \cZ(G^\der),
\end{align}
which is the multiplication of the two natural projections from the two factors on the LHS. Let $\cL_{\Lambda_{G_\flat}}$ be the local system $(\fq_{\Lambda})_*\underline{\bC}$ on $\Lambda_{G_\flat}$, which is corresponding to the representation of $\pi_1(\Lambda_{G_\flat})$ on the space of functions $\bC[\cZ(G^\der)]$ through (\ref{eq: fqcorresppi_1}). Let 
\begin{align*}
T_{\flat}=T/\cZ(G^\der)=T_\ad\times (\cZ(G)/\cZ(G^\der)),
\end{align*}
and let $q_T: T\rightarrow T_\flat$ to the covering map. 

In the following, for a Lagrangian skeleton $\Lambda$ of a Weinstein sector $X$, let $\mu\Shv_\Lambda$ be the sheaf of microlocal categories (over $\bk=\bC$) on $\Lambda$, in the sense of \cite{NaSh}. The definition of $\mu\Shv_\Lambda$ requires the same data
for defining the wrapped Fukaya category, and they are equivalent for a fixed choice of data (as remarked in  \emph{loc. cit.}). For $X=\overline{J}_G$, $\mu\Shv_{\Lambda_G}$ can be defined over $\bZ$, for which the data amount to a trivialization of $\overline{J}_G\rightarrow B^2\bZ\times B^3(\bZ/2\bZ)$. Since $J_G$ is hyperKahler, there is a canonical trivialization of $\overline{J}_G\rightarrow B^2\bZ$ (cf. beginning of Section \ref{sec: wrapping Ham}). The map $\overline{J}_G\rightarrow B^3(\bZ/2\bZ)$ is trivial (this is always the case), and different choices of trivializations give equivalent categories.

\begin{prop}\label{prop: res,qstar}
We have a natural diagram of adjoint functors
\begin{equation}\label{diagram: muShvLambdaG}
\begin{tikzcd}[arrow style=tikz,>=stealth,row sep=2em, column sep=2em]
&\mu\Shv_{\Lambda_G}(\Lambda_G)\ar[r, shift left=.4ex, "(\fq_\Lambda)_*\simeq (\fq_\Lambda)_!"]\ar[d, shift left=.4ex,  "res"]&\mu\Shv_{\Lambda_{G_\flat}}(\Lambda_{G_\flat})\ar[l, shift left=.4ex, "\fq_\Lambda^*\simeq \fq_\Lambda^!"]\ar[d, shift left=.4ex, "res_\flat"]\ar[r, "\simeq"]&\Ind\Coh((T_{\flat})^\vee\sslash W)\\
\Ind\Coh(T^\vee)&\Loc(T)\ar[l,"\simeq"']\ar[r, shift left=.4ex, "(\fq_T)_*\simeq (\fq_T)_!"]\ar[u,shift left=.4ex, "co\text{-}res"]&\Loc(T_\flat)\ar[l, shift left=.4ex, "\fq_T^*\simeq \fq_T^!"]\ar[u, shift left=.4ex, "co\text{-}res_\flat"]\ar[r,"\simeq"]&\Ind\Coh((T_\flat)^\vee) 
\end{tikzcd}
\end{equation}
satisfying the following natural relations: 
\begin{align*}
&res\circ\fq_\Lambda^*\simeq (\fq_T)^*\circ res_\flat: \mu\Shv_{\Lambda_{G_\flat}}(\Lambda_{G_\flat})\rightarrow \Loc(T),\\
&(\fq_\Lambda)_*\circ co\text{-}res\simeq co\text{-}res_\flat\circ (\fq_T)_*: \Loc(T)\rightarrow \mu\Shv_{\Lambda_{G_\flat}}(\Lambda_{G_\flat})\\
&(\fq_T)_*\circ res\simeq res_\flat\circ (\fq_\Lambda)_*: \mu\Shv_{\Lambda_G}(\Lambda_G)\rightarrow \Loc(T_\flat)\\
&\fq_\Lambda^*\circ co\text{-}res_\flat\simeq co\text{-}res\circ\fq_T^*: \Loc(T_\flat)\rightarrow \mu\Shv_{\Lambda_G}(\Lambda_G)\\
&(\fq_\Lambda)_*\fq_\Lambda^*\simeq (-)\otimes \cL_{\Lambda_{G_\flat}}, 
\ \fq_\Lambda^*(\fq_\Lambda)_*\simeq (-)\otimes  \underline{\bC[\cZ(G^\der)]}.
\end{align*}
\end{prop}

\begin{proof}
The top right equivalence of the diagram (\ref{diagram: muShvLambdaG})  follows from the HMS of $J_{G_\ad}$ and the Kunneth formula for microlocal sheaves or wrapped Fukaya category (cf. \cite{GPS2}). The bottom two equivalences in the diagram are well known. 
The first and the third identification of functors (involving $res$) are obvious, and the second and the fourth follow from taking left adjoint of the former. 
The last line on the two compositions of $\fq_\Lambda^*$ and $(\fq_\Lambda)_*$ are obvious. 
\end{proof}

\begin{thm}\label{thm: mushvLambda_G}
For any connected complex reductive Lie group $G$, we have 
\begin{align*}
\mu\Shv_{\Lambda_G}(\Lambda_G)\simeq \Ind\Coh((T_\flat)^\vee\sslash W)^{\cZ(G^\der)^*},
\end{align*}
where $\cZ(G^\der)^*$ is the Pontryagin dual of $\cZ(G^\der)$. 
If $\cZ(G)$ is connected, then $\cZ(G^\der)^*$ acts freely on $(T/\cZ(G^\der))^\vee\sslash W$, so we have 
\begin{align*}
\mu\Shv_{\Lambda_G}(\Lambda_G)\simeq \Ind\Coh(T^\vee\sslash W).
\end{align*}
\end{thm}

\begin{proof}
By Proposition \ref{prop: res,qstar} and the Barr--Beck--Lurie theorem \cite[Theorem 4.7.4.5]{Lurie}, we have 
\begin{align*}
&\mu\Shv_{\Lambda_G}(\Lambda_G)
\simeq  res_\flat\circ (\fq_\Lambda)_*\fq_\Lambda^*\circ co\text{-}res_\flat-\Mod(\Loc(T_\flat))\\
&\simeq (res_\flat\circ co\text{-}res_\flat)\circ (\fq_T)_*\fq_T^*-\Mod(\Loc(T_\flat)).
\end{align*}
Consider the Cartesian diagram in which all the maps are the natural ones
\begin{align*}
\xymatrix{T^\vee_{\flat}\ar[r]^{\pi_u\ \ \ \ \ \ }\ar[d]_{p_\ell}&T^\vee_{\flat}/\cZ(G^\der)^*=T^\vee\ar[d]^{p_r}\\
T_{\flat}^\vee\sslash W\ar[r]_{\pi_d\ \ \ \ \ \ } &(T^\vee_{\flat}\sslash W)/\cZ(G^\der)^*
}. 
\end{align*}
Then we have 
\begin{align*}
(res_\flat\circ co\text{-}res_\flat)\circ (\fq_T)_*\fq_T^*\simeq (\pi_d p_\ell)^*(\pi_d p_\ell)_*: \Ind\Coh(T_\flat^\vee)\rightarrow \Ind\Coh(T_\flat^\vee). 
\end{align*}
Therefore, 
\begin{align*}
\mu\Shv_{\Lambda_G}(\Lambda_G)&\simeq  (\pi_d p_\ell)^*(\pi_d p_\ell)_*-\Mod(\Ind\Coh(T_\flat^\vee))\\
&\simeq \Ind\Coh(T^\vee_{\flat}\sslash W)^{\cZ(G^\der)^*}.
\end{align*}

Since 
\begin{align*}
T_\flat^\vee\sslash W= (T_\ad)^\vee\sslash W\times (\cZ(G)/\cZ(G^\der))^\vee=(T_\ad)^\vee\sslash W\times (\cZ(G)_0/\cZ(G^\der)_0)^\vee,
\end{align*}
if $\cZ(G)$ is connected, then $\cZ(G^\der)^*$ acts freely on it. 
\end{proof}

We remark that throughout the paper, we can replace all $\Ind\Coh$ by $\text{QCoh}$, because they are only taken on smooth (Deligne-Mumford) stacks.

\subsection{Proof of HMS using wrapped Fukaya categories}\label{subsec: Pf of HMS}

Choose a wrapping Hamiltonian $H_1$ (resp. $H_2$) on $J_{G_\ad}$ as in Section \ref{subsec: Hamiltonians} (resp. on $T^*(\cZ(G)_0/\cZ(G^\der)_0)$), and let $H$ be the sum of $H_1$ and $H_2$ on the product, then $\fq^*H$ is a well defined $\cZ(G^\der)$-invariant wrapping Hamiltonian on $J_G$. For any cylindrical Lagrangian $L\subset J_{G_\flat}$, we have 
\begin{align}\label{eq: fq, H, varphi}
\fq^{-1}\varphi_H^t(L)=\varphi_{\fq^*H}^t(\fq^{-1}(L)),
\end{align}
and $\varphi_H^t(L), t\in \bR_{\geq 0}$ is cofinal in the wrapping category $(L\rightarrow -)^+$ if and only if (every connected component of) $\varphi_{\fq^*H}^t(\fq^{-1}(L))$ is cofinal in the the wrapping category $(\fq^{-1}(L)\rightarrow -)^+$.

\begin{lemma}\label{lemma: F_L, F_R}
The quotient map $\mathsf{q}$ (\ref{eq: fq}) induces an adjoint pair of functors (with $F^L$ viewed as the left adjoint)
on wrapped Fukaya categories:
\begin{equation}\label{eq: F_L, F_R}
\begin{tikzcd}[arrow style=tikz,>=stealth,row sep=4em]
\cW(J_G)\arrow[rr, shift left=.4ex, "F_R"]
  && \cW(J_{G_\flat})\ar[ll, shift left=.4ex, "F^L"]
\end{tikzcd}. 
\end{equation}

\end{lemma}

\begin{proof}
First, the functor $F^L$ can be defined as follows. 
Since $\fq$ is a regular covering map, for any  cylindrical Lagrangians $L, L'\in \cW(J_{G_\flat}) $, define $F^L$ as an $A_\infty$-functor following the notations in \cite[Chapter I. 1 (1b)]{Seidel1}
\begin{align}
\nonumber (F^L)^1: CF_{J_{G_\flat}}(\varphi_H^t(L), L')&\longrightarrow CF_{J_G}(\varphi_{\fq^*H}^t(\fq^{-1}(L)), \fq^{-1}(L'))\\
\label{eq: FL1def} x&\mapsto \sum\limits_{y\in \fq^{-1}(x)} y
\end{align}
for every generator $x\in \varphi_H^t(L)\cap L', t\geq 0$.  Let $(F^L)^d=0, d>1$. This is clearly compatibile with continuation maps.   For any collection of cylindrical Lagrangians $L_1, \cdots, L_k\subset J_{G_\ad}$ and $t_1\geq t_2\geq \cdots \geq t_{k-1}\geq 0$, we have a canonical commutative diagram
\begin{align*}
\xymatrix{CF(L^{t_{k-1}}_{k-1}, L_k)\otimes \cdots\otimes CF(L_1^{t_1}, L^{t_2}_2)\ar[r]^{\ \ \ \ \ \mu^{k-1}}\ar[d]&CF(L_1^{t_1}, L_k)\ar[d]\\
CF(\fq^{-1}(L_{k-1}^{t_{k-1}}), \fq^{-1}(L_k))\otimes \cdots\otimes CF(\fq^{-1}(L_1^{t_1}), \fq^{-1}(L_2^{t_2}))\ar[r]^{\ \ \ \ \ \ \ \ \ \ \ \ \ \ \ \ \ \ \ \ \ \ \  \mu^{k-1}}&CF(\fq^{-1}(L_1^{t_1}), \fq^{-1}(L_k))
}
\end{align*}
where $L_j^{t_j}:=\varphi_H^{t_j}(L_j)$ and $\fq^{-1}(L_{j})^{t_{j}}$ for $\varphi_{\fq^*H}^{t_j}(\fq^{-1}(L_j))$. This follows from the unique lifting property of a $\cJ$-holomorphic disc with the only incoming vertex at $y\in L_1^{t_1}\cap L_k$ to a $\fq^*\cJ$-holomorphic\footnote{Here $\fq^*\cJ$ is well defined for $\fq$ is a local isomorphism.} disc with the incoming vertex at each given point in $\fq^{-1}(y)$.

Second, we give a definition of $F_R$ on the full subcategory of (connected and) simply connected Lagrangian branes $L\subset J_G$ satisfying $\fq(L)$ is embedded. Equivalently, this means $\fq: L\rightarrow \fq(L)$ is an isomorphism. 
Such Lagrangians include the Kostant sections $\Sigma_z, z\in \cZ(G)$ (as a product cylindrical Lagrangian), which are generators of $\cW(J_G)$, so the definition of $F_R$ extends to $\cW(J_G)$ in a unique way (up to equivalences). Fix an indexing of $\cZ(G^\der)$ and denote the elements in $\cZ(G^\der)$ by $z_j$.  
The Kostant section for $z_j\in \cZ(G^\der)$ is just the product of the Kostant section for $z_j$ in $J_{G^\der}$, denoted by $\Sigma_{z_j}'$, and the cotangent fiber at identity in $T^*\cZ(G)_0$. Then $\Sigma_{z_j}, j\in  \cZ(G^\der)$, generate $\cW(J_G)$\footnote{It actually suffices to choose just one $z_j$ for each connected component of $\cZ(G)$ to generate $\cW(J_G)$.}. 
On the object level, $F_R(L)=\fq(L)$. 
\begin{align*}
(F_R)^1: CF_{J_{G}}(L_1^{t_1}, L_2)&\longrightarrow CF_{J_{G_\flat}}(\fq(L_1^{t_1}), \fq(L_2))\\
x&\mapsto \fq(x), 
\end{align*}
for every generator $x\in L_1^{t_1}\cap L_2$. 
Let $(F_R)^d=0, d>1$.  For any collection  $L_1, \cdots, L_k\subset J_{G}$ and $t_1\geq t_2\geq \cdots \geq t_{k-1}\geq 0$, we have an obvious commutative diagram
\begin{align}\label{diagram: }
\xymatrix{CF(L^{t_{k-1}}_{k-1}, L_k)\otimes \cdots\otimes CF(L_1^{t_1}, L^{t_2}_2)\ar[r]^{\ \ \ \ \ \mu^{k-1}}\ar[d]&CF(L_1^{t_1}, L_k)\ar[d]\\
CF(\fq(L_{k-1}^{t_{k-1}}), \fq(L_k))\otimes \cdots\otimes CF(\fq(L_1^{t_1}), \fq(L_2^{t_2}))\ar[r]^{\ \ \ \ \ \ \ \ \ \ \ \ \ \ \ \ \ \ \ \ \ \ \  \mu^{k-1}}&CF(\fq(L_1^{t_1}), \fq(L_k)). 
}
\end{align}
due to the isomorphism $\fq: L\overset{\sim}{\rightarrow} \fq(L)$

Lastly, we verify the adjunction property about $F^L, F_R$. The co-unit map on each generator $\Sigma_{z_j}\in \cW(J_G)$,  
\begin{align*}
F^LF_R(\Sigma_{z_j})=\bigoplus\limits_{z_i\in \cZ(G^\der)}\Sigma_{z_i}\longrightarrow \Sigma_{z_j}
\end{align*}
is given by the projection $\proj_{z_j}$ to the $z_j$-component. For each $x\in \Sigma_{z_j}^{t}\cap\Sigma_{z_k}$ corresponding to a generator of 
\begin{align*}
CF_{J_{G_\flat}}(\Sigma_{z_j}^t, \Sigma_{z_k})\cong CF_{J_{G^\der}}((\Sigma'_{z_j})^t, \Sigma'_{z_k})\times CF_{T^*(\cZ(G))_0}((T^*_I\cZ(G)_0)^t, T^*_I\cZ(G)_0)
\end{align*}
(recall the cochain complex is concentrated in degree $0$ for a sequence of $t\rightarrow\infty$; cf. Proposition \ref{prop: A_G, center}), we have a strictly commutative diagram (after taking $t\rightarrow \infty$)
\begin{align*}
\xymatrix{
F^LF_R(\Sigma_{z_j})=\bigoplus\limits_{z_i\in \cZ(G^\der)}\Sigma_{z_i}\ar[r]^{\ \ \ \ \ \ \ \ \ \ \ \ \ \ \proj_{z_j}}\ar[d]_{\sum\limits_{u\in \cZ(G^\der)}u\cdot x}&\Sigma_{z_j}\ar[d]^{x}\\
F^LF_R(\Sigma_{z_k})\cong \bigoplus\limits_{z_i\in \cZ(G^\der)}\Sigma_{z_i}\ar[r]^{\ \ \ \ \ \ \ \ \ \ \ \ \ \  \proj_{z_k}} &\Sigma_{z_k}
}. 
\end{align*}
where $u\cdot x\in CF(\Sigma_{u\cdot z_j}, \Sigma_{u\cdot z_k})$ under the action of $u\in \cZ(G^\der)$. 
Such data completely determine the co-unit map. The unit map on the generator $\Sigma_{I}=\Sigma_{I}''\times T^*_I(\cZ(G)/\cZ(G^\der))\in \cW(J_{G_\ad}\times T^*(\cZ(G)/\cZ(G^\der)))$, where $\Sigma_{I}''$ denotes the Kostant section in $J_{G_\ad}$, 
\begin{align*}
\Sigma_{I}\longrightarrow F_RF^L(\Sigma_{I})=\Sigma_I^{\oplus \cZ(G^\der)}
\end{align*}
is given by the diagonal embedding. For any $x\in \Sigma_I^t\cap \Sigma_I$ corresponding to a generator of $CF(\Sigma_{I}^t, \Sigma_{I})$, which projects to $x''\in (\Sigma_I'')^t\cap \Sigma_I''\subset J_{G_\ad}$, we have a strictly commutative diagram
\begin{align}\label{diagram: Sigma_I, x, FRL}
\xymatrix{\Sigma_I\ar[r]\ar[d]_{x}&F_RF^L(\Sigma_{I})=\Sigma_I^{\oplus \cZ(G^\der)}\ar[d]^{(x^{\oplus \cZ(G^\der)})\circ\sigma_x}\\
\Sigma_I\ar[r]&F_RF^L(\Sigma_{I})=\Sigma_I^{\oplus \cZ(G^\der)}
},
\end{align}
where $\sigma_x$ is the permutation on the indexed set $\cZ(G^\der)$ given by multiplying the inverse of the element in $\cZ(G^\der)$ corresponding to $x''\in X_*(T_\ad)^+$ (cf. the same proposition just mentioned),  and $x^{\oplus \cZ(G^\der)}$ means the morphism given by a $\cZ(G^\der)\times \cZ(G^\der)$-matrix with diagonal entries all equal to $x$. 

The identities 
\begin{align*}
&(F_R\overset{(unit)\circ F_R}{\longrightarrow} F_RF^LF_R\overset{F_R\circ (co\text{-}unit)}{\longrightarrow} F_R)\simeq id_{F_R}\\
&(F^L\overset{F^L\circ (unit)}{\longrightarrow} F^LF_RF^L\overset{(co\text{-}unit)\circ F^L}{\longrightarrow} F^L)\simeq id_{F^L}
\end{align*}
can be directly checked on the generators. We leave the details to the interested reader.
\end{proof}

We can also easily deduce that 
\begin{lemma}\label{lemma: F_R, F_L}
There is another adjoint pair
\begin{equation}\label{eq: F_R, F_L}
\begin{tikzcd}[arrow style=tikz,>=stealth,row sep=4em]
\cW(J_{G_\flat})\arrow[rr, shift left=.4ex, "F^L"]
  && \cW(J_{G})\ar[ll, shift left=.4ex, "F_R"]
\end{tikzcd},
\end{equation}
where $F_R$ now serves as the left adjoint. 
\end{lemma}
\begin{proof}
The co-unit map on $\Sigma_I\in \cW(J_{G_\flat})$ 
\begin{align*}
\Sigma_I^{\oplus \cZ(G^\der)}\cong F_RF^L(\Sigma_I)\longrightarrow \Sigma_I
\end{align*}
is given by $(id_{\Sigma_I}, \cdots, id_{\Sigma_I})$. For any $x\in \Sigma_I^t\cap \Sigma_I$ corresponding to a generator of $CF(\Sigma_{I}^t, \Sigma_{I})$, we have a strictly commutative diagram
\begin{align*}
\xymatrix{F_RF^L(\Sigma_I)\cong \Sigma_I^{\oplus \cZ(G^\der)}\ar[r]\ar[d]_{(x^{\oplus \cZ(G^\der)})\circ\sigma_x}&\Sigma_I\ar[d]^{x}\\
F_RF^L(\Sigma_I)\cong \Sigma_I^{\oplus \cZ(G^\der)}\ar[r]&\Sigma_I
},
\end{align*}
where $\sigma_x$ and $x^{\oplus \cZ(G^\der)}$ are as in the proof of Lemma \ref{lemma: F_L, F_R}. 

The unit map on $\Sigma_{z_j}\in \cW(J_G)$ 
\begin{align*}
\Sigma_{z_j}\longrightarrow F^LF_R(\Sigma_{z_j})\cong \bigoplus\limits_{z_i\in \cZ(G^\der)}\Sigma_{z_i}
\end{align*}
is the embedding into the component of $\Sigma_{z_j}$. For any generator $x\in CF(\Sigma_{z_j}^t, \Sigma_{z_k})$, we have a strictly commutative diagram
\begin{align*}
\xymatrix{\Sigma_{z_j}\ar[r]\ar[d]_{x}& F^LF_R(\Sigma_{z_j})\cong \bigoplus\limits_{z_i\in \cZ(G^\der)}\Sigma_{z_i}\ar[d]^{\sum\limits_{u\in \cZ(G^\der)}u\cdot x}\\
\Sigma_{z_k}\ar[r]&F^LF_R(\Sigma_{z_k})\cong \bigoplus\limits_{z_i\in \cZ(G^\der)}\Sigma_{z_i}
}. 
\end{align*}

The identities 
\begin{align*}
&(F^L\overset{(unit)\circ F^L}{\longrightarrow} F^LF_RF^L\overset{F^L\circ (co\text{-}unit)}{\longrightarrow} F^L)\simeq id_{F^L}\\
&(F_R\overset{F_R\circ (unit)}{\longrightarrow} F_RF^LF_R\overset{(co\text{-}unit)\circ F_R}{\longrightarrow} F_R)\simeq id_{F_R}
\end{align*}
can be directly checked on the generators. We leave the details to the interested reader.

\end{proof}

Consider the diagram
\begin{equation}\label{diagram: F_L, co-res, w_0}
\begin{tikzcd}[arrow style=tikz,>=stealth,row sep=4em]
\cW(J_G) \arrow[rr, shift left=.4ex, "res"]\arrow[d, shift left=.4ex, "F_R"]
  && \cW(\cB_{w_0})\simeq \cW(T^*T)\ar[ll, shift left=.4ex, "co\text{-}res"]\ar[d, shift left=.4ex, "F_{R,w_0}"]\\
\cW(J_{G_\flat})\arrow[rr, shift left=.4ex, "res_\flat"]\arrow[u, shift left=.4ex, "F^L"]&& \cW(\cB_{w_0,\flat})\simeq \cW(T^*T_{\flat})\ar[ll, shift left=.4ex, "co\text{-}res_\flat"]\ar[u, shift left=.4ex, "F^L_{w_0}"]
\end{tikzcd}, 
\end{equation}
where $F_{w_0}^L$ and $F_{R,w_0}$ are defined in the same way as $F^L$ and $F_R$,  $\cB_{w_0}$ and $\cB_{w_0,\flat}$ are the open Bruhat ``cells" in $J_G$ and $J_{G_\flat}$, respectively, viewed as subsectors, and $res_\flat$ and $co\text{-}res_\flat$ are the restriction and co-restriction functors for the subsector inclusion in $J_{G_\flat}$. 

\begin{lemma}\label{lemma: comm, functors}
We have canonical isomorphisms of functors
\begin{itemize}
\item[(i)]
\begin{align}
\label{eq: co-res, F_L}&co\text{-}res\circ F_{w_0}^L\simeq  F^L\circ co\text{-}res_\flat\\
\label{eq: res, F_R}&F_{R,w_0}\circ res\simeq res_\flat\circ F_R.
\end{align}
\item[(ii)]
\begin{align}
\label{eq: co-res, F_R} &F_R\circ co\text{-}res\simeq  co\text{-}res_\flat \circ F_{R, w_0}\\
\label{eq: res, F_L} & res\circ F^L\simeq  F_{w_0}^L\circ res_\flat.
\end{align}
\end{itemize}
In other words, the four squares in (\ref{diagram: F_L, co-res, w_0}) with initial and terminal vertices lying on any diagonal are all commutative. 
\end{lemma}

\begin{proof}
(i) Here the first one (\ref{eq: co-res, F_L}) follows from (1) the compatibility of cofinality of wrapping in $J_G$ and $J_{G_\flat}$ through $\fq^{-1}$; (2) the definition of $co\text{-}res$ and $co\text{-}res_{\flat}$ as ``wrapping more". To be more precise, following the notations in \cite[Section 3.5, 3.6]{GPS1}, let $\cI=\{T^*_IT_{\flat}\}$ and $\cI'=\{\Sigma_I, T^*_IT_\flat\}$ be collections of Lagrangian branes in $J_{G_\flat}$.  
let $\scrO$ (resp. $\scrO'$) be the $A_\infty$-category of Lagrangians for $\cB_{w_0,\flat}$ (resp. $J_{G_\flat}$) associated to $\cI$ (resp. $\cI'$). 
Then $\scrO\hookrightarrow \scrO'$ is an inclusion of a full subcategory\footnote{This is because by Definition 3.35 in \emph{loc. cit.}, the objects in $\scrO$ is a subset of $\scrO'$, and the morphisms between objects in $\scrO$ calculated in the subsector $\cB_{w_0,\flat}$ is the same as those calculated in $J_{G_\flat}$, which is due to that holomorphic discs with boundaries in Lagrangians in $\scrO$ do not cross the sector boundary of $\cB_{w_0,\flat}$.}. Let $\fq^{-1}(\scrO)$ (resp. $\fq^{-1}(\scrO')$) be the $A_\infty$-category from taking the inverse image of every element (together with each of their connected components) in $\scrO$ (resp. $\scrO'$) through $\fq$. Let $C$ (resp. $C'$) be the set of all continuation elements in $\scrO$ (resp. $\scrO'$), then $\fq^{-1}(C)$ and $\fq^{-1}(C')$ give the 
set of all continuation elements in $\fq^{-1}(\scrO)$ and $\fq^{-1}(\scrO')$, respectively. There are the natural inclusions $C\hookrightarrow C'$ and $\fq^{-1}(C)\hookrightarrow \fq^{-1}(C')$. The commutative diagram of $A_\infty$-categories
\begin{align*}
\xymatrix{\fq^{-1}(\scrO')&\fq^{-1}(\scrO)\ar@{_{(}->}[l]\\
\scrO'\ar[u]^{F^L} & \scrO\ar@{_{(}->}[l]\ar[u]_{F^L_{w_0}=F^L|_{\scrO}}
}
\end{align*}
induces the commutative diagram on localizations
\begin{align*}
\xymatrix{\fq^{-1}(\scrO')_{(\fq^{-1}(C'))^{-1}}&\fq^{-1}(\scrO)_{(\fq^{-1}(C))^{-1}}\ar[l]\\
\scrO'_{(C')^{-1}}\ar[u]^{F^L} & \scrO_{C^{-1}}\ar[l]\ar[u]_{F^L_{w_0}}
}
\end{align*}
which gives  (\ref{eq: co-res, F_L}). 

The second isomorphism of functors (\ref{eq: res, F_R}) is from taking the right adjoint on both sides of  the first one. 

(ii) From (\ref{eq: co-res, F_L}) and adjunction, we get a morphism of functors
\begin{align}\label{eq: functor, adjuntion}
&co\text{-}res_\flat\circ F_{R, w_0}\longrightarrow F_R\circ F^L\circ co\text{-}res_\flat\circ F_{R, w_0}\simeq F_R\circ co\text{-}res\circ F^L_{w_0}\circ F_{R,w_0}\\
\nonumber&\longrightarrow F_R\circ co\text{-}res,
\end{align}
where the first map is from the unit map for $(F^L, F_R)$ and the last map is from the co-unit map for $(F^L_{w_0}, F_{R, w_0})$. To confirm (\ref{eq: co-res, F_R}), we just need to show that (\ref{eq: functor, adjuntion}) on the generator $T^*_IT\in \cW(\cB_{w_0})$ is an isomorphism. This is straightforward, and we leave the details to the interested reader. 

The last isomorphism of functors (\ref{eq: res, F_L}) follows from taking the right adjoint on both sides of (\ref{eq: co-res, F_R}), using Lemma \ref{lemma: F_R, F_L}. 

\end{proof}

\begin{thm}\label{thm: HMS for reductive}
For any reductive $G\cong G^\der\underset{\cZ(G^\der)}{\times} \cZ(G)$, we have 
\begin{align}\label{eq: thm HMS reductive}
\cW(J_G)\simeq \Coh((T/\cZ(G^\der))^\vee\sslash W)^{\cZ(G^\der)^*}.
\end{align}
If $\cZ(G)$ is connected, then $\cZ(G^\der)^*$ acts freely on $(T/\cZ(G^\der))^\vee\sslash W$, and we have equivalently 
\begin{align*}
\cW(J_G)\simeq \Coh(T^\vee\sslash W).
\end{align*}
\end{thm}

\begin{proof}
We again use the Barr--Beck--Lurie theorem. For this, we take the cocompletion of each wrapped Fukaya category $\cW(M)$, and denote them by $\Ind \cW(M)$ using the standard notation. The functors $F^L, F_R$ (resp. $F^L_{w_0}, F_{R,w_0}$) between $\cW(J_G)$ and $\cW(J_{G_\flat})$ (resp. $\cW(\cB_{w_0})$ and $\cW(\cB_{w_0, \flat})$) uniquely extend, and Lemma \ref{lemma: F_L, F_R}, Lemma \ref{lemma: F_R, F_L} and Lemma \ref{lemma: comm, functors}
remain unchanged after the extension. 

First, we check that $F_R$ is conservative. Suppose an object $L\in \Ind\cW(J_G)$ is sent to the zero object in $\Ind\cW(J_{G_\flat})$ through $F_R$, then by adjunction 
\begin{align*}
0&\simeq \Hom_{\Ind\cW(J_{G_\flat})}(\Sigma_I, F_R(L))\simeq \Hom_{\Ind\cW(J_G)}(F^L(\Sigma_I), L)\\
&\simeq \Hom_{\Ind\cW(J_G)}(\bigoplus\limits_{z\in \cZ(G^\der)}\Sigma_z, L),
\end{align*}
so $L\simeq 0$. Similarly, one gets that $F_{R, w_0}, res_{\flat}$ and $res$ are also conservative.   
 Second, by definition, all functors mentioned above preserve filtered colimits, so in particular preserve geometric realizations.

Using Lemma \ref{lemma: comm, functors}, we get a commuting pair of monads $res_\flat co\text{-}res_\flat$ and $F_{R,w_0}F_{w_0}^L$ on $\cW(\cB_{w_0,\flat})$. 
Now we apply the Barr--Beck--Lurie theorem and get an equivalence
\begin{align}\label{eq: Barr-Beck-Lurie}
res_\flat\circ F_R: \Ind\cW(J_G)&\overset{\sim}{\longrightarrow} (res_\flat F_RF^L co\text{-}res_\flat)-\Mod(\Ind\cW(\cB_{w_0,\flat}))\\
\nonumber&\overset{\sim}{\longrightarrow} (F_{R,w_0}F_{w_0}^L res_\flat co\text{-}res_\flat)-\Mod(\Ind\cW(\cB_{w_0,\flat})). 
\end{align}

By Theorem \ref{thm: sec G adjoint} and its proof, the monad $res_\flat co\text{-}res_\flat$ is isomorphic to $\mathsf{f}^*\mathsf{f}_*$ on $\Ind\Coh((T_\flat)^\vee)$ for 
\begin{align*}
\mathsf{f}: (T_\flat)^\vee\longrightarrow (T_\flat)^\vee\sslash W. 
\end{align*}
It is clear that the monad $F_{R,w_0}F_{w_0}^L$ is isomorphic to $\varpi^*\varpi_*$ on $\Ind\Coh((T_\flat)^\vee)$ for the quotient map  
\begin{align*}
\varpi: (T_\flat)^\vee\longrightarrow T^\vee\cong (T_\flat)^\vee/\cZ(G^\der)^*.
\end{align*}
Therefore, the last line of (\ref{eq: Barr-Beck-Lurie}) is equivalent to $\Ind\Coh(((T_\flat)^\vee\sslash W)/\cZ(G^\der)^*)= \Ind\Coh((T_\flat)^\vee\sslash W))^{\cZ(G^\der)^*}$ through the Cartesian diagram
\begin{align*}
\xymatrix{(T_\flat)^\vee\ar[r]^{\varpi}\ar[d]_{\mathsf{f}}&(T_\flat)^\vee/\cZ(G^\der)^*\ar[d]\\
(T_\flat)^\vee\sslash W\ar[r] &((T_\flat)^\vee\sslash W)/\cZ(G^\der)^*
}. 
\end{align*}
Lastly, taking compact objects, we get (\ref{eq: thm HMS reductive}) as desired. 
\end{proof}

\subsection{Induction pattern}\label{subsec: induction}

For any $S\subset \Pi$,  Theorem \ref{thm: HMS for reductive} tells us that  
\begin{align*}
\cW(J_{L_S})\simeq \Coh((T/\cZ(L_S^\der))^\vee\sslash W_S)^{\cZ(L_S^\der)^*}.
\end{align*}
In the following, let $T_{\flat, S}:=T/\cZ(L_S^\der)$. 
The restriction and co-restriction functors for the inclusion of Liouville sectors $\overline{J}_{L_\emptyset}\subset \overline{J}_{L_S}$ (cf. Subsection \ref{ss: sectorinclusion}) correspond to 
\begin{equation}\label{eq: res, co-res, L_S}
\begin{tikzcd}[arrow style=tikz,>=stealth,row sep=4em]
\Coh(T_{\flat, S}^\vee\sslash W_S)^{\cZ(L_S^\der)^*}\arrow[rr, shift left=.4ex, "res\simeq p_{\emptyset, S}^*"]
  &&\Coh(T^\vee)\ar[ll, shift left=.4ex, "co\text{-}res\simeq (p_{\emptyset, S})_*"]
\end{tikzcd},
\end{equation}
where 
\begin{align*}
p_{\emptyset, S}: T^\vee=T^\vee_{\flat, S}/\cZ(L_S^\der)^*\longrightarrow (T^\vee_{\flat, S}\sslash W_S)/\cZ(L_S^\der)^*
\end{align*}
is the natural projection. 

\begin{prop}\label{prop: JLS'JLS}
Let $G$ be any complex reductive group. For any $S\subset S'\subset \Pi$, we have the restriction and co-restriction functors between $\cW(J_{L_{S'}})$ and $\cW(J_{L_S})$ given by 
\begin{equation}\label{eq: res, co-res, S, S'}
\begin{tikzcd}[arrow style=tikz,>=stealth,row sep=4em]
\cW(J_{L_{S'}})\simeq \Coh(T^\vee_{\flat, S'}\sslash W_{S'})^{\cZ(L_{S'}^\der)^*}\arrow[rr, shift left=.4ex, "res\simeq p_{S, S'}^*"]
  &&\cW(J_{L_S})\simeq \Coh(T_{\flat, S}^\vee\sslash W_{S})^{\cZ(L_S^\der)^*}\ar[ll, shift left=.4ex, "co\text{-}res\simeq (p_{S, S'})_*"]
\end{tikzcd},
\end{equation}
where 
\begin{align*}
p_{S, S'}: (T_{\flat, S}^\vee\sslash W_{S})/\cZ(L_S^\der)^*   \longrightarrow (T_{\flat, S'}^\vee\sslash W_{S'})/\cZ(L_{S'}^\der)^* 
\end{align*}
is the natural projection. 

\end{prop}

\begin{proof}
First, we assume $G$ is of adjoint type. For any $S\subset S'$, we have the diagram 
\begin{equation}\label{eq: res, co-res, S, S', empty}
\begin{tikzcd}[arrow style=tikz,>=stealth,row sep=4em]
&\cW(J_{L_\emptyset})\simeq \Coh(T^\vee)\arrow[dl, shift left=.4ex, "co\text{-}res\simeq (p_{\emptyset, S'})_*"]\arrow[dr, shift left=.4ex, "co\text{-}res\simeq (p_{\emptyset, S})_*"]&\\
\cW(J_{L_{S'}})\simeq \Coh(T^\vee\sslash W_{S'})\arrow[rr, shift left=.4ex, "res_{S,S'}"]\arrow[ur, shift left=.4ex, "res\cong p_{\emptyset, S'}^*"]
  &&\cW(J_{L_S})\simeq \Coh(T^\vee\sslash W_{S})\ar[ll, shift left=.4ex, "co\text{-}res_{S,S'}"]\arrow[ul, shift left=.4ex, "res\cong p_{\emptyset, S}^*"]
\end{tikzcd},
\end{equation}
with the co-restriction functors forming a commutative triangle (cf. Proposition \ref{prop: sector_cover}), and the restriction functors forming another commutative triangle from adjunction. It implies that $co\text{-}res_{S, S'}$ takes any skyscraper sheaf on $T^\vee\sslash W_S$ to a skyscraper sheaf on $T^\vee\sslash W_{S'}$. In particular, it can be identified with the pushforward functor for a morphism of schemes 
\begin{align*}
f_{S,S'}: T^\vee\sslash W_S\longrightarrow T^\vee\sslash W_{S'}.
\end{align*}
It is clear that  $f_{S, S'}=p_{S,S'}$ since the $W_S$-orbit of any closed point in $T^\vee$ is sent to the $W_{S'}$-orbit of the same point.

For a general reductive $G$, it follows from first geting the result for $G_\flat$ by the Kunneth formula on wrapped Fukaya categories, and then using the monad $F_RF^L$ on $\cW(J_{L_S}/\cZ(G^\der))$ for each $S\subset \Pi$ and 
its compatibility with restriction and co-restriction functors for inclusions of subsectors for different $S$, as in Lemma \ref{lemma: comm, functors} (whose argument directly generalizes to the case that $\cB_{w_0}$ is replaced by any subsector associated with $S\subset \Pi$). 
\end{proof}

Let $\DMStk_\bC^{\text{prop}}$ be the ordinary (2,1)-category of Deligne-Mumford stacks over $\bC$ with proper morphisms. Then we have the functor $\Coh_*: \DMStk_\bC^{\text{prop}}\rightarrow \OneCat_\bC^L$ that takes each stack $X$ to $\Coh(X)$ and a proper morphism $X\rightarrow Y$ to the pushforward functor on coherent sheaves. 

\begin{cor}
Under the canonical equivalences $\cW(J_{L_S})\simeq \Coh(T_{\flat, S}^\vee\sslash W_S)^{\cZ(L_S^\der)^*}$, $S\subset\Pi$, the functor (\ref{eq: functor, sector-cover}) is canonically equivalent to first taking the functor
\begin{align*}
(\{S\subset \Pi\}, \subset)&\longrightarrow \DMStk_\bC^{\mathrm{prop}}\\
 S&\mapsto (T_{\flat, S}^\vee\sslash W_S)/\cZ(L_S^\der)^*,
\end{align*}
that sends each inclusion to the natural projection on stacks, and then taking $\Coh_*$. 
\end{cor}
\begin{proof}
The canonical equivalences follow from 
\begin{itemize}
\item the proof of  Theorem \ref{thm: sec G adjoint} in which the morphism $\hat{\mathsf{f}}$ is uniquely determined; 

\item the maximal torus $T$ is canonically identified with the abstract maximal torus, through the inclusion $T\subset B$ to the fixed Borel $B$. 
\end{itemize}
The rest is an immediate consequence of Proposition \ref{prop: JLS'JLS}.
\end{proof}


\begin{thebibliography}{99}

\bibitem[Aur]{Auroux} D. Auroux, ``A beginner's introduction to Fukaya categories," Contact and Symplectic Topology, Bolyai Society Mathematical Studies 26 (2014), 85--136.


\bibitem[Bal1]{Balibanu1} A. Balibanu, ``The Peterson variety and the wonderful compactification", Represent. Theory 21 (2017), 132--150.

\bibitem[Bal2]{Balibanu} A. Balibanu, ``The partial compactification of the universal centralizer", arXiv:1710.06327.

\bibitem[Bie]{Bielawski} R. Bielawski, ``HyperK$\ddot{\text{a}}$hler structures and group actions," J. London Math. Soc. (2) 55 (1997), no. 2, 400--414.


\bibitem[BFM]{BFM} R. Bezrukavnikov, M. Finkelberg, I. Mirkovic, ``Equivariant ($K$-)homology of affine Grassmannian and Toda lattice,"  Compos. Math. 141 (2005), no. 3, 746--768.

\bibitem[BFN]{BFN} A. Braverman, M. Finkelberg, H. Nakajima, ``Towards a mathematical definition of Coulomb branches of $3$-dimensional $N=4$ gauge theories, II," Adv. Theor. Math. Phys. 22 (2018), no. 5, 1071--1147.

\bibitem[BZG]{BZG} D. Ben-Zvi, S. Gunningham, ``Symmetries of categorical representations and the quantum Ng$\hat{\text{o}}$ action," arXiv:1712.01963.

\bibitem[BZGN]{BZGN} D. Ben-Zvi, S. Gunningham, D. Nadler,  ``The character field theory and homology of character varieties," MRL 26 (2019) no. 5, 1313--1342.

\bibitem[BZN1]{BZN1} D. Ben-Zvi, D. Nadler, ``The Character Theory of a Complex Group," arXiv:0904.1247.

\bibitem[BZN2]{BZN2} D. Ben-Zvi, D. Nadler, ``Betti Geometric Langlands," Algebraic geometry: Salt Lake City 2015, AMS Proceedings of symposia in pure mathematics, vol. 97, part 2.

\bibitem[CDGG]{CDGG} B. Chantraine, G. Dimitroglou Rizell, P. Ghiggini, R. Golovko, ``Geometric generation of the wrapped Fukaya category of Weinstein manifolds and sectors." arXiv:1712.09126.

\bibitem[Che]{Chen} T-H. Chen, ``Non-linear Fourier transforms and the Braverman-Kazhdan conjecture," arXiv:1609.03221.

\bibitem[ChGi]{CG} N. Chriss and V. Ginzburg, ``Representation Theory and Complex Geometry", Birkh$\ddot{\text{a}}$user Boston, Inc., Boston, MA (1997).


\bibitem[CiEl]{CE} K. Cieliebak and Y. Eliashberg, `` From Stein to Weinstein and Back," Symplectic geometry of affine complex manifolds. AMS Colloquium Publications, 59 (2012).

\bibitem[Eli]{El} Y. Eliashberg, ``Weinstein manifolds revisited," arXiv: 1707.03442


\bibitem[GPS1]{GPS1} S. Ganatra, J. Pardon and V. Shende, ``Covariantly functorial wrapped Floer theory on Liouville sectors," Publ. Math. Inst. Hautes $\acute{\text{E}}$tudes Sci. 131 (2020), 73--200.

\bibitem[GPS2]{GPS2} S. Ganatra, J. Pardon and V. Shende, ``Sectorial descent for wrapped Fukaya categories," arXiv:1809.03427.

\bibitem[Gin]{Ginzburg} V. Ginzburg, ``Nil Hecke algebras and Whittaker $D$-modules," Lie groups, geometry, and representation theory, 137--184, Progr. Math., 326, Birkh$\ddot{\text{a}}$user/Springer, Cham, 2018.


\bibitem[Jin]{Jin1} X. Jin, ``Holomorphic Lagrangian branes correspond to perverse sheaves," Geom. Topol. 19 (2015), no. 3, 1685--1735.


 



\bibitem[Kos1]{Kostant} B. Kostant, ``On Whittaker vectors and representation theory," Invent. Math. 48 (1978), 101--184.

\bibitem[Kos2]{Kostant2} B. Kostant, ``The solution to a generalized Toda lattice and representation theory," Adv. in Math. 34, 3 (1979), 195--338.

\bibitem[Lon]{Lonergan} G. Lonergan, ``A Fourier transform for the quantum Toda lattice," Sel. Math. New Ser. 24, 4577--4615 (2018).

\bibitem[Lur]{Lurie} J. Lurie, ``Higher algebra", https://www.math.ias.edu/~lurie/papers/HA.pdf. 


\bibitem[Lus]{Lusztig} G. Lusztig, "Coxeter orbits and eigenspaces of Frobenius," Invent. Math. 38 (1976) pages 101--159


\bibitem[McSa]{McSa} D. McDuff, D. Salamon, ``Introduction to symplectic topology (third edition)," Oxford Mathematical Monographs, Oxford University Press, New York (2017).

\bibitem[Nad1]{Nadler} D. Nadler, ``Arboreal singularities", Geometry \& Topology 21 (2017) 1231--1274.

\bibitem[Nad2]{Nadler2} D. Nadler, ``Wrapped microlocal sheaves on pairs of pants", arXiv:1604.00114.


\bibitem[NaSh]{NaSh} D. Nadler, V. Shende, ``Sheaf quantization in Weinstein symplectic manifolds," arXiv:2007.10154.

\bibitem[NaZa]{NaZa} D. Nadler, E. Zaslow, ``Constructible sheaves and the Fukaya category," J. Amer. Math. Soc. 22 (2009), 233--286.

\bibitem[Pas]{Pascaleff} J. Pascaleff, ``Poisson geometry, monoidal Fukaya categories, and commutative Floer cohomology rings," arXiv:1803.07676.


\bibitem[Poz]{Pozniak} M. Po$\acute{z}$niak, ``Floer homology, Novikov rings and clean intersections," Ph.D. thesis, University of Warwick, 1994.

\bibitem[Sch]{Schmaschke} F. Schm$\ddot{\text{a}}$schke, ``Floer homology of Lagrangians in clean intersection," arXiv:1606.05327.

\bibitem[Sei1]{Seidel1} P. Seidel, ``Fukaya categories and Picard-Lefschetz Theory," Zurich Lectures in Advanced Mathematics. European Mathematical Society (EMS), Zurich (2008).

\bibitem[Sei2]{Seidel2} P. Seidel, ``Lagrangian two-spheres can be symplectically knotted," J. Differential Geom. 52 (1999), no. 1, 145--171.

\bibitem[SP]{stacks-project} The Stacks project authors, ``The Stacks project", \url{https://stacks.math.columbia.edu}, 2022.

\bibitem[Ste]{Steinberg} R. Steinberg, ``On a theorem of Pittie," Topology, 14, 173--177 (1975).


\bibitem[Syl]{Sylvan} Z. Sylvan, ``On partially wrapped Fukaya categories," Journal of Topology, Volume 12, Issue 2, 372--441 (2019).

\bibitem[Tel]{Teleman} C. Teleman, ``Gauge theory and mirror symmetry." Proceedings of the International Congress of Mathematicians--Seoul 2014. Vol. II, 1309--1332.




\end{thebibliography}
\end{document}